%% file: bounded_cohomology_book.tex
\documentclass [11pt,makeidx]{amsbook}

\usepackage{comment}
\usepackage{amssymb} \usepackage{amsfonts} \usepackage{amsmath}
\usepackage{amsthm} \usepackage{epsfig} 
\usepackage[cmtip,arrow]{xy}
\usepackage{amsmath,amssymb,enumerate,pb-diagram,pb-xy}%,pdfsync}

\usepackage[applemac]{inputenc}
\input xy
\xyoption{all}

\usepackage{caption}

\newcommand{\bb}{\partial}
\DeclareMathOperator{\vol}{Vol}
\DeclareMathOperator{\Hom}{Hom}

\newcommand{\cb}{C_b}
\newcommand{\bqn}{\begin{equation*}}
\newcommand{\eqn}{\end{equation*}}
\newcommand{\bq}{\begin{equation}}
\newcommand{\eq}{\end{equation}}
\newcommand{\ba}{\begin{aligned}}
\newcommand{\ea}{\end{aligned}}
\newcommand{\be}{\begin{enumerate}}
\newcommand{\ee}{\end{enumerate}}
\newcommand{\Ff}{{\mathcal F}}
\newcommand{\sgn}{{\Sigma_{g,n}}}

\newcommand{\mathno}{\ensuremath{\overline{\matH^n}}}

\newcommand{\todo}[1]{\vspace{5mm}\par \noindent
\framebox{\begin{minipage}[c]{0.95 \textwidth} \tt #1\end{minipage}} \vspace{5mm} \par}

 \newcommand{\inparens}[2][flex]{\csname #1l\endcsname(#2%
                                 \csname #1r\endcsname)\mathclose{}}
 \newcommand{\inangles}[2][flex]{\csname #1l\endcsname\langle#2%
                                 \csname #1r\endcsname\rangle\mathclose{}} 
 \newcommand{\innorm}[2][flex]{\csname #1l\endcsname|#2%
                                 \csname #1r\endcsname|\mathclose{}}
 \newcommand{\indnorm}[2][flex]{\csname #1l\endcsname\|#2%
                                 \csname #1r\endcsname\|\mathclose{}}
 \newcommand{\indnorml}[4][flex]{\csname #1l\endcsname\|#2%
                                 \csname #1r\endcsname\|_{#3}^{#4}\mathclose{}}

\newcommand{\sv}[2][flex]{\indnorm[#1]{#2}}%
\newcommand{\isv}[2][norm]{\indnorml[#1]{#2}{\mathbb{Z}}{}}
\newcommand{\pfcl}[2][flex]{\csname #1l\endcsname[#2%
                            \csname #1r\endcsname]}
\newcommand{\ifsv}[2][norm]{\csname #1l\endcsname\bracevert\!#2\!%
                            \csname #1r\endcsname\bracevert}
\newcommand{\stisv}[2][flex]{\indnorml[#1]{#2}{\mathbb{Z}}{\infty}}

\newcommand{\str} {\ensuremath {\mathrm{str}}}
\newcommand{\strtil} {\ensuremath {\widetilde{\mathrm{str}}}}
\newcommand{\ttob}[1]{\stackrel{#1}{\longleftarrow}}
\newcommand{\tto}[1]{\stackrel{#1}{\longrightarrow}}
\newtheorem{lemma}{Lemma}[chapter]
\newtheorem{teo}[lemma]{Theorem}
\newtheorem{thm}[lemma]{Theorem}
\newtheorem{prop}[lemma]{Proposition}
\newtheorem{cor}[lemma]{Corollary} 
\newtheorem{conj}[lemma]{Conjecture}

\theoremstyle{definition}
\newtheorem{defn}[lemma]{Definition}

\newtheorem{quest}[lemma]{Question}

\newtheorem{rem}[lemma]{Remark}

\newcommand{\wdtX}{\widetilde{X}}
\newcommand{\bc}{\ensuremath{\overline{C}}}

\newcommand{\matN}{\ensuremath {\mathbb{N}}}
\newcommand{\R} {\ensuremath {\mathbb{R}}}
\newcommand{\Z} {\ensuremath {\mathbb{Z}}}

\newcommand{\matZ} {\ensuremath {\mathbb{Z}}}

\newcommand{\matH} {\ensuremath {\mathbb{H}}}

\newcommand{\hl}{\ensuremath{H^{\ell_1}}}
\newcommand{\cl}{\ensuremath{C^{\ell_1}}}
\newcommand{\eu}{\ensuremath{\mathrm{eu}}}
\newcommand{\eul}{\ensuremath{\mathrm{eul}}}
\newcommand{\E}{\ensuremath{\mathrm{eul}}}

\newcommand{\cst} {\ensuremath{C^\bullet}}
\newcommand{\climst} {\ensuremath{{C}_b^\bullet}}
\newcommand{\ccst} {\ensuremath{C_c^\bullet}}
\newcommand{\cclimst} {\ensuremath{{C}_{b,c}^\bullet}}

\newcommand{\h} {\ensuremath{H}}

\newcommand{\deltab}{\ensuremath{\overline{\delta}}}
\newcommand{\alt}{\ensuremath{\mathrm{alt}}}

\newcommand{\omeon}{\ensuremath{{\rm Homeo}^+(S^n)}}
\newcommand{\omeo}{\ensuremath{{\rm Homeo}^+(S^1)}}
\newcommand{\omeot}{\widetilde{\ensuremath{{\rm Homeo}_+(S^1)}}}
\DeclareMathOperator{\rot}{rot}
\DeclareMathOperator{\rott}{\widetilde{rot}}
\DeclareMathOperator{\cone}{Cone}

\newcommand{\clim} {\ensuremath{C_b}}
\newcommand{\cc} {\ensuremath{{C}_c}}
\newcommand{\cclim} {\ensuremath{C}_{b,c}}

\newcommand{\vare} {\ensuremath{\varepsilon}}

\newcommand{\hst} {\ensuremath{H^\bullet}}
\newcommand{\hcst} {\ensuremath{H_c^\bullet}}
\newcommand{\hlimst} {\ensuremath{{H}_b^\bullet}}
\newcommand{\hclimst}{\ensuremath{{H}_{b,c}^\bullet}}

\newcommand{\G}{\ensuremath {\Gamma}}
\newcommand{\calC} {\ensuremath {\mathcal{C}}}

\newcommand{\calS} {\ensuremath {\mathcal{S}}}

\newcommand{\hx} {\ensuremath {h_{\widetilde{M}}}}
\newcommand{\xtil} {\ensuremath {\widetilde{M}}}
\newcommand{\res} {\ensuremath {{\rm res}}}

\newcommand{\tr} {\ensuremath {{\rm trans}}}

\newcommand{\GL}{{\rm GL}}
\newcommand{\Isom}{\ensuremath{{\rm Isom}}}

\author{Roberto Frigerio}

\address{Dipartimento di Matematica \\
Universit\`a di Pisa \\
Largo B.~Pontecorvo 5 \\
56127 Pisa, Italy}

\email{frigerio@dm.unipi.it}

\title{Bounded cohomology of discrete groups}

\subjclass[2000]{}

\keywords{}

\thanks{}

\makeindex

\begin{document}

\frontmatter

%\usetikzlibrary{arrows}
\begin{abstract}
{Bounded cohomology of groups was first defined by Johnson and Trauber during the seventies in the context of Banach algebras. As an independent and very
active research field, however, bounded cohomology started to develop in 1982, thanks to the pioneering paper ``Volume and Bounded Cohomology''
by M. Gromov, where the definition of bounded cohomology was extended to deal also with topological spaces.

The aim of this monograph is to provide an introduction to bounded cohomology of discrete groups and of topological spaces. We also describe some applications
of the theory to related active research fields (that have been chosen according to the taste and the knowledge of the author). 
The book is essentially self-contained. Even if a few statements do not appear elsewhere and some proofs are slighlty different from the ones already available in the literature, the monograph does
not contain original results. 

In the first part of the book we settle the fundamental definitions of the theory, and
we prove some (by now classical) results on low-dimensional bounded cohomology and on bounded cohomology of topological spaces. Then
we describe how bounded cohomology has proved useful in the study of the simplicial volume of manifolds, for the classification of circle actions, 
for the definition and the description of maximal representations of surface groups, and in the study of higher rank flat vector bundles (also in relation with the Chern conjecture).}
\end{abstract}
\maketitle

\tableofcontents

\printindex

\mainmatter 

\chapter*{Introduction}
Bounded cohomology of groups was first defined by Johnson~\cite{Johnson} and Trauber during the seventies in the context of Banach algebras. As an independent and very
active research field, however, bounded cohomology started to develop in 1982, thanks to the pioneering paper ``Volume and Bounded Cohomology''
by M.~Gromov~\cite{Gromov}, where the definition of bounded cohomology was extended to deal also with topological spaces.

Let $C^\bullet(M,\R)$ denote the complex of real singular cochains with values in the topological space $M$. A cochain
$\varphi\in C^n(M,\R)$ is \emph{bounded} if it takes uniformly bounded values on the set of singular $n$-simplices. Bounded cochains
provide a subcomplex $C^\bullet_b(M,\R)$ of singular cochains, and the bounded cohomology $H^\bullet_b(M,\R)$ of $M$ (with trivial real coefficients) is just the (co)homology
of the complex $C^\bullet_b(M,\R)$. An analogous definition of boundedness applies to group cochains with (trivial) real coefficients,
and the bounded cohomology $H^\bullet_b(\G,\R)$ of a group $\G$ (with trivial real coefficients)
is the (co)homology of the complex $C^\bullet_b(\G,\R)$ of the bounded cochains on $\G$. A classical result which dates back to the forties 
ensures that the singular cohomology of an aspherical CW-complex is canonically isomorphic to the cohomology of its fundamental group.
In the context of bounded cohomology a stronger result holds: the bounded cohomology of a countable CW-complex is canonically isomorphic to the bounded
cohomology of its fundamental group, even without any assumption on the asphericity of the space~\cite{Gromov,Brooks,Ivanov}. For example,
the bounded cohomology of spheres is trivial in positive degree. On the other hand, the bounded cohomology of the wedge of two circles
is infinite-dimensional in degrees $2$ and $3$, and still unknown in any degree bigger than $3$.
As we will see in this monograph, this phenomenon eventually depends on the fact
that higher homotopy groups are abelian, and 
abelian groups are invisible to bounded cohomology, while ``negatively curved'' groups, such as non-abelian free groups, tend to have 
very big bounded cohomology modules. 

The bounded cohomology of a group $\G$ can be defined with coefficients in any normed (e.g.~Banach) $\G$-module, where the norm
is needed to make sense of the notion of boundedness of cochains. Moreover, if $\G$ is a topological group, then one may restrict
to considering only \emph{continuous} bounded cochains, thus obtaining the \emph{continuous} bounded cohomology of $\G$.
In this monograph, we will often consider arbitrary normed $\G$-modules, but we will restrict our attention to
bounded cohomology of discrete groups. The reason for this choice is twofold. First, the (very powerful) theory of continuous bounded cohomology
is based on a quite sophisticated machinery, which is not needed in the case of discrete groups. Secondly, 
Monod's book~\cite{Monod} and Burger-Monod's paper~\cite{BM2} already provide an excellent introduction to continuous bounded cohomology, while to the author's knowledge no
reference is available where the fundamental properties of bounded cohomology of discrete groups are collected and proved in detail. However, we should emphasize 
that the theory of continuous bounded cohomology is essential in proving important results also in the context of discrete groups: many vanishing theorems for
the bounded cohomology 
of lattices in Lie groups can be obtained by comparing the bounded cohomology of the lattice with the continuous bounded cohomology of the ambient group.

This monograph is devoted to provide a self-contained introduction to bounded cohomology of discrete groups and topological spaces. Several 
(by now classical) applications of the theory will be described in detail, while many others will be completely omitted. 
Of course,
the choice of the topics discussed here is largely arbitrary, and based on the taste (and on the knowledge) of the author. 

Before describing the content
of each chapter, let us provide a brief overview on the relationship between bounded  cohomology and 
other research fields.

\subsection*{Geometric group theory and quasification}
The bounded cohomology of a closed manifold is strictly related
to the curvature of the metrics that the manifold can support. For example, if the closed manifold $M$ is flat or positively curved, then
the fundamental group of $M$ is amenable, and $H^n_b(M,\R)=H^n_b(\pi_1(M),\R)=0$ for every $n\geq 1$. On the other hand, if $M$ is negatively curved,
then it is well-known that the comparison map $H^\bullet_b(M,\R)\to H^\bullet(M,\R)$ induced by the inclusion $C^\bullet_b(M,\R)\to C^\bullet(M,\R)$
is surjective in every degree bigger than one. 

In fact, it turns out that the surjectivity of the comparison map is related in a very clean way to 
the notion of Gromov hyperbolicity, 
 which represents
the coarse geometric version of negative curvature. Namely,
a group $\G$ is Gromov hyperbolic if and only if the comparison map $H^n_b(\G,V)\to H^n(\G,V)$
is surjective for every $n\geq 2$ and for every Banach $\G$-module $V$~\cite{Mineyev1,Mineyev2}. 

Coarse geometry comes into play also when studying the
(non-)injectivity of the comparison map. In fact, let $EH^{n}_b(\G,V)$ denote the kernel of the comparison map in degree $n$. 
It follows by the very definitions that an  element of $H^{n}_b(\G,V)$ lies in $EH^{n}_b(\G,V)$  
if and only if any of its representatives is the coboundary of a (possibly unbounded) cocycle.
More precisely, a cochain is usually called a \emph{quasi-cocycle} if its differential is bounded, and
a quasi-cocycle is \emph{trivial} if it is the sum of a cocycle and a bounded cochain. Then $EH^{n}_b(\G,V)$
is canonically isomorphic to the space of $(n-1)$-quasi-cocycles modulo trivial $(n-1)$-quasi-cocycles. 
When $V=\R$, quasi-cocycles of degree one are usually called \emph{quasimorphisms}. There exists a large literature which concerns
the construction of non-trivial quasimorphisms in presence of (weaker and weaker versions of) negative curvature. Brooks~\cite{Brooks} first constructed
infinitely many quasimorphisms on the free group with two generators $F_2$, that were shown to define linearly indepenedent
elements in $EH^2(F_2,\R)$ by Mitsumatsu~\cite{Mitsu}. In particular, this proved that $EH^2_b(F_2,\R)$ (which coincides with $H^2_b(F_2,\R)$) is infinite-dimensional.

Quasi-cocycles are cochains which satisfy the cocycle equation only up to a finite error, and geometric group theory provides 
tools that are particularly well-suited to study notions which involve finite errors in their definition. Therefore, 
it is not surprising that Brooks' and Mitsumatsu's result has been generalized to larger and larger classes of groups, which include
now the class of non-elementary relatively hyperbolic groups, and most mapping class groups. We refer the reader to 
Section~\ref{further:2}
for a more detailed account on this issue. It is maybe worth mentioning that, even if in the cases cited above $EH^2_b(\G,\R)$ is always infinite-dimensional,
there exist lattices $\G$ in non-linear Lie groups for which $EH^2_b(\G,\R)$ is of finite non-zero dimension~\cite{Mann-Monod}.

We have mentioned the fact that $H^2_b(G,\R)$ is infinite-dimensional for negatively curved groups according to a suitable notion of negative curvature for groups.
On the other hand, bounded cohomology vanishes for ``positively curved'' (i.e.~finite) or ``flat'' (i.e.~abelian) groups. In fact, $H^n_b(\G,\R)=0$ for any $n\geq 1$ and any
$G$ belonging to the distinguished class of \emph{amenable} groups. The class of amenable groups contains all virtually solvable groups, it is closed
under quasi-isometries, and it admits a nice characterization in terms of bounded cohomology (see Section~\ref{Johnson:sec}). These
facts provide further evidence  that bounded cohomology detects coarse non-positive curvature as well as coarse non-negative curvature.

\subsection*{Simplicial volume}
The $\ell^\infty$-norm of an $n$-cochain is the supremum of the values it takes on single singular $n$-simplices (or on single $(n+1)$-tuples of
elements of the group, when dealing with group cochains rather than with singular cochains). So a cochain $\varphi$ is bounded if and only if
it has a finite $\ell^\infty$-norm, and the $\ell^\infty$-norm induces a natural quotient $\ell^\infty$-seminorm on bounded cohomology. 
The $\ell^\infty$-norm on singular cochains arises as the dual of a natural $\ell^1$-norm on singular \emph{chains}.
This $\ell^1$-norm induces an $\ell^1$-seminorm on homology. If $M$ is a closed oriented manifold, then  the \emph{simplicial volume} $\|M\|$ of $M$
is the $\ell^1$-seminorm of the real fundamental class of $M$~\cite{Gromov}. 
Even if it depends only on the homotopy type of the manifold, the simplicial volume is deeply related
to the geometric structures that a manifold can carry. 
As one of the main motivations for its definition, Gromov himself showed that the simplicial volume provides a lower bound
for the \emph{minimal volume} of a manifold, which is the infimum of the volumes of the Riemannian metrics 
that are supported by the manifold and that satisfy suitable curvature bounds.

An elementary duality result relates the simplicial volume of an $n$-dimensional manifold $M$ to the bounded cohomology module
$H^n_b(M,\R)$. For example, if $H^n_b(M,\R)=0$ then also $\|M\|=0$. In particular, the simplicial volume of simply connected manifolds
(or, more in general, of manifolds with  amenable fundamental group) is vanishing. It is worth stressing the fact that no homological proof
of this statement is available: in many cases, the fact that $\|M\|=0$ cannot be proved by exhibiting fundamental cycles with arbitrarily small norm.
Moreover, the exact value of non-vanishing simplicial volumes is known only in the following very few cases:
hyperbolic manifolds~\cite{Gromov, Thurston}, some classes of $3$-manifolds with boundary~\cite{BFP}, and the product of two surfaces~\cite{Bucher3}. 
In the last case, it is not known to the author any description of
a sequence of fundamental cycles whose $\ell^1$-norms approximate the simplicial volume. In fact, Bucher's computation of the simplicial volume of the product of surfaces
heavily relies on deep properties of bounded cohomology that have no counterpart in the context of singular homology. 
%It should be stressed here that the seminorm on bounded cohomology plays a fundamental role in this computation:
%in order to exploit duality to reduce the computation 
%of the $\ell^1$-seminorm of homology classes (whence, of the simplicial volume of a manifold) to computations in bounded cohomology, it is necessary
%to get a precise control on the $\ell^\infty$-seminorm of the bounded coclasses that are relevant for the problem.

\subsection*{Characteristic classes}
%It is very natural to investigate the boundedness of classical cohomology classes which may arise e.g.~in the obstruction theory of bundles.
A fundamental theorem by Gromov~\cite{Gromov} states that, if $G$ is an algebraic subgroup of $\GL_n(\R)$, then every characteristic class of flat $G$-bundles 
lies in the image of the comparison map (i.e.~it can be represented
by a bounded cocycle). (See~\cite{Bucher:thesis} for an alternative proof of a stronger result). Several natural questions arise from this result. First of all, one may ask whether
such characteristic classes admit a canonical representative in bounded cohomology: since the comparison map is often non-injective, this would produce more refined invariants.
Even when one is not able to find natural bounded representatives for a characteristic class, the seminorm on bounded cohomology (which induces a seminorm on the image of the comparison 
map by taking the infimum over the bounded representatives)
can be used to produce numerical invariants, or to provide useful estimates.

In this context, the most famous example is certainly represented by the Euler class. To every oriented circle bundle there is associated its Euler class, which 
arises as an obstruction to the existence of a section, and
completely
classifies the topological type of the bundle. When restricting to \emph{flat} circle bundles, a \emph{bounded} Euler class can be defined, which now depends on the flat structure
 of the bundle (and, in fact, classifies the isomorphism type  of flat bundles with minimal holonomy~\cite{Ghys0, Ghys2}). Moreover, the seminorm of the bounded Euler 
class is equal to $1/2$. Now the Euler number of a circle bundle over a surface is obtained by evaluating the Euler class on the fundamental class of the surface. As a 
consequence, the Euler number of any flat circle bundle over a surface 
is bounded above by the product of $1/2$ times the simplicial volume of the surface.  This yields the celebrated \emph{Milnor-Wood inequalities}~\cite{Milnor,Wood}, which provided the firstnot so economic
explicit and easily computable obstructions for a circle bundle to admit a flat structure. Of course, Milnor's and Wood's original proofs did not explicitly use bounded cohomology,
 which was not defined yet. However, their arguments introduced ideas and techniques which are now fundamental in the theory of quasimorphisms, and they were implicitly 
 based on the construction of a bounded representative for the Euler class.
 
These results also extend to higher dimensions. It was proved by Sullivan~\cite{Sullivan} that the $\ell^\infty$-seminorm of the (simplicial) 
Euler class of an oriented flat vector bundle is bounded above by $1$. A clever trick due to Smillie allowed to sharpen this bound from $1$ to $2^{-n}$, where $n$ is the rank of the bundle. 
Then, Ivanov and Turaev~\cite{IvTu}
gave a different proof of Smillie's result, also constructing a natural bounded Euler class in every dimension. A very clean characterization of this bounded cohomology class,
as well as the proof that its seminorm is equal to $2^{-n}$, have recently been provided by Bucher and Monod in~\cite{BuMo}.

 \subsection*{Actions on the circle}
 If $X$ is a topological space with fundamental group $\G$, then
 any orientation-preserving topological action of $\G$ on the circle gives rise to
 a flat circle bundle over $X$. Therefore, we can associate to every such action a bounded Euler class.
 It is a fundamental result of Ghys~\cite{Ghys0,Ghys2} that the bounded Euler class 
 encodes the most essential features of the dynamics of an action. 
 For example, an action admits a global fixed point if and only if its bounded Euler class vanishes. Furthermore, in the almost opposite case of \emph{minimal} actions (i.e.~of actions
 every orbit of which is dense), the bounded Euler class provides a complete conjugacy invariant:
% In fact, a fundamental result by Ghys~\cite{Ghys0,Ghys2} ensures that
 two minimal circle actions share the same bounded Euler class if and only if they are topologically conjugate~\cite{Ghys0,Ghys2}. These results establish a deep connection
 between bounded cohomology and a fundamental problem in one-dimensional dynamics.

\subsection*{Representations and Rigidity}
Bounded cohomology has been very useful in proving rigidity results for representations. It is known that an epimorphism
between discrete groups induces an injective map on 2-dimensional bounded cohomology with real coefficients (see Theorem~\ref{Bua:thm}). As a consequence, if $H_b^2(\G,\R)$ is finite-dimensional
and $\rho\colon \G\to G$ is any representation, then the second bounded cohomology of the image of $\rho$ must also be finite-dimensional. In some cases, this information suffices to
ensure that $\rho$ is almost trivial. For example, if $\G$ is a uniform irreducible lattice in a higher rank semisimple Lie group, 
then by work of Burger and Monod we have that $H^2_b(\G,\R)$ is finite-dimensional~\cite{BM1}. On the contrary,
non-virtually-abelian subgroups of mapping class groups of hyperbolic surfaces admit many non-trivial quasimorphisms~\cite{BestFuji}, whence an infinite-dimensional second bounded cohomology group.
As a consequence, the image of any representation of a higher rank lattice into a mapping class group is virtually abelian, whence finite. This provides an independent
proof of a result by Farb, Kaimanovich and Masur~\cite{FarbMasur, KaMa}.

Rigidity results of a different nature arise when exploiting bounded cohomology to get more refined invariants with respect to the ones provided by classical cohomology. For example, 
we have seen that the norm of the Euler class can be used to bound the Euler number of flat circle bundles. When considering representations
into $PSL(2,\R)$ of the fundamental group of closed surfaces
of negative Euler characteristic, a celebrated result by Goldman~\cite{Goldth}
implies that a representation has maximal Euler number if and only if it is faithful and discrete, i.e.~if and only if it is the holonomy of a hyperbolic structure (in this case,
one usually says that the representation is \emph{geometric}).
A new proof of this fact (in the more general case of surfaces with punctures) has been recently provided in~\cite{BIW1} via a clever use of the bounded Euler class of a representation
(see also~\cite{iozzi} for another proof of Goldman's Theorem based on bounded cohomology).
In fact, one may define maximal representations of surface groups into a much more general class of Lie groups (see e.g.~\cite{BIW1}). 
This strategy has been followed e.g.~in~\cite{BIuseful, BIgeo} to establish deformation rigidity for representations into $SU(m,1)$ of lattices in $SU(n,1)$ also in the 
non-uniform case (the uniform case having being settled by Goldman and Millson in~\cite{GoldM}).
Finally,
bounded cohomology has proved very useful in studying rigidity phenomena 
also  
in the context of orbit equivalence and measure equivalence for countable groups
(see e.g.~\cite{MonShal} and~\cite{BFS}).

\bigskip

\subsection*{Content of the book}
Let us now briefly outline the content of each chapter. In the {\bf first chapter} we introduce the basic definitions about the cohomology and the bounded cohomology of groups.
We introduce the comparison map between the bounded cohomology and the classical cohomology of a group, 
and we briefly discuss the relationship between the classical cohomology of a group and the singular cohomology of its classifying space, postponing to Chapter 5 the investigation 
of the same topic in the context of bounded cohomology. 

In {\bf Chapter~2} we study the bounded cohomology of groups in low degrees. When working with trivial coefficients, bounded cohomology in degree 0 and 1 is completely understood, so our attention
is mainly devoted to degree 2. We recall that degree-2 classical cohomology is related to group extensions, while degree-2 bounded cohomology is related to the existence of quasimorphisms.
We exhibit many non-trivial quasimorphism on the free group, thus showing that $H^2_b(F,\R)$ is infinite-dimensional for every non-abelian free group $F$. We also describe a somewhat neglected result by Bouarich, 
which states that every class in $H^2_b(\G,\R)$ admits a canonical ``homogeneous'' representative, and, following~\cite{Bua}, we use this fact to show that any group epimorphism induces an injective map on bounded cohomology
(with trivial real coefficients) in degree 2. A stronger result (with more general coefficients allowed, and where the induced map is shown to be an isometric embedding) may be found in~\cite{Huber}.

{\bf Chapter~3} is devoted to amenability, which represents a fundamental notion in the theory of bounded cohomology. We briefly review some results on amenable groups, 
also providing a complete proof of von Neumann's Theorem, which ensures that abelian groups are amenable. Then we show that bounded cohomology of amenable groups vanishes
(at least for a very wide class of coefficients), and 
we describe Johnson's characterization
of amenability in terms of bounded cohomology.

In {\bf Chapter 4} we introduce the tools from homological algebra which are best suited to deal with bounded cohomology. Following~\cite{Ivanov} and~\cite{BM1,BM2}, we define the notion
of relatively injective $\G$-module, and we establish the basic results that allow to compute bounded cohomology via relatively injective strong resolutions.
We also briefly discuss how this part of the theory iteracts with amenability, also defining the notion of amenable action (in the very restricted context of discrete groups acting on discrete spaces).

We come back to the topological interpretation of bounded cohomology of groups in {\bf Chapter 5}. By exploiting the machinery developed in the previous chapter, we describe Ivanov's proof of
a celebrated result due to Gromov, which asserts that the bounded cohomology of a space is isometrically isomorphic to the bounded cohomology of its fundamental group.
Then, following~\cite{Ivanov}, we show that the existence of amenable covers of a topological space implies the vanishing of the comparison map in 
degrees higher than the multiplicity of the cover.
We also introduce the relative bounded cohomology of topological pairs, and prove that $H^n_b(X,Y)$ is isometrically isomorphic to $H^n_b(X)$ whenever every component
of $Y$ has an amenable fundamental group. 

In {\bf Chapter 6} we introduce  $\ell^1$-homology of groups and spaces. Bounded cohomology naturally provides a dual theory to $\ell^1$-homology. Following some works by L\"oh,
we prove several statements on duality. As an application, we describe L\"oh's proof of the fact that the $\ell^1$-homology
of a space is canonically isomorphic to the $\ell^1$-homology of its fundamental group. We also review some results by Matsumoto and Morita, showing for example that in degree 2 the seminorm on bounded cohomology
is always a norm, and providing a characterization of the injectivity of the comparison map in terms of the so-called \emph{uniform boundary condition}. 
Following~\cite{BBFIPP},
we also make use of duality (and of the results on relative bounded cohomology proved in the previous chapter) to obtain a proof of Gromov equivalence theorem.

{\bf Chapter 7} is devoted to an important application of bounded cohomology: the computation of simplicial volume. The simplicial volume of a closed oriented manifold is equal to the $\ell^1$-seminorm
of its real fundamental class. Thanks to the duality between $\ell^1$-homology and bounded cohomology, the study of the simplicial volume of a manifold often benefits from the study of its bounded cohomology.
Here we introduce the basic definitions and the most elementary properties of the simplicial
volume, and we state several results concerning it, postponing the proofs to the subsequent chapters.

Gromov's proportionality principle states that, for closed Riemannian manifolds admitting the same Riemannian universal covering, the ratio between the simplicial volume and the Riemannian volume is a constant only depending
on the universal covering. In {\bf Chapter 8} we prove a simplified version of the proportionality principle, which applies only to non-positively curved manifolds. The choice of restricting to the non-positively curved context allows us 
to avoid a lot of technicalities, while introducing the most important ideas on which also the proof of the general statement is based. In particular, we introduce a bit of continuous cohomology (of topological spaces) and
the trasfer map, which relates the (bounded) cohomology of a lattice in a Lie group to the continuous (bounded) cohomology of the ambient group. Even if our use of the transfer map is very limited, we feel
worth introducing it, since this map plays a very important role in the theory of continuous bounded cohomology of topological groups, as developed by Burger and Monod~\cite{BM1,BM2}.
As an application, we carry out the computation of the simplicial volume of closed hyperbolic manifolds following the strategy described in~\cite{Bucher}.

In {\bf Chapter 9} we prove that the simplicial volume is additive with respect to gluings along $\pi_1$-injective boundary components with an amenable fundamental group.
Our proof of this fundamental theorem (which is originally due to Gromov) is based on a slight variation of the arguments described in~\cite{BBFIPP}. In fact,
we deduce additivity of the simplicial volume from a result on bounded cohomology, together with a suitable application of duality.

As mentioned above, bounded cohomology has found interesting applications in the study of the dynamics of homemorphisms
of the circle. In {\bf Chapter 10} we introduce the Euler class and the bounded Euler class of a circle action, and
we review a fundamental result due to Ghys~\cite{Ghys0,Ghys1,Ghys2}, who proved that semi-conjugacy classes
of circle actions are completely classified by their bounded Euler class. The bounded Euler class is thus a much finer invariant than the classical Euler class,
and this provides a noticeable instance of a phenomenon mentioned above:
passing from classical to bounded cohomology often allows to refine classical invariants.
We also relate the bounded Euler class of a cyclic subgroup of homeomorphisms of the circle to the classical \emph{rotation number} of the generator of the subgroup,
and prove some properties of the rotation number that will be used in the next chapters. Finally, following Matsumoto~\cite{Matsu:numerical} we describe the canonical representative of the \emph{real} bounded Euler class,
also proving a characterization of semi-conjugacy in terms of the real bounded Euler class.

{\bf Chapter 11} is devoted to a brief digression from the theory of bounded cohomology. The main aim of the chapter is a detailed description
of the Euler class of a sphere bundle. However, our treatment of the subject is a bit different from the usual one, in that we define the Euler cocycle
as a \emph{singular} (rather than cellular) cocycle. Of course,
cellular and singular cohomology are canonically isomorphic for every CW-complex. However, 
cellular cochains are not useful to compute the \emph{bounded} cohomology of a cellular complex: for example,
it is not easy to detect whether
the singular coclass corresponding to a cellular one admits a bounded representative or not, and this motivates us to work directly in the singular context, even if this choice
could seem a bit costly  at first glance.

In {\bf Chapter 12} we specialize our study of sphere bundles to the case of \emph{flat} bundles. 
The theory of flat bundles builds a bridge between the theory of 
fiber bundles
and the theory of representations. For example, the Euler class of a flat circle bundle corresponds (under a canonical
morphism) to the Euler class of a representation canonically associated to the bundle. This leads to the definition of the  \emph{bounded} Euler class
of a flat circle bundle. By putting together an estimate on the norm of the bounded Euler class and the computation of the simplicial volume of surfaces carried out in the previous chapters,
we are then able to prove
\emph{Milnor-Wood inequalities}, which provide sharp estimates on the possible Euler numbers of flat circle bundles over surfaces. 
We then concentrate our attention on \emph{maximal} representations, i.e.~on representations of surface groups
which attain the extremal value allowed by Milnor-Wood inequalities. A celebrated result by Goldman~\cite{Goldth} states that maximal representations
of surface groups into the group of orientation-preserving isometries of the hyperbolic plane are exactly the holonomies of hyperbolic structures. 
Following~\cite{BIW1}, we will give a proof of Goldman's theorem based on the use of bounded cohomology. In doing this, we will describe the Euler number of flat bundles over surfaces with boundary,
as defined in~\cite{BIW1}. 

{\bf Chapter 13} is devoted to higher-dimensional generalizations of Milnor-Wood inequalities. We introduce Ivanov-Turaev's bounded Euler cocycle in dimension $n\geq 2$,
and from an estimate (proved to be optimal in~\cite{BuMo}) on its norm we deduce several inequalities on the possible Euler numbers of flat vector bundles over closed manifolds.
We also discuss the relationship between this topic and the Chern conjecture, which predicts that the Euler characteristic of a closed affine manifold should vanish.

\bigskip

\subsection*{Acknowledgements}
This book originates from the notes of a Ph.~D. course I taught at the University of Rome ``La Sapienza'' in Spring 2013. I would like to thank Gabriele Mondello, who asked me to teach the course: without his 
invitation this book would have never been written. Many thanks go to the mathematicians with whom I discussed many topics touched in this monograph; writing down a list may be dangerous,
because I am certainly forgetting someone, but I would like to mention at least Michelle Bucher, Marc Burger, Chris Connell, Stefano Francaviglia, Federico Franceschini, Tobias Hartnick, Alessandra Iozzi, Jean-Fran{\c c}ois Lafont, Clara L\"oh, 
Bruno Martelli,
Cristina Pagliantini, Maria Beatrice Pozzetti, Roman Sauer, Ben Schmidt, Alessandro Sisto.
Finally, special thanks go to Elia Fioravanti, Federico Franceschini and Cristina Pagliantini, who pointed out to me some inaccuracies in previous versions of this manuscript.

\chapter{(Bounded) cohomology of groups}\label{construction:chapter}

\section{Cohomology of groups}\label{construction:sec}
Let $\G$ be a group (which has to be thought as endowed with the discrete topology).
We begin by recalling the usual definition of the cohomology of 
$\G$ with coefficients in a $\G$-module $V$. 
The following definitions are classical, and date back to the works of Eilenberg and Mac Lane, Hopf, Eckmann, and Freudenthal in the 1940s, 
and to the contributions of Cartan and Eilenberg to
the birth of the theory of homological algebra in the early 1950s.

Throughout the whole monograph, unless otherwise stated,  groups actions will always be on the \emph{left}, and modules over (possibly non-commutative) rings will always be \emph{left} modules.
If $R$ is any commutative ring (in fact, we will be interested only in the cases
$R=\matZ$, $R=\R$), we denote by $R[\G]$ the group ring associated to $R$ and $\G$, 
i.e.~the set of finite linear combinations
of elements of $\G$ with coefficients in $R$, endowed with the operations
\begin{align*}
\left(\sum_{g\in \G} a_g g\right) + \left(\sum_{g\in \G} b_g g\right) & =\sum_{g\in \G} (a_g+b_g)g\, ,\\
\left(\sum_{g\in \G} a_g g\right) \cdot \left(\sum_{g\in \G} b_g g\right) & =
\sum_{g\in \G} \sum_{h\in \G} a_{gh}b_{h^{-1}} g
\end{align*}
(when the sums on the left-hand sides of these equalities are finite, then the same
is true for the corresponding right-hand sides).
Observe that an $R[\G]$-module $V$ is just an $R$-module $V$ endowed with an action of $\G$
by $R$-linear maps, and that an $R[\G]$-map between $R[\G]$-modules is just
an $R$-linear map which commutes with the actions
of $\G$. When $R$ is understood, we will often refer to $R[\G]$-maps
as to $\G$-maps.
If $V$ is an $R[\G]$-module, then we denote by $V^\G$ the subspace of \emph{$\G$-invariants} of $V$, {i.e.}~the set 
$$V^\G=\{v\in V\, |\, g\cdot v=v\ {\rm for\ every}\ g\in \G\}\ .$$

We are now ready to describe the complex of cochains which defines the cohomology
of $\G$ with coefficients in $V$.
For every $n\in\mathbb{N}$ we set
$$
C^n(\G,V)=\{f\colon \G^{n+1}\to V\}\ ,
$$
and we define  $\delta^n\colon C^n(\G,V)\to C^{n+1}(\G,V)$ as follows:
$$
\delta^n f(g_0,\ldots,g_{n+1})=\sum_{i=0}^{n+1} (-1)^i f(g_0,\ldots,\widehat{g_i},\ldots,g_{n+1})\ .
$$
It is immediate to check that $\delta^{n+1}\circ\delta^n=0$ for every $n\in\matN$,
so the pair $(C^\bullet(\G,V),\delta^\bullet)$ is indeed a complex, which is usually known
as the \emph{homogeneous} complex associated to the pair $(\G,V)$.
The formula 
$$
(g\cdot f) (g_0,\ldots,g_n)=g(f(g^{-1}g_0,\ldots,g^{-1}g_n))
$$
endows $C^n(\G,V)$ with an action of $\G$, whence with the structure of an $R[\G]$-module.
It is immediate to check that
$\delta^n$ is a $\G$-map, so the $\G$-invariants $C^\bullet(\G,V)^\G$ provide
a subcomplex of $C^\bullet(\G,V)$, whose homology is by definition the 
cohomology of $\G$ with coefficients in $V$. More precisely, if
$$
Z^n(\G,V)=C^n(\G,V)^\G\cap \ker \delta^n\, ,\quad
B^n(\G,V)=\delta^{n-1}(C^{n-1}(\G,V)^\G)
$$
(where we understand that $B^0(\G,V)=0$), then 
$B^n(\G,V)\subseteq Z^n(\G,V)$, and
$$
H^n(\G,V)=Z^n(\G,V)/B^n(\G,V)\ .
$$

\begin{defn}
 The $R$-module $H^n(\G,V)$ is the $n$-th cohomology module
of $\G$ with coefficients in $V$.
\end{defn}

For example, it readily follows from the definition than $H^0(\G,V)=V^\G$ for every $R[\G]$-module
$V$.

\section{Functoriality}\label{functor} 
Group cohomology provides a binary functor. If $\G$ is a group and
$\alpha\colon V_1\to V_2$ is an $R[\G]$-map, then $f$ induces an obvious
\emph{change of coefficients map} $\alpha^\bullet \colon C^\bullet(\G,V_1)\to C^\bullet(\G,V_2)$
obtained by composing any cochain with values in $V_1$ with $\alpha$. It is not difficult
to show that, if 
$$
0\tto{} V_1\tto{\alpha} V_2\tto{\beta} V_3\tto{} 0
$$
is an exact sequence of $R[\G]$-modules, then the induced sequence
of complexes
$$
0\tto{} C^\bullet(\G,V_1)^\G\tto{\alpha^\bullet}
C^\bullet(\G,V_2)^\G\tto{\beta^\bullet} C^\bullet(\G,V_3)^\G\tto{} 0
$$
is also exact, so there is a long exact sequence
$$
0\tto{} H^0(V_1,\G)\tto{} H^0(V_2,\G)\tto{} H^0(V_3,\G)\tto{} H^1(\G,V_1)\tto{} H^1(\G,V_2)\tto{}\ldots
$$
in cohomology.

Let us now consider the functioriality of cohomology with respect to the first variable.
Let $\psi\colon \G_1\to \G_2$ be a group homomorphism, and let $V$ be an $R[\G_2]$-module.
Then $\G_1$ acts on $V$ via $\psi$,
so $V$ is endowed with a natural structure of $R[\G_2]$-module. If we denote
this structure by $\psi^{-1}V$, then the maps 
$$
\psi^n\colon C^n(\G_2,V)\to C^n(\G_1,\psi^{-1}V)\, ,\qquad
\psi^n(f)(g_0,\ldots,g_n)=f(\psi(g_0),\ldots,\psi(g_n))
$$
provide a chain map such that 
$\psi^n(C^n(\G_2,V)^{\G_2})\subseteq C^n(\G_1,\psi^{-1}V)^{\G_1}$. As a consequence, we get a well-defined map
$$
H^n(\psi^n)\colon H^n(\G_2,V)\to H^n(\G_1,\psi^{-1}V)
$$ 
in cohomology. We will consider this map mainly in the case when
$V$ is the trivial module $R$.
In that context, the discussion above shows that every homomorphism
$\psi\colon \G_1\to \G_2$ induces a map 
$$
H^n(\psi^n)\colon H^n(\G_2,R)\to H^n(\G_1,R)
$$ 
in cohomology.

\section{The topological interpretation of group cohomology}\label{topol:sec}
Let us recall the well-known topological interpretation of group cohomology. 
We restrict our attention to the case when $V=R$ is a trivial $\G$-module. If $X$ is any topological space, then
we denote by $C_\bullet(X,R)$ (resp.~$C^\bullet(X,R)$)
the complex of singular chains (resp.~cochains) with coefficients in $R$, and by
$H_\bullet(X,R)$ (resp.~$H^\bullet(X,R)$) the corresponding singular homology module.

Suppose now that
%In this case
%$H^n(\G)$ is isomorphic to the singular homology module
%(with real coefficients)
%$H^n(X)$, where 
$X$ is any path-connected topological space satisfying the following properties:
\begin{enumerate}
 \item[(1)] 
the fundamental group of $X$ is isomorphic to $\G$,
\item[(2)] the space $X$ admits
a universal covering $\widetilde{X}$, and 
\item[(3)] 
$\widetilde{X}$ is $R$-acyclic, i.e.~$H_n(\widetilde{X},R)=0$
for every $n\geq 1$.
\end{enumerate} 
Then $H^\bullet(\G,R)$ is canonically isomorphic to $H^\bullet(X,R)$
(see Section~\ref{revisited:sec}).
If $X$ is a CW-complex, then condition (2) is automatically satisfied,
and Whitehead Theorem implies that condition (3) may be replaced by one of the following equivalent conditions:
\begin{enumerate}
 \item[(3')] $\widetilde{X}$ is contractible, or
\item[(3'')] $\pi_n(X)=0$ for every $n\geq 2$.
\end{enumerate}
If a CW-complex satisfies conditions (1), (2) and (3) (or (3'), or (3'')), then one usually says
that $X$ is a $K(\G,1)$, or an Eilenberg-MacLane space. Whitehead Theorem implies that the homotopy type of a $K(\G,1)$ only depends on $\G$, so if $X$ is any $K(\G,1)$, then it makes
sense to \emph{define} $H^i(\G,R)$ by setting $H^i(\G,R)=H^i(X,R)$. It is not difficult to show that this definition agrees with 
the definition of group cohomology given above. In fact, 
associated to $\G$ there is the $\Delta$-complex $B\G$, having one $n$-simplex
for every ordered $(n+1)$-tuple of elements of $\G$, and obvious face operators
(see e.g.~\cite[Example 1.B.7]{Hatcher}). 

When endowed with the weak topology,
$B\G$ is clearly contractible. Moreover,
the group $\G$ acts
freely and simplicially on $B\G$, so the quotient $X_\G$ of $B\G$ by the action
of $\G$ inherits the structure of a $\Delta$-complex, and the projection $B\G\to X_{\G}$ is a universal covering. Therefore, 
$\pi_1(X_{\G})=\G$ and  $X_\G$ is a $K(\G,1)$.
By definition, the space of the simplicial $n$-cochains on $B\G$ coincides
with the module $C^n(\G,R)$ introduced above, and
the simplicial cohomology of $X_\G$ is isomorphic
to the cohomology of the $\G$-invariant simplicial cochains on $B\G$. 
Since the simplicial cohomology of $X_{\G}$ is canonically isomorphic to the singular cohomology of $X_{\G}$,
this allows us to conclude that
$H^\bullet(X_\G,R)$ is canonically isomorphic to $H^\bullet(\G,R)$.

\section{Bounded cohomology of groups}\label{bounded:coh:sec}
Let us now shift our attention to \emph{bounded} cohomology of groups.
To this aim we first need to define the notion of \emph{normed}
$\G$-module.
Let $R$ and $\G$ be as above. For the sake of simplicity,
we assume that $R=\matZ$ or $R=\R$, and we denote by $|\cdot |$ the usual
absolute value on $R$.
A \emph{normed} $R[\G]$-module $V$ is an $R[\G]$-module endowed with an invariant norm, i.e.~a map
$\|\cdot \|\colon V\to \R$ such that:
\begin{itemize}
\item $\|v\|=0$ if and only if $v=0$,
 \item $\|r\cdot v\|\leq |r|\cdot \|v\|$ for every $r\in R$, $v\in V$,
\item $\|v+w\|\leq \|v\|+\|w\|$ for every $v,w\in V$,
\item $\|g\cdot v\|= \|v\|$ for every $g\in \G$, $v\in V$.
\end{itemize}
A $\G$-map between normed $R[\G]$-modules is a $\G$-map between the underlying
$R[\G]$-modules, which is bounded with respect to the norms.

Let $V$ be a normed $R[\G]$-module, and recall that $C^n(\G,V)$ is endowed
with the structure of an $R[\G]$-module. For every $f\in C^n(\G,V)$ one may consider the $\ell^\infty$-norm
$$
\|f\|_\infty =\sup \{\|f(g_0,\ldots,g_n)\|\, |\, (g_0,\ldots,g_n)\in \G^{n+1}\}\ \in \ [0,+\infty]\ .
$$
We set
$$
C_b^n(\G,V)=\{f\in C^n(\G,V)\, |\, \|f\|_\infty<\infty\}
$$
and we observe that $C_b^n(\G,V)$ is an $R[\G]$-submodule of $C^n(\G,V)$. 
Therefore, $C_b^n(\G,V)$
is a normed $R[\G]$-module. The differential $\delta^n\colon C^n(\G,V)\to C^{n+1}(\G,V)$
restricts to a $\G$-map of normed $R[\G]$-modules $\delta^n\colon C^n_b(\G,V)\to C^{n+1}_b(\G,V)$, so one may define as usual 
$$
Z^n_b(\G,V)=\ker \delta^n\cap C^n_b(\G,V)^\G\, ,\quad
B^n_b(\G,V)=\delta^{n-1}(C_b^{n-1}(\G,V)^\G)
$$
(where we understand that $B^0_b(\G,V)=\{0\}$), and set
$$
H^n_b(\G,V)=Z_b^n(\G,V)/B^n_b(\G,V)\ .
$$
The $\ell^\infty$-norm on $C_b^n(\G,V)$ restricts to a norm on $Z_b^n(\G,V)$,
which descends to a seminorm on $H^n_b(\G,V)$ by taking the infimum over all the representatives of a coclass: namely, for every $\alpha\in H^n_b(\G,V)$ one sets
$$
\|\alpha\|_\infty=\inf \{\|f\|_\infty\, |\, f\in Z^n_b(\G,V),\, [f]=\alpha\}\ .
$$

\begin{defn}
The $R$-module $H^n_b(\G,V)$ is the $n$-th bounded cohomology module
of $\G$ with coefficients in $V$. The seminorm $\|\cdot \|_\infty\colon H^n_b(\G,V)\to \R$
is called the \emph{canonical seminorm} of $H^n_b(\G,V)$.
\end{defn}

The canonical seminorm on $H^n_b(\G,V)$ is a norm if and only if the subspace $B^n_b(\G,V)$
is closed in $Z^n_b(\G,V)$ (whence in $C^n_b(\G,V)$). However, this is not always the case:
for example, 
$H^3_b(\G,\R)$ contains
non-trivial elements with null seminorm
if $\G$ is free non-abelian~\cite{Soma2}, and contains even an infinite-dimensional
subspace of elements with null norm if $\G$ is 
the fundamental group of a closed hyperbolic
surface.

On the other hand, 
it was proved independently by Ivanov~\cite{Ivanov2} and
Matsumoto and Morita~\cite{Matsu-Mor} that $H^2_b(\G,\R)$ is a Banach space
for every group $\G$ (see Corollary~\ref{Banach:h2}).

\section{Functoriality}
Just as in the case of ordinary cohomology, also 
bounded cohomology provides a binary functor. 
The discussion carried out in Section~\ref{functor} applies
word by word to the bounded case. Namely, 
every exact sequence of normed $R[\G]$-modules
$$
0\tto{} V_1\tto{\alpha} V_2\tto{\beta} V_3\tto{} 0
$$
induces a long exact sequence
$$
0\tto{} H^0_b(V_1,\G)\tto{} H^0_b(V_2,\G)\tto{} H^0_b(V_3,\G)\tto{} H_b^1(\G,V_1)\tto{} H_b^1(\G,V_2)\tto{}\ldots
$$
in bounded cohomology. Moreover, 
if $\psi\colon \G_1\to \G_2$ is a group homomorphism, and $V$ is  a normed $R[\G_2]$-module,
then $V$ admits a natural structure of normed $R[\G_1]$-module, which
is denoted by $\psi^{-1}V$. Then, the homomorphism $\psi$ induces a
 well-defined map
$$
H_b^n(\psi^n)\colon H_b^n(\G_2,V)\to H_b^n(\G_1,\psi^{-1}V)
$$ 
in bounded cohomology. In particular, in the case of trivial coefficients we get a map
$$
H_b^n(\psi^n)\colon H_b^n(\G_2,R)\to H_b^n(\G_1,R)\ .
$$ 
\medskip

\section{The comparison map and exact bounded cohomology}
The inclusion $C^\bullet_b(\G,V)\hookrightarrow C^\bullet(\G,V)$
induces a map
$$
c^\bullet\colon H_b^\bullet (\G,V)\to H^\bullet (\G,V)
$$
called the \emph{comparison map}. We will see soon that the comparison
map is neither injective nor surjective in general. 
In particular, then kernel of $c^n$ is denoted by $EH^n_b(\G,V)$, and called the \emph{exact}
bounded cohomology of $\G$ with coefficients in $V$.
One could
approach the study of $H^n_b(\G,V)$ by looking separately at the
kernel and at the image of the comparison map. 
This strategy is described in Chapter~\ref{lowdegree:chap} for the case of
low degrees.

Of course, if $V$ is a normed $R[\G]$-module, then a (possibly infinite) seminorm can be put also on $H^\bullet(\G,V)$ by setting
$$
\|\alpha\|_\infty=\inf \{\|f\|_\infty\, |\, f\in Z^n(\G,V),\, [f]=\alpha\}\in [0,+\infty]
$$
 for every $\alpha\in H^\bullet(\G,V)$. It readily follows from the definitions that 
  $$
 \|\alpha\|_\infty=\inf \{\|\alpha_b\|_\infty\, |\, \alpha_b\in H^n_b(\G,V)\, ,\  c(\alpha_b)=\alpha\} 
 $$
  for every $\alpha\in H^n(\G,V)$, where we understand that $\sup \emptyset =+\infty$.

\section{The bar resolution}\label{bar:sec}
We have defined the cohomology (resp.~the bounded cohomology) of $\G$ as the cohomology
of the complex $C^\bullet(\G,V)^\G$ (resp.~$C_b^\bullet(\G,V)^\G$). Of course,
an element $f\in C^\bullet(\G,V)^\G$ is completely determined by the values
it takes on $(n+1)$-tuples having $1$ as first entry.
More precisely, 
if we set 
$$\overline{C}^0(\G,V)=V\, ,\qquad \overline{C}^n(\G,V)=C^{n-1}(\G,V)=\{f\colon \G^n\to V\}\ ,$$
and we consider $\overline{C}^n(\G,V)$ simply as 
an $R$-module for every $n\in\matN$, 
then we have $R$-isomorphisms
$$
\begin{array}{rcl}
C^n(\G,V)^\G & \to &  \overline{C}^{n}(\G,V)\\
\varphi & \mapsto & \left( (g_1,\ldots,g_n)\mapsto \varphi (1,g_1,g_1g_2,\ldots,g_1\cdots g_n)\right)\ .
\end{array}
$$
Under these isomorphisms, 
the diffential $\delta^\bullet\colon C^\bullet(\G,V)^\G\to C^{\bullet+1}(\G,V)^\G$
translates into the differential $\overline{\delta}^\bullet\colon \overline{C}^\bullet(\G,V)\to \overline{C}^{\bullet+1}(\G,V)$ defined by
$$
\overline{\delta}^0(v)(g)=g\cdot v -v\, \qquad v\in V=\overline{C}^0(\G,V),\ g\in \G\ ,
$$
and 
\begin {align*}
\overline{\delta}^n(f)(g_1,\ldots,g_{n+1}) &=
g_1\cdot f(g_2,\ldots,g_{n+1})\\ &+\sum_{i=1}^n (-1)^i f(g_1,\ldots,g_ig_{i+1},\ldots,g_{n+1})\\ &
+(-1)^{n+1} f(g_1,\ldots,g_n)\ ,
\end{align*}
for $n\geq 1$.
The complex 
$$
0\longrightarrow \overline{C}^0(\G,V)\tto{\overline{\delta}^0} \overline{C}^1(\G,V)\tto{\overline{\delta}^2}\ldots \tto{\overline{\delta}^{n-1}}
\overline{C}^n(\G,V)\tto{\overline{\delta}^n}\ldots
$$
is usually known as the \emph{bar resolution} (or as the \emph{inhomogeneous} complex) associated to the pair $(\G,V)$. 
By construction,
the cohomology of this complex is canonically isomorphic to $H^\bullet (\G,V)$.

Just as we did  for the homogeneous complex, if $V$ is a normed $R[\G]$-module, then we can define the submodule $\overline{C}^n_b(\G,V)$ of bounded elements of $\overline{C}^n(\G,V)$.
For every $n\in\matN$, the isomorphism $C^n(\G,V)^\G\cong \overline{C}^n(\G,V)$ restricts
to an isometric isomorphism $C_b^n(\G,V)^\G\cong \overline{C}_b^n(\G,V)$,
so $\overline{C}_b^\bullet(\G,V)$ is a subcomplex
of $\overline{C}^\bullet(\G,V)$, whose cohomology is canonically isometrically isomorphic
to $H^\bullet_b(\G,V)$.

\section{Topology and bounded cohomology}\label{bounded:inter}
Let us now restrict to the case when $V=R$ is a trivial normed $\G$-module. 
We would like to compare the bounded cohomology of $\G$ 
with a suitably defined singular bounded cohomology of a suitable topological model for $\G$.
There is a straightforward notion of boundedness for singular cochains, so the notion of bounded singular
cohomology of a topological space is easily defined (see Chapter~\ref{bounded:space:chap}). 
It is still true that the bounded cohomology of a group
$\G$ is canonically isomorphic to the bounded singular cohomology of a $K(\G,1)$ (in fact, even more is true:
in the case of real coefficients, the bounded cohomology of $\G$ is isometrically isomorphic to the bounded singular cohomology
of \emph{any} countable CW-complex $X$ such that $\pi_1(X)=\G$ -- see Theorem~\ref{gro-iva:thm}).
However,  
the proof described in Section~\ref{topol:sec} for classical cohomology does not carry over to the bounded context.
It is still true that homotopically equivalent spaces have isometrically isomorphic
bounded cohomology, but,
if $X_\G$ is the $K(\G,1)$ defined in Section~\ref{topol:sec}, then there is no clear reason why the bounded simplicial cohomology
of $X_\G$ (which coincides with the bounded cohomology of $\G$)
should be isomorphic to the bounded singular cohomology of $X_\G$. 
For example, if $X$ is a finite $\Delta$-complex, then every simplicial cochain on $X$ is obviously bounded, so the cohomology
of the complex of bounded simplicial cochains on $X$ coincides with the classical singular cohomology of $X$, which in general is very different
from the bounded singular cohomology of $X$. 

\section{Further readings}
Of course, the most natural reference for an introduction to bounded cohomology is given by~\cite{Gromov}.
However, Gromov's paper is mainly devoted to the study of bounded cohomology of topological spaces, and to the applications
of that theory to the construction and the analysis of invariants of manifolds. An independent approach to bounded cohomology of groups 
was then developed by Ivanov~\cite{Ivanov} (in the case of trivial real coefficients) and by Noskov~\cite{Noskov} (in the case with twisted coefficients).
Ivanov's theory was then extended to deal with \emph{topological} groups by Burger and Monod~\cite{BM1,BM2,Monod}. In particular, Monod's book~\cite{Monod} settles the
foundations of bounded cohomology of locally compact groups, and covers many results which are of great interest also in the case of discrete groups. 
We refer the reader also to L{\"o}h's book~\cite{Clara:book} for an introduction to the (bounded) cohomology of discrete groups.

There are plenty of sources where the definition and the basic properties of classical cohomology of groups can be found. For example, we refer the reader
to Brown's book~\cite{brown} for a definition of group cohomology
as a derived functor.

\chapter{(Bounded) cohomology of groups in low degree}\label{lowdegree:chap}
In this chapter we analyze the (bounded) cohomology modules of a group $\G$ in degree $0,1,2$.
We restrict our attention to the case when $V=R$ is equal either to $\matZ$, or to $\R$,
both endowed with the structure of trivial $R[\G]$-module. In order to simplify the computations, it will be convenient to work with the \emph{inhomogeneous}
complexes of (bounded) cochains.

\section{(Bounded) group cohomology in degree zero and one}\label{degree1}
By the very definitions (see Section~\ref{bar:sec}), we have $\overline{C}^0(\G,R)=\overline{C}^0_b(\G,R)=R$
and $\overline{\delta}^0=0$, so
$$
H^0(\G,R)=H^0_b(\G,R)=R\ .
$$

Let us now describe what happens in degree one. By definition, for every
$\varphi\in \bc^1(\G,R)$ we have
$$
\deltab(f)(g_1,g_2)=f(g_1)+f(g_2)-f(g_1g_2)
$$
(recall that we are assuming that the action of $\G$ on $R$ is trivial). In other words,
if we denote by $\overline{Z}^\bullet(\G,R)$ and $\overline{B}^\bullet(\G,R)$ the spaces
of cocycles and coboundaries of  the inhomogeneous complex $\bc^\bullet(\G,R)$, then
we have
$$
H^1(\G,R)=\overline{Z}^1(\G,R)={\rm Hom} (\G,R)\ .
$$
Moreover, every bounded homomorphism with values in $\matZ$ or in $\R$ is obviously
trivial, so
$$
H^1_b(\G,R)=\overline{Z}_b^1(\G,R)=0 
$$
(here and henceforth, we denote by $\overline{Z}_b^\bullet(\G,R)$ (resp.~$\overline{B}_b^\bullet(\G,R)$) the space
of cocycles (resp.~coboundaries) of  the bounded inhomogeneous complex
$\bc_b^\bullet(\G,R)$).

\section{Group cohomology in degree two}\label{extension:sec}
A \emph{central extension} of $\G$
by $R$ is an exact sequence
$$
1\tto{} R\tto{\iota} \G'\tto{\pi} \G\to 1
$$
such that $\iota(R)$ is contained in the center of $\G'$. 
In what follows, in this situation we will identify $R$ with $\iota(R)\in \G'$
via $\iota$.
Two such extensions are equivalent
if they may be put into a commutative diagram as follows:
$$
\xymatrix{
1\ar[r] & R \ar[r]^{\iota_1} \ar[d]^{\rm Id} & \G_1\ar[r]^{\pi_1} \ar[d]^f & \G \ar[d]^{\rm Id} \ar[r] & 1\\
1\ar[r] & R \ar[r]^{\iota_2} & \G_2 \ar[r]^{\pi_2} & \G \ar[r]  & 1
}
$$
(by the commutativity of the diagram, the map $f$ is necessarily an isomorphism). 
Associated to an exact sequence 
$$
1\tto{} R\tto{\iota} \G'\tto{\pi} \G\to 1
$$
there is a cocycle $\varphi\in \overline{C}^2(\G,R)$ which is defined as follows.
Let $s\colon \G\to \G'$ be any map such that $\pi\circ s={\rm Id}_\G$. Then we set
$$
\varphi(g_1,g_2)=s(g_1g_2)^{-1}s(g_1)s(g_2)\ .
$$
By construction we have $\varphi(g_1,g_2)\in \ker\pi=R$, so  $\varphi$ is indeed an element
in $\overline{C}^2(\G,R)$. It is easy to check that $\deltab^2 (\varphi)=0$,
so $\varphi\in \overline{Z}^2(\G,R)$. Moreover, different choices for the section $s$
give rise to cocycles which differ one from the other by a coboundary, so any central extension $\calC$ of $\G$ by $R$ defines an element $e(\calC)$ in $H^2(\G,R)$
(see~\cite[Chapter 4, Section 3]{brown}).

It is not difficult to reverse
this construction to show that every element in $H^2(\G,R)$ is represented by a central
extension:

\begin{lemma}\label{centr-group}
For every class $\alpha\in H^2(\G,R)$, there exists a central extension $\calC$ of $\G$ by $R$ such 
that $\alpha=e(\calC)$.
\end{lemma}
\begin{proof}
It is easy to show that the class $\alpha$ admits a representative $\varphi$ such that
\begin{equation}\label{normalized}
\varphi(g,1)=\varphi(1,g)=0
\end{equation}
for every $g\in \G$ (see Lemma~\ref{balanced} below). We now define a group $\G'$ as follows. As a set, $\G'$ is just the cartesian product
$A\times\G$. We define a multiplication on $\G'=A\times \G$ by setting
$$
(a,g)(b,h)=(a+b+f(g,h),gh)\ .
$$
Using that $\varphi$ is a cocycle and equation~\eqref{normalized} one can check that this operation indeed defines a group law
(see~\cite[Chapter 4, Section 3]{brown} for the details). Moreover, the obvious inclusion $A\hookrightarrow A\times \G=\G'$ and
the natural projection $\G'=A\times \G\to G$ are clearly group homomorphisms, and every element in the image of $A$ is indeed in the center of $\G'$.
Finally, the cocycle corresponding to the section $g\to (0,g)\in \G'$ is equal to $\varphi$ (see again~\cite[Chapter 4, Section 3]{brown}), and this concludes the proof. 
\end{proof}

\begin{lemma}\label{balanced}
 Take $\varphi\in \overline{Z}^2(\G,\R)$. Then, there exists a bounded cochain $b\in \overline{C}^1_b(\G,\R)$
 such that $\varphi'=\varphi+\overline{\delta}^1(b)$ satisfies
 $$
 \varphi'(g,1)=\varphi'(1,g)=0\qquad {\rm for\ every}\ g\in\G\ .
 $$
\end{lemma}
\begin{proof}
 Since $\varphi$ is a cocycle, for every $g_1,g_2,g_3\in\G$ we have
  \begin{equation}\label{cocycle:eq}
  c(g_2,g_3)-c(g_1g_2,g_3)+c(g_1,g_2g_3)-c(g_1,g_2)=0\ .
 \end{equation}
By substituting $g_2=1$ we obtain $c(1,g_3)=c(g_1,1)$, whence $c(g,1)=c(1,g)=k_0$ for some constant $k_0\in\R$ and every $g\in\G$.
Therefore, the conclusion follows by setting
 $b(g)=-k_0$ for every $g\in\G$.
\end{proof}

%and two central extensions are equivalent if and only if they define
%the same element in $H^2(\G,R)$ (see e.g.~\cite{brown} for full details). Therefore:

The following results
imply that an extension is substantially determined by its class, and 
will prove useful in the study of the Euler class of groups actions on the circle (see Chapter~\ref{actions:chapter}).

\begin{lemma}\label{euler-null}
 Let us consider  the central extension $\calC$
$$
1\tto{} R \tto{\iota} \G'\tto{p} \G\to 1\ .
$$
Then $H^2(p)(e(\calC))=0\in H^2(\G',R)$.
\end{lemma}
\begin{proof}
 Let us fix a section
 $s\colon \G\to\G'$, and let $\varphi\colon \overline{C}^2(\G,R)$ be the corresponding $2$-cocycle.
 An easy computation shows that the pull-back of $\varphi$ via $p$ is equal to
 $\overline{\delta}^1 \psi$, where $\psi(g)=g^{-1}s(p(g))$, so
$H^2(p)(e(\calC))=0\in H^2(\G',R)$. 
\end{proof}

\begin{lemma}\label{Lift}
Let us consider the central extension $\calC$
$$
1\tto{} R \tto{\iota} \G'\tto{p} \G\to 1\ .
$$
If $\rho\colon G \to \G$ is a homomorphism, then there exists a lift
\[\begin{xy}
\xymatrix{
0\ar[r]& _R \ar[r]^i&\G' \ar[r]^p& \G \ar[r]&\{e\}\\
&&& G \ar[u]_\rho \ar@{.>}[ul]^{\widetilde{\rho}}&
}
\end{xy}\]
if and only if $H^2(\rho)(e(\calC)) = 0 \in H^2(G, R)$.
\end{lemma}
\begin{proof}
 If the lift $\widetilde{\rho}$ exists, then by Lemma~\ref{euler-null} we have
$H^2(\rho)(e(\calC))=H^2(\widetilde{\rho})(H^2(p)(e(\calC)))=0$.

On the other hand, let us fix a section $s\colon \G\to \G'$ with corresponding $2$-cocycle $\varphi$ and assume that 
the pull-back of $\varphi$ via $\rho$ is equal to $\overline{\delta}^1 u$ 
for some $u\colon G \to R$. Then it is easy to check that a homomorphic lift $\widetilde{\rho}\colon G \to \G'$ is given by the formula
\[
\widetilde{\rho}(g) =s(\rho(g)) \cdot \iota(-u(g)).
\]
Indeed, we have
\begin{eqnarray*}
s(\rho(g_1g_2)) &=& s(\rho(g_1))\left(s(\rho(g_1))^{-1}s(\rho(g_1)\rho(g_2))s(\rho(g_2))^{-1}\right)s(\rho(g_2))\\
&=& s(\rho(g_1))\cdot \varphi(\rho(g_1),\rho( g_2))^{-1}\cdot s(\rho(g_2))\\
&=&  s(\rho(g_1))\cdot i(-\overline{\delta}^1u(g_1, g_2))\cdot s(\rho(g_2))\\
&=&  s(\rho(g_1))i(-u(g_1))\cdot  s(\rho(g_2))i(-u(g_2)) \cdot i(-u(g_1g_2))^{-1}
\end{eqnarray*}
for all $g_1, g_2 \in \Gamma$. Multiplying both sides by $ i(-u(g_1g_2))$ now yields $\widetilde{\rho}(g_1g_2) =\widetilde{ \rho}(g_1)\widetilde{\rho}(g_2)$ and finishes the proof.
%In fact, since $\iota(R)$ is central in $\G'$ we obtain
%\begin{eqnarray*}
%s(\rho(g_1g_2)) &=& s(\rho(g_1))\left(s(\rho(g_1))^{-1}s(\rho(g_1)\rho(g_2))s(\rho(g_2))^{-1}\right)s(\rho(g_2))\\
%&=& s(\rho(g_1))\cdot c_s(\rho(g_1),\rho( g_2))^{-1}\cdot s(\rho(g_2))\\
%&=&  s(\rho(g_1))\cdot i(-du(g_1, g_2))\cdot s(\rho(g_2))\\
%&=&  s(\rho(g_1))i(-u(g_1))\cdot  s(\rho(g_2))i(-u(g_2)) \cdot i(-u(g_1g_2))^{-1}
%\end{eqnarray*}
%for all $g_1, g_2 \in G$. Multiplying both sides by $\iota (-u(g_1g_2))$ now yields $\widetilde{\rho}(g_1g_2) =\widetilde{ \rho}(g_1)\widetilde{\rho}(g_2)$ and finishes the proof.
\end{proof}

Putting together the results stated above we get the following:

\begin{prop}\label{central:extensions}
 The group $H^2(\G,R)$ is in natural bijection with the set of equivalence classes
of central extensions of $\G$ by $R$.
\end{prop}
\begin{proof}
 Lemma~\ref{centr-group} implies that every class in $H^2(\G,R)$ is realized by a central extension. On the other hand, using Lemmas~\ref{euler-null} and~\ref{Lift}
 it is immediate to check that extensions sharing the same cohomology class are equivalent.
\end{proof}

\begin{rem}
 Proposition~\ref{central:extensions} can be easily generalized to the case of
arbitrary (i.e.~non-necessarily central) 
extensions of $\G$ by $R$. For any such extension
$$
1\tto{} R \tto{\iota} \G'\tto{p} \G\to 1\ ,
$$
 the action of $\G'$ on $R=i(R)$ by conjugacy descends to a well-defined 
action of $\G$ on $R$, thus endowing $R$ with the structure of a (possibly non-trivial) $\G$-module (this structure is trivial precisely when $R$ is central in $\G'$). Then,  
one can associate to the extension an element
of the cohomology group $H^2(\G,R)$ with coefficients in the (possibly non-trivial)
$R[\G]$-module $R$.
\end{rem}

\section{Bounded group cohomology in degree two: quasimorphisms}\label{quasimorphisms:chap}
As mentioned above, the study of $H_b^2(\G,R)$ can be reduced to the study of the kernel and of the image of the comparison map
$$
c^2\colon H_b^2(\G,R)\to H^2(\G,R)\ .
$$
We will describe in Chapter~\ref{duality:chap} a characterization of groups
with injective comparison map due to Matsumoto and Morita~\cite{Matsu-Mor}.
In this section we describe the relationship between the kernel $EH^2_b(\G,R)$ of the comparison map
and the space of quasimorphisms on $\G$. 

\begin{defn}
 A map $f\colon \G\to R$ is a \emph{quasimorphism} if there exists a constant $D\geq 0$
such that
$$
|f(g_1)+f(g_2)-f(g_1g_2)|\leq D
$$
for every $g_1,g_2\in \G$. The least $D\geq 0$ for which the above inequality is satisfied
is the \emph{defect} of $f$, and it is denoted by $D(f)$. 
The space of quasimorphisms is an $R$-module, and it is denoted by $Q(\G,R)$. 
\end{defn}

By the very definition, a quasimorphism is an element
of $\overline{C}^1(\G,R)$ having bounded differential.
Of course, both bounded functions (i.e.~elements in $\overline{C}^1_b(\G,R)$)
and homomorphisms (i.e.~elements in $\overline{Z}^1(\G,R)={\rm Hom} (\G,R)$) are quasimorphisms.
If every quasimorphism could be obtained just by adding a bounded function to a homomorphism, the notion of quasimorphism would not lead to
anything new. Observe that every bounded homomorphism with values in $R$ is necessarily trivial,
so $\overline{C}^1_b(\G,R)\cap {\rm Hom} (\G,R)=\{0\}$, and we may think
of $\overline{C}^1_b(\G,R)\oplus {\rm Hom} (\G,R)$ as of the space of ``trivial''
quasimorphisms on $\G$.

The following result is an immediate consequence of the definitions, and shows that the existence of ``non-trivial''
quasimorphisms is equivalent to the vanishing of
$EH^2_b(\G,R)$.

\begin{prop}\label{quasi:bounded:prop}
 There exists an exact sequence
 $$
 \xymatrix{
 0\ar[r] & \overline{C}^1_b(\G,R)\oplus {\rm Hom}_\matZ(\G,R) \ar@{^{(}->}[r] & Q(\G,R)\ar[r] & EH^2_b(\G,R)\ar[r] & 0\ , 
 }$$
 where the map $Q(\G,R)\to EH^2_b(\G,R)$ is induced by the differential $\deltab^1\colon Q(\G,R)\to \overline{Z}^2_b(\G,R)$.
 In particular,
$$
Q(\G,R)\big/\left(\overline{C}^1_b(\G,R)\oplus {\rm Hom}_\matZ(\G,R)\right)\ \cong\ 
EH^2_b(\G,R)\ .
$$
\end{prop}

Therefore, in order to show that $H^2_b(\G,R)$ is non-trivial it is sufficient
to construct quasimorphisms which do not stay at finite distance from a homomorphism.

\section{Homogeneous quasimorphisms}\label{hom:qm}
Let us introduce the following:

\begin{defn}
 A quasimorphism $f\colon \G\to R$ is \emph{homogeneous} if $f(g^n)=n\cdot f(g)$
for every $g\in \G$, $n\in\matZ$. The space of homogeneous quasimorphisms is a submodule
of $Q(\G,R)$, and it is denoted by $Q^h(\G,R)$.
\end{defn}

Of course, there are no non-zero  bounded homogeneous quasimorphisms. In particular,
for every quasimorphism $f$ there exists at most one homogeneous quasimorphism
$\overline{f}$ such that $\|\overline{f}-f\|_\infty<+\infty$, so a homogeneous quasimorphism which is not a homomorphism cannot stay at finite distance from a homomorphism. When $R=\matZ$, it may happen that homogeneous quasimorphisms are
quite sparse in $Q(\G,\matZ)$: for example, for every $\alpha\in \mathbb{R}\setminus \matZ$,
if we denote by $\lfloor x\rfloor$ the largest integer which is not bigger than $x$, then 
the quasimorphism $f_\alpha\colon \matZ\to\matZ$
defined by 
\begin{equation}\label{fal}
f_\alpha (n)=\lfloor \alpha n\rfloor
\end{equation}
is not at finite distance from any element in $Q^h(\matZ,\matZ)={\rm Hom} (\matZ,\matZ)$. When $R=\R$
homogeneous quasimorphisms play a much more important role, due to the following:

\begin{prop}\label{homogeneous:prop}
 Let $f\in Q(\G,\R)$ be a quasimorphism. Then, there exists a unique
element $\overline{f}\in Q^h(\G,\R)$ 
that stays at finite distance from $f$. Moreover, we have
$$
\|f-\overline{f}\|_\infty \leq D(f)\, ,\qquad D(\overline{f})\leq 4 D(f)\ .
$$
\end{prop}
\begin{proof}
For every $g\in \G$, $m,n\in\mathbb{N}$ we have
$$
|f(g^{mn})-nf(g^m)|\leq (n-1)D(f)\ ,
$$
so
$$
\left| \frac{f(g^n)}{n}-\frac{f(g^{m})}{m}\right|
\leq \left| \frac{f(g^n)}{n}-\frac{f(g^{mn})}{mn}\right| +
\left|\frac{f(g^{mn})}{mn}- \frac{f(g^m)}{m}\right|\leq \left(\frac{1}{n}+\frac{1}{m}\right) D(f)\ .
$$
Therefore, the sequence $f(g^n)/n$ is a Cauchy sequence, and the same is true for
the sequence $f(g^{-n})/(-n)$. Since $f(g^n)+f(g^{-n})\leq f(1)+D(f)$ for every $n$,
we may conclude that the limits
$$
\lim_{n\to \infty} \frac{f(g^n)}{n}=\lim_{n\to -\infty} \frac{f(g^n)}{n}=\overline{f} (g)
$$
exist for every $g\in \G$. 
Moreover, the inequality $|f(g^n)-nf(g)|\leq (n-1)D(f)$ implies that
$$
\left| f(g)-\frac{f(g^n)}{n}\right|\leq D(f)
$$
so by passing to the limit we obtain that $\|\overline{f}-f\|_\infty\leq D(f)$.
This immediately implies that $\overline{f}$ is a quasimorphism such that
$D(\overline{f})\leq 4D(f)$. Finally, the fact that $\overline{f}$ is homogeneous is obvious. 
\end{proof}

In fact, the stronger inequality $D(\overline{f})\leq 2D(f)$ holds
(see e.g.~\cite[Lemma 2.58]{scl} for a proof).
Propositions~\ref{quasi:bounded:prop} and \ref{homogeneous:prop} imply the following:

\begin{cor}\label{decomp:cor}
The space $Q(\G,\R)$ decomposes as a direct sum
$$
Q(\G,\R)=Q^h(\G,\R)\oplus \overline{C}^1_b(\G,\R)\ .
$$
Moreover, the restriction of $\overline{\delta}$ to $Q^h(\G,\R)$ induces an isomorphism
$$
Q^h(\G,\R)/{\rm Hom}(\G,\R)\ \cong EH^2_b(\G,\R)\ .
$$
\end{cor}

\section{Quasimorphisms on abelian groups}\label{quasi:abelian:sec}
Suppose now that $\G$ is abelian. Then, for every $g_1,g_2\in \G$, every element
$f\in Q^h(\G,R)$, and every $n\in\matN$ we have
\begin{align*}
|nf(g_1g_2)-nf(g_1)-nf(g_2)|& =|f((g_1g_2)^n)-f(g_1^n)-f(g_2^n)|\\ & =|f(g_1^ng_2^n)-f(g_1^n)-f(g_2^n)|\leq D(f)\ .
\end{align*}
Dividing by $n$ this inequality and passing to the limit for $n\to\infty$ we get
that $f(g_1g_2)=f(g_1)+f(g_2)$, i.e.~$f$ is a homomorphism. Therefore, 
every homogeneous quasimorphism on $\G$ is a homomorphism. Putting together
Proposition~\ref{quasi:bounded:prop} and Corollary~\ref{decomp:cor}
we obtain the following:

\begin{cor}\label{inj:cor}
 If $\G$ is abelian, then  $EH^2_b(\G,\R)=0$.
\end{cor}

In fact, a much stronger result holds: if $\G$ is abelian,
then it is amenable (see Definition~\ref{amenable:def} and Theorem~\ref{abelian:amenable}), so $H_b^n(\G,\R)=0$ for every
$n\geq 1$ (see Corollary~\ref{amenable:real:cor}). We stress that Corollary~\ref{inj:cor} does not hold in the case
with integer coefficients. For example, we have the following:

\begin{prop}\label{h2zz}
We have $H^2_b(\matZ,\matZ)=\R/\matZ$. More precisely, an isomorphism between $\R/\matZ$ and $H^2_b(\matZ,\matZ)$ is given by the map
$$
\R/\matZ\ \ni\ [\alpha]\ \mapsto\ [c]\ \in\ H^2_b(\matZ,\matZ)\, ,
$$
where
$$
c(n,m)=\lfloor \alpha (n+m)\rfloor - \lfloor \alpha n\rfloor-\lfloor \alpha m\rfloor \ ,
$$
where $\alpha\in\R$ is any representative of $[\alpha]$ and $\lfloor x\rfloor$ denotes the largest integer which does not exceed
$x$. 
\end{prop}
\begin{proof}
The short exact sequence 
$$
0\tto{} \matZ \tto{} \R \tto{} \R/\matZ\tto{} 0
$$
induces a short exact sequence of complexes
$$
0\tto{} \overline{C}_b^\bullet(\matZ,\matZ)\tto{} \overline{C}^\bullet_b (\matZ,\R)\tto{} \overline{C}^\bullet (\matZ,\R/\matZ)
\tto{} 0\ .
$$
Let us consider the following portion of the long exact sequence induced in cohomology:
$$
\ldots\tto{} H^1_b(\matZ,\R)\tto{} H^1(\matZ,\R/\matZ)\tto{} H_b^2(\matZ,\matZ)\tto{} H_b^2(\matZ,\R)
\tto{} \ldots 
$$
Recall from Section~\ref{degree1} that $H^1_b(\matZ,\R)=0$. Moreover, 
we have $H^2(\matZ,\R)=H^2(S^1,\R)=0$, so $H_b^2(\matZ,\R)$ is equal to $EH_b^2(\matZ,\R)$, which vanishes
by Corollary~\ref{inj:cor}. Therefore, we have
$$
H_b^2(\matZ,\matZ)\cong H^1(\matZ,\R/\matZ)={\rm Hom} (\matZ,\R/\matZ)\cong \R/\matZ\ .
$$

Let us now look more closely at the map realizing this isomorphism. The connecting homomorphism of the long exact sequence above, which is usually denoted by $\delta$,
may be described as follows: 
for every $[\alpha]\in \R/\matZ$, if we denote by $f_{[\alpha]}\in H^1(\matZ,\R/\matZ)={\rm Hom}(\matZ,\R/\matZ)$ the homomorphism such that $f_{[\alpha]}(n)=n[\alpha] $ and by $g_{[\alpha]}\in C^1_b(\matZ,\R)$ a cochain such that
$[g_{[\alpha]}(n)]=f_{[\alpha]}(n)$, then $\delta f_{[\alpha]}=[\overline{\delta}^1 g_{[\alpha]}]\in H^2_b(\matZ,\matZ)$ (by construction, the cochain $\overline{\delta}^1 g_{[\alpha]}$ is automatically bounded and with integral values).

If $\alpha\in\R$ is a representative of $[\alpha]$, then a possible choice for $g_{[\alpha]}$ is given by
$$
g_{[\alpha]}(n)=\alpha n-\lfloor \alpha n\rfloor\ ,
$$
so 
\begin{align*}
\overline{\delta}^1 g_{[\alpha]}(n,m)&=\alpha n-\lfloor \alpha n\rfloor + \alpha m-\lfloor \alpha m\rfloor -(\alpha (n+m)-\lfloor \alpha (n+m)\rfloor)\\
&=-\lfloor \alpha n\rfloor-\lfloor \alpha m\rfloor + \lfloor \alpha (n+m)\rfloor=c(n,m)\ ,
\end{align*}
whence the conclusion.
\end{proof}

%\begin{rem}\label{h2zz:rem}
%It follows from the proof of the previous proposition that the isomorphism
%$\R/\matZ\to H^2_b(\matZ,\matZ)$ of Proposition~\ref{h2zz} is given by the map
%$$
%\R/\matZ\ni [\alpha]\ \mapsto \ [-\delta f_\alpha] \in H^2_b(\matZ,\matZ)\ ,
%$$
%where $f_\alpha$ is the quasimorphism defined in equation~\eqref{fal}.
%\end{rem}

It is interesting to notice that the module $H^2_b(\matZ,\matZ)$ is \emph{not} finitely
generated over $\matZ$. This fact already shows that bounded cohomology can be very different from classical cohomology. The same phenomenon may occur in the case
of real coefficients:
in the following section we will show that, if $F_2$ is the free group
on $2$ elements, then $H^2_b(F_2,\R)$ is an infinite dimensional vector space.

\section{The bounded cohomology of free groups in degree 2}
Let $F_2$ be the free group on two generators.
Since $H^2(F_2,\R)=H^2(S^1\vee S^1,\R)=0$ we have $H^2_b(F_2,\R)=EH^2_b(F_2,\R)$, so the computation of $H^2_b(F_2,\R)$
is reduced to the analysis of the space $Q(\G,\R)$ of real quasimorphisms on 
$F_2$. There are several constructions of elements in $Q(\G,\R)$ available in the literature. The first such example is probably due to Johnson~\cite{Johnson},
who proved that $H^2_b(F_2,\R)\neq 0$. Afterwords,
%The first family of such quasimorphisms is due to 
Brooks~\cite{Brooks} produced an infinite family of quasimorphisms, which were shown to define independent elements 
$EH^2_b(F_2,\R)$ by Mitsumatsu~\cite{Mitsu}.
%Brooks' quasimorphisms were shown to define linearly independent elements
%of $H^2_b(F_2,\R)$ by Mitsumatsu~\cite{}. 
Since then, other constructions have been provided by many authors for more
general classes of groups (see Section~\ref{further:2} below). We describe here a family of quasimorphisms which is due to Rolli~\cite{Rolli}.

Let $s_1,s_2$ be the generators of $F_2$, and let 
$$\ell^\infty_{\rm odd}(\matZ)=\{\alpha\colon \matZ\to \R\, |\, \alpha(n)=-\alpha(-n)\ {\rm for\ every}\ n\in\matZ\} \ .
$$ 
For every $\alpha\in \ell^\infty_{\rm odd}(\matZ)$ we consider the map
$$
f_\alpha \colon F_2\to \R
$$ 
defined by 
$$
f_\alpha (s_{i_1}^{n_1}\cdots s_{i_k}^{n_k})=\sum_{j=1}^k \alpha (n_k)\ ,
$$
where we are identifying every element of $F_2$ with the unique reduced word
representing it. It is elementary to check that $f_\alpha$ is a quasimorphism
(see~\cite[Proposition 2.1]{Rolli}). Moreover,
we have the following:

\begin{prop}[\cite{Rolli}]\label{rolli:prop}
 The map
$$
\ell^\infty_{\rm odd}(\matZ)\to H^2_b(F_2,\R)\, \qquad \alpha\mapsto [\delta(f_\alpha)]
$$
is injective.
\end{prop}
\begin{proof}
Suppose that $[\delta(f_\alpha)]=0$.
By Proposition~\ref{quasi:bounded:prop}, this implies that
$f_\alpha=h+b$, where $h$ is a homomorphism and $b$ is bounded. Then, for $i=1,2$,
$k\in\matN$,  we have $k\cdot h(s_i)=h(s_i^k)=f_\alpha(s_i^k)-b(s_i^k)=\alpha(k)-b(s_i^k)$.
By dividing by $k$ and letting $k$ going to $\infty$ we get $h(s_i)=0$, so $h=0$,
and $f_\alpha=b$ is bounded.

Observe now that for every $k,l\in\matZ$ we have 
$f_\alpha((s_1^ls_2^l)^k)=2k\cdot \alpha(l)$. Since $f_\alpha$ is bounded, this implies
that $\alpha=0$, whence the conclusion.
\end{proof}

\begin{cor}\label{infinitedim:cor}
Let $\G$ be a group admitting an epimorphism
$$
\varphi\colon \G\to F_2\ .
$$
Then the vector space $H^2_b(\G,\R)$ is infinite-dimensional.
\end{cor}
\begin{proof}
 Since $F_2$ is free, the epimorphism admits a right inverse $\psi$. The composition
$$
\xymatrix{
H^2_b(F_2,\R)\ar[rr]^{H^2_b(\varphi)} &  & H^2_b(\G,\R) \ar[rr]^{H^2_b(\psi)} & & H^2_b(F_2,\R)
}
$$
of the induced maps in bounded cohomology 
is the identity, and Proposition~\ref{rolli:prop} ensures that
$H^2_b(F_2,\R)$ is infinite-dimensional. The conclusion follows. 
\end{proof}

For example, the second bounded cohomology module (with trivial real coefficients) of the fundamental 
group of any closed hyperbolic surface is infinite-dimensional. We refer the reader to~\cite{BrooksSeries, Mitsu, BarGhys}
for different and more geometric constructions of non-trivial bounded cohomology classes for surface groups.

\section{Homogeneous $2$-cocycles}\label{homo:coc}
It was a question of Gromov whether any group epimorphism induces an isometric embedding on bounded cohomology (with Banach coefficients) in degree 2. 
A complete proof of this statement (which exploits the powerful machinery of Poisson boundaries) may be found in~\cite[Theorem 2.14]{Huber}.
In this section we describe an argument, due to Bouarich~\cite{Bua1,Bua2}, which proves that group epimorphisms induce biLipschitz embeddings on bounded
cohomology in degree 2 (with trivial real coefficients).

We have seen in Section~\ref{hom:qm} that $EH^2_b(\G,\R)$ is isomorphic to the quotient of the space of real homogeneous quasimorphisms by the space of
real homomorphisms. Therefore, every element in $EH^2_b(\G,\R)$ admits a canonical representative, which is obtained by taking the differential of a suitably chosen homogeneous quasimorphism. 
It turns out that  that also generic (i.e.~not necessarily exact) bounded classes of degree 2 admit
special representatives. Namely, let us say that a bounded $2$-cocycle 
$\varphi\in \overline{C}^2_b(\G,\R)$ is \emph{homogeneous} if and only if 
$$
\varphi(g^n,g^m)=0\quad {\rm for\ every}\ g\in\G\, ,\ n,m\in\mathbb{Z}\ .
$$
If $f\in Q(\G,\R)$ is a quasimorphism, then
$$
\overline{\delta}^1(f)(g^n,g^m)=f(g^n)+f(g^m)-f(g^{n+m})\ ,
$$
so a quasimorphism is homogeneous if and only if its differential is.
The next result shows that every bounded class of degree 2 admits a unique homogeneous representative. We provide here  the details of the proof
sketched in~\cite{Bua1}.

\begin{prop}\label{homo:eu}
 For every $\alpha\in H^2_b(\G,\R)$ there exists a unique homogeneous cocycle $\varphi_\alpha\in \overline{Z}^2_b(\G,\R)$ such that
 $[\varphi_\alpha]=\alpha$ in $H^2_b(\G,\R)$. 
\end{prop}
\begin{proof}
Let $\varphi\in \overline{Z}^2_b(\G,\R)$ be a representative of $\alpha$. 
By Lemma~\ref{balanced} we may assume that
$$
\varphi(g,1)=\varphi(1,g)=0\qquad {\rm for\ every}\ g\in\G\ .
$$
The class of $\varphi$ in $H^2(\G,\R)$ corresponds to a central extension 
 $$
1\tto{} \R \tto{\iota} \G'\tto{p} \G\to 1\ ,
$$
and Lemma~\ref{euler-null} implies that the pull-back $p^*\varphi$ of $\varphi$ in $\overline{Z}^2_b(\G',\R)$ is exact, 
so that
$H^2_b(p)\colon H^2_b(\G,\R)\to H^2_b(\G',\R)$ sends $\alpha$ to an exact class $\alpha'\in EH^2_b(\G',\R)$. 
As a consequence we have $p^*\varphi=\overline{\delta}^1 (f)$ for some $f\in Q(\G',\R)$.
We denote by $f_h\in Q^h(\G',\R)$  the homogeneous quasimorphism such that
$[\overline{\delta}^1(f)]=[\overline{\delta}^1(f_h)]$ in $EH^2_b(\G',\R)$ (see Corollary~\ref{decomp:cor}).
We will now show that $f_h$ can be exploited to define the desired cocycle $\varphi_\alpha$ on $\G$.

If $t$ is any element in the kernel of $p$, then for every $g'\in\G'$ we have
$$
0=\varphi(1,p(g'))=\varphi(p(t),p(g'))=\overline{\delta}^1(f)(t,g')=
f(t)+f(g')-f(tg')\ ,
$$
i.e.~$f(tg')=f(t)+f(g')$. In particular, an easy inductive argument shows that $f(t^n)=nf(t)$ for every $n\in\mathbb{Z}$, so that $f_h(t)=f(t)$ for every $t\in\ker p$. 
Also observe that, since $\ker p$ is contained in the center
of $\G'$, for every $t\in\ker p$ and every $g'\in\G'$ the subgroup $H$ of $\G'$ generated by $g'$ and $t$ is abelian. But abelian groups do not admit non-trivial quasimorphisms
(see Section~\ref{quasi:abelian:sec}),
so the restriction of $f_h$ to $H$, being homogeneous, is a homomorphism, and
$f_h(tg')=f_h(t)+f_h(g')$. 

Since the coboundaries of $f$ and of $f_h$ define the same bounded cohomology class, we have $f=f_h+b'$ for some
$b'\in \overline{C}^1_b(\G',\R)$. Moreover, for every $t\in\ker p$, $g'\in\G'$ we have
$$
b'(tg')=f(tg')-f_h(tg')=f(t)+f(g')-(f_h(t)+f_h(g'))=f(g')-f_h(g')=b'(g')\ .
$$
In other words, $b'=p^*b$ for some bounded cochain $b\in \overline{C}^1_b(\G,\R)$.

Let us now define a cochain $\varphi_\alpha\in \overline{C}^2(\G,\R)$ as follows: if $(g_1,g_2)\in \G^2$, choose $g_i'\in\G'$ such that
$p(g_i')=g_i$ and set $\varphi_\alpha(g_1,g_2)=\overline{\delta}^1(f_h)(g_1',g_2')$. For every $t_1,t_2\in\ker p$ we have
$$
\overline{\delta}^1(f_h)(t_1g_1',t_2g_2')=f_h(t_1g_1')+f_h(t_2g_2')-f_h(t_1g_1't_2g_2')=f_h(g_1')+f_h(g_2')-f_h(g_1'g_2')\ ,
$$
so $\varphi_\alpha$ is well defined, and it is immediate to check that $\varphi_\alpha$ is indeed homogeneous. Moreover, by construction we have $\varphi=\varphi_\alpha+\overline{\delta}^1(b)$,
so $\alpha=[\varphi]=[\varphi_\alpha]$ in $H^2_b(\G,\R)$. 

We are now left to show that $\varphi_\alpha$ is the \emph{unique} homogeneous bounded cocycle representing $\alpha$.
In fact, if $\widehat{\varphi}_\alpha$ is another such cocycle, then $\varphi_\alpha-\widehat{\varphi}_\alpha$ is a homogeneous $2$-cocycle representing the trivial
bounded coclass in $H^2_b(\G,\R)$. As a consequence, $\varphi_\alpha-\widehat{\varphi}_\alpha=\overline{\delta}^1(q)$ for some $q\in Q^h(\G,\R)$
such that $[\overline{\delta}^1(q)]=0$ in $EH^2_b(\G,\R)$. By Corollary~\ref{decomp:cor} we thus get that $q$ is a homomorphism, so 
$\overline{\delta}^1(q)=0$ and $\varphi_\alpha=\widehat{\varphi}_\alpha$.
\end{proof}

We are now ready to prove the aforementioned result, which first appeared in~\cite{Bua1}:

\begin{thm}\label{Bua:thm}
Let $\varphi\colon \G\to H$ be a surjective homomorphism. Then the induced map
$$
H^2_b(\varphi)\colon H^2_b(H,\R)\to H^2_b(\G,\R)$$
is injective.
\end{thm}
\begin{proof}
 Let $\alpha\in\ker H^2_b(\varphi)$. By Theorem~\ref{homo:eu} there exists a homogeneous bounded $2$-cocycle
 $c$ such that $\alpha=[c]$ in $H^2_b(H,\R)$. The pull-back $\varphi^*c\in \overline{Z}^2_b(\G,\R)$ is a homogeneous
 bounded cocycle representing the trivial class in $H^2_b(\G,\R)$, so $\varphi^*c=0$ in $\overline{Z}^2_b(\G,\R)$
 by uniqueness of homogeneous representatives (see again Theorem~\ref{homo:eu}. Since $\varphi$ is surjective, this readily implies that $c=0$, so
 $\alpha=0$.
\end{proof}

Recall that replacing a quasimorphism with the associated homogeneous one  at most doubles its defect. Using this, it is not difficult
to show that the map $H^2_b(\varphi)\colon H^2_b(H,\R)\to H^2_b(\G,\R)$ induced by an epimorphism $\varphi$ not only is injective, but it also
satisfies $\|H^2_b(\varphi)(\alpha)\|/2\leq \|\alpha\|\leq H^2_b(\varphi)(\alpha)$ for every $\alpha\in H^2_b(H,\R)$. As mentioned above, an independent proof
of Theorem~\ref{Bua:thm} due to Huber shows that $H^2_b(\varphi)$ is in fact an isometric embedding~\cite[Theorem 2.14]{Huber}.

\section{The image of the comparison map}
In Section~\ref{extension:sec} we have interpreted elements in $H^2(\G,R)$ as equivalence classes of 
central extensions of $\G$ by $R$. Of course, one may wonder whether elements of $H^2(\G,R)$
admitting bounded representatives represent peculiar central extensions. 
It turns out that this is indeed the case: bounded classes represent central extensions 
which are quasi-isometrically trivial (i.e.~quasi-isometrically equivalent to product extensions).
Of course, in order to give a sense to this statement we need to restrict our attention to the case
when $\G$ is finitely generated, and $R=\matZ$ (in fact, one could require a bit less, and suppose just that $R$ is a finitely generated abelian group). 
Under these assumptions, let us consider the central extension
$$
\xymatrix{
1 \ar[r] & \matZ \ar[r] & 
\G' \ar[r]^\pi & \G \ar[r] & 1\ ,
}
$$
and the following condition:
\begin{itemize}
 \item[(*)] 
there exists a quasi-isometry $q\colon \G' \to \G\times\matZ$ which makes the following diagram commute:
$$
\xymatrix{
\G'\ar[r]^\pi \ar[d]_q & \G\ar[d]^{\rm Id}\\
\G\times\matZ \ar[r] & \G
}
$$
where the horizontal arrow on the bottom represents the obvious projection.
\end{itemize}
Condition (*) is equivalent to the existence of a Lipschitz section
$s\colon \G\to \G'$ such that $\pi\circ s={\rm Id}_\G$ (see {e.g.}~\cite[Proposition 8.2]{klelee}).
Gersten proved that a sufficient condition for a
central extension
to satisfy condition (*) is that its coclass in $H^2(\G,\matZ)$ admits a \emph{bounded}
representative (see \cite[Theorem 3.1]{Ger:last}). Therefore,
the image of the comparison map $H^2_b(\G,\matZ)\to H^2(\G,\matZ)$ determines extensions
which satisfy condition (*). As far as the author knows, it is not known whether the converse implication
holds, i.e.~if every central extension satisfying (*) is represented
by a cohomology class lying in the image of the comparison map.

Let us now consider the case with real coefficients, and list some results about the surjectivity of the comparison map:
\begin{enumerate}
 \item If $\G$ is the fundamental group of a closed $n$-dimensional locally symmetric
space of non-compact type, then the comparison map $c^n\colon H^n_b(\G,\R)\to H^n(\G,\R)\cong \R$ is surjective~\cite{Lafont-Schmidt, michelleprimo}.
\item If $\G$ is word hyperbolic, then the comparison map $c^n\colon H^n_b(\G,V)\to H^n(\G,V)$
is surjective for every $n\geq 2$ when $V$ is any normed $\R[\G]$ module
or any finitely generated abelian group (considered as a trivial $\matZ[\G]$-module)~\cite{Mineyev1}.
\item
If $\G$ is finitely presented and the comparison map $c^2\colon H^2_b(\G,V)\to H^2(\G,V)$
is surjective for every normed $\R[\G]$-module, then $\G$ is word hyperbolic~\cite{Mineyev2}.
\item The set of closed $3$-manifolds $M$ for which the comparison map
$H^\bullet_b(\pi_1(M),\R)\to H^\bullet(\pi_1(M),\R)$ is surjective in every degree is completely characterized
in~\cite{FujiSoma}.
\end{enumerate}
Observe that points~(2) and (3) provide a characterization of word hyperbolicity
in terms of bounded cohomology. Moreover, putting together item (2) with the discussion above, we see that
every extension of a word hyperbolic group by $\matZ$ is quasi-isometrically trivial.

In this section we have mainly dealt with the case of real coefficients. However, for later purposes we point out  the following result~\cite[Theorem 15]{Mineyev1}:

\begin{prop}\label{surjiffsurj}
 Let $\G$ be any group, and take an element $\alpha\in H^n(\G,\matZ)$ such that
 $\|\alpha_\R\|_\infty<+\infty$, where $\alpha_\R$ denotes the image of $\alpha$ via the change of coefficients
 homomorphism induced by the inclusion $\matZ\to \R$. Then $\|\alpha\|_\infty< +\infty$.
\end{prop}
\begin{proof}
 Let $\varphi\in C^{n}(\G,\matZ)^\G$ be a representative of $\alpha$, which we also think of as a \emph{real} cocycle. 
 %Then, if we denote by $\varphi_\R$ the cocycle
 %$\varphi$, thought of as a \emph{real} cocycle, we have
 By our assumptions, we have
 $$
 \varphi=\varphi'+\delta\psi\ ,
 $$
 where $\varphi'\in C^{n}_b(\G,\matZ)^\G$ and $\psi\in C^{n-1}(\G,\R)^\G$. Let us denote by $\overline{\psi}$ the cochain
 defined by $\overline{\psi}(g_0,\ldots,g_{n-1})=\lfloor \psi(g_0,\ldots,g_{n-1})\rfloor$, where
for every $x\in \R$ we denote by $\lfloor x\rfloor$ the largest integer which does not exceed
$x$. Then $\overline{\psi}$ is still $\G$-invariant, so $\overline{\psi}\in C^{n-1}(\G,\matZ)^\G$. 
Moreover, we obviously have $\psi-\overline{\psi}\in C^{n-1}_b(\G,\R)$. Let us now consider the equality
$$
\varphi-\delta \overline{\psi}=\varphi'+\delta(\psi-\overline{\psi})\ .
$$
The left-hand side is a cocycle with integral coefficients, while the right-hand side
is bounded. Therefore, the cocycle $\varphi-\delta \overline{\psi}$ is a bounded representative
of $\alpha$, and we are done.
\end{proof}

\begin{cor}\label{surjiffsurj2}
 Let $\G$ be any group. Then, the comparison map $c_\matZ^n\colon H^n_b(\G,\matZ)\to H^n(\G,\matZ)$ is surjective if and only
 if the comparison map $c_\R^n\colon H^n_b(\G,\R)\to H^n(\G,\R)$ is.
\end{cor}
\begin{proof}
Suppose first that $c_\matZ^n$ is surjective. Then, the image of $c_\R$ is a vector subspace of $H^n(\G,\R)$
which contains the image of $H^n(\G,\matZ)$ via the change of coefficients homomorphism. By the Universal Coefficients Theorem,
this implies that $c_\R$ is surjective. The converse implication readily follows from Proposition~\ref{surjiffsurj}.
\end{proof}

\section{Further readings}\label{further:2}

\subsection*{Bounded cohomology in degree one}
We have seen at the beginning of this chapter that $H^1_b(\G,\mathbb{R})=0$ for every group $\G$.
This result extends also to the case of non-constant coefficient modules, under suitable additional hypotheses.
In fact, it is proved in~\cite{Johnson} (see also~\cite[Section 7]{Noskov}) that $H^1_b(\G,V)=0$
for every group $\G$ and every \emph{reflexive} normed $\R[\G]$-module $V$.
Things get much more interesting when considering non-reflexive coefficient modules.
In fact, it turns out that a group $\G$ is amenable (see Definition~\ref{amenable:def})
if and only if $H^1_b(\G,V)=0$ for every dual normed $\R[\G]$-module $V$ (see Corollary~\ref{am:char:cor}). 
Therefore, for every non-amenable group $\G$ there exists a (dual) normed $\R[\G]$-module $\G$
such that $H^1_b(\G,V)\neq 0$ (this module admits a very explicit description, see Section~\ref{Johnson:sec}).
In fact, even when $\G$ is amenable, there can exist (non-dual) Banach $\G$-modules $V$ such that
$H^1_b(\G,V)\neq 0$: in~\cite[Section 7]{Noskov} a Banach $\matZ$-module $\mathcal{A}$ is constructed such that
$H^1_b(\matZ,\mathcal{A})$ is infinite-dimensional.

\subsection*{Quasimorphisms}
The study of quasimorphisms is nowadays a well-developed research field. The interest towards this topic was probably renovated by the pioneering work of Bavard~\cite{Bavard},
where quasimorphisms were used to estimate the stable commutator length of a commutator $g\in [\G,\G]<\G$ (i.e.~the limit
$\lim_{n\to\infty} c(g^n)/n$, where 
$c(g)$ is 
the minimal number of factors in a product of commutators equaling $g$).  
We refer the reader to Calegari's book~\cite{scl} for an introduction to quasimorphisms, with a particular attention payed to their relationship with the stable commutator length.
Among the most recent developments of the theory of quasimorphisms is it maybe worth mentioning two distinct generalizations of the notion of quasimorphism that deal also with non-commutative targets,
due respectively to Hartnick--Schweitzer~\cite{HS} and to Fujiwara--Kapovich~\cite{FuKa}.

There exist many papers devoted to the construction of non-trivial quasimorphisms for several classes of groups. Corollary~\ref{infinitedim:cor}
implies that 
 a large class of groups has infinite-dimensional second bounded cohomology module. 
In fact, this holds true for every group belonging to one of the following families: 
\begin{enumerate}
 \item non-elementary Gromov hyperbolic groups~\cite{EpsteinFuji};
\item groups acting properly discontinuously via isometries on Gromov hyperbolic spaces with limit set
consisting of at least three points~\cite{Fujiwara1};
\item groups admitting a non-elementary \emph{weakly proper discontinuous} action on 
a Gromov hyperbolic space~\cite{BestFuji} (see also~\cite{Hamen});
\item acylindrically hyperbolic groups~\cite{HullOsin} (we refer to~\cite{Osin:ac} for the definition of acylindrically hyperbolic group);
\item groups having infinitely many ends~\cite{Fujiwara2}.
\end{enumerate}
(Conditions (1), (2), (3) and (4) define bigger and bigger classes of groups: the corresponding cited papers generalize one the results
of the other.) 
An important application of case (3) is that every subgroup of the mapping
class group of a compact surface either is virtually abelian, or has infinite-dimensional second bounded cohomology.

A nice geometric interpretation of (homogeneous) quasimorphisms is given in~\cite{Manning}, where it is proved that, if $\G$ admits a \emph{bushy} quasimorphism,
then $EH^2_b(\G,\R)$ is infinite-dimensional. However, contrary to what has been somewhat expected, 
there exist lattices $\G$ in non-linear Lie groups for which $EH^2_b(\G,\R)$ is of finite non-zero dimension~\cite{Mann-Monod}.

\subsection*{Quasimorphisms with non-trivial coefficients}
Quasimorphisms with twisted coefficients have also been studied by several authors.
For example, for acylindrically hyperbolic groups, non-trivial quasimorphisms with coefficients in non-trivial Banach $\G$-modules
were constructed by Hull and Osin in~\cite{HullOsin} (see also~\cite{Rolli2} for a construction which applies to a more restricted class of groups, but is closer to the one described in Proposition~\ref{rolli:prop}).
Moreover, Bestvina, Bromberg and Fujiwara recently extended Hull and Osin's results to show that
the second bounded cohomology of any acylindrically hyperbolic group is infinite dimensional  when taking
(possibly non-trivial) coefficients in \emph{any} uniformly convex Banach space~\cite{BeBrFu}. 

Monod and Shalom showed the importance of bounded cohomology with
coefficients in $\ell^2(\G)$ in the study of rigidity of $\G$~\cite{MonShal0,MonShal}, and
proposed the condition $H^2_b(\G,\ell^2(\G))\neq 0$ 
as a cohomological definition of
negative curvature for groups. More in general, bounded cohomology with
coefficients in $\ell^p(\G)$, $1\leq p<\infty$
has been widely studied as a powerful
tool to prove (super)rigidity results (see e.g.~\cite{Hamen,CFI}).

\subsection*{An exact sequence for low-dimensional bounded cohomology}
Theorem~\ref{Bua:thm} is part of the following more general result:

\begin{thm}[\cite{Bua1, Bua2}]\label{bucompl}
 An exact sequence of groups
 $$
 1\tto{} K\tto{} \G\tto{} H\tto{} 1
 $$
 induces an exact sequence of vector spaces
 $$
 0\tto{} H^2_b(H,\R)\tto{} H^2_b(\G,\R)\tto{} H^2_b(K,\R)^H \tto{} H^3_b(H,\R)\tto{} 0 ,
 $$
 where $H^2_b(K,\R)^H$ denotes the subspace of invariants with respect to the natural action of $H$ induced
 by conjugation.
\end{thm}

Let us just mention an interesting application of this theorem, pointed out to us by M.~Boileau. 
Let $M$ by a closed hyperbolic $3$-manifold fibering over the circle with fiber $\Sigma_g$,
where $\Sigma_g$ is a closed oriented surface of genus $g\geq 2$. 
Then $M$ is obtained by gluing the boundary components of $\Sigma_g \times [0,1]$ via the homeomorphism
$(x,0)\mapsto (f(x),0)$, where $f$ is an orientation-preserving pseudo-Anosov homeomorphism of $\Sigma_g$.
If $\G=\pi_1(M)$ and $\G_g=\pi_1(\Sigma_g)$, then we have an exact sequence
$$
 1\tto{} \G_g\tto{} \G\tto{} \mathbb{Z}\tto{} 1\ .
 $$
 Since $\mathbb{Z}$ is amenable we have $H^2_b(\mathbb{Z},\R)=H^3_b(\mathbb{Z},\R)=0$ (see the next chapter), so 
Theorem~\ref{bucompl} implies that
$$
H^2_b(\G,\R)\ \cong\  H^2_b(\G_g,\R)^\mathbb{Z}\ .
$$
Now $M$ and $\Sigma_g$ are aspherical, so we have canonical isomorphisms $H^2_b(\G,\R)\cong H^2_b(M,\R)$, $H^2_b(\G_g,\R)\cong H^2_b(\Sigma_g,\R)$
(see Chapter~\ref{bounded:space:chap}), and Bouarich's result says that
$$
H^2_b(M,\R)\ \cong\ H^2_b(\Sigma_g,\R)^{f_*}\ ,
$$
where $f_*$ denotes the map induced by $f$ on bounded cohomology. This implies the somewhat counterintuitive fact that the action induced by a pseudo-Anosov homeomorphism 
of $\Sigma_g$ onto $H^2_b(\Sigma_g,\R)$ admits 
a huge invariant subspace.

\subsection*{Surjectivity of the comparison map in the relative context}
We have mentioned the fact proved by Mineyev that a finitely presented group $\G$  is Gromov hyperbolic if and only if the comparison map 
$H^2_b(\G,V)\to H^2(\G,V)$ is surjective for every normed $\mathbb{R}[\G]$-module $V$. A relative version of this result has been recently established by Franceschini~\cite{Franceschini2},
who proved an analogous result for Gromov hyperbolic pairs (see also~\cite{Yaman}).

\chapter{Amenability}\label{amenable:chap}
The original definition of amenable group is due to von Neumann~\cite{VN}, who introduced the class of amenable groups
while studying the Banach-Tarski paradox. As usual, we will restrict our attention to the definition of amenability in the context 
of discrete groups, referring the reader e.g.~to~\cite{Pier, Paterson} for  a thorough account on  amenability in the wider context of locally compact groups.

Let $\G$ be a group. We denote by $\ell^\infty(\G)=C^0_b(\G,\R)$ the space of bounded real functions on $\G$, endowed with the usual structure
of normed $\R[\G]$-module (in fact, of Banach $\G$-module). A \emph{mean} on $\G$ is a map $m\colon \ell^\infty(\G)\to \R$ satisfying the following properties:
\begin{enumerate}
\item
$m$ is linear;
\item
if ${\bf 1}_\G$ denotes the map taking every element of $\G$ to $1\in\R$, then $m({\bf 1}_\G)=1$;\
\item
$m(f)\geq 0$ for every non-negative $f\in\ell^\infty(\G)$.
\end{enumerate}
If condition (1) is satisfied, then conditions (2) and~(3) are equivalent to 
\begin{enumerate}
\item[(2+3)] $\inf_{g\in \G} f(g)\leq m(f)\leq \sup_{g\in \G} f(g)$ for every $f\in\ell^\infty(\G)$.
\end{enumerate}
In particular, the operator norm of a  mean on $\G$, when considered as a functional on $\ell^\infty(\G)$, is equal to $1$.
A mean $m$ is \emph{left-invariant} (or simply \emph{invariant}) if it satisfies the following additional condition:
\begin{enumerate}
\item[(4)]
$m(g\cdot f)=m(f)$ 
for every $g\in \G$, $f\in\ell^\infty(\G)$.
\end{enumerate}

\begin{defn}\label{amenable:def}
A group $\G$ is \emph{amenable} if it admits an invariant mean.
\end{defn}

The following lemma describes some equivalent definitions of amenability that will prove useful in the sequel.

\begin{lemma}\label{equivalent:amenable}
Let $\G$ be a group. Then,
the following conditions are equivalent:
\begin{enumerate}
\item
$\G$ is amenable;
\item
there exists a non-trivial left-invariant continuous functional $\varphi\in\ell^\infty(\G)'$; 
\item
$\G$ admits a left-invariant finitely additive probability measure.
\end{enumerate}
\end{lemma}
\begin{proof}
If $A$ is a subset of $\G$, we denote by $\chi_A$ the characteristic function of $A$.

(1) $\Rightarrow$ (2): If $m$ is an invariant mean on $\ell^\infty(\G)$, then the map $f\mapsto m(f)$ defines
the desired functional on $\ell^\infty(\G)$.

(2) $\Rightarrow$ (3): Let $\varphi\colon \ell^\infty(\G)\to\R$ be a non-trivial continuous functional.
For every $A\subseteq \G$ we define a non-negative real number $\mu(A)$ as follows. For every partition
$\mathcal{P}$ of $A$ into a finite number of subsets $A_1,\ldots,A_n$, we set 
$$
\mu_{\mathcal{P}} (A)=|\varphi(\chi_{A_1})|+\ldots+|\varphi(\chi_{A_n})|\ .
%\quad
%\mu (A)=\sup_{\mathcal{P}} \mu_\mathcal{P} (A)\  .
$$
Observe that, if $\varepsilon_i$ is the sign of $\varphi (\chi_{A_i})$, then
$\mu_{\mathcal{P}}(A)=\varphi(\sum_i \varepsilon_i\chi_{A_i})\leq \|\varphi\|$, so 
$$
\mu (A)=\sup_{\mathcal{P}} \mu_\mathcal{P} (A)
$$
is a finite non-negative number for every $A\subseteq \G$. 
By the linearity of $\varphi$, if $\mathcal{P}'$ is a refinement of $\mathcal{P}$,
then $\mu_{\mathcal{P}}(A)\leq \mu_{\mathcal{P}'} (A)$, and this easily implies that $\mu$ is 
a finitely additive measure on $\G$. Recall now that the set of characteristic functions generates a dense subspace of
$\ell^{\infty}(\G)$. As a consequence, since $\varphi$ is non-trivial, there exists a subset $A\subseteq \G$ such that
$\mu(A)\geq |\varphi(\chi_A)|>0$. In particular, we have $\mu(\G)>0$, so after rescaling we may assume that
$\mu$ is a probability measure on $\G$. The fact that $\mu$ is $\G$-invariant is now a consequence of the 
$\G$-invariance of $\varphi$.

(3) $\Rightarrow$ (1): Let $\mu$ be a left-invariant finitely additive measure on $\G$,
and let $Z\subseteq \ell^\infty(\G)$ be the subspace generated by the functions
$\chi_A$, $A\subseteq \G$. 
Using the finite additivity of $\mu$, it is not difficult to show that there exists a linear functional 
$m_Z\colon Z\to \R$ such that $m_Z(\chi_A)=\mu(A)$ for every $A\subseteq \G$. Moreover,
it is readily seen that 
\begin{equation}\label{infme}
\inf_{g\in \G} f(g)\leq m_Z(f)\leq \sup_{g\in \G} f(g)\qquad {\rm for\ every}\ f\in Z
\end{equation}
(in particular, $m_Z$ is continuous).
Since $Z$ is dense in $\ell^\infty(\G)$, the functional $m_Z$ uniquely extends to a continuous functional $m\in\ell^\infty(\G)'$
that satisfies condition~\eqref{infme} for every $f\in\ell^\infty(\G)$. 
Therefore, $m$ is a mean,
and the $\G$-invariance of $\mu$ implies that $m_Z$, whence $m$, is also invariant. 
\end{proof}

The action of $\G$ on itself by \emph{right} translations endows $\ell^\infty(\G)$ also with the structure of a right Banach $\G$-module, and 
it is well-known that the existence of a left-invariant mean on $\G$ implies the existence of a bi-invariant mean on $\G$. 
Indeed, let $m\colon \ell^\infty(\G)\to \R$ be a left-invariant mean, and take an element $f\in\ell^\infty(\G)$. For every $a\in\G$ we denote by $R_a$ (resp.~$L_a$) the right (resp.~left)
multiplication by $a$.
For every $g_0\in\G$ the map $f_{g_0}\colon \G\to\R$, $f_{g_0}(g)=f(g_0g^{-1})$ is bounded, so we may take its mean
$\overline{f}(g_0)=m(f_{g_0})$. Then, by taking the mean of $\overline{f}$ we set $\overline{m}(f)=m(\overline{f})$.
It is immediate to verify that $\overline{m}\colon \ell^\infty(\G)\to\R$ is itself
a mean. Let us now prove that $\overline{m}$ is left-invariant. So let $f'=f\circ L_a$ and observe that
$$
f'_{g_0}(g)=f(ag_0g^{-1})=f_{ag_0}(g)\ ,
$$
so 
$$
\overline{f'}(g_0)=m(f_{ag_0})=\overline{f}(ag_0)=(\overline{f}\circ L_a)(g_0)
$$
and
$$
\overline{m}(f')=m(\overline{f'})=m(\overline{f}\circ L_a)=m(\overline{f})=\overline{m}(f)\ ,
$$ as desired.
Moreover, if now $f'=f\circ R_a$, then 
$$
f'_{g_0}(g)=f(g_0g^{-1}a)=f(g_0(a^{-1}g)^{-1})=f_{g_0}\circ L_{a^{-1}}\ ,
$$
so $\overline{f'}(g_0)=m(f'_{g_0})=m(f_{g_0}\circ L_{a^{-1}})=m(f_{g_0})=\overline{f}(g_0)$
and a fortiori 
$$
\overline{m}(f')=m(\overline{f'})=m(\overline{f})=\overline{m}(f)\ .
$$
We can thus conclude that a group is amenable if and only if it admits a bi-invariant mean. However, we won't use this fact in the sequel.

\section{Abelian groups are amenable}
Finite groups are obviously amenable: an invariant mean $m$ is obtained by setting $m(f)=(1/|\G|)\sum_{g\in \G} f(g)$
for every $f\in\ell^\infty(\G)$. Of course, we are interested in finding less obvious
examples of amenable groups.
A key result in the theory of amenable groups is the following theorem, which is originally due
to von Neumann:
 
 \begin{thm}[\cite{VN}]\label{abelian:amenable}
 Every abelian group is amenable.
 \end{thm}
\begin{proof}
We follow here the proof given in~\cite{Paterson},
which is based on the Markov-Kakutani Fixed Point Theorem. 

Let $\G$ be an abelian group, and
let $\ell^\infty(\G)'$ be the topological dual
of $\ell^\infty(\G)$, endowed with the weak-$*$ topology. We consider the subset $K\subseteq \ell^\infty(\G)'$
of (not necessarily invariant) means on $\G$, i.e.~we set
$$
K=\{\varphi\in\ell^\infty(\G)'\, |\, \varphi({\bf 1}_\G)=1\, ,\  \varphi(f)\geq 0 \ {\rm for\ every}\ f\geq 0\}\ .
$$
The set $K$ is non-empty (it contains, for example, the evaluation on the identity of $\G$). Moreover, it is  closed, convex, and contained in the closed ball of radius one in $\ell^\infty(\G)'$.
Therefore, $K$ is compact by the Banach-Alaouglu Theorem. 

For every $g\in \G$ let us now consider the map $L_g\colon \ell^\infty(\G)'\to \ell^\infty(\G)'$ defined by
$$
L_g(\varphi) (f)=\varphi(g\cdot f)\ .
$$
We want to show that the $L_g$ admit a common fixed point in $K$. 
To this aim, one may apply the Markov-Kakutani Theorem cited above, which
we prove here (in the case we are interested in) for completeness.

We fix $g\in\G$ and
 show that $L_g$ admits a fixed point in $K$. Take $\varphi\in H$. For every $n\in\matN$ we set
$$
\varphi_n=\frac{1}{n+1}\sum_{i=0}^n L_g^i(\varphi)\ .
$$
By convexity, $\varphi_n\in K$ for every $n$, 
so the set 
$$
\Omega_n=\overline{\bigcup_{i\geq n} \{\varphi_i\}}\ \subseteq\ K
$$
is compact. Since the family $\{\Omega_n,\, n\in\mathbb{N}\}$ obviously satisfies the finite intersection property,
there exists a point $\overline{\varphi}\in\bigcap_{i\in\mathbb{N}} \Omega_n$. Let us now fix $f\in \ell^\infty(\G)$ and let $\varepsilon>0$ be given.
We choose $n_0\in\mathbb{N}$ such that $n_0\geq 6\|f\|/\varepsilon$. Since $\overline{\varphi}\in \Omega_{n_0}$, there exists $n_1\geq n_0$
such that $|\varphi_{n_1}(f)-\overline{\varphi}(f)|\leq\varepsilon/3$ and
$|\varphi_{n_1}(g\cdot f)-\overline{\varphi}(g\cdot f)|\leq\varepsilon/3$. Moreover, we have
\begin{align*}
|\varphi_{n_1}(g\cdot f)-\varphi_{n_1}(f)| &=
|(L_g(\varphi_{n_1})-\varphi_{n_1})(f)|
=\frac{|(L_g^{n_1+1}(\varphi)-\varphi)(f)|}{n_1+1}\\ & \leq
\frac{ 2\|f\|}{n_1+1}\leq \frac{ 2\|f\|}{n_0}\leq \frac{\varepsilon}{3}\ ,
\end{align*}
so that
$$
|\overline{\varphi}(g\cdot f)-\overline{\varphi}(f)|\leq 
|\overline{\varphi}(g\cdot f)-\varphi_{n_1}(g\cdot f)|+
|\varphi_{n_1}(g\cdot f)-\varphi_{n_1}(f)|+|\overline{\varphi}(f)-\varphi_{n_1}(f)|\leq\varepsilon
$$
and $\overline{\varphi}(g\cdot f)=\overline{\varphi}(f)$ since $\varepsilon$ was arbitrarily chosen.
%re exists a subnet $\varphi_{n_i}$, $i\in I$, tending to $\overline{\varphi}\in H$. Recall
%that $\|\varphi\|=\|L_g^{n_i+1}(\varphi)\|=1$, so for every
%$f\in\ell^\infty(\G)$ we have
%$$
%|\varphi_{n_i}(g\cdot f)-\varphi_{n_i}(f)|=
%|(L_g(\varphi_{n_i})-\varphi_{n_i})(f)|
%| \varphi(L_g(x(n_i)))-(x(n_i)))| 
%=\frac{|(L_g^{n_i+1}(\varphi)-\varphi)(f)|}{n_i+1}\leq
%\frac{ 2\|f\|}{n_i+1}\ .
%$$
%By passing to the limit for
%$n_i$ tending to $\infty$, 
%we obtain that $L_g(\overline{\varphi})(f)=\overline{\varphi}(g\cdot f)=\overline{\varphi}(f)$
%for every $f\in\ell^\infty(\G)$, so
We have thus shown that
$L_g$ has a fixed point in $K$.

Let us now observe that $K$ is $L_g$-invariant for every $g\in \G$. Therefore, for any fixed $g\in\G$ the set $K_g$ of points of $K$
which are fixed by $L_g$ is non-empty. Moreover, it is easily seen that $K_g$ is 
closed (whence compact) and convex.
Since $\G$ is abelian, if $g_1,g_2$ are elements of $\G$, then
the maps $L_{g_1}$ and $L_{g_2}$ commute with each other, so $K_{g_1}$ is $L_{g_2}$-invariant. 
The previous argument (with $K$ replaced by $K_{g_1}$) shows that $L_{g_2}$ has a fixed point in $K_{g_1}$, so $K_{g_1}\cap K_{g_2}\neq \emptyset$.
One may iterate this argument to show that every finite intersection $K_{g_1}\cap\ldots\cap K_{g_n}$ is non-empty.
In other words, the family $K_g$, $g\in \G$ satisfies the finite intersection property,
so
$K_\G=\bigcap_{g\in \G} K_g$ 
is non-empty. But $K_\G$ is precisely the set of invariant means on $\G$. 
\end{proof}

\section{Other amenable groups}
Starting from abelian groups, it is possible to construct larger classes of amenable groups:

\begin{prop}\label{amenable:groups}
Let $\G,H$ be amenable groups. Then:
\begin{enumerate}
\item Every subgroup of $\G$ is amenable. 
\item If the sequence
$$
1\tto{} H\tto{} \G'\tto{} \G\tto{} 1
$$
is exact, then $\G'$ is  amenable (in particular, the direct product of a finite
number of amenable groups is amenable).
\item Every direct union of amenable groups is amenable.
\end{enumerate}
\end{prop}
\begin{proof}
(1): Let $K$ be a subgroup of $\G$. Let $S\subseteq \G$ be a set of representatives  for the set $K\backslash \G$ of right cosets of $K$ in $\G$.
For every $f\in \ell^\infty(K)$ we define $\hat{f}\in\ell^\infty(\G)$ by setting $\hat{f}(g)=f(k)$, where $g=ks$, $s\in S$, $k\in K$. 
If $m$ is an invariant mean on $\G$, then we set $m_K(f)=m(\hat{f})$. It is easy to check that $m_K$ is an invariant mean on $K$, so $K$ is amenable.

(2): Let $m_H,m_\G$ be invariant means on $H,\G$ respectively.
For  every $f\in \ell^\infty(\G')$ we construct a map $f_\G\in\ell^{\infty}(\G)$ as follows. For $g\in \G$, we take $g'$ in the preimage of $g$ in $\G'$,
and we define $f_{g'}\in \ell^\infty(H)$ by setting $f_{g'}(h)=f(g'h)$. Then, we set $f_\G(g)=m_H(f_{g'})$. Using that $m_H$ is $H$-invariant
one can show that $f_\G(g)$ does not depend on the choice of $g'$, so $f_\G$ is well defined. We obtain the desired mean
on $\G'$ by setting $m(f)=m_\G(f_\G)$.

(3): 
Let $\G$ be the direct union of the amenable groups $\G_i$, $i\in I$.
%As in the proof of Theorem~\ref{abelian:amenable}, 
For every $i$ we consider the set 
$$
K_i=\{\varphi\in\ell^\infty(\G)'\, |\, \varphi({\bf 1}_\G)=1\, ,\  \varphi(f)\geq 0 \ {\rm for\ every}\ f\geq 0,\ \varphi\ {\rm is}\ \G_i-{\rm invariant} \}\ .
$$
If $m_i$ is an invariant mean on $\G_i$, then the map $f\mapsto m_i(f|_{\G_i})$ defines
an element of $K_i$. Therefore, each $K_i$ is non-empty. Moreover, the fact that 
the union of the $\G_i$ is direct implies that the family $K_i$, $i\in I$ satisfies the finite intersection property. Finally, each $K_i$ is closed and compact in the weak-$*$
topology on $\ell^\infty(\G)'$, so the intersection $K=\bigcap_{i\in I} K_i$
is non-empty. But  $K$ is precisely the set of $\G$-invariants means on $\G$,
whence the conclusion.
\end{proof}

The class of \emph{elementary amenable groups} is the smallest class of groups which contains all finite and all abelian groups,
and is closed under the operations of taking subgroups, forming quotients, forming extensions, and taking direct unions~\cite{Day}.
For example, virtually solvable groups are elementary amenable.
Proposition~\ref{amenable:groups} (together with Theorem~\ref{abelian:amenable}) shows that every elementary amenable group is indeed amenable.
However, any group of intermediate growth is amenable~\cite{Paterson} but is not elementary amenable~\cite{Chou}. In particular, there exist amenable groups which 
are not virtually solvable.

\begin{rem}\label{amenable:prod}
 We have already noticed that, if $\G_1,\G_2$ are amenable groups, then
so is $\G_1\times \G_2$. More precisely, in the proof of Proposition~\ref{amenable:groups}-(2) we have shown how to construct an invariant
mean $m$ on $\G_1\times \G_2$ starting from invariant means
$m_i$ on $\G_i$, $i=1,2$. Under the identification of means with finitely additive
probability measures, the mean $m$ corresponds to 
the product measure $m_1\times m_2$, so it makes sense to say that $m$ is the product
of $m_1$ and $m_2$. If $\G_1,\ldots,\G_n$ are amenable groups with invariant means
$m_1,\ldots,m_n$,
then we denote by $m_1\times \ldots \times m_n$ the invariant mean on $\G$
inductively defined by the formula $m_1\times\ldots\times m_n=m_1\times (m_2\times
\ldots \times m_n)$.
\end{rem}

%Point (2) of the previous proposition also implies that the direct product of a finite number of amenable subgroups is amenable. In particular, if $\G$ is amenable, then so is $\G^n$. In fact, if $\mu$ is an element of $\mathbb{P}_f (G)$, then
%the product measure $\mu^{\otimes n}$ is an element of $\mathbb{P}_f(G^n)$. We point out that this argument is not really different from the one
%described in the previous proposition: in the case of the trivial extesion $1\to G\to G\times G\to G\to 1$, the invariant mean
%on $G\times G$ constructed from an invariant mean $m$ on $G$ corresponds to the product of the finitely additive measures associated to $m$.

\section{Amenability and bounded cohomology}
In this section we show that amenable groups are somewhat invisible to bounded cohomology with coefficients in many normed $\R[\G]$-modules
(as already pointed out in Proposition~\ref{h2zz}, things are quite different in the case of $\matZ[\G]$-modules). We say that an $\R[\G]$-module $V$  is a \emph{dual} normed $\R[\G]$-module
if $V$ is isomorphic (as a normed $\R[\G]$-module) to the topological dual of some normed $\R[\G]$-module $W$.  
In other words, $V$ is the space of bounded linear functionals on the normed $\R[\G]$-module $W$, endowed with the action 
defined by $g\cdot f (w)=f(g^{-1}w)$, $g\in \G$, $w\in W$.

\begin{thm}\label{amenable:vanishing}
Let $\G$ be an amenable group, and let $V$ be a dual normed $\R[\G]$-module. Then $H^n_b(\G,V)=0$ for every $n\geq 1$.
\end{thm}
\begin{proof}
Recall that the bounded cohomology $H^\bullet_b(\G,V)$ is defined as the cohomology of the $\G$-invariants
of the homogeneous complex
$$
0\tto{} C^0_b(\G,V)\tto{\delta^0} C_b^1(\G,V)\to\ldots
\to C^n_b(\G,V)\to\ldots 
$$
Of course, the cohomology of the complex $C^\bullet_b(\G,V)$ of possibly non-invariant cochains vanishes in positive degree, since
for $n\geq 0$
the maps
$$
k^{n+1}\colon C_b^{n+1}(\G,V)\to C_b^n(\G,V)\, ,\quad k^{n+1}(f)(g_0,\ldots,g_n)=f(1,g_0,\ldots,g_n)
$$
provide a  homotopy between the identity and the zero map of $C^\bullet_b(\G,V)$ in positive degree.
%are such that $k^{n+1}\delta^n+\delta^{n-1}k^n={\rm Id}_{\rm C^b_n(\G,V)}$ for every $n\geq 1$.
In order to prove the theorem, we are going to show that, under the assumption that $\G$ is amenable, a similar homotopy can be defined on the complex of invariant cochains. Roughly speaking,
while $k^\bullet$ is obtained by coning over the identity of $\G$,
we will define a $\G$-invariant homotopy $j^\bullet$ by averaging the cone operator over all the elements of $\G$.

Let us fix an invariant mean $m$ on $\G$, and
% To this aim, using the amenability of $\G$, we modify 
%$k^\bullet$ into a $\G$-equivariant map $j^\bullet$, which is defined as folllows. 
let $f$ be an element of $C_b^{n+1}(\G,V)$.
Recall that $V=W'$ for some $\R[\G]$-module
$W$, so $f(g,g_0,\ldots,g_n)$ is a bounded functional on $W$ for every $(g,g_0,\ldots,g_n)\in \G^{n+2}$. 
For every $(g_0,\ldots,g_n)\in \G^{n+1}$, $w\in W$, we consider the function
$$
f_w\colon \G\to\R\, ,\quad f_w(g)=f(g,g_0,\ldots,g_n)(w)\ .
$$
It follows from the definitions that $f_w$ is an element of $\ell^\infty(\G)$, so we may set
$(j^{n+1}(f))(g_0,\ldots,g_n) (w)=m(f_w)$. It is immediate to check that this formula
defines a continuous functional on $W$, whose norm is bounded in terms of the norm
of $f$. In other words, $j^{n+1}(f)(g_0,\ldots,g_n)$ is indeed an element of $V$, and the map
 $j^{n+1}\colon C^{n+1}_b(\G,V)\to C^n_b(\G,V)$ is bounded. The $\G$-invariance of the mean $m$
 implies that $j^{n+1}$ is $\G$-equivariant. The collection of maps $j^{n+1}$, $n\in\matN$ provides
 the required partial $\G$-equivariant homotopy between the identity and the zero map of
 $C^n_b(\G,V)$, $n\geq 1$.
\end{proof}

One may wonder whether the assumption that $V$ is a \emph{dual} normed $\R[\G]$-module is really necessary in Theorem~\ref{amenable:vanishing}.
It turns out that this is indeed the case: the vanishing of the bounded cohomology of $\G$ with coefficients
in \emph{any} normed $\R[\G]$-module is equivalent to the fact that $\G$ is finite
(see Theorem~\ref{finite:char}).

\begin{cor}\label{amenable:real:cor}
Let $\G$ be an amenable group. Then $H^n_b(\G,\R)=0$ for every $n\geq 1$.
\end{cor}

Recall from Chapter~\ref{quasimorphisms:chap} that the second bounded cohomology group with trivial real coefficients
is strictly related to the space of real quasimorphisms. Putting together Corollary~\ref{amenable:real:cor}
with Proposition~\ref{quasi:bounded:prop} and Corollary~\ref{decomp:cor} we get the following result, which
extends to amenable groups the characterization of real quasimorphisms on abelian groups given in Section~\ref{quasi:abelian:sec}:

\begin{cor}
Let $\G$ be an amenable group. Then every real quasimorphism on $\G$ is at bounded distance from a homomorphism.
Equivalently, every homogeneous quasimorphism on $\G$ is a homomorphism.
\end{cor}

From Corollary~\ref{infinitedim:cor} and Corollary~\ref{amenable:real:cor} we deduce that there exist many groups which are not amenable:

\begin{cor}
Suppose that the group $\G$ contains a non-abelian free subgroup. Then $\G$ is not amenable.
\end{cor}
\begin{proof}
We know from Corollary~\ref{infinitedim:cor} that non-abelian free groups have non-trivial second bounded cohomology
group with real coefficients, so they cannot be amenable. The conclusion follows from the fact that any subgroup of an amenable
group is amenable.
\end{proof}

It was a long-standing problem to understand whether the previous corollary could be sharpened into a characterization
of amenable groups as those groups which do not contain any non-abelian free subgroup. This question
is usually attributed to von Neumann, and appeared first in~\cite{Day}. Ol'shanskii answered von Neumann's question in the negative
in~\cite{Olsh}. The first examples of finitely presented groups which are not amenable but do not contain any non-abelian free group
are due to Ol'shanskii and Sapir~\cite{OlshSapir}.

\section{Johnson's characterization of amenability}\label{Johnson:sec}
We have just seen that the bounded cohomology of any amenable group with values in any dual coefficient module vanishes. 
In fact, also the converse implication holds. More precisely, amenability of $\G$ is implied by the vanishing of a specific coclass 
in a specific bounded cohomology
module (of degree one). 
%For the sake of clarity, in what follows we denote by $\ell^\infty(\G)$ the normed $\R[\G]$-module $C^0_b(\G,\R)$ of 
%bounded real functions on $\G$. 
We denote by $\ell^\infty(\G)/\R$ the quotient of $\ell^\infty(\G)$ by 
the (trivial) $\R[\G]$-submodule of constant functions. Such a quotient is itself endowed with the structure of a normed $\R[\G]$-module.
Note that the topological dual $(\ell^\infty(\G)/\R)'$ may be identified with the subspace of elements of $\ell^\infty(\G)'$ that vanish on
constant functions. For example, if $\delta_{g}$ is the element of $\ell^\infty(\G)'$ such that $\delta_g(f)=f(g)$, 
then for every pair $(g_0,g_1)\in \G^2$ the map $\delta_{g_0}-\delta_{g_1}$ defines an element in $(\ell^\infty(\G)/\R)'$.
For later purposes, we observe that the action of $\G$ on $\ell^\infty(\G)'$ is such that 
$$
g\cdot \delta_{g_0} (f)=\delta_{g_0}(g^{-1}\cdot f)=f(gg_0)=\delta_{gg_0} (f)\ .
$$
%For $g_0,g_1\in \G$, we denote by $\delta_{g_0}-\delta_{g_1}$ the element of the topological dual
%$(\ell^\infty/\R)'$ taking a class $[f]\in \ell^\infty(\G)/\R$ into $f(g_1)-f(g_2)$. 
Let us consider the element
$$
J\in C^1_b(\G,(\ell^\infty(\G)/\R)')\, ,\qquad J(g_0,g_1)=\delta_{g_1}-\delta_{g_0}\ .
$$
Of course we have $\delta J=0$. Moreover, for every $g,g_0,g_1\in \G$ we have 
$$
(g\cdot J)(g_0,g_1)=g(J(g^{-1}g_0,g^{-1}g_1))=g(\delta_{g^{-1}g_0}-\delta_{g^{-1}g_1})=\delta_{g_0}-\delta_{g_1}=J(g_0,g_1)
%((g\cdot J)(g_0,g_1))([f])=J(g^{-1}g_0,g^{-1}g_1)(g^{-1}\cdot [f])=f(g_0)-f(g_1)=J(g_0,g_1)([f])\ ,
$$
so $J$ is $\G$-invariant, and
defines an element $[J]\in H^1_b(\G,(\ell^\infty(\G)/\R)')$, called the \emph{Johnson class} of $\G$.

\begin{thm}\label{Johnson:class}
Suppose that the Johnson class of $\G$ vanishes. Then $\G$ is amenable.
\end{thm}
\begin{proof}
By Lemma~\ref{equivalent:amenable}, it is sufficient to show that the topological dual
$\ell^\infty(\G)'$ contains
a non-trivial $\G$-invariant element.

We keep notation from the preceding paragraph. If $[J]=0$, then $J=\delta \psi$ for some $\psi\in C^0_b(\G,(\ell^\infty(\G)/\R)')^\G$.
For every $g\in \G$ we denote by $\hat{\psi}(g)\in\ell^\infty(\G)'$ the pullback of $\psi(g)$ via the projection map
$\ell^\infty(\G)\to \ell^\infty(\G)/\R$. 
We consider the element $\varphi\in \ell^\infty(\G)'$ defined by 
$\varphi=\delta_1-\hat{\psi}(1)$.
%$$
%\varphi\colon \ell^\infty(\G)\to \R\, ,\qquad \varphi(f)=f(1)-\psi(1)([f])\ .
%$$
Since $\hat{\psi}(1)$ vanishes on constant functions,
we have $\varphi({\bf 1}_\G)=1$, so $\varphi$ is non-trivial, and we are left to show that $\varphi$ is $\G$-invariant.

The $\G$-invariance of $\psi$ implies the $\G$-invariance of $\hat{\psi}$, so
\begin{equation}\label{equi2:eq}
\hat{\psi}(g)=(g\cdot\hat{\psi})(g)=g\cdot(\hat{\psi}(1))\quad {\rm for\ every}\ g\in \G\ .
\end{equation}
From $J=\delta\psi$ we deduce that
$$
\delta_{g_1}-\delta_{g_0}=\hat{\psi}(g_1)-\hat{\psi}(g_0)\, \quad {\rm for\ every}\ g_0,g_1\in \G\ .
$$
In particular, we have
\begin{equation}\label{equi:eq}
\delta_g-\hat{\psi}(g)=\delta_1-\hat{\psi}(1)\, \quad {\rm for\ every}\ g \in \G\ .
\end{equation}
Therefore, using equations~\eqref{equi2:eq} and \eqref{equi:eq}, for every $g\in \G$ we get
$$
g\cdot\varphi=g\cdot (\delta_{1}-\hat{\psi}(1))=(g\cdot\delta_1)-g\cdot (\hat{\psi}(1))=\delta_g-\hat{\psi}(g)=
\delta_1-\hat{\psi}(1)=\varphi\ ,
$$
and this concludes the proof.
\end{proof}

Putting together Theorems~\ref{amenable:vanishing} and~\ref{Johnson:class} we get the following result, which characterizes amenability
in terms of bounded cohomology:

\begin{cor}\label{am:char:cor}
Let $\G$ be a group. Then the following conditions are equivalent:
\begin{itemize}
 \item 
$\G$ is amenable,
\item
$H^1_b(\G,V)=0$ for every 
dual normed $\R[\G]$-module $V$,
\item 
$H^n_b(\G,V)=0$ for every 
dual normed $\R[\G]$-module $V$ and every $n\geq 1$.
\end{itemize}
\end{cor}
\qed

\section{A characterization of finite groups via bounded cohomology}
As already mentioned above, amenability is not sufficient to guarantee the vanishing of bounded cohomology
with coefficients in any normed $\R[\G]$-module. Let $\ell^1(\G)=C^0_b(\G,\R)$ be the space of summable real functions on $\G$,
endowed with the usual structure of normed $\R[G]$-module.
Let us denote by $e\colon \ell^1(\G)\to \R$ the evaluation map such that
$e(f)=\sum_{g\in\G} f(g)$. Since $e$ is clearly $\G$-invariant, the space 
$\ell_0^1(\G)=\ker e$ is a normed $\R[G]$-space. What is more, since $\ell^1(\G)$
is Banach and $e$ is continuous, $\ell^1(\G)$ is itself a Banach $\G$-module.

With an abuse, for $g\in \G$ we denote by $\delta_g\in \ell^1(\G)$ the
characteristic function of the singleton $\{g\}$.
Let us now consider the element
$$
K\in C^1_b(\G,\ell^1_0(\G))\, ,\qquad K(g_0,g_1)=\delta_{g_0}-\delta_{g_1}\ .
$$
Of course we have $\delta K=0$, and
it is immediate to check that $g\cdot \delta_{g'}=\delta_{gg'}$ for every $g,g'\in\G$, so
$$
(g\cdot K)(g_0,g_1)=g(K(g^{-1}g_0,g^{-1}g_1))=g(\delta_{g^{-1}g_0}-\delta_{g^{-1}g_1})=
\delta_{g_0}-\delta_{g_1}=K(g_0,g_1)
$$
and $K$ defines a bounded coclass $[K]\in H^1_b(\G,\ell^1_0(\G))$, which we call
the \emph{characteristic coclass} of $\G$. 

We are now ready to prove that finite groups may be characterized
by the vanishing of bounded cohomology with coefficients in any normed vector space.

\begin{thm}\label{finite:char}
 Let $\G$ be a group. Then the following conditions are equivalent:
 \begin{enumerate}
  \item $\G$ is finite;
  \item 
  $H^1_b(\G,V)=0$ for every Banach $\G$-module $V$;
   \item $H^n_b(\G,V)=0$ for every normed $\R[\G]$-module $V$ and every
  $n\geq 1$;
  \item 
  the characteristic coclass of $\G$ vanishes.
 \end{enumerate}
\end{thm}
\begin{proof}
 (1) $\Rightarrow$ (3): Let $n\geq 1$. For every cochain $f\in C^n_b(\G,V)$ we define a cochain
 $k^n(f)\in C^{n-1}_b(\G,V)$ by setting 
 $$
 k^n(f)(g_1,\ldots,g_n)=\frac{1}{|G|}\sum_{g\in \G} f(g,g_1,\ldots,g_n)\ ,
 $$
 where $|\G|$ denotes the cardinality of $\G$. The maps $k^n$, $n\geq 1$, provide
 a (partial) equivariant $\R[\G]$-homotopy between the identity and the null map
 of $C^n_b(\G,V)$, $n\geq 1$, whence the conclusion.
 
 The implications (3) $\Rightarrow$ (2) and (2) $\Rightarrow$ (4) are obvious, so we are left to show that the vanishing
 of the characteristic coclass implies that $\G$ is finite.
 
 So, suppose that $[K]=0$, i.e.~there exists a $\G$-equivariant map $\psi\colon \G\to \ell^1_0(\G)$
 such that 
\begin{equation}\label{deltaeq}
\delta_{g_1}-\delta_{g_0}={\psi}(g_1)-{\psi}(g_0)\, \quad {\rm for\ every}\ g_0,g_1\in \G\ .
\end{equation}
We then set $f=\delta_1-\psi(1)\in\ell^1(\G)$, and we observe that $f\neq 0$ since $e(f)=e(\delta_1)-e(\psi(1))=0$. Moreover, by equation~\eqref{deltaeq}
we have $f=\delta_g-\psi(g)$ for every $g\in\G$, so 
%In particular, we have
%\begin{equation}\label{equi3:eq}
%\delta_g-{\psi}(g)=\delta_1-{\psi}(1)\, \quad {\rm for\ every}\ g \in \G\ .
%\end{equation}
%Therefore, using Equations~\eqref{equi3:eq}, 
%and 
the $\G$-equivariance of $\psi$ implies that
$$
g\cdot f=g\cdot (\delta_{1}-{\psi}(1))=(g\cdot\delta_1)-g\cdot ({\psi}(1))=\delta_g-{\psi}(g)=
\delta_1-{\psi}(1)=f\ ,
$$
which implies that $f$ is constant. Since $f$ is not null and summable, this implies
in turn that $\G$ is finite.
\end{proof}

\section{Further readings}
There are plenty of equivalent characterizations of amenable groups. Probably one of the most important (both beacuse it establishes a clear brigde towards geometric group theory,
and because it is often useful in applications) is the one which involves F\"olner sequences. If $\G$ is any group, a F\"olner sequence for $\G$ is a sequence
$F_n$, $n\in\mathbb{N}$, of finite subsets of $\G$ such that, for every $g\in\G$, the following holds:
$$
\lim_{n\to\infty} \frac{|F_n\, \triangle (g\cdot F_n)|}{|F_n|}=0\ .
$$
(Here we denote by $A\, \triangle\, B$ the symmetric difference between $A$ and $B$, and by $|A|$ the cardinality of the set $A$.) 
When $\G$ is finitely generated, the condition above admits a nice interpretation in terms of the geometry of Cayley graphs of $\G$: namely,
if we define the boundary of a subset $A$ of vertices of a graph $G$ (with unitary edges) as the set of vertices of $G$ at distance from $A$ equal to $1$, then
a F\"olner sequence is a sequence of subsets of $\G$ whose boundaries (in any Cayley graph of $\G$) grow sublinearly with respect to the growth of the size of the subsets. In other words,
amenable groups are those groups that do \emph{not} admit any linear isoperimetric inequality. 

Other well-known and very useful of characterizations of amenability involve fixed-point properties of group actions. For example, it can be proved that a group $\G$ is amenable if 
the following holds:
if $\G$ acts by isometries on a (separable) Banach space $E$, leaving a weakly closed convex subset $C$ of the closed unit ball of the topological dual $E'$ invariant, then $\G$ has a fixed point in $C$.

It is maybe worth mentioning that a slightly different notion of amenability 
is usually considered 
in the context of \emph{topological} groups: usually, a locally compact topological group $\G$ (endowed with its Haar measure $\mu$) is said to be amenable if
there exists a left-invariant mean on the space $L^\infty(\G,\mu)$ of essentially bounded Borel functions on $\G$.  

We refer the reader to the books~\cite{Pier,Paterson} for a thorough treatment of (topological) amenable groups.

\chapter{(Bounded) group cohomology via resolutions}\label{res:chapter}

%Following Ivanov and Monod,
%we ask even more: namely, we consider only complete normed vector spaces, i.e.~Banach spaces.
 
%\begin{defn}[\cite{Ivanov, Monod}]
%A \emph{Banach $\G$-module} is a Banach space $V$ with a left action of $\G$ such that
%$\|g\cdot v\|\leq\|v\|$ for every $g \in \G$ and every $v \in V$. A $\G$-morphism
%of Banach $\G$-modules is a bounded $\G$-equivariant linear operator.
%\end{defn} 
%A $\G$-morphism between Banach $\G$-modules is a $\G$-morphism between the underlying
%$\R[\G]$-modules, which is bounded with respect to the norms.

Computing group cohomology by means of its definition is usually very hard. The topological interpretation
of group cohomology already provides a powerful tool for computations: for example, one may estimate 
the cohomological dimension of a group $\G$ (i.e.~the maximal $n\in\matN$ such that $H^n(\G,R)\neq 0$)
in terms of the dimension of a $K(\G,1)$ as a CW-complex. We have already mentioned that a topological
interpretation for the bounded cohomology of a group is still available,
but in order to prove this fact more machinery has to be introduced. 
Before going into the bounded case, we describe some well-known results which hold in the classical case:
namely, we will show that 
the cohomology of $\G$ may be computed by looking at 
several complexes of cochains.  This crucial fact was 
already observed in the pioneering works on group cohomology of the 1940s and the 1950s.

There are several ways to define group cohomology in terms of resolutions. We privilege here an approach
that better extends to the case of bounded cohomology. 
%As we will see soon, this approach works well only in the case with real coefficients. 
We briefly compare our approach to more traditional ones in Section~\ref{alternative}.

\section{Relative injectivity}
We begin by introducing the notions of relatively injective
$R[\G]$-module and of strong $\G$-resolution of an $R[\G]$-module.  
The counterpart of these notions in the context of normed $R[\G]$-modules will play an important role 
in the theory of bounded cohomology of groups. 
The importance
of these notions is due to the fact that the cohomology of $\G$
may be computed by looking at \emph{any} strong $\G$-resolution
of  the coefficient module by relatively injective modules (see Corollary~\ref{fund1:cor} below).

A $\G$-map $\iota\colon A\rightarrow B$ between $R[\G]$-modules is
\emph{strongly injective}
%\emph{admissible}
if there is an $R$-linear map
$\sigma\colon B\rightarrow A$  such that
$\sigma\circ\iota={\rm Id}_A$ (in particular, $\iota$ is injective). 
%In particular, if $\eta $ is injective,
%it is admissible if and only if it admits a left inverse morphism
%of norm at most one. 
We emphasize that, even if $A$ and $B$ are 
%endowed with a structure of 
$\G$-modules, the map $\sigma$ is \emph{not} required
to be  $\G$-equivariant.

\begin{defn}\label{inj2:def}
An $R[\G]$-module $U$ is \emph{relatively injective} (over $R[\G]$) if the following holds: 
whenever $A,B$ are $\G$-modules, $\iota\colon A\to B$ is a strongly injective $\G$-map
and $\alpha\colon A\to U$ is a $\G$-map, there exists a $\G$-map $\beta\colon B\to U$
such that $\beta\circ\iota = \alpha$.
\end{defn}

$$
\xymatrix{
0 \ar[r] & A  \ar[r]_{\iota} \ar[d]_{\alpha} & B \ar@{-->}[dl]^{\beta} \ar@/_/[l]_{\sigma}\\
& U
}
$$

In the case when $R=\R$, a map is strongly injective if and only if it is injective,
so in the context of $\R[\G]$-modules the notions of relative injectivity and
the traditional notion of injectivity coincide. However, 
relative injectivity is weaker than injectivity in general.
For example, if $R=\matZ$ and $\G=\{1\}$, then 
$C^i(\G,R)$ is isomorphic to the trivial $\G$-module $\matZ$, which of course is not injective over $\matZ[\G]=\matZ$. On the other hand, we have the following:

\begin{lemma}\label{inj:basic}
For every $R[\G]$-module $V$ and every $n\in\matN$, the $R[\G]$-module $C^n(\G,V)$ is 
relatively injective.
\end{lemma}
\begin{proof}
Let us consider the extension problem described in Definition~\ref{inj2:def},
with $U=C^n(\G,V)$. 
Then we 
define $\beta$ as follows:
\begin{align*}
\beta(b)(g_0,\ldots,g_n)&=\alpha(g_0\sigma(g_0^{-1}b))(g_0,\ldots,g_n)
\\ &=g_0\big(\alpha(\sigma(g_0^{-1}b))(1,g_0^{-1}g_1,\ldots,g_0^{-1}g_n)\big)
\ .
\end{align*}
The fact that $\beta$ is $R$-linear and that $\beta\circ\iota=\alpha$ is straightforward, so we just need to show that $\beta$ is $\G$-equivariant. But for every 
$g\in \G$, $b\in B$ and $(g_0,\ldots,g_n)\in \G^{n+1}$ 
%if we denote
%by $\omega\in C^n(\G,V)$ the map $\omega=\alpha(g_0\sigma(g_0^{-1}gb))$, then 
we have
%\begin{align*}
%\beta(g\cdot b)(g_0,\ldots,g_n) & = \omega(g_0,\ldots,g_n)%\\
%\end{align*}
%and
\begin{align*}
 g(\beta (b))(g_0,\ldots,g_n) & = g(\beta(b)(g^{-1}g_0,\ldots,g^{-1}g_n))\\
& = g(\alpha (g^{-1}g_0\sigma (g_0^{-1}gb))(g^{-1}g_0,\ldots,g^{-1}g_n))\\
& = g(g^{-1}(\alpha (g_0\sigma (g_0^{-1}gb)))(g^{-1}g_0,\ldots,g^{-1}g_n))\\
& = \alpha(g_0\sigma(g_0^{-1}gb))(g_0,\ldots,g_n)\\
& = \beta(gb)(g_0,\ldots,g_n)\ .
%& = g((g^{-1}\omega)(g^{-1}g_0,\ldots,g^{-1}g_n))\\
%& = g(g^{-1}(\omega(g_0,\ldots,g_n)))\\
%& = \omega (g_0,\ldots,g_n)
\end{align*}
\end{proof}

\section{Resolutions of $\G$-modules}
An \emph{$R[\G]$-complex} (or simply a \emph{$\G$-complex} or a \emph{complex}) is a sequence of 
$R[\G]$-modules $E^i$ and $\G$-maps $\delta^i\colon
E^i\to E^{i+1}$, $i\in\matN$, such that $\delta^{i+1}\circ\delta^i=0$ for every $i$:
% where $i$ runs over
%$\matN\cup\{-1\}$:
$$
0\longrightarrow E^0\tto{\delta^0} E^1\tto{\delta^1}\ldots
\tto{\delta^n} E^{n+1}\tto{\delta^{n+1}}\ldots
$$
Such a sequence will be denoted by $(E^\bullet,\delta^\bullet)$. Moreover, we set
$Z^n(E^\bullet)=\ker \delta^n\cap (E^n)^\G$, $B^n(E^\bullet)=\delta^{n-1}((E^{n-1})^\G)$
(where again we understand that $B^0(E^\bullet)=0$),
and we define  the cohomology of the complex $E^\bullet$ by setting
$$
H^n(E^\bullet)=Z^n(E^\bullet)/B^n(E^\bullet)\ .
$$

A \emph{chain map} between $\G$-complexes $(E^\bullet,\delta_E^\bullet)$ and $(F^\bullet,\delta_F^\bullet)$
is a sequence of $\G$-maps $\{\alpha^i\colon E^i\to F^i\, |\, i\in\matN\}$
such that $\delta_F^i\circ\alpha^i=\alpha^{i+1}\circ\delta_E^{i}$ for every $i\in\matN$.
If $\alpha^\bullet,\beta^\bullet$ are chain maps between $(E^\bullet,\delta_E^\bullet)$ and $(F^\bullet,\delta_F^\bullet)$, a 
\emph{$\G$-homotopy} 
%(or simply a \emph{homotopy}) 
between $\alpha^\bullet$ and $\beta^\bullet$ is 
a sequence of $\G$-maps $\{T^i \colon E^i\to F^{i-1}\, |\, i\geq 0\}$ such that
$T^1\circ \delta_E^0=\alpha^0-\beta^0$ and 
$\delta_F^{i-1}\circ T^i + T^{i+1}\circ \delta_E^i=\alpha^i-\beta^i$   for every $i\geq 1$. 
Every chain map induces a well-defined map in cohomology, and $\G$-homotopic chain
maps induce the same map in cohomology.

If $E$ is an $R[\G]$-module, 
an augmented $\G$-complex $(E,E^\bullet,\delta^\bullet)$ with augmentation map
$\varepsilon\colon E\to E^0$ 
is a complex
$$
0\longrightarrow E\tto{\varepsilon} E^0\tto{\delta^0} E^1\tto{\delta^1}\ldots
\tto{\delta^n} E^{n+1}\tto{\delta^{n+1}}\ldots
$$
A \emph{resolution} of $E$  (over $\G$) is an exact augmented complex $(E,E^\bullet,\delta^\bullet)$ (over $\G$).
A resolution 
$(E,E^\bullet,\delta^\bullet)$ is \emph{relatively injective}
if $E^n$ is relatively injective for every $n\geq 0$. 
It is well-known that any map between modules extends to a chain map between \emph{injective}
resolutions of the modules.
Unfortunately, the same result for relatively injective resolutions does not hold.
The point is that relative injectivity guarantees the needed extension property only
for \emph{strongly} injective maps. Therefore,  we need to introduce the notion of \emph{strong} resolution.

A \emph{contracting homotopy} for a resolution $(E,E^\bullet,\delta^\bullet)$ is a 
sequence of $R$-linear maps $k^i\colon E^i\to E^{i-1}$ such that 
$\delta^{i-1}\circ k^i+
k^{i+1}\circ\delta^i = {\rm Id}_{E^i}$ if $i\geq 0$, and $k^0 \circ \varepsilon=
{\rm Id}_E$:
$$
\xymatrix{
0\ar[r] & E\ar[r]_{\varepsilon} & E^0 \ar[r]_{\delta^0} \ar@/_/[l]_{k^0} &
E^1  \ar[r]_{\delta^1} \ar@/_/[l]_{k^1} & \ldots \ar@/_/[l]_{k^2} \ar[r]_{\delta^{n-1}}&
E^{n}
\ar[r]_{\delta^n} \ar@/_/[l]_{k^n} & \ldots \ar@/_/[l]_{k^{n+1}}
}
$$
Note however that it is not required that 
$k^i$ be $\G$-equivariant. A resolution is \emph{strong}
if it admits a contracting homotopy.

The following proposition shows that the chain complex $C^\bullet(\G,V)$ 
provides a relatively injective strong resolution of $V$:

\begin{prop}\label{standard:res}
Let $\varepsilon \colon V\to C^0(\G,V)$ be defined by $\varepsilon(v)(g)=v$ for every $v\in V$, $g\in \G$. Then the
augmented complex
$$
0\longrightarrow V\tto{\varepsilon} C^0(\G,V)\tto{\delta^0} C^1(\G,V)\to\ldots
\to C^n(\G,V)\to\ldots
$$
provides a relatively  injective strong resolution of $V$.
\end{prop}
\begin{proof}
 We already know that each $C^i(\G,V)$ is relatively injective. In order to show that the augmented complex $(V,C^\bullet(\G,V),\delta^\bullet)$ is a strong resolution 
it is sufficient to observe that the maps
$$
k^{n+1}\colon C^{n+1}(\G,V)\to C^n(\G,V)\,\qquad k^{n+1}(f)(g_0,\ldots,g_n)=f(1,g_0,\ldots,g_n)
$$
provide the required contracting homotopy.
\end{proof}

The resolution of $V$ described in the previous proposition is obtained just by augmenting
the homogeneous complex associated to $(\G,V)$, and it is usually called the 
\emph{standard} resolution of $V$ (over $\G$). 

The following result may be proved
by means of standard homological algebra arguments (see e.g.~\cite{Ivanov}, \cite[Lemmas 7.2.4 and 7.2.6]{Monod} for a detailed
proof in the bounded case -- the argument there extends to this context
just by forgetting any reference to the norms). It 
implies that any relatively injective strong resolution of a $\G$-module $V$
may be used to compute the cohomology modules $H^\bullet(\G,V)$ (see Corollary~\ref{fund1:cor}).

\begin{thm}\label{ext:thm}
Let $\alpha\colon E\to F$ be a $\G$-map between 
$R[\G]$-modules,
let $(E,E^\bullet,\delta_E^\bullet)$ be a strong resolution of $E$, and suppose
that $(F,F^\bullet,\delta_F^\bullet)$ is an augmented complex
such that $F^i$ is relatively injective for every $i\geq 0$.
Then $\alpha$ extends to a chain map $\alpha^\bullet$, and any two extensions
of $\alpha$ to chain maps are $\G$-homotopic.
\end{thm}

\begin{cor}\label{fund1:cor}
Let $(V,V^\bullet,\delta_V^\bullet)$ be a relatively injective strong resolution of $V$. Then for every $n\in\matN$
there is a canonical isomorphism
$$
H^n(\G,V)\cong H^n(V^\bullet)\ .
$$
\end{cor}
\begin{proof}
 By Proposition~\ref{standard:res}, both $(V,V^\bullet,\delta_V^\bullet)$ and the standard resolution of $V$ are relatively injective strong
resolutions of $V$ over $\G$. Therefore, Theorem~\ref{ext:thm} provides chain maps
between $C^\bullet(\G,V)$ and $V^\bullet$, which are one the $\G$-homotopy inverse of the other.
Therefore, these chain maps induce isomorphisms in cohomology.
\end{proof}

\section{The classical approach to group cohomology via resolutions}\label{alternative}
In this section we describe the relationship
between 
the description of group cohomology via resolutions given above and more
traditional approaches to the subject (see e.g.~\cite{brown}).
The reader who is not interested
can safely skip the section, since the results cited below will not be used elsewhere in this monograph.

%The category of abelian groups endowed with a left action by $\G$ obviously
%coincides with the category of left $\matZ[\G]$-modules. 
If $V,W$ are $\matZ[\G]$-modules,
then the space ${\rm Hom}_{\matZ} (V,W)$ is endowed with the structure of a $\matZ[\G]$-module by setting $(g\cdot f)(v)=g(f(g^{-1}v))$ for every $g\in \G$,
$f\in {\rm Hom}_{\matZ} (V,W)$, $v\in V$.
Then, the cohomology $H^\bullet(\G,V)$ is often defined in the following equivalent ways
(we refer e.g. to~\cite{brown}
for the definition of projective module (over a ring $R$); for our discussion, it is sufficient
to know that any free $R$-module is projective over $R$, so
any free resolution is projective):
\begin{enumerate}
 \item {\bf (Via injective resolutions of $V$):} Let $(V,V^\bullet,\delta^\bullet)$ be an injective resolution of $V$ over $\matZ[\G]$, and take
the complex $W^\bullet={\rm Hom}_{\matZ}(V^\bullet,\matZ)$, endowed with the action 
$g\cdot f(v)=f(g^{-1}v)$. Then $(W^\bullet)^\G={\rm Hom}_{\matZ[\G]}(V^\bullet,\matZ)$
(where $\matZ$ is endowed with the structure of trivial $\G$-module), and one may define
$H^\bullet(\G,V)$ as the homology of the $\G$-invariants of $W^\bullet$.
\item 
{\bf (Via projective resolutions of $\matZ$):}
Let $(\matZ,P_\bullet,d_\bullet)$ be a \emph{projective} resolution of $\matZ$ over $\matZ[\G]$,
and take the complex $Z^\bullet={\rm Hom}_{\matZ} (P_\bullet,V)$. We have again that
$(Z^\bullet)^\G={\rm Hom}_{\matZ[\G]} (P_\bullet,V)$, and again we may define
$H^\bullet(\G,V)$ as the homology of the $\G$-invariants of the complex $Z^\bullet$. 
\end{enumerate}
The fact that these two definitions are indeed equivalent is proved e.g.~in~\cite{brown}.

We have already observed that, if $V$ is a $\matZ[\G]$-module, the module $C^n(\G,V)$ is not injective in general. However, this 
is not really a problem, since the complex $C^\bullet (\G,V)$ may be recovered from a projective resolution of $\matZ$ over $\matZ[\G]$. 
Namely, let $C_n(\G,\matZ)$ be the free $\matZ$-module admitting the set $\G^{n+1}$ as a basis. The diagonal action
of $\G$ onto $\G^{n+1}$  endows $C_n(\G,\matZ)$ with the structure of a $\matZ[\G]$-module.
The modules $C_n(\G,\matZ)$ may be arranged into a resolution
$$
0\longleftarrow \matZ \ttob{\varepsilon} C_0(\G,\matZ)\ttob{d_1} C_1(\G,\matZ)\ttob{d_2}\ldots
\ttob{d_{n}} C_n(\G,\matZ) \ttob{d_{n+1}}\ldots
$$
of the trivial $\matZ[\G]$-module $\matZ$ over $\matZ[\G]$ (the homology of the $\G$-coinvariants of this resolution is by definition
the \emph{homology} of $\G$, see Section~\ref{homol:group:sec}).
Now, it is easy to check that $C_n(\G,\matZ)$ is free, whence projective,
as a $\matZ[\G]$-module. Moreover, the module
${\rm Hom}_{\matZ} (C_n(\G,\matZ),V)$ is $\matZ[\G]$-isomorphic to $C^n(\G,V)$.
This shows that, in the context of the traditional definition of group cohomology via resolutions,
the complex $C^\bullet(\G,V)$ arises from a projective resolution of $\matZ$,
rather than from an injective resolution of $V$.

\section{The topological interpretation of group cohomology revisited}\label{revisited:sec}
Corollary~\ref{fund1:cor} may be exploited to prove that the cohomology 
of $\G$ is isomorphic to the singular cohomology of any path-connected topological
space $X$ satisfying conditions (1), (2) and (3) described in Section~\ref{topol:sec}, which we recall here for the reader's convenience:
\begin{enumerate}
 \item[(1)] 
the fundamental group of $X$ is isomorphic to $\G$,
\item[(2)] the space $X$ admits
a universal covering $\widetilde{X}$, and 
\item[(3)] 
$\widetilde{X}$ is $R$-acyclic, i.e.~$H_n(\widetilde{X},R)=0$
for every $n\geq 1$.
\end{enumerate} 

%We denote by $C_\bullet(X)$ (resp.~$C^\bullet(X)$)
%the complex of singular chains (resp.~cochains) with real coefficients.
We fix an identification
between $\G$ and the group of the covering automorphisms of the universal covering $\widetilde{X}$ of $X$. The action of $\G$ on $\widetilde{X}$ induces an
action of $\G$ on $C_\bullet(\widetilde{X},R)$, whence an action of $\G$
on
$C^\bullet(\widetilde{X},R)$, which is defined by $(g\cdot\varphi)(c)=\varphi(g^{-1}c)$
for every $c\in C_\bullet(\widetilde{X},R)$, $\varphi\in C^\bullet(\widetilde{X},R)$,
$g\in \G$. Therefore, for $n\in\matN$ both $C_n(\widetilde{X},R)$ and $C^n(\widetilde{X},R)$ 
are endowed with the structure of $R[\G]$-modules.
We have a natural identification $C^\bullet(\widetilde{X},R)^\G\cong C^\bullet(X,R)$.
%The following result provides an important step in the proof that,
%if $\widetilde{X}$ is acyclic, then
%the cohomology of $\G$ turns out to be isomorphic with the singular cohomology of
%$X$. 

\begin{lemma}\label{sing:inj}
For every $n\in\matN$, the singular cochain module
$C^n(\widetilde{X},R)$ is relatively injective.
\end{lemma}
\begin{proof}
For every topological space $Y$, let us denote by $S_n(Y)$ the set
of singular simplices with values in $Y$.

We denote by $L_n(\widetilde{X})$ a set of representatives
for the action of $\G$ on $S_n(\widetilde{X})$
(for example, if $F$ is a set of representatives for the action of $\G$ on
$\widetilde{X}$, we may define $L_n(\widetilde{X})$ as the set of singular
$n$-simplices whose first vertex lies in $F$). Then, for every $n$-simplex
${s}\in S_n(\widetilde{X})$, there exist a unique
$g_s\in \G$ and a unique $\overline{s}\in L_n(\widetilde{X})$
such that $g_s\cdot\overline{s}=s$.
Let us now consider the extension problem:
$$
\xymatrix{
0 \ar[r] & A  \ar[r]_{\iota} \ar[d]_{\alpha} & B \ar@{-->}[dl]^{\beta} \ar@/_/[l]_{\sigma}\\
& C^n(\widetilde{X},R)
}
$$ 
We define the desired extension $\beta$ by setting
$$
\beta(b)(s)=\alpha(g_s\sigma(g_s^{-1}\cdot b)) ({s})=\alpha(\sigma(g_s^{-1}\cdot b)) (\overline{s})
$$
for every $s \in S_n(\widetilde{X})$. It is easy to verify that
the map $\beta$ is an $R[\G]$-map, and that $\alpha=\beta\circ \iota$.

An alternative proof (producing the same solution to the extension problem)
is the following.
Using that
the action of $\G$ on $\widetilde{X}$ is free, it is easy to show
that $L_n(\widetilde{X})$, when considered as a subset of $C_n(\widetilde{X},R)$, is a 
free basis of $C_n(\widetilde{X},R)$ over $R[\G]$. 
In other words, every $c\in C_n(\widetilde{X},R)$ may be expressed uniquely
as a sum of the form
$c=\sum_{i=1}^k a_ig_i{s}_i$, $a_i\in R$, $g_i\in \G$, ${s}_i\in L_n(\widetilde{X})$.
Therefore, the map
$$
\psi\colon C^0(\G,C^0(L_n(\widetilde{X}),R))\to C^n(\widetilde{X},R)\, ,\quad
\psi(f)\left(\sum_{i=1}^k a_ig_i s_i\right)=\sum_{i=1}^k a_if(g_i)({s}_i)
$$
is well defined. 
If we endow $C^0(L_n(\widetilde{X}),R)$ with the structure of trivial 
$\G$-module, then
a straightforward computation shows that $\psi$ is in fact
a $\G$-isomorphism, so the conclusion follows from Lemma~\ref{inj:basic}.
\end{proof}

\begin{prop}\label{sing:res}
Let $\varepsilon\colon R\to C^0(\widetilde{X},R)$ be defined
by $\varepsilon (t)(s)=t$ for every singular $0$-simplex $s$
in $\widetilde{X}$.
 Suppose that $H_i(\widetilde{X},R)=0$ for every $i\geq 1$. Then the augmented complex
$$
0\longrightarrow R\tto{\varepsilon} C^0(\widetilde{X},R)\tto{\delta^1}C^1(\widetilde{X},R)\to
\ldots \tto{\delta^{n-1}} C^n(\widetilde{X},R)\tto{\delta^n} C^{n+1}(\widetilde{X},R)
$$
is a relatively injective strong resolution of the trivial $R[\G]$-module $R$.
\end{prop}
\begin{proof}
Since $\widetilde{X}$ is path-connected, we have
that ${\rm Im}\, \varepsilon =\ker\delta^0$.
Observe now that $C_n(\widetilde{X},R)$ is $R$-free for every $n\in\matN$. As a consequence,
the obvious augmented complex associated to $C_\bullet(\widetilde{X},R)$, being acyclic, is homotopically trivial over $R$. 
Since $C^n(\widetilde{X},R)\cong {\rm Hom}_{R}(C_n(\widetilde{X},R),R)$, we may conclude that the augmented
complex described in the statement is a strong resolution of $R$ over $R[\G]$. 
Then the conclusion follows from Lemma~\ref{sing:inj}.
\end{proof}

Putting together Propositions~\ref{standard:res}, \ref{sing:res} and
Corollary~\ref{fund1:cor} we can provide the following topological
description of $H^n(\G,R)$:

\begin{cor}
 Let $X$ be a path-connected space admitting a universal covering $\widetilde{X}$, and suppose that
$H_i(\widetilde{X},R)=0$ for every $i\geq 1$. Then
$H^i(X,R)$ is canonically isomorphic to $H^i(\pi_1(X),R)$ for every $i\in\matN$.
\end{cor}

\section{Bounded cohomology via resolutions}\label{res:bounded}
Just as in the case of classical cohomology, 
it is often useful to have alternative ways to compute the bounded cohomology of a group. This section is devoted to an approach to bounded cohomology which closely follows the traditional approach to classical cohomology via homological algebra.
The circle of ideas we are going to  describe first appeared (in the case with trivial real coefficients) in a paper by Brooks~\cite{Brooks}, where it was exploited to
 give an independent proof of Gromov's result that the isomorphism type of the bounded cohomology of a space (with real coefficients)
 only depends on its fundamental group~\cite{Gromov}. 
 Brooks' theory was then 
developed by Ivanov in his foundational paper~\cite{Ivanov} (see also~\cite{Noskov} for the case of coefficients in general normed $\G$-modules). Ivanov 
gave a new proof of the vanishing of the bounded cohomology (with real coefficients) of a simply connected
space (this result played an important role in Brooks' argument, and was originally due to Gromov~\cite{Gromov}), and
managed to incorporate the seminorm into the homological algebra approach to bounded cohomology, thus proving that the 
bounded  cohomology of a space is \emph{isometrically}
isomorphic to the bounded cohomology of its fundamental group (in the case with real coefficients).
Ivanov-Noskov's theory was further developed by Burger and Monod~\cite{BM1,BM2,Monod}, who paid a particular attention to the 
\emph{continuous} bounded cohomology of topological groups. 

Both Ivanov-Noskov's and Monod's theory are concerned with \emph{Banach} $\G$-modules, which are in particular
$\R[\G]$-modules. For the moment, we prefer to consider also the (quite different) case with integral coefficients. Therefore,
we let $R$ be either $\matZ$ or $\R$, and we concentrate our attention on the category of normed $R[\G]$-modules introduced in 
Section~\ref{bounded:coh:sec}. In the next sections we will see that relative injective modules and strong resolutions may be defined in this context just by adapting
to normed $R[\G]$-modules the analogous definitions for generic $R[\G]$-modules.

\section{Relatively injective normed $\G$-modules}
Throughout the whole section, unless otherwise stated, we will deal only with \emph{normed} $R[\G]$-modules.
Therefore, $\G$-morphisms will be always assumed to be bounded. 

The following definitions are taken from~\cite{Ivanov}
(where only the case when $R=\R$ and $V$ is Banach is dealt with).
A bounded linear map $\iota\colon A\rightarrow B$ of normed $R$-modules is 
\emph{strongly injective}
%\emph{admissible}
if there is an $R$-linear map 
$\sigma\colon B\rightarrow A$ with $\|\sigma\|\leq 1$ and 
$\sigma\circ\iota={\rm Id}_A$ (in particular, $\iota$ is injective). 
%In particular, if $\eta $ is injective,
%it is admissible if and only if it admits a left inverse morphism
%of norm at most one. 
We emphasize that, even when $A$ and $B$ are 
%endowed with a structure of 
$R[\G]$-modules, the map $\sigma$ is \emph{not} required
to be  $\G$-equivariant.

\begin{defn}\label{relativelyinjective}
A normed $R[\G]$-module $E$ is \emph{relatively injective} if for every 
strongly injective $\G$-morphism $\iota\colon A\rightarrow B$
of normed $R[\G]$-modules and every $\G$-morphism $\alpha\colon A\rightarrow E$
there is a $\G$-morphism $\beta\colon B\rightarrow E$ satisfying $\beta\circ \iota=
\alpha$ and $\|\beta\|\leq \|\alpha\|$.
$$
\xymatrix{
0 \ar[r] & A  \ar[r]_{\iota} \ar[d]_{\alpha} & B \ar@{-->}[dl]^{\beta} \ar@/_/[l]_{\sigma}\\
& E
}
$$
\end{defn}

\begin{rem}
Let $E$ be a normed $R[\G]$-module, and let $\hat{E}$ be the underlying $R[\G]$-module. Then
no obvious implication exists between the fact that $E$ is relatively injective
(in the category of normed $R[\G]$-modules, i.e.~according to Definition~\ref{relativelyinjective}),
and the fact that $\hat{E}$ is (in the category of $R[\G]$-modules, i.e.~according to Definition~\ref{inj2:def}).
This could suggest that the use of the same name for these different notions could indeed be an abuse. However,
unless otherwise stated, henceforth we will deal with relatively injective modules only in the context of normed
$R[\G]$-modules, so the reader may safely take Definition~\ref{relativelyinjective} as the only definition
of relative injectivity.  
\end{rem}

The following result is due to Ivanov~\cite{Ivanov} in the case of real coefficients, and to Monod~\cite{Monod} in the general case
(see also Remark~\ref{compare}), and shows that the modules involved in the definition of bounded cohomology are relatively injective. 

\begin{lemma}\label{relinj:standard}
 Let $V$ be a normed $R[\G]$-module. Then the normed $R[\G]$-module $C_b^n(\G,V)$
is relatively injective.
\end{lemma}
\begin{proof}
Let us consider the extension problem described in Definition~\ref{relativelyinjective},
with $E=C_b^n(\G,V)$. 
Then we 
define $\beta$ as follows:
$$
\beta(b)(g_0,\ldots,g_n)=\alpha(g_0\sigma(g_0^{-1}b))(g_0,\ldots,g_n)\ .
$$
It is immediate to check that $\beta\circ \iota=\alpha$. Moreover, since
$\|\sigma\|\leq 1$, we have $\|\beta\|\leq\|\alpha\|$. Finally, the fact that
$\beta$ commutes with the actions of $\G$ may be proved by the very same computation
described in the proof of
Lemma~\ref{inj:basic}. 
\end{proof}

\section{Resolutions of normed $\G$-modules}
A normed $R[\G]$-complex is an $R[\G]$-complex whose modules are normed $R[\G]$-spaces,
and whose differential is a bounded $R[\G]$-map in every degree.
% If $(E^\bullet,\delta_E^\bullet)$,
%$(F^\bullet,\delta_F^\bullet)$ are normed $R[\G]$-complexes, then 
A chain map
between $(E^\bullet,\delta_E^\bullet)$
$(F^\bullet,\delta_F^\bullet)$ is  a chain map between the underlying $R[\G]$-complexes
which is bounded in every degree, and a $\G$-homotopy between two such chain maps
%$f^\bullet,g^\bullet$ 
is just a $\G$-homotopy between the underlying 
maps of $R[\G]$-modules, which is bounded in every degree. The cohomology
$H_b^\bullet(E^\bullet)$ of the normed $\G$-complex $(E^\bullet,\delta_E^\bullet)$
is defined as usual by taking the cohomology of the subcomplex of $\G$-invariants.
The norm on $E^n$ restricts to a norm on $\G$-invariant cocycles, which induces in turn a seminorm
on $H_b^n(E^\bullet)$ for every $n\in\matN$.

An augmented normed $\G$-complex 
$(E,E^\bullet,\delta^\bullet)$
with augmentation map $\varepsilon\colon E\to E^0$ 
is a $\G$-complex 
$$
0\longrightarrow E\tto{\varepsilon} E^0\tto{\delta^0} E^1\tto{\delta^1}\ldots
\tto{\delta^n} E^{n+1}\tto{\delta^{n+1}}\ldots\ 
$$
We also ask that $\varepsilon$ is an isometric embedding.
A resolution of $E$ (as a normed $R[\G]$-complex) is an exact augmented normed complex
$(E,E^\bullet,\delta^\bullet)$. It is \emph{relatively injective} 
if $E^n$ is relatively injective for every $n\geq 0$. From now on, 
we will call simply \emph{complex}
any normed $\G$-complex.

Let $(E,E^\bullet,\delta_E^\bullet)$ be a resolution of $E$, and suppose that $(F,F^\bullet,\delta_F^\bullet)$
is a relatively injective
 resolution of $F$. We would like to be able to extend any $\G$-map $E\to F$ to a chain map between $E^\bullet$
 and $F^\bullet$.
As observed in the preceding section,
to this aim we need to require the resolution $(E,E^\bullet,\delta_E^\bullet)$ to be strong, according to the following definition.
%the assumption that $(F,F^\bullet,\delta_F^\bullet)$ is relatively injective
%is not sufficient to ensure that any $\G$-map between $E$ and $F$ extends to a chain map
%between $(E,E^\bullet,\delta_E^\bullet)$ and $(F,F^\bullet,\delta_F^\bullet)$.
%The point is that relative injectivity guarantees the needed extension property only
%for \emph{strongly} injective maps, so a corresponding notion
%of \emph{strong} resolution is needed.

A \emph{contracting homotopy} for a resolution $(E,E^\bullet,\delta^\bullet)$ is a 
sequence of linear maps $k^i\colon E^i\to E^{i-1}$ such that 
$\|k^i\|\leq 1$ for every $i\in\matN$, 
$\delta^{i-1}\circ k^i+
k^{i+1}\circ\delta^i = {\rm Id}_{E^i}$ if $i\geq 0$, and $k^0 \circ \varepsilon=
{\rm Id}_E$:
$$
\xymatrix{
0\ar[r] & E\ar[r]_{\varepsilon} & E^0 \ar[r]_{\delta^0} \ar@/_/[l]_{k^0} &
E^1  \ar[r]_{\delta^1} \ar@/_/[l]_{k^1} & \ldots \ar@/_/[l]_{k^2} \ar[r]_{\delta^{n-1}}&
E^{n}
\ar[r]_{\delta^n} \ar@/_/[l]_{k^n} & \ldots \ar@/_/[l]_{k^{n+1}}
}
$$
Note however that it is not required that 
$k^i$ be $\G$-equivariant. A resolution is \emph{strong}
if it admits a contracting homotopy. 

\begin{prop}\label{standard:ok}
Let $V$ be a normed $R[\G]$-space, and let
$\varepsilon\colon V\to C^0_b(\G,V)$ be defined by $\varepsilon(v)(g)=v$ for every
$v\in V$, $g\in \G$. Then the
augmented complex
$$
0\longrightarrow V\tto{\varepsilon} C_b^0(\G,V)\tto{\delta^0} C_b^1(\G,V)\to\ldots
\to C_b^n(\G,V)\to\ldots
$$
provides a relatively injective strong resolution of $V$.
\end{prop}
\begin{proof}
 We already know that each $C_b^i(\G,V)$ is relatively injective, so in order to conclude it is sufficient to
observe that the map
$$
k^{n+1}\colon C^{n+1}_b(\G,V)\to C^n_b(\G,V)\,\qquad k^{n+1}(f)(g_0,\ldots,g_n)=f(1,g_0,\ldots,g_n)
$$
provides a contracting homotopy for the resolution $(V,C^\bullet_b(\G,V),\delta^\bullet)$.
\end{proof}

The resolution described in Proposition~\ref{standard:ok} is the \emph{standard resolution}
of $V$ as a normed $R[\G]$-module.

\begin{rem}\label{compare}
 Let us briefly compare our notion of standard resolution
with Ivanov's and Monod's ones. In~\cite{Ivanov}, for every $n\in\matN$
the set $C_b^n(\G,\R)$ is denoted
by $B(\G^{n+1})$, and it is endowed with the structure of a \emph{right}
Banach $\G$-module by the action
$g\cdot f (g_0,\ldots,g_n)=f(g_0,\ldots,g_n\cdot g)$. Moreover, the sequence
of modules
$B(\G^{n})$, $n\in\matN$, is equipped with a structure of $\G$-complex
$$
0\longrightarrow \R\tto{d_{-1}} B(\G)\tto{d_0} B(\G^2)\tto{d_1}\ldots
\tto{d_n} B(\G^{n+2})\tto{d_{n+1}}\ldots,
$$
where $d_{-1}(t)(g)=t$ and 
\begin{align*}
%d_{-1}(t)(g)=t,\qquad 
d_n(f)(g_0,\ldots,g_{n+1})&=(-1)^{n+1}f(g_1,\ldots,g_{n+1})\\ & +\sum_{i=0}^n (-1)^{n-i}
f(g_0,\ldots,g_ig_{i+1},\ldots,g_{n+1})
\end{align*}
for every $n\geq 0$
(here we are using Ivanov's notation also for the differential). 
Now, it is readily seen that (in the case with trivial real coefficients) Ivanov's resolution is 
isomorphic to our standard resolution via the isometric $\G$-chain isomorphism
$\varphi^\bullet\colon B^\bullet (\G)\to C^\bullet_b(\G,\R)$
defined by
$$
%\begin{array}{lll}
\varphi^n(f)(g_0,\ldots,g_n)=f(g_n^{-1},g_n^{-1}g_{n-1}^{-1},\ldots,g_n^{-1}g_{n-1}^{-1}
\cdots g_1^{-1}g_0^{-1})
$$
(with inverse $(\varphi^n)^{-1}(f)(g_0,\ldots,g_n)=f(g_n^{-1}g_{n-1},g_{n-1}^{-1}g_{n-2},\ldots,g_1^{-1}g_0,g_0^{-1})$).
%Moreover, $\varphi^\ast$ induces an isometry in cohomology, since it admits the isometric inverse
%$\psi^\ast\colon
%B(\G^{\ast+1})\to B^\ast(\G)$, where
%$$
%\psi^n(f)(g_0,\ldots,g_n)=f(g_n^{-1}g_{n-1},g_{n-1}^{-1}g_{n-2},\ldots,g_1^{-1}g_0,g_0^{-1})\ .
%\end{array}
%$$
We also observe that
the contracting homotopy described in Proposition~\ref{standard:ok}
is conjugated by $\varphi^\bullet$ into
Ivanov's contracting homotopy for the complex $(B(\G^\bullet),d_\bullet)$ (which is defined in~\cite{Ivanov}).

Our notation is much closer to Monod's one. In fact, 
in~\cite{Monod} the more general case of a topological group $\G$ is addressed,
and the 
$n$-th module
of the standard $\G$-resolution of an $\R[\G]$-module $V$ is inductively defined by setting
$$
C^0_b(\G,V)=C_b(\G,V),\qquad C_b^n(\G,V)=C_b(\G,C_b^{n-1}(\G,V))\ ,
$$
where $C_b(\G,E)$ denotes the space of \emph{continuous} bounded maps from
$\G$ to the Banach space $E$. 
However, as observed in~\cite[Remarks 6.1.2 and 6.1.3]{Monod},
the case when $\G$ is an abstract group may be recovered from the general case 
just by equipping $\G$ with the discrete topology. In that case, 
our notion of standard resolution coincides with Monod's one 
(see also~\cite[Remark 7.4.9]{Monod}).
\end{rem}

The following result can  be proved by means of 
standard homological algebra arguments
(see~\cite{Ivanov}, \cite[Lemmas 7.2.4 and 7.2.6]{Monod} for full details):

\begin{thm}\label{ext:bounded:thm}
Let $\alpha\colon E\to F$ be a $\G$-map between 
normed $R[\G]$-modules,
let $(E,E^\bullet,\delta_E^\bullet)$ be a strong resolution of $E$, and suppose
$(F,F^\bullet,\delta_F^\bullet)$ is an augmented complex such that $F^i$ is relatively injective
for every $i\geq 0$. 
Then $\alpha$ extends to a chain map $\alpha^\bullet$, and any two extensions
of $\alpha$ to chain maps are $\G$-homotopic.
\end{thm}

\begin{cor}\label{fund1:bounded:cor}
Let $V$ be a normed $R[\G]$-modules, and 
let $(V,V^\bullet,\delta_V^\bullet)$ be a relatively injective strong resolution of $V$. Then for every $n\in\matN$
there is a canonical isomorphism
$$
H_b^n(\G,V)\cong H^n_b(V^\bullet)\ .
$$
Moreover, this isomorphism is bi-Lipschitz with respect to the seminorms
of $H_b^n(\G,V)$ and $H^n(V^\bullet)$.
\end{cor}
\begin{proof}
 By Proposition~\ref{standard:ok}, the standard resolution of $V$ is a relatively injective strong
resolution of $V$ over $\G$. Therefore, Theorem~\ref{ext:thm} provides chain maps
between $C_b^\bullet(\G,V)$ and $V^\bullet$, which are one the $\G$-homotopy inverse of the other.
Therefore, these chain maps induce isomorphisms in cohomology. The conclusion
follows from the fact that bounded chain maps induce bounded maps
in cohomology.
\end{proof}

By Corollary~\ref{fund1:bounded:cor},
every relatively injective strong resolution of $V$ induces a seminorm on $H_b^\bullet(\G,V)$. Moreover, the seminorms defined in this way are pairwise equivalent.
However,
in many applications, it is important to be able to compute the exact \emph{canonical} seminorm of elements
in $H^n_b(\G,V)$, i.e.~the seminorm induced on $H^n_b(\G,V)$ by the standard
resolution $C^\bullet_b(\G,V)$. Unfortunately, it is not possible
to capture the isometry type of $H^n_b(\G,V)$ via arbitrary relatively injective strong resolutions. Therefore, a special role is played by those resolutions
which compute the canonical seminorm. The following fundamental result is due to Ivanov, and implies that these distinguished 
resolutions are in some sense extremal:

\begin{thm}\label{norm:prop}
Let $V$ be a normed $R[\G]$-module, and let
$(V,V^\bullet,\delta^\bullet)$ be any strong resolution of 
$V$.
Then the identity of $V$ can be extended to a chain map 
$\alpha^\bullet$ between $V^\bullet$ and the standard resolution of $V$,
in such a way that $\|\alpha^n\|\leq 1$ for every $n\geq 0$. In particular,
the canonical seminorm of $H_b^\bullet(\G,V)$
is not bigger than the seminorm induced on $H^\bullet_b(\G,V)$ by
any relatively injective strong resolution.
\end{thm}
\begin{proof}
One can inductively define $\alpha^n$  by setting, for every $v\in E^n$ and $g_j\in \G$:  
$$
\alpha^n (v) (g_0,\ldots,g_n)=\alpha^{n-1}\left(g_0\Bigl(k^{n}\bigl(g_0^{-1}(v)\bigr)\Bigr)\right)(g_1,\ldots,g_n),
$$
where $k^\bullet$ is a contracting homotopy for the given resolution
$(V,V^\bullet,\delta^\bullet)$. It is not difficult 
to prove by induction that  
$\alpha^\bullet$ is indeed a norm non-increasing chain $\G$-map (see~\cite{Ivanov}, \cite[Theorem 7.3.1]{Monod}
for the details).
\end{proof}

\begin{cor}\label{norm:cor}
 Let $V$ be a normed $R[\G]$-module, let
$(V,V^\bullet,\delta^\bullet)$ be a relatively injective strong resolution of 
$V$, and suppose that the identity of $V$ can be extended to a chain map
$\alpha^\bullet\colon  C^\bullet_b(\G,V)\to V^\bullet$ such that 
$\|\alpha^n\|\leq 1$ for every $n\in\matN$. Then $\alpha^\bullet$ induces an isometric isomorphism between  $H^\bullet_b(\G,V)$ and $H^\bullet(V^\bullet)$. In particular,
the seminorm induced by the resolution $(V,V^\bullet,\delta^\bullet)$ coincides
with the canonical seminorm on $H^\bullet_b(\G,V)$.
\end{cor}

\section{More on amenability}
The following result establishes an interesting relationship between the amenability of $\G$ and the relative injectivity
of normed $\R[\G]$-modules.

\begin{prop}\label{amenable:relinj}
The following facts are equivalent:
\begin{enumerate}
\item
The group $\G$ is amenable.
\item 
Every dual normed $\R[\G]$-module is relatively injective.
\item
The trivial $\R[\G]$-module $\R$ is relatively injective.
\end{enumerate}
\end{prop}
\begin{proof}
(1) $\Rightarrow$ (2):
Let $W$ be a normed $\R[\G]$-module, and let $V=W'$ be the dual normed $\R[\G]$-module of $W$.
We first construct a left inverse (over $\R[\G]$) of the augmentation map
$\varepsilon \colon V\to C^0_b(\G,V)$.
We fix an invariant mean $m$ on $\G$.
% To this aim, using the amenability of $\G$, we modify 
%$k^\bullet$ into a $\G$-equivariant map $j^\bullet$, which is defined as folllows. 
For  $f\in C_b^{0}(\G,V)$ and $w\in W$
we consider the function
$$
f_w\colon \G\to\R\, ,\quad f_w(g)=f(g)(w)\ .
$$
It follows from the definitions that $f_w$ is an element of $\ell^\infty(\G)$, so we may define a map
$r\colon C^0_b(\G,V)\to V$ by setting
$r(f) (w)=m(f_w)$. It is immediate to check that $r(f)$ is indeed a bounded functional on $W$, whose norm is bounded 
by $\|f\|_\infty$. In other words, the map $r$ is well defined and norm non-increasing. The $\G$-invariance of the mean $m$
 implies that $r$ is $\G$-equivariant, and an easy computation shows that
$r\circ \varepsilon={\rm Id}_V$.
 
 Let us now consider the diagram 
 $$
\xymatrix{
0 \ar[r] & A  \ar[r]_{\iota} \ar[d]_{\alpha} & B \ar@/^/@{-->}[ddl]^{\beta'} \ar@{-->}[dl]^{\beta} \ar@/_/[l]_{\sigma}\\
& V \ar@/_/[d]_\varepsilon \\
& C^0_b(\G,V)\ar@/_/[u]_r
}
$$
By Lemma~\ref{relinj:standard}, $C^0_b(\G,V)$ is relatively injective, so there exists a bounded $\R[\G]$-map $\beta'$
such that $\|\beta'\|\leq \|\varepsilon\circ\alpha\|\leq \|\alpha\|$ and $\beta'\circ\iota=\varepsilon\circ\alpha$.  
The $\R[\G]$-map $\beta:=r\circ \beta'$ satisfies $\|\beta\|\leq \|\alpha\|$ and
$\beta\circ\iota=\alpha$. This shows that $V$ is relatively injective.

(2) $\Rightarrow$ (3) is obvious, so we are left to show that (3) implies (1).  
If $\R$ is relatively injective and 
$\sigma\colon \ell^\infty(\G)\to \R$ is the map defined by $\sigma (f)=f(1)$, then there exists an $\R[\G]$-map
$\beta\colon \ell^\infty(\G)\to \R$ such that $\|\beta\| \leq 1$ and the following diagram commutes:
$$
\xymatrix{
 \R  \ar[rr]_{\varepsilon} \ar[d]_{\mathrm{Id}} & & \ell^\infty(\G)  \ar[dll]^{\beta} \ar@/_/[ll]_{\sigma}\\
 \R
}
$$
By construction, since $\beta({\bf 1}_\G)=1$, the map $\beta$ is an invariant continuous non-trivial functional on $\ell^\infty(\G)$, so $\G$ is amenable by Lemma~\ref{equivalent:amenable}.
\end{proof}

The previous proposition allows us to provide an alternative proof of Theorem~\ref{amenable:vanishing}, which we recall here
for the convenience of the reader:

\begin{thm}
Let $\G$ be an amenable group, and let $V$ be a dual normed $\R[\G]$-module. Then $H^n_b(\G,V)=0$ for every $n\geq 1$.
\end{thm}
\begin{proof}
The complex
$$
0\tto{} V\tto{\mathrm{Id}} V\tto{} 0
$$
provides a relatively injective strong resolution of $V$, so the conclusion follows from Corollary~\ref{fund1:bounded:cor}.
\end{proof}

\section{Amenable spaces}
The notion of 
amenable space 
was introduced by Zimmer~\cite{Zimmer} in the context of
actions of topological groups
on standard measure spaces
(see e.g.~\cite[Section 5.3]{Monod} for several equivalent definitions).
%We refer the reader to~\cite[Section 5.3]{Monod_book} for several equivalent definitions of amenability for the action of a topological group on a standard  $\\G$-spaces in a much more general context.
In our case of interest, i.e.~when $\G$ is a discrete group acting on a set $S$ (which may be thought as endowed with the discrete topology), the amenability of $S$ as a $\G$-space is equivalent to the amenability
of the stabilizers in $\G$ of elements of $S$~\cite[Theorem 5.1]{AEG}. Therefore,
we may take this characterization as a definition:

\begin{defn}\label{amenable:space}
A left action $\G\times S\to S$ of a group $\G$ on a set $S$ is \emph{amenable}
if the stabilizer of every $s\in S$ is an amenable subgroup of $\G$. In this
case, we equivalently say that $S$ is an amenable $\G$-set.
\end{defn}

The importance of amenable $\G$-sets is due to the fact that they may be exploited
to isometrically compute the bounded cohomology of $\G$. If $S$ is any $\G$-set and $V$ is any normed $\R[\G]$-module, then we denote by
$\ell^\infty(S^{n+1},V)$ the space of bounded functions from $S^{n+1}$ to $V$.
This space may be endowed with the structure of a 
normed $\R[\G]$-module via the action
$$
(g\cdot f)(s_0,\ldots,s_n)=g\cdot (f(g^{-1}s_0,\ldots,g^{-1}s_n))\ .
$$
The differential
$\delta^n\colon \ell^\infty(S^{n+1},V)\to \ell^\infty(S^{n+2},V)$
defined by 
$$
\delta^n(f)(s_0,\ldots,s_{n+1})=\sum_{i=0}^{n}  (-1)^i 
f(s_0,\ldots,\hat{s_i},\ldots,s_n)
$$
endows the pair $(\ell^\infty(S^{\bullet+1},V),\delta^\bullet)$ with the structure
of a normed $\R[\G]$-complex. Together with the augmentation $\varepsilon\colon V\to
\ell^\infty(S,V)$ given by $\varepsilon (v)(s)=v$ for every $s\in S$, such a complex
provides a strong resolution of $V$:

\begin{lemma}\label{Sstrong}
 The augmented complex
$$
0\tto{} V\tto{} \ell^\infty(S,V)\tto{\delta^0} \ell^\infty (S^2,V)\tto{\delta^1} 
\ell^\infty (S^3,V)\tto{\delta^2}\ldots
$$
provides a strong resolution of $V$.
\end{lemma}
\begin{proof}
 Let ${s}_0$ be a fixed element of $S$. Then the maps
$$
k^{n}\colon \ell^\infty(S^{n+1},V)\to \ell^\infty(S^n,V)\, ,\qquad
k^n(f)(s_1,\ldots,s_{n})=f({s}_0,s_1,\ldots,s_{n})\, 
$$
provide the required contracting homotopy.
\end{proof}

\begin{lemma}\label{Srelinj}
 Suppose that $S$ is an amenable $\G$-set, and that $V$ is a dual normed
$\R[\G]$-module. Then $\ell^\infty(S^{n+1},V)$ is relatively injective for every $n\geq 0$.
\end{lemma}
\begin{proof}
Since any intersection of amenable subgroups is amenable, the $\G$-set $S$ is amenable if and only if $S^n$ is. Therefore, it is sufficient
to deal with the case $n=0$.

Let $W$ be a normed $\R[\G]$-module
such that $V=W'$, and consider the extension problem described in Definition~\ref{relativelyinjective},
with $E=\ell^\infty(S,V)$:
$$
\xymatrix{
0 \ar[r] & A  \ar[r]_{\iota} \ar[d]_{\alpha} & B \ar@{-->}[dl]^{\beta} \ar@/_/[l]_{\sigma}\\
& \ell^\infty(S,V)
}
$$ 
We denote by $R\subseteq S$ a set of representatives for the
action of $\G$ on $S$, and
for every $r\in R$ we denote by $\G_r$ the stabilizer of $r$, endowed
with the invariant mean $\mu_r$. 
Moreover, for every $s\in S$ we choose an element $g_s\in \G$ such that
$g_s^{-1} (s)=r_s\in R$. Then $g_s$ is uniquely determined up to right
multiplication by elements in $\G_{r_s}$. 

Let us fix an element 
$b\in B$. In order to define $\beta(b)$, for every $s\in S$ we need to know the value
taken by $\beta(b)(s)$ on every $w\in W$. Therefore, we fix $s\in S$, $w\in W$, and we
set
$$
(\beta(b)(s))(w)=\mu_{r_s}(f_b)\ ,
$$
where $f_b\in \ell^\infty(\G_{r_s},\R)$ is defined by
$$
f_b(g)=\left((g_sg)\cdot \alpha(\sigma (g^{-1}g_s^{-1} b))\right) (s)\, (w) \ .
$$ 
Since $\|\sigma\|\leq 1$ we have that 
$\|\beta\|\leq \|\alpha\|$, and
the behaviour of means
on constant functions implies that 
$\beta\circ \iota=\alpha$. 

Observe that the element $\beta(b)(s)$ does not depend on the choice of
$g_s\in \G$. In fact, if we replace $g_s$ by $g_sg'$ for some $g'\in \G_{r_s}$,
then the function $f_b$ defined above is replaced by the function
 $$
f_b'(g)=\left((g_sg'g)\cdot \alpha(\sigma (g^{-1}(g')^{-1}g_s^{-1} b))\right) (s)\, 
(w)=f_b(g'g)\ , 
$$ 
and $\mu_{r_s}(f_b')=\mu_{r_s}(f_b)$ by the invariance of the mean $\mu_{r_s}$. This fact allows us to prove that $\beta$ is a $\G$-map. In fact,
let us fix $\overline{g}\in \G$ and let
$\overline{s}=\overline{g}^{-1}(s)$. Then 
we may assume that $g_{\overline{s}}=\overline{g}^{-1}g_s$, so
$$
\overline{g}(\beta(b))(s)(w)=
\overline{g}(\beta(b)(\overline{g}^{-1}s))(w)
=\beta(b)(\overline{s})(\overline{g}^{-1}w)=\mu_{r_s}(\overline{f}_b)\ ,
$$
where $\overline{f}_b\in \ell^\infty(\G_{r_s},\R)$ is given by
\begin{align*}
\overline{f}_b(g) & =\left((\overline{g}^{-1}g_sg)\cdot \alpha(\sigma (g^{-1}g_s^{-1}\overline{g} b))\right) (\overline{s})\, (\overline{g}^{-1}w) \\
& = \left((g_sg)\cdot \alpha(\sigma (g^{-1}g_s^{-1}\overline{g} b))\right) (s)\, (w)\\
& = f_{\overline{g}b} (g)\ .
\end{align*}
Therefore, 
$$
\overline{g}(\beta(b))(s)(w)=\mu_{r_s}(f_{\overline{g}b})=\beta(\overline{g}b)(s)(w)
$$
for every $s\in S$, $w\in W$, i.e.~$\overline{g}\beta(b)=\beta(\overline{g}b)$, and we are done. 
\end{proof}

As anticipated above, we are now able to show that amenable spaces
may be exploited to compute bounded cohomology:

\begin{thm}\label{Siso}
 Let $S$ be an amenable $\G$-set and let $V$ be a dual normed $\R[\G]$-module. 
 Let also $(V,V^\bullet,\delta_V^\bullet)$ be a strong resolution of $V$. Then, there exists a $\G$-chain map
 $\alpha^\bullet\colon V^\bullet\to \ell^\infty(S^{\bullet +1},V)$ which extends ${\rm Id}_V$ and is norm non-increasing in every degree. In particular, 
 the homology of the complex
$$
0\tto{} \ell^\infty(S,V)^\G\tto{\delta^0} \ell^\infty (S^2,V)^\G\tto{\delta^1} 
\ell^\infty (S^3,V)^\G\tto{\delta^2}\ldots
$$
is canonically isometrically isomorphic to $H_b^\bullet(\G,V)$.
\end{thm}
\begin{proof}
By Theorem~\ref{norm:prop}, in order to prove the first statement of the theorem it is sufficient to 
assume that $V^\bullet=C^\bullet_b(\G,V)$, so we are left to 
construct a norm non-increasing chain map
$$
\alpha^\bullet\colon C^\bullet_b(\G,V)\to \ell^\infty(S^{\bullet+1},V)\ .
$$

We keep notation from the proof of the previous lemma, i.e.~we fix a set of representatives
$R$ for the action of $\G$ on $S$, and for every $s\in S$ we choose an element
$g_s\in \G$ such that $g_s^{-1}s=r_s\in R$. 
For every $r\in R$ we also fix an invariant mean $\mu_r$ on the stabilizer $\G_r$ of $r$.

Let us fix an element $f\in C^n_b(\G,V)$,
and take $(s_0,\ldots,s_n)\in S^{n+1}$. 
If $V=W'$, we also fix an element $w\in W$. 
For every $i$ we denote by $r_i\in R$ the representative of the orbit
of $s_i$, and we consider the invariant mean 
$\mu_{r_0}\times \ldots \times \mu_{r_n}$
on $\G_{r_0}\times \ldots \times \G_{r_n}$ (see Remark~\ref{amenable:prod}).
Then we consider the function $f_{s_0,\ldots,s_n}\in \ell^\infty(\G_{r_0}\times\ldots \times \G_{r_n},\R)$ defined by
$$
f_{s_0,\ldots,s_n}(g_0,\ldots,g_n)=f(g_{s_0}g_0,\ldots,g_{s_n}g_n)(w)\ .
$$
By construction we have $\|f_{s_0,\ldots,s_n}\|_\infty\leq \|f\|_\infty\cdot \|w\|_W$,
so we may set
$$
(\alpha^n(f)(s_0,\ldots,s_n))(w)=(\mu_{r_0}\times \ldots\times \mu_{r_n}) (f_{s_0,\ldots,s_n})\ ,
$$
thus defining an element $\alpha^n(f)(s_0,\ldots,s_n)\in W'=V$ such that
$\|\alpha^n(f)\|_V\leq \|f\|_\infty$. We have thus shown that $\alpha^n\colon C^n_b(\G,V)\to \ell^\infty(S^{n+1},V)$ is a well-defined norm non-increasing linear map.
The fact that $\alpha^n$ commutes with the action of $\G$ follows from the invariance
of the means $\mu_r$, $r\in R$, and the fact that $\alpha^\bullet$ is a chain
map is obvious.

The fact that the complex $\ell^\infty(S^{\bullet +1},V)$ isometrically computes the bounded cohomology of $\G$ with coefficients in $V$ is now straightforward. 
Indeed, the previous lemmas imply that the augmented complex $(V,\ell^\infty(S^{\bullet+1},V),\delta^\bullet)$ provides a relatively injective strong
resolution of $V$, so Corollary~\ref{fund1:bounded:cor} implies that
the homology of the $\G$-invariants of $(V,\ell^\infty(S^{\bullet+1}),\delta^\bullet)$
is isomorphic to the bounded cohomology of $\G$ with coefficients in $V$. 
Thanks to the existence of the norm non-increasing chain map $\alpha^\bullet$, the fact that this 
isomorphism is isometric is then a consequence of Corollary~\ref{norm:cor}.
\end{proof}

Let us prove some direct corollaries of the previous results. Observe that, if
$W$ is a dual normed $\R[K]$-module and $\psi\colon \G\to K$ is a homomorphism, then the induced module 
$\psi^{-1} (W)$ is a dual normed $\R[\G]$-module.

\begin{thm}
Let $\psi\colon \G\to K$ be  a surjective homomorphism with amenable kernel, and let $W$ be a dual normed $\R[K]$-module. Then the induced map
$$
 H^\bullet_b(K,W)\to H^\bullet_b(\G,\psi^{-1}W)
$$ 
is an isometric isomorphism.
\end{thm}
\begin{proof}
The action $\G\times K\to K$ defined by $(g,k)\mapsto \psi(g)k$
endows $K$ with the structure of an amenable $\G$-set. Therefore, Theorem~\ref{Siso}
implies that the bounded cohomology of $\G$ with coefficients in $\psi^{-1}W$ is isometrically isomorphic to the cohomology
of the complex $\ell^\infty(K^{\bullet+1},\psi^{-1}W)^\G$. However, since
$\psi$ is surjective, we have a (tautological) isometric identification between
$\ell^\infty(K^{\bullet+1},\psi^{-1}W)^\G$ and $C^\bullet_b(K,W)^K$, whence the conclusion.
\end{proof}

\begin{cor}\label{map:am:cor}
 Let $\psi\colon \G\to K$ be a surjective homomorphism with amenable kernel. Then the induced map
 $$
 H^n_b(K,\R)\to H^n_b(\G,\R)
 $$
is an isometric isomorphism for every
$n\in\mathbb{N}$.
\end{cor}

\section{Alternating cochains}\label{altern:group:sec}
Let $V$ be a normed $\R[\G]$-module. 
A cochain $\varphi\in C^n(\G,V)$ is \emph{alternating} if it satisfies the following condition: 
for every permutation $\sigma\in\mathfrak{S}_{n+1}$ of the set $\{0,\ldots,n\}$, 
if we denote by ${\rm sgn}(\sigma)=\pm 1$ the sign of $\sigma$, then
the equality
$$
\varphi(g_{\sigma(0)},\ldots,g_{\sigma(n)})={\rm sgn}(\sigma)\cdot \varphi (g_0,\ldots,g_n)
$$
holds for every $(g_0,\ldots,g_n)\in\G^{n+1}$.
We denote by $C^n_{\alt}(\G,V)\subseteq C^n(\G,V)$ the subset of alternating cochains, and we set
$C^n_{b,\alt}(\G,V)=C^n_{\alt}(\G,V)\cap C^n_b(\G,V)$.  
It is well-known that (bounded) alternating cochains provide a $\G$-subcomplex of general (bounded) cochains.
In fact, it turns out that, in the case of real coefficients, one can compute the (bounded) cohomology of $\G$ 
via the complex of alternating cochains:

\begin{prop}\label{alternating:prop}
The complex $C^\bullet_{\alt}(\G,V)$ (resp.~$C^\bullet_{b,\alt}(\G,V)$) isometrically computes
the cohomology (resp.~the bounded cohomology) of $\G$ with real coefficients.
\end{prop}
\begin{proof}
We concentrate our attention on bounded cohomology, the case of ordinary cohomology being very similar.
The inclusion $j^\bullet\colon C^\bullet_{b,\alt}(\G,V)\to C^\bullet_b(\G,V)$ 
induces a norm non-increasing map on bounded cohomology, so in order to prove the proposition it is sufficient to construct
a norm non-increasing $\G$-chain map $\alt^\bullet_b\colon C^\bullet_b(\G,V)\to C^\bullet_{b,\alt}(\G,V)$ which satisfies the following properties:
\begin{enumerate}
\item
$\alt_b^n$ is a retraction onto the subcomplex of alternating cochains, i.e.~$\alt_b^n\circ j^n={\rm Id}$ for every $n\geq 0$;
\item
$j^\bullet\circ\alt_b^\bullet$ is $\G$-homotopic to the identity of $C^\bullet_b (\G,V)$ 
(as usual, via a homotopy which is bounded in every degree).
\end{enumerate}
For every $\varphi\in C^n_b(\G,V)$, $(g_0,\ldots,g_n)\in\G^{n+1}$ we set
$$
\alt_b^n(\varphi)(g_0,\ldots,g_n)=\frac{1}{(n+1)!}\sum_{\sigma\in \mathfrak{S}_{n+1}} {{\rm sgn}(\sigma)}
\cdot \varphi(g_{\sigma(0)},\ldots,g_{\sigma(n)})\ .
$$
It is easy to check that $\alt_b^\bullet$ satisfies the required properties 
(the fact that it is indeed homotopic to the identity of $C^\bullet_b(\G,V)$ may be deduced, for example,
by the computations carried out in the context of singular chains in \cite[Appendix B]{FM}).
\end{proof}

Let $S$ be an amenable $\G$-set. We have seen in Theorem~\ref{Siso} that the bounded cohomology of $\G$  with coefficients
in the normed $\R[\G]$-module $V$ is isometrically isomorphic
to the cohomology of the complex $\ell^\infty(S^{\bullet +1},V)^\G$. The very same argument described in the proof of Proposition~\ref{alternating:prop}
shows that the bounded cohomology of $\G$ is computed also by the subcomplex of alternating elements of  $\ell^\infty(S^{\bullet +1},V)^\G$.
More precisely, let us denote by $\ell^\infty_{\alt}(S^n,V)$ the submodule of alternating elements of  $\ell^\infty(S^n,V)$ (the definition of alternating being obvious).
Then we have the following:
 
\begin{thm}\label{Sisoalt}
 Let $S$ be an amenable $\G$-set and let $V$ be a dual normed $\R[\G]$-module. Then the homology of the complex
$$
0\tto{} \ell^\infty_{\alt}(S,V)^\G\tto{\delta^0} \ell^\infty_{\alt} (S^2,V)^\G\tto{\delta^1} 
\ell^\infty_{\alt}(S^3,V)^\G\tto{\delta^2}\ldots
$$
is canonically isometrically isomorphic to $H_b^\bullet(\G,V)$.
\end{thm}

\section{Further readings}
The study of bounded cohomology of groups by means of homological algebra was initiated by Brooks~\cite{Brooks} and developed by Ivanov~\cite{Ivanov}, who first managed to construct a theory
that allowed the use of resolutions for the \emph{isometric} computation of bounded cohomology (see also~\cite{Noskov} for the case with twisted coefficients). 
Ivanov's ideas were then generalized and further developed by Burger and Monod~\cite{BM1,BM2,Monod} in order to deal
with topological groups (their theory lead also to a deeper understanding of the bounded cohomology of discrete groups, for example in the case of lattices in Lie groups). 
In the case of discrete groups, an introduction to bounded cohomology following Ivanov's approach may also be found in~\cite{Clara:book}.

From the point of view of homological algebra, bounded cohomology is a pretty exotic theory, since it fails excision and thus it cannot be easily studied via any (generalized) Mayer-Vietoris principle. 
As observed in the introduction of~\cite{Bualg}, 
this seemed to suggest that (continuous) bounded cohomology could not be interpreted as a derived functor and that triangulated methods 
could not apply to its study. On the contrary, it is proved in~\cite{Bualg} that the formalism of exact categories and their derived categories can be exploited
to construct a classical derived functor on the category of Banach $\G$-modules with values in Waelbroeck's abelian category. 
Building on this fact, B\"uhler 
provided an axiomatic characterization of bounded cohomology, and illustrated how the theory of bounded cohomology 
can be interpreted in the standard framework of homological and homotopical algebra.

\subsection*{Amenable actions}
There exist at least two quite different notions of amenable actions in the literature. Namely, suppose that the group $\G$ acts via Borel isomorphisms on a topological space $X$. Then
one can say that the action is amenable if $X$ supports a $\G$-invariant Borel probability measure. This is the case, for example, whenever the action has a finite orbit $O$, because
in that case a $\G$-invariant probability measure is obtained just by taking the (normalized) counting measure supported on $O$. This definition of amenable action dates back to~\cite{Green}, and it is not the best suited
to bounded cohomology. Indeed, our definition of amenable action (for a discrete group on a discrete space) is a very particular instance of a more general notion due to Zimmer~\cite{Zimmer,Zimmer:book}. Zimmer's
original definition of amenable action is quite involved, but it turns out to be equivalent to the existence of a $\G$-invariant conditional expectation $L^\infty(\G\times X)\to L^\infty(X)$, or to the fact that the equivalence relation
on $X$ defined by the action is amenable, together with the amenability of the stabilizer of $x\in X$ for almost every $x\in X$ (this last characterization readily implies that
our Definition~\ref{amenable:space} indeed coincides with Zimmer's one in the case of actions on discrete spaces). For a discussion of the equivalence of these definitions of Zimmer's
amenability we refer the reader to~\cite[II.5.3]{Monod}, which summarizes and discusses results from~\cite{AEG}. Amenable actions also admit a characterization in the language of homological algebra:
an action of $\G$ on a space $X$ is amenable if and only if the $\G$-module $L^\infty(X)$ is relatively injective. This implies that amenable spaces may be exploited to compute
(and, in fact, to isometrically compute) the bounded cohomology of groups (in the case of discrete spaces, this is just Theorem~\ref{Siso}).

Both definitions of amenable actions we have just discussed may be generalized to the case when $\G$
is a locally compact topological group. As noted in~\cite{GlaMon}, these two notions are indeed quite different, and in fact somewhat ``dual'' one to the other: for example, a trivial action is always amenable
in the sense of Greenleaf, while it is amenable in the sense of Zimmer precisely when the group $\G$ is itself amenable.

An important example of amenable action is given by the action of any countable group $\G$ on its Poisson boundary~\cite{Zimmer}. Such an action is also doubly ergodic~\cite{Kaim},
and this readily implies that the second bounded cohomology of a group $\G$ may be isometrically identified with the space of $\G$-invariant $2$-cocycles
on its Poisson boundary. This fact has proved to be extremely powerful for the computation of low-dimensional bounded cohomology modules for discrete and locally compact groups
(see e.g.~\cite{BM1,BM2}). As an example, let us recall that Bouarich proved that, 
if $\varphi\colon \G\to H$ is a surjective homomorphism, then the induced map
$H^2_b(\varphi)\colon H^2_b(H,\R)\to H^2_b(\G,\R)$
is injective (see Theorem~\ref{Bua:thm}).
 Indeed, the morphism $\varphi$ induces a  measurable map from a suitably chosen Poisson boundary of $\G$ to a suitably chosen Poisson boundary of $H$, and this implies
in turn that the map $H^2_b(\varphi)$ is  
 in fact an isometric isomorphism~\cite[Theorem 2.14]{Huber}.

\chapter{Bounded cohomology of topological spaces}\label{bounded:space:chap}
Let $X$ be a topological space, and let $R=\matZ,\R$. Recall that $C_\bullet (X,R)$ (resp.~$C^\bullet(X,R)$) denotes the usual
complex of singular chains (resp.~cochains) on $X$ with coefficients in $R$, and
$S_i (X)$ is the set of singular $i$--simplices
in $X$. We also
regard $S_i (X)$ as a subset of $C_i (X,R)$, so that
for any cochain $\varphi\in C^i (X,R)$ it makes sense to consider its restriction
$\varphi|_{S_i (X)}$. For every $\varphi\in C^i (X,R)$, we set
$$
\| \varphi \|=\|\varphi\|_\infty  = \sup \left\{|\varphi (s)|\ |\ s\in S_i (X)\right\}\in [0,\infty].
$$
We denote by $C^\bullet_b (X,R)$ the submodule of bounded cochains, i.e.~we set
$$C^\bullet_b (X,R)= \left\{\varphi\in C^\bullet (X,R)\ | \ \|\varphi\|<\infty\right\}\ .$$ Since
the differential takes bounded cochains into bounded cochains, $C^\bullet_b (X,R)$
is a subcomplex of $C^\bullet(X,R)$.
We denote by $H^\bullet(X,R)$ (resp.~$H_b^\bullet(X,R)$) 
the homology of the complex $C^\bullet(X,R)$ (resp.~$C_b^\bullet(X,R)$).
Of course, $H^\bullet(X,R)$ is the usual singular cohomology module of $X$ with coefficients in $R$, while $H_b^\bullet(X,R)$
is the \emph{bounded cohomology module} of $X$ with coefficients in $R$. 
Just as in the case of groups,
the norm on $C^i (X,R)$ descends (after taking
the suitable restrictions) to a seminorm on each of the modules 
$H^\bullet(X,R)$, $H_b^\bullet(X,R)$. More precisely, if
$\varphi\in H$ is a class in one of these modules, which is obtained as a quotient
of the corresponding module of cocycles $Z$, then we set 
$$
\|\varphi\|=\inf \left\{\|\psi\|\ |\  \psi\in Z,\, [\psi]=\varphi\ {\rm in}\ H\right\}.
$$
This seminorm may take the value $\infty$ on elements in $H^\bullet(X,R)$
and may be null on non-zero elements in $H_b^\bullet(X,R)$
(but not on non-zero elements in $H^\bullet(X,R)$: this is clear in the case
with integer coefficients, and it is a consequence of the Universal Coefficient Theorem in the case
with real coefficients, since a real cohomology class with vanishing seminorm
has to be null on any cycle, whence null in $H^\bullet(X,\R)$).

The inclusion of bounded cochains into possibly unbounded cochains
induces the \emph{comparison map}
$$
%\begin{array}{llllll}
c^\bullet \colon  H_b^\bullet(X,R)\to H^\bullet(X,R)\ .
%\end{array}
$$

\section{Basic properties of bounded cohomology of spaces}
Bounded cohomology enjoys some of the fundamental properties of classical singular
cohomology: for example, $H^i_b(\{\text{pt.}\},R)=0$ if $i>0$, and
$H^0_b(\{\text{pt.}\},R)=R$ (more in general, $H^0_b(X,R)$ is canonically isomorphic
to $\ell^\infty (S)$, where $S$ is the set of  path connected components of $X$).
The usual proof of the homotopy invariance of singular homology is based on
the construction of an algebraic homotopy which maps every $n$-simplex to the sum
of at most $n+1$  $(n+1)$-dimensional simplices. As a consequence, the homotopy operator
induced in cohomology preserves bounded cochains. This implies that bounded cohomology
is a homotopy invariant of topological spaces. Moreover, if $(X,Y)$ is a topological pair,
then there is an obvious definition of $H^\bullet_b(X,Y)$, and it is immediate to check that the analogous of the long exact sequence of the pair in classical singular cohomology
also holds in the bounded case.

Perhaps the most important phenomenon that distinguishes  bounded cohomology from  classical singular cohomology is the lacking of any Mayer-Vietoris sequence (or, equivalently,
of any excision theorem). In particular, spaces with finite-dimensional bounded cohomology may be tamely glued to each other to get spaces with infinite-dimensional
bounded cohomology (see Remark~\ref{wedge} below).

Recall from Section~\ref{topol:sec} that $H^n(X,R)\cong H^n(\pi_1(X),R)$ for every
aspherical CW-complex $X$.  As anticipated in Section~\ref{bounded:inter}, a fundamental result by Gromov provides
an isometric isomorphism $H_b^n(X,\R)\cong H_b^n(\pi_1(X),\R)$ even without
any assumption on the asphericity of $X$. This section is devoted to a proof
of Gromov's result. We will closely follow Ivanov's argument~\cite{Ivanov},
which deals with the case when $X$ is (homotopically equivalent to)
a countable CW-complex. Before going into Ivanov's argument, we will concentrate our attention
on the easier case of aspherical spaces.

\section{Bounded singular cochains as relatively injective modules}\label{sing:inj:sec}
Henceforth, we assume that $X$ is a path connected topological space
admitting the universal covering $\widetilde{X}$, we 
denote by $\G$ the fundamental group
of $X$, and we fix an identification of $\G$ with the group of covering automorphisms
of $\widetilde{X}$. Just as we did in Section~\ref{revisited:sec} for $C^\bullet(\widetilde{X},R)$, we endow
$C^\bullet_b (\widetilde{X},R)$ with the structure of a normed $R[\G]$-module. 
Our arguments are based on the obvious but fundamental isometric identification
$$
C^\bullet_b(X,R)\ \cong \ C^\bullet_b(\widetilde{X},R)^\G\ ,
$$
which induces a canonical isometric identification
$$
H^\bullet_b(X,R)\ \cong H^\bullet(C^\bullet_b (\widetilde{X},R)^\G)\ .
$$
As a consequence, in order to prove the isomorphism $H^\bullet_b(X,R)\cong H^\bullet_b(\G,R)$ it is sufficient to show that the complex
$C^\bullet_b(\widetilde{X},R)$ provides a relatively injective strong
resolution of $R$ (an additional argument shows that this isomorphism
is also isometric). The relative injectivity of the modules $C^n_b(\widetilde{X},R)$
can be easily deduced from the argument described in the proof of
Lemma~\ref{sing:inj}, which applies verbatim in the context of bounded singular cochains:

\begin{lemma}\label{sing:bd:inj}
For every $n\in\matN$, the bounded cochain module
$C_b^n(\widetilde{X},R)$ is relatively injective.
\end{lemma}

Therefore, in order to show that the bounded cohomology of $X$ is isomorphic to the bounded cohomology of $\G$ we need to show that
the (augmented) complex $C_b^\bullet (\widetilde{X},R)$ provides a strong resolution of $R$. We will show that this is the case if ${X}$
is aspherical. In the general case, it is not even true that $C_b^\bullet(\widetilde{X},R)$ is acyclic (see Remark~\ref{znotwork}). However, in the case
when $R=\R$ a deep result by Ivanov shows that the complex $C_b^\bullet(\widetilde{X},\R)$ indeed provides a strong resolution
of $\R$.  A sketch of Ivanov's proof will be given in Section~\ref{ivanov:sec}.

Before going on, we point out that we always have a norm non-increasing map
from the bounded cohomology of $\G$ to the bounded cohomology of $X$:

\begin{lemma}\label{normnon}
Let us endow the complexes $C^\bullet_b(\G,R)$ and $C^\bullet_b(\widetilde{X},R)$
with the obvious augmentations. Then,
there exists a norm non-increasing $\G$-equivariant chain map 
$$
r^\bullet \colon C^\bullet_b(\G,R)\to C^\bullet_b(\widetilde{X},R)
$$
extending the identity
of $R$.
\end{lemma}
\begin{proof}
 Let us choose a set of representatives $R$ for the action of $\G$ on $\widetilde{X}$. We consider the map $r_0\colon S_0(\widetilde{X})=\widetilde{X}\to \G$ taking a point $x$ to the unique
$g\in \G$ such that $x\in g(R)$. For $n\geq 1$, 
we define $r_n\colon S_n(\widetilde{X})\to \G^{n+1}$ by setting
$r_n(s)=(r_0(s (e_0)),\ldots,r_n(s(e_n)))$, where
$s (e_i)$ is the $i$-th vertex of $s$.
Finally, we extend $r_n$ to $C_n(\widetilde{X},R)$ by $R$-linearity
and define $r^n$ as the dual map of $r_n$. Since $r_n$ takes every single simplex
to a single $(n+1)$-tuple, it is readily seen that $r^n$ is  norm non-increasing
(in particular, it takes bounded cochains into bounded cochains). The fact that
$r^n$ is a $\G$-equivariant chain map is obvious. 
\end{proof}

\begin{cor}\label{classifyingmap}
Let $X$ be a
path connected topological space. Then, for every $n\in\matN$ there exists 
a natural norm non-increasing map
$$
H^n_b(\G,R)\to H^n_b(X,R)\ .
$$
\end{cor}
\begin{proof}
 By Lemma~\ref{normnon}, there exists a norm non-increasing chain map 
$$
r^\bullet \colon C^\bullet_b(\G,R)\to C^\bullet_b(\widetilde{X},R)
$$
extending the identity
of $R$, so we may define the desired map $H^n_b(\G,R)\to H^n_b(X,R)$
to be equal to $H^n_b(r^\bullet)$. We know that every $C^n_b(\widetilde{X},R)$ is relatively injective
as an $R[\G]$-module, while $C^\bullet_b(\G,R)$ provides a strong resolution of $R$ over $R[\G]$, so
Theorem~\ref{ext:bounded:thm} ensures that $H^n_b(r^\bullet)$ does not depend on the choice
of the particular chain map $r^\bullet$ which extends the identity of $R$.
\end{proof}

A very natural question is whether the map provided by Corollary~\ref{classifyingmap}
is an (isometric) isomorphism. In the following sections we will see that this holds true when
$X$ is aspherical or when $R$ is equal to the field of real numbers. The fact that $H^n_b(\G,R)$ is isometrically isomorphic
to $H^n_b(X,R)$ when $X$ is aspherical is not surprising, and just generalizes the analogous result for classical cohomology:
the (bounded) cohomology of a group $\G$ may be defined as the (bounded) cohomology of any aspherical space having $\G$
as fundamental group; this is a well posed definition since any two such spaces are homotopically equivalent, and
(bounded) cohomology is a homotopy invariant. On the other hand, the fact that the isometric isomorphism
$H^n_b(\G,\R)\cong H^n_b(X,\R)$ holds even without the assumption that $X$ is aspherical is a very deep result
due to Gromov~\cite{Gromov}.

We begin by proving the following:

\begin{lemma}\label{cones}
Let $Y$ be a path connected topological space. If $Y$ is aspherical,
then the augmented complex $C^\bullet_b(Y,R)$ admits a contracting homotopy (so it is a strong resolution of $R$). If $\pi_i(Y)=0$ for every $i\leq n$, 
then there exists a partial
contracting homotopy
$$
\xymatrix{
R &  C^0_b(Y,R) \ar[l]_(.6){k^0} &
C^1_b(Y,R) \ar[l]_{k^1} &
 \quad \ldots \quad \ar[l]_(.4){k^2} & C^n_b(Y,R) \ar[l]_(.5){k^n}
& C^{n+1}_b(Y,R) \ar[l]_(.5){k^{n+1}}
}
$$
(where we require that the equality $\delta^{m-1}k^m+k^{m+1}\delta^m={\rm Id}_{C^m_b(Y,R)}$
holds for every $m\leq n$, and $\delta^{-1}=\varepsilon$ is the usual augmentation).
\end{lemma}
 \begin{proof}
For every $-1\leq m\leq n$ we construct a map 
$T_m\colon C_{m}(Y,R)\to C_{m+1}(Y,R)$ sending every single simplex into a single simplex
in such a way that $d_{m+1}T_m+T_{m-1}d_m={\rm Id}_{C_m(Y,R)}$, where we understand that $C_{-1}(Y,R)=R$ and
$d_{0}\colon C_0(Y,R)\to R$ is the augmentation map
$d_0(\sum r_i y_i)=\sum r_i$. 
 Let us fix 
a point $y_0\in Y$, and
define $T_{-1}\colon R\to C_0(Y,R)$ by setting
$T_{-1}(r)=ry_0$. For $m\geq 0$ we define $T_m$ as the $R$-linear extension of 
a map $T_m\colon S_m(Y)\to S_{m+1}(Y)$ having the following property: for every
$s\in S_m(Y)$, the $0$-th vertex of $T_m(s)$ is equal to $y_0$, and 
has $s$ as opposite face. 
We proceed by induction, and 
suppose that $T_i$ has been defined for every $-1\leq i\leq m$.
Take $s\in S_m(Y)$. Then, using the fact that
$\pi_m(Y)=0$ and the properties of $T_{m-1}$, it is not difficult to show that a simplex 
$s'\in S_{m+1}(Y)$ exists which satisfies both 
the equality $d_{m+1}s'=s-T_{m-1}(d_m s)$ and the additional requirement described above.
We set $T_{m+1}(s)=s'$, and we are done.

Since $T_{m-1}$ sends every single simplex to a single simplex, its dual map
$k^{m}$ sends bounded cochains into bounded cochains, and has operator norm
equal to one. Therefore,  
the maps $k^m\colon C^m_b(Y,R)\to C_b^{m-1}(Y,R)$, $m\leq n+1$, provide the desired
(partial) contracting homotopy.
 \end{proof}

\section{The aspherical case}
We are now ready to show that, under the assumption that
$X$ is aspherical, 
the bounded cohomology of $X$ is isometrically isomorphic to the bounded cohomology
of $\G$ (both with integral and with real coefficients):

\begin{thm}\label{aspherical:thm}
 Let $X$ be an aspherical space, i.e.~a path connected topological
space such that $\pi_n(X)=0$ for every $n\geq 2$.
Then
$H^n_b(X,R)$ is isometrically isomorphic to $H^n_b(\G,R)$ for every $n\in\mathbb{N}$.
\end{thm}
\begin{proof}
Lemmas~\ref{sing:bd:inj} and~\ref{cones}
imply that 
the complex
$$
0\tto{} R\tto{\varepsilon} C^0_b(\widetilde{X},R)\tto{\delta^0}
C^1_b(\widetilde{X},R)\tto{\delta^1}\ldots
$$
provides a relatively injective strong resolution of $R$ as a trivial $R[\G]$-module,
so $H^n_b(X,R)$ is canonically isomorphic to $H^n_b(\G,R)$ for every $n\in\mathbb{N}$. 
The fact the the isomorphism $H^n_b(X,R)\cong H^n_b(\G,R)$ is isometric is a consequence
of Corollary~\ref{norm:cor} and Lemma~\ref{normnon}.
\end{proof}

\begin{rem}\label{wedge}
 Let $S^1\vee S^1$ be the wedge of two copies of the circle. Then
Theorem~\ref{aspherical:thm} implies that $H^2_b(S^1\vee S^1,\R)\cong H^2_b(F_2,\R)$ is infinite-dimensional,
while $H^2_b(S^1,\R)\cong H^2_b(\mathbb{Z,\R})=0$. This shows that bounded cohomology of spaces
cannot satisfy any Mayer-Vietoris principle.
\end{rem}

\section{Ivanov's contracting homotopy}\label{ivanov:sec}
We now come back to the general case, i.e.~we do not assume that $X$
is aspherical. In order to show that $H^n_b(X,R)$ is isometrically isomorphic
to $H^n_b(\G,R)$ we need to prove that the complex of singular bounded cochains 
on $\wdtX$ provides
a strong resolution of $R$. In the case when $R=\mathbb{Z}$, this is false in general,
since the complex $C^\bullet_b(\wdtX,\mathbb{Z})$ may even be  non-exact (see Remark~\ref{znotwork}). On the other hand, a deep result by Ivanov ensures that 
$C^\bullet_b(\wdtX,\R)$ indeed provides  a strong resolution of $\R$. 

Ivanov's argument makes use of sophisticated techniques 
from algebraic topology, which work under the assumption that $X$, whence
$\wdtX$, is a countable CW-complex (but see Remark~\ref{bu}). We begin with the following:

\begin{lemma}[\cite{Ivanov}, Theorem 2.2]\label{principal}
Let $p\colon Z\to Y$ be a principal $H$-bundle, where $H$ is an abelian topological group.
Then there exists a chain map $A^\bullet\colon C^\bullet_b(Z,\R)\to C^\bullet_b(Y,\R)$
such that $\|A^n\|\leq 1$ for every $n\in\mathbb{N}$ and $A^\bullet\circ p^\bullet$ is the identity
of $C^\bullet_b(Y,\R)$.
\end{lemma}
\begin{proof}
For $\varphi\in C^n_b(Z,\R)$, $s\in S_n(Y)$, the value
$A(\varphi)(s)$ 
is obtained by suitably averaging the value of $\varphi(s')$, where $s'$ ranges
over the set
$$
P_s=\{s'\in S_n(Z)\, |\, p\circ s'=s\}\ .
$$

More precisely, let $K_n=S_n(H)$ be the space of continuous functions from the standard $n$-simplex
to $H$, and define on $K_n$ the operation given by pointwise multiplication
of functions. With this structure, $K_n$ is an abelian group,
so it admits an invariant mean $\mu_n$. Observe that the permutation
group $\mathfrak{S}_{n+1}$ acts on $\Delta^{n}$ via affine
transformations. This action induces an action on $K_n$, whence
on $\ell^\infty(K_n)$, and on the space of means on $K_n$,
so there is an obvious notion of 
$\mathfrak{S}_{n+1}$-invariant mean on $K_n$. By averaging over the action
of $\mathfrak{S}_{n+1}$, the space of $K_n$-invariant means may be retracted
onto the space $\mathcal{M}_n$ of $\mathfrak{S}_{n+1}$-invariant $K_n$-invariant means
on $K_n$, which, in particular, is non-empty. Finally, observe that
any affine identification of $\Delta^{n-1}$
with a face of $\Delta^n$ induces a map $\mathcal{M}_n\to \mathcal{M}_{n-1}$.
Since elements of $\mathcal{M}_n$ are $\mathfrak{S}_{n+1}$-invariant, this map
does not depend on the chosen identification. Therefore, we get a sequence of maps
$$
\xymatrix{
\mathcal{M}_0 & \mathcal{M}_1 \ar[l] & \mathcal{M}_2 \ar[l]
& \mathcal{M}_3 \ar[l] & \ar[l] \ldots
}
$$
Recall now that the Banach-Alaouglu Theorem implies that
every $\mathcal{M}_n$ is compact (with respect to the weak-$*$ topology on $\ell^\infty(K_n)'$). 
This easily implies that there exists a
sequence $\{\mu_n\}$ of means such that $\mu_n\in\mathcal{M}_n$ and
$\mu_n\mapsto \mu_{n-1}$ under the map $\mathcal{M}_n\to\mathcal{M}_{n-1}$.
We say that such a sequence is \emph{compatible}.

Let us now fix 
$s\in S_n(Y)$, and observe that there is a bijection between
$P_s$ and $K_n$. This bijection is uniquely determined up to 
the choice of an element in $P_s$, i.e.~up to left multiplication by an element of $K_n$.
In particular, if $f\in\ell^\infty(P_s)$, and $\mu$ is a left-invariant mean on $K_n$,
then there is a well-defined value $\mu(f)$.

Let us now choose a compatible sequence of means $\{\mu_n\}_{n\in\mathbb{N}}$.
We define
the operator $A^n$ by setting
$$
A^n(\varphi)(s)=\mu_n (\varphi|_{P_s})\qquad \text{for every}\
\varphi\in C^n_b(Z),\ s\in S_n(Y)\ .
$$
Since the sequence $\{\mu_n\}$ is compatible, the sequence of maps $A^\bullet$
is a chain map. The inequality
$\|A^n\|\leq 1$ is obvious, and the fact that $A^\bullet\circ p^\bullet$ is the identity
of $C^\bullet_b(Y,\R)$ may be deduced from the behaviour of means on constant functions.
\end{proof}

\begin{thm}[\cite{Ivanov}]\label{ivanov:thm}
 Let $X$ be a path connected countable CW-complex with universal covering $\wdtX$. Then
the (augmented) complex $C^\bullet_b(\wdtX,\R)$ provides a relatively injective
strong resolution of $\R$.
\end{thm}
\begin{proof}
 We only sketch Ivanov's argument, referring the reader to~\cite{Ivanov}
for full details.
Building on results by Dold and Thom~\cite{DT},
Ivanov constructs an infinite 
tower of bundles
$$
\xymatrix{
X_1 & X_2 \ar[l]_(.45){p_1} & X_3 \ar[l]_{p_2} & \ldots \ar[l] & X_n \ar[l] & \ldots \ar[l]
}
$$
%\ldots \ar[r]^{p_m} & X_m \ar[r]^(.4){p_{m-1}} & X_{m-1}\ar[r]^{p_{m-2}} & \ldots & \ar[r]^{p_2} & X_2 \ar[r]^{p_{1}} & X_1\ ,
%}
%$$
where $X_1=\wdtX$, $\pi_i(X_m)=0$ for every $i\leq m$, $\pi_i(X_m)=\pi_i(X)$ for every $i>m$ and each map
$p_m\colon X_{m+1}\to X_m$ is a principal $H_m$-bundle for some topological connected abelian group
$H_m$, which has the homotopy type of a $K(\pi_{m+1}(X),m)$. 

By Lemma~\ref{cones}, for every $n$ we may construct a partial contracting homotopy
$$
\xymatrix{
\R  &  C^0_b(X_n,\R) \ar[l]_(.6){k^0_n} &
C^1_b(X_n,\R) \ar[l]_{k^1_n} &
\quad  \ldots\quad  \ar[l]_(.4){k^2_n} & C^{n+1}_b(X_n,\R) \ar[l]_(.55){k^{n+1}_n}\ .
}
$$

Moreover, Lemma~\ref{principal} implies that for every $m\in\mathbb{N}$ the chain map $p^\bullet_m\colon C_b^\bullet (X_m,\R)\to C_b^\bullet(X_{m+1},\R)$ admits a left inverse chain map
$A_m^\bullet\colon C_b^\bullet (X_{m+1},\R)\to C_b^\bullet(X_{m},\R)$ which is norm non-increasing. This allows us to define 
a partial contracting homotopy
$$
\xymatrix{
\R  &  C^0_b(X,\R) \ar[l]_(.6){k^0} &
C^1_b(X,\R) \ar[l]_{k^1} &
 \quad \ldots \quad \ar[l]_(.4){k^2} & C^{n+1}_b(X,\R) \ar[l]_(.55){k^{n+1}}\ & \ldots \ar[l] 
}
$$
via the formula
$$
k^i=A^{i-1}_1 \circ \ldots \circ A^{i-1}_{n-1} \circ k_{n}^i \circ p_{n-1}^i   \circ \ldots \circ p_2^i \circ  p_1^i\qquad {\rm for\ every\ }i\leq n+1\ .
$$
The existence of such a partial homotopy is sufficient for all our applications.
In order to construct a complete contracting homotopy 
one should check that the definition of $k^i$ does not depend on $n$. This is not automatically true, and it is equivalent to the fact that, for $i\leq n+1$,
the equality $A^i_n\circ k^{i}_{n+1}\circ p^{i}_n=k^i_n$ holds. In order to get this, one has to coherently choose both the basepoints and the cones involved in the contruction providing the 
(partial) contracting homotopy $k_n^\bullet$ (see Lemma~\ref{cones}). 
The fact that this  can be achieved ultimately depends on the fact that the fibers of the bundle
$X_{n+1}\to X_n$ have the homotopy type of a $K(\pi_{n+1}(X),n)$. 
We refer the reader to~\cite{Ivanov} for the details.
\end{proof}

\section{Gromov's Theorem}
The discussion in the previous section implies the following:

\begin{thm}[\cite{Gromov,Ivanov}]\label{gro-iva:thm}
 Let $X$ be a countable CW-complex. Then $H^n_b(X,\R)$ is canonically isometrically
isomorphic 
to $H^n_b(\G,\R)$. An explicit isometric isomorphism is induced by the map
$r^\bullet \colon C^\bullet_b(\G,\R)^\G\to C^\bullet_b(\widetilde{X},\R)^\G= C^\bullet_b(X,\R)$ described in Lemma~\ref{normnon}.
\end{thm}
\begin{proof}
Observe that, if $X$ is a countable CW-complex, then $\G=\pi_1(X)$
is countable, so $\widetilde{X}$ is also a countable CW-complex. Therefore,
 Lemma~\ref{sing:bd:inj} and Theorem~\ref{ivanov:thm}
imply that 
the complex
$$
0\tto{} \R\tto{\varepsilon} C^0_b(\widetilde{X},\R)\tto{\delta^0}
C^1_b(\widetilde{X},\R)\tto{\delta^1}\ldots
$$
provides a relatively injective strong resolution of $\R$ as a trivial $\R[\G]$-module,
so $H^n_b(X,\R)$ is canonically isomorphic to $H^n_b(\G,\R)$ for every $n\in\mathbb{N}$. 
The fact the the isomorphism $H^n_b(X,\R)\cong H^n_b(\G,\R)$ induced by $r^\bullet$ is isometric is a consequence
of Corollary~\ref{norm:cor} and Lemma~\ref{normnon}.
\end{proof}

\begin{rem}\label{bu}
Theo B\"ulher recently proved 
that Ivanov's argument may be generalized to show that $C^\bullet_b(Y,\R)$ is a strong
resolution of $\R$ whenever $Y$ is a simply connected topological space~\cite{Theo}.
As a consequence, Theorem~\ref{gro-iva:thm} holds even when the assumption that $X$ is a countable CW-complex is replaced by the weaker condition that $X$ is path connected
and admits a universal covering.
\end{rem}

\begin{cor}[Gromov Mapping Theorem]\label{mapping:cor}
 Let $X,Y$ be path connected countable CW-complexes and let $f\colon X\to Y$
be a continuous map inducing an epimorphism $f_*\colon \pi_1(X)\to \pi_1(Y)$ with amenable kernel. Then
$$
H^n_b(f)\colon H^n_b(Y,\R)\to H^n_b(X,\R)
$$
is an isometric isomorphism for every $n\in\mathbb{N}$.
\end{cor}
\begin{proof}
The explicit description of the isomorphism between the bounded cohomology
of a space and the one of its fundamental group implies that the diagram
$$
\xymatrix{
H^n_b(Y)\ar@{<->}[d]\ar[r]^{H^n_b(f)} & H^n_b(X)\ar@{<->}[d]\\
H^n_b(\pi_1(Y)) \ar[r]^{H^n_b(f_*)}& H^n_b(\pi_1(X)) 
}
$$
is commutative, where the vertical arrows represent the isometric isomorphisms of Theorem~\ref{gro-iva:thm}. Therefore, the conclusion follows from Corollary~\ref{map:am:cor}.
\end{proof}

\begin{rem}\label{znotwork}
Theorem~\ref{gro-iva:thm} does not hold for bounded cohomology with integer coefficients.
In fact, if $X$ is any topological space, then
the short exact sequence 
$0\to \matZ \to \R \to \R/\matZ\to 0$
induces an exact sequence
$$
H^n_b(X,\R)\tto{} H^n(X,\R/\matZ)\tto{} H_b^{n+1}(X,\matZ)\tto{} H_b^{n+1}(X,\R)
$$
(see the proof of Proposition~\ref{h2zz}).
If $X$ is a simply connected CW-complex, then $H^n_b(X,\R)=0$ for every $n\geq 1$, so we have
$$
H^n_b(X,\matZ)\ \cong \ H^{n-1}_b(X,\R/\matZ)\qquad \text{for every}\ n\geq 2.
$$
For example, in the case of the 2-dimensional sphere we have
$H^3_b(S^2,\matZ)\cong H^2(S^2,\R/\matZ)\cong \R/\matZ$. 
\end{rem}

\section{Alternating cochains}\label{altern:sec}
We have seen in Section~\ref{altern:group:sec} that (bounded) cohomology of groups may be computed 
via the complex of alternating cochains. The same holds true also in the context of (bounded) singular cohomology
of topological spaces. 

If $\sigma\in \mathfrak{S}_{n+1}$ is any permutation of $\{0,\ldots,n\}$, then 
we denote by $\overline{\sigma}\colon \Delta^n\to \Delta^n$ the affine automorphism
of the standard simplex which induces the permutation $\sigma$ on the vertices of $\Delta^n$.
Then we say that a cochain $\varphi\in C^n(X,\R)$ is \emph{alternating} if
$$
\varphi(s)={\rm sgn}(\sigma)\cdot \varphi (s\circ \overline{\sigma})
$$
for every $s\in S_n(X)$, $\sigma\in\mathfrak{S}_{n+1}$. We also denote by $C^\bullet_{\alt}(X,\R)\subseteq C^\bullet(X,\R)$
the subcomplex of alternating cochains, and we set $C^\bullet_{b,\alt}(X,\R)=C^\bullet_{\alt}(X,\R)\cap C^\bullet_b(X,\R)$.
Once a generic cochain $\varphi\in C^n(X,\R)$ is given, we may alternate it by setting
$$
\alt^n(\varphi)(s)=\frac{1}{(n+1)!}\sum_{\sigma\in \mathfrak{S}_{n+1}} {{\rm sgn}(\sigma)}\cdot \varphi(s\circ \overline{\sigma})
$$
for every $s\in S_n(X)$.
 Then
the very same argument exploited in the proof of Proposition~\ref{alternating:prop}
applies in this context to give the following:

\begin{prop}\label{alternating:sing:prop}
The complex $C^\bullet_{\alt}(X,\R)$ (resp.~$C^\bullet_{b,\alt}(X,\R)$) isometrically computes
the cohomology (resp.~the bounded cohomology) of $X$ with real coefficients.
\end{prop}

\section{Relative bounded cohomology}\label{relbounded:sec}
When dealing with manifolds with boundary, it is often useful to study \emph{relative} homology and cohomology.
For example, in Section~\ref{dual:simpl:sec} we will show how the simplicial volume of a manifold with boundary $M$
can be computed via the analysis of the \emph{relative} bounded cohomology module of the pair
$(M,\partial M)$. This will prove useful to show that the simplicial volume is additive with respect to gluings
along boundary components with amenable fundamental groups.

Until the end of the chapter, all the cochain and cohomology modules will be assumed to be with real coefficients.
Let $Y$ be a subspace of the topological space $X$. We denote by $C^n(X,Y)$ the submodule of cochains which vanish
on simplices supported in $Y$. In other words, $C^n(X,Y)$ is the kernel of the map
$C^n(X)\to C^n(Y)$ induced by the inclusion $Y\hookrightarrow X$. We also set $C^n_b(X,Y)=C^n(X,Y)\cap C^n_b(X)$,
and we denote by $H^\bullet(X,Y)$ (resp.~$H^\bullet_b(X,Y)$) the cohomology of the complex
$C^\bullet(X,Y)$ (resp.~$C_b^\bullet(X,Y)$). The well-known short exact sequence of the pair for ordinary cohomology
also holds for bounded cohomology: the short exact sequence of complexes
$$ 
0\tto{} C^\bullet_b(X,Y)\tto{} C_b^\bullet (X)\tto{} C_b^\bullet(Y)\tto{} 0
$$
induces the long exact sequence
\begin{equation}\label{long:pair}
 \ldots \tto{} H^n_b(Y)\tto{} H^{n+1}_b(X,Y)\tto{} H^{n+1}_b(X)\tto{} H^{n+1}_b (Y)\tto{}\ldots
\end{equation}
Recall now that, if the fundamental group of every component of $Y$ is amenable, then
$H^n_b(Y)=0$ for every $n\geq 1$, so the inclusion $j^n\colon C^n_b(X,Y)\to C^n_b(X)$
induces a norm non-increasing isomorphism
$$
H^n_b(j^n)\colon  H^n_b(X,Y)\to H^n_b(X)
$$
for every $n\geq 2$. 
The following result is proved in~\cite{BBFIPP}, and shows that this isomorphism is in fact isometric:

\begin{thm}\label{reliso:thm}
 Let $(X,Y)$ be a pair of countable CW-complexes, and suppose that the fundamental group of each component
 of $Y$ is amenable. Then the map
 $$
H^n_b(j^n)\colon  H^n_b(X,Y)\to H^n_b(X)
$$
is an isometric isomorphism for every $n\geq 2$.
\end{thm}
The rest of this section is devoted to the proof of Theorem~\ref{reliso:thm}.

Henceforth we assume that $(X,Y)$ is a pair of countable CW-complexes such that
the fundamental group of every component of $Y$ is amenable. 
Let $p\colon \widetilde{X}\to X$ be a universal covering of $X$, and
set $\widetilde{Y}=p^{-1}(Y)$. 
As usual, we denote by $\G$ the fundamental group of $X$, we fix an identification of $\G$ with the group
of the covering automorphisms of $p$, and we consider the induced 
identification
$$
C_b^n(X)=C_b^n(\widetilde{X})^{\G}\ .
$$

\begin{defn}\label{special:section}
We say that a cochain $\varphi\in C_b^n(\widetilde{X})$
%=C_b^n(\widetilde{X})^\G$ 
is \emph{special} (with respect to $\widetilde{Y}$) if the following conditions
hold:
\begin{itemize}
\item
$\varphi$ is alternating;
 \item
  let ${s},{s}'$ be singular $n$-simplices with values
in $\widetilde{X}$ and suppose that, for every $i=0,\ldots,n$, 
either ${s}(w_i)={s}'(w_i)$, or ${s}(w_i)$ and ${s}'(w_i)$
belong to the same connected component of $\widetilde{Y}$,
where $w_0,\ldots,w_n$ are the vertices of the standard
$n$-simplex. Then $\varphi({s})=\varphi({s}')$.
\end{itemize}
We denote by $ C_{bs}^\bullet(\widetilde{X},\widetilde{Y})\subseteq C_b^\bullet (\widetilde{X})$ the 
subcomplex
of special cochains, and we set 
$$ C_{bs}^\bullet(X,Y)= C_{bs}^\bullet(\widetilde{X},\widetilde{Y})\cap C_b^\bullet(\widetilde{X})^\G\subseteq C_b^\bullet(\widetilde{X})^\G
=C_b^\bullet(X)\ .$$
\end{defn}

\begin{rem}\label{special:rem}
Any cochain $\varphi\in  C_{bs}^\bullet(\widetilde{X},\widetilde{Y})$ vanishes on every simplex 
having two vertices on the same connected component of $\widetilde{Y}$. 
In particular
$$
 C_{bs}^n(X,Y)\subseteq C_b^n(X,Y)\subseteq C_b^n(X)
$$
for every $n\geq 1$.
\end{rem}

We denote by $l^\bullet\colon  C_{bs}^\bullet(X,Y)\to C_b^\bullet(X)$ the natural inclusion.

%We denote by $H_bs^\bullet (X,Y)$ the homology of the complex $ C_{bs}^\bullet(X,Y)$. The inclusion
%$l^\bullet\colon  C_{bs}^\bullet(X,Y)\to C_b^\bullet(X)$ induces a map $H_b^\bullet(l^\bullet)\colon
%H_bs^\bullet(X,Y) \to H_b^\bullet(X)$ in bounded cohomology.

\begin{prop}\label{special:prop}
There exists a norm non-increasing chain map 
$$
\eta^\bullet\colon C_b^\bullet(X)\to C_{bs}^\bullet (X,Y)
$$
such that the composition $l^\bullet\circ\eta^\bullet$ is chain-homotopic
to the identity of $C_b^\bullet(X)$.
\end{prop}
\begin{proof}
Let us briefly describe the strategy of the proof. First of all, we will define a $\G$-set $S$ which 
provides a sort of discrete approximation of the pair $(\widetilde{X},\widetilde{Y})$. As usual, the group $\G$ already provides 
an approximation of $\widetilde{X}$. However, in order to prove that $Y$ is completely irrelevant from the point of view of bounded cohomology,
we need to approximate every component of $\widetilde{Y}$ by a single point, and this implies that the set $S$ cannot coincide with
$\G$ itself. Basically, we add one point for each component of $\widetilde{Y}$. 
Since the fundamental group of every component of $Y$ is amenable,
the so obtained $\G$-set $S$ is amenable: therefore, the bounded cohomology of $\G$, whence of $X$, may be isometrically
computed using the complex of 
alternating cochains on $S$. Finally, alternating cochains on $S$ can  be isometrically translated into special cochains on $X$. 

Let us now give some more details.
Let $Y=\sqcup_{i\in I}C_i$ be the decomposition of $Y$
into the union of its connected components. If  $\check C_i$ is a choice of a connected component of $p^{-1}(C_i)$ 
and $\Gamma_i$ denotes the stabilizer of $\check C_i$ in $\Gamma$ then
$$
p^{-1}(C_i)=\bigsqcup_{\gamma\in\Gamma/\Gamma_i}\gamma\check C_i\, .
$$ 

We endow the set
$$
S=\Gamma\sqcup\bigsqcup_{i\in I}\Gamma/\Gamma_i
$$
with the obvious structure of $\G$-set
and we choose a fundamental domain
$\Ff\subset\widetilde X\smallsetminus \widetilde{Y}$ 
for the $\Gamma$-action on $\widetilde X\smallsetminus \widetilde{Y}$. 
We define a $\Gamma$-equivariant map 
$r\colon\widetilde X\to S$
as follows:
\bqn
r(\gamma x)=
\begin{cases}
\hphantom{\Gamma}\,\gamma\in\Gamma&\text{ if }x\in\Ff,\\
\gamma\Gamma_i\in \Gamma/\Gamma_i &\text{ if } x \in\check C_i\,.
\end{cases}
\eqn
For every $n\geq0$ we set $\ell^\infty_{\alt}(S^{n+1})=\ell^\infty_{\alt}(S^{n+1},\R)$ and we define
$$
r^n\colon\ell^\infty_{\mathrm{alt}}(S^{n+1})\to C_{bs}^n(\widetilde X,\widetilde{Y})
\, ,\qquad 
r^n(f)(s)=f(r(s(e_0)),\dots,r(s(e_n)))\ .
$$
The fact that $r^n$ takes values in the module of special cochains is immediate, and
clearly $r^\bullet$ is a norm non-increasing $\Gamma$-equivariant chain map  extending the identity on $\R$.

Recall now that, by Theorem~\ref{ivanov:thm}, the complex $C^\bullet_b(\widetilde{X})$ provides
a relatively injective strong resolution of $\R$, so there exists a norm non-increasing $\G$-equivariant chain map
$C^\bullet_b(\widetilde{X})\to C^\bullet_b(\G,\R)$. Moreover, by composing the map provided by 
Theorem~\ref{Siso} and the obvious alterating operator we get a 
norm non-increasing $\G$-equivariant chain map
$C^\bullet_b(\G,\R)\to \ell^\infty_{\alt}(S^{\bullet +1})$. By composing these morphisms of normed $\G$-complexes we finally get a 
norm non-increasing $\G$-chain map
$$
\zeta^\bullet\colon C^\bullet_b(\widetilde{X})\to \ell^\infty_{\alt}(S^{\bullet +1})
$$
which extends the identity of $\R$.

Let us now consider the composition $\theta^\bullet=r^\bullet\circ \zeta^\bullet\colon C^\bullet_b(\widetilde{X})\to C^\bullet_{bs}(\widetilde{X},\widetilde{Y})$:
it is a norm non-increasing $\G$-chain map which extends the identity of $\R$. 
It is now easy to check that the chain map
$\eta^\bullet\colon C^\bullet_b({X})\to C^\bullet_{bs}({X},{Y})$ induced by $\theta^\bullet$ satisfies the required properties.
In fact, $\eta^\bullet$ is obvously norm non-increasing. Moreover,
the composition
of $\theta^\bullet$ with the inclusion $C^\bullet_{bs}(\widetilde{X},\widetilde{Y})\to C^\bullet_{b}(\widetilde{X})$
extends the identity of $\R$. Since $C^\bullet_{b}(\widetilde{X})$ is a relatively injective strong resolution of $\R$,
this implies in turn that this composition is $\G$-homotopic to the identity of $\R$, thus concluding the proof.
\end{proof}

\begin{cor}\label{special:cor}
Let $n\geq 1$, take $\alpha\in H^n_b(X)$ and let $\varepsilon>0$ be given. Then there exists a special cocycle
$f\in C^n_{bs}(X,Y)$ such that 
$$
[f]=\alpha\, ,\qquad \|f\|_\infty\leq \|\alpha\|_\infty+\varepsilon\ .
$$
\end{cor}

Recall now that $C^n_{bs}(X,Y)\subseteq C^n_b(X,Y)$ for every $n\geq 1$. Therefore, Corollary~\ref{special:cor}
implies that, for $n\geq 1$, the norm of every coclass in $H^n_b(X)$ may be computed
by taking the infimum over relative cocycles. 
Since we already know that the inclusion $C^\bullet_{b}(X,Y)\hookrightarrow C^\bullet_b(X)$ induces an isomorphism in bounded cohomology
in degree greater than one,
this concludes the proof of Theorem~\ref{reliso:thm}.

\section{Further readings}
It would be interesting to extend to the relative case Gromov's and Ivanov's results on the coincidence of bounded cohomology of (pairs of) spaces
with the bounded cohomology of their (pairs of) fundamental groups. The case when the subspace is path connected corresponds to the case when the pair of groups indeed consists of a group
and one of its subgroups, and it is somewhat easier than the general case. It was first treated by Park in~\cite{Park}, and then by Pagliantini and the author in~\cite{FP2}. Rather disappointingly,
Ivanov's cone construction runs into some difficulties in the relative case, so some extra assumptions on higher homotopy groups is needed in order to get the desired isometric isomorphism
(if one only requires a bi-Lipschitz isomorphism, then the request that the subspace is $\pi_1$-injective in the whole space is sufficient).

In the case of a disconnected subspace it is first necessary to define bounded cohomology for pairs $(\G,\G')$, where $\G'$ is a family of subgroups
of $\G$. This was first done for classical cohomology of groups by Bieri and Eckmann~\cite{Bieri}, and extended to the case of bounded cohomology
by Mineyev and Yaman~\cite{Yaman} (see also~\cite{Franceschini2}). Probably the best suited approach to bounded cohomology of generic pairs is via the theory 
of bounded cohomology for \emph{groupoids}, as introduced and developed by Blank in~\cite{Blank}. Many results from~\cite{FP2} are extended in~\cite{Blank} to the case
of disconnected subspaces (under basically the same hypotheses on higher homotopy groups that were required in~\cite{FP2}).

The theory of multicomplexes initiated by Gromov in~\cite{Gromov} may also be exploited to study (relative) bounded cohomology of (pairs of) spaces. We refer the reader to~\cite{Kuessner} for some results in this direction.

\chapter{$\ell^1$-homology and duality}\label{duality:chap}
Complexes of cochains naturally arise by taking duals of complexes of chains. Moreover, the Universal Coefficient Theorem ensures
that, at least when working with real coefficients, taking (co)homology commutes with taking duals. Therefore, cohomology with real coefficients
is canonically isomorphic to the dual of homology with real coefficients. 
In this chapter we describe analogous results in the context of 
bounded cohomology, showing that also in the bounded case duality plays an important role in the study of the relations between homology and cohomology.
We restrict our attention to the case with real coefficients. 

\section{Normed chain complexes and their topological duals}
Before going into the study of the cases we are interested in, we introduce some general terminology and recall some general results proved 
in~\cite{Matsu-Mor,Loeh}.
A normed chain complex is a complex
$$
\xymatrix{
0 & C_0 \ar[l] & C_1 \ar[l]^{d_1} & C_2 \ar[l]^{d_2} & \ldots \ar[l]
}
$$
where every $C_i$ is a normed real vector space, and $d_i$ is bounded for every $i\in \mathbb{N}$. 
For notational convenience, in this section we will denote such a complex by  the symbol $(C,d)$, rather than
by  $(C_\bullet,d_\bullet)$.
If $C_n$ is complete for every $n$, then we say that $(C,d)$ is a Banach chain complex (or simply a Banach complex).
Let $(C')^i=(C_i)'$ be the \emph{topological} dual
of $C_i$, endowed with the operator norm, and denote by $\delta^i\colon (C')^i\to (C')^{i+1}$ the dual map of $d_{i+1}$. Then the normed cochain complex
$$
\xymatrix{
0\ar[r] & (C')^0\ar[r]^{\delta^0} & (C')^1\ar[r]^{\delta^1} & (C')^2\ar[r]^{\delta^2} & \ldots
}
$$
is called the normed dual complex of $(C,d)$, and it is denoted by $(C',\delta)$. Observe that the normed dual complex of any normed chain complex
is Banach.

In several interesting cases, the normed spaces $C_i$ are not complete (for example, this is the case of singular chains on topological spaces, endowed with the $\ell^1$-norm -- see below). Let us denote
by $\widehat{C}_i$ the completion of $C_i$. Being bounded, the differential $d_i\colon C_i\to C_{i-1}$
extends to a bounded map $\hat{d}_i\colon \widehat{C}_i\to \widehat{C}_{i-1}$, and it is obvious that 
$\hat{d}_i\circ \hat{d}_{i-1}=0$ for every $i\in\mathbb{N}$, so we may consider the normed
chain complex $(\widehat{C},\hat{d})$. The topological dual of $\widehat{C}_i$ is again
$(C')^i$, and $\delta^i$ is the dual map of $\hat{d}_{i+1}$, so 
$(C',\delta)$ is the normed dual complex also 
of $(\widehat{C},\hat{d})$. The homology of the complex $(\widehat{C},\hat{d})$ is very different in general
from the homology of $(C,d)$; however the inclusion $C\to C'$ induces a seminorm-preserving map in homology (see Corollary~\ref{semipr}).

The homology (resp.~cohomology) of the complex $(C,d)$ (resp.~$(C',\delta)$) is denoted by $H_\bullet(C)$ (resp.~$H_b^\bullet(C')$). As usual, the norms on
$C_n$ and $(C')^n$ induce seminorms on 
$H_n(C)$ and $H_b^n(C')$ respectively. We will denote these (semi)norms respectively by $\|\cdot\|_1$ and $\|\cdot\|_\infty$.
The duality pairing between $(C')^n$ and $C_n$ induces the \emph{Kronecker product}
$$
\langle \cdot,\cdot \rangle\colon H_b^n(C')\times \h_n(C)\to \R\ .
$$

\section{$\ell^1$-homology of groups and spaces}\label{homol:group:sec}
Let us introduce some natural examples of normed dual cochain complexes.
Since we are restricting our attention to the case of real coefficients,
for every group $\G$ (resp.~space $X$)
we simply denote by $H_b^\bullet(\G)$ (resp.~by $H_b^\bullet(X)$) the bounded  cohomology of $\G$ (resp.~of $X$) with real
coefficients. 

For every $n\geq 0$, the space $C_n(X)=C_n(X,\R)$ may be endowed with the $\ell^1$-norm
$$
\left\| \sum a_i s_i\right\|_1 =\sum |a_i|
$$
(where the above sums are finite), which descends to a seminorm $\|\cdot \|_1$ on $H_n(X)$. 
The complex $C^\bullet_b(X)$, endowed with the usual $\ell^\infty$-norm, coincides with the normed dual complex of 
$C_\bullet(X)$. 

Since the differential is bounded in every degree,
to the normed chain complex $C_\bullet(X)$ 
there is associated the normed chain complex obtained by taking the completion
of $C_n(X)$ for every $n\in\mathbb{N}$. Such a complex is denoted
by $\cl_\bullet(X)$. An element $c\in \cl_n(X)$ is a sum
$$
\sum_{s\in S_n(X)} a_s s
$$
such that
$$
\sum_{s\in S_n(X)} |a_s|<+\infty 
$$
and it is called an \emph{$\ell^1$-chain}. The \emph{$\ell^1$-homology} of $X$
is just the homology of the complex $\cl_\bullet(X)$, and it is denoted by
$\hl_\bullet(X)$. The inclusion of singular chains into $\ell^1$-chains
induces a norm non-increasing map
$$
H_\bullet(X)\to\hl_\bullet(X)\ ,
$$
which is in general neither injective nor surjective.

The same definitions may be given in the context
of groups. If $\G$ is a group, for every $n\in\mathbb{N}$ we denote by $C_n(\G)=C_n(\G,\R)$ the $\R$-vector space having $\G^n$ as a basis,
where we understand that $C_0(\G)=\R$.
For every $n\geq 2$ we define
$d_n\colon C_n(\G)\to C_{n-1}(\G)$ as the linear extension of the map
\begin{align*}
(g_1,\ldots,g_n)& \mapsto (g_2,\ldots,g_n)+\sum_{i=1}^{n-1} (-1)^i (g_1,\ldots,g_ig_{i+1},\ldots,g_n) \\ &+ (-1)^n (g_1,\ldots,g_{n-1})
\end{align*}
and we set $d_1=d_0=0$. It is easily seen that $d_{n-1}d_n=0$ for every $n\geq 1$,
and the homology $H_\bullet(\G)$ of $\G$ (with real coefficients) is defined as the homology of the complex $(C_\bullet(\G),d_\bullet)$.

Observe that $C_i(\G)$ admits the $\ell^1$-norm defined by
$$
\left\|\sum a_{g_1,\ldots,g_i} (g_1,\ldots,g_i)\right\|_1=\sum | a_{g_1,\ldots,g_i}|
$$
(where all the sums in the expressions above are finite),
which descends in turn to a seminorm $\|\cdot \|_1$ on $H_i(\G)$. Moreover,
we denote by $\cl_i(\G)$ the completion of $C_i(\G)$ with respect to the
$\ell^1$-norm, and by $\hl_i(\G)=H_i(\cl_\bullet(\G))$ the corresponding
$\ell^1$-homology module, which is endowed with the induced seminorm
(it is immediate to check that the differential $d_n\colon C_n(\G)\to C_{n-1}(\G)$ has norm bounded above by $n$).

It follows from the very definitions that the topological duals 
of $(C_i(\G),\|\cdot \|_1)$ and $(\cl_i(\G),\|\cdot\|_1)$ both coincide with the Banach space 
 $(\overline{C}^n_b(\G,\R),\|\cdot\|_\infty)$ of inhomogeneous
cochains introduced in Section~\ref{bar:sec}. Moreover, the dual map of $d_i$ coincides with the differential 
$\overline{\delta}^{i+1}$, so $C^\bullet_b(\G)$ is the dual normed cochain complex both of $C_\bullet(\G)$ and of $\cl_\bullet(\G)$.

\section{Duality: first results}
As mentioned above, by the Universal Coefficient Theorem, taking (co)homology commutes with taking \emph{algebraic} duals.
However, this is no more true when replacing algebraic duals with topological duals,
so $H_b^n(C')$ is not isomorphic to
the topological dual of $H_n(C)$ in general (but see Theorems~\ref{MM1} and~\ref{Loeh}  
for the case when $C$ is Banach). Nevertheless, the following 
results establish several useful relationships between $H_b^n(C')$ and $H_n(C)$. 

If $H$ is any seminormed vector space, then we denote by $H'$ the space of bounded linear functionals
on $H$. So $H'$ is canonically identified with the topological dual of the quotient of $H$ by the subspace
of elements with vanishing seminorm.

\begin{lemma}\label{lemma:duality}
 Let $(C,d)$ be a normed chain complex with dual normed chain complex $(C',\delta)$. 
 Then, the map
$$
H^q_b(C')\to (H_q(C))'
$$
induced by the Kronecker product is surjective. Moreover,
for every $\alpha\in H_n(C)$ we have 
$$
\|\alpha\|_1=\max \{ \langle\beta,\alpha\rangle\, |\,  \beta\in H_b^n(C'),\,  \|\beta\|_\infty \leq 1\}\ .
$$
\end{lemma}
\begin{proof}
Let $B_q,Z_q$ be the spaces of cycles and boundaries in $C_q$.
The space $(H_{q}(C))'$ is isomorphic to the space  of 
bounded functionals $Z_q\to \R$
that vanish on $B_{q}$. Any such element
admits a bounded extension to $C_{q}$ by Hahn-Banach, and this implies
the first statement of the lemma.

Let us come to the second statement. The inequality $\geq$ is obvious.
Let $a\in C_n$ be a representative of $\alpha$. In order to conclude 
it is enough to find an element $b\in (C')^n$
such that $\delta^nb=0$, $b(a)=\|\alpha\|_1$ and $\| b\|_\infty\leq 1$.
If $\|\alpha\|_1=0$ we may take $b=0$. Otherwise,
let $V\subseteq C_n$ be the closure of 
$B_n$ in $C_n$, and put on the quotient
$W:=C_n/V$ the induced seminorm $\|\cdot \|_W$. Since $V$ is closed,
such seminorm is in fact a norm.
By construction, $\|\alpha\|_1=\| [a]\|_W$. Therefore, Hahn-Banach Theorem provides
a functional $\overline{b}\colon W\to \mathbb{R}$ with operator norm one such that $\overline{b}([a])=\|\alpha\|_1$.
We obtain the desired element $b\in (C')^n$ by pre-composing $\overline{b}$  with the projection $C_n\to W$.
\end{proof}

\begin{cor}\label{semipr0}
Let $(C,d_C)$ , $(D,d_D)$ be normed chain complexes and let $\psi\colon C\to D$ be a chain map of normed chain complexes
(so $\psi$ is bounded in every degree). If the induced map in bounded cohomology
$$H^n_b(\psi)\colon H^n_b(D')\to H^n_b(C')$$
is an isometric isomorphism, then the induced map in homology
$$H_n(\psi)\colon H_n(C)\to H_n(D)$$
preserves the seminorm.
\end{cor}
\begin{proof}
From  the naturality of the Kronecker product, for every $\alpha\in H_n(C)$, $\varphi\in H^n_b(D')$ we get
$$
%\langle \varphi,\alpha\rangle=
\langle H^n_b(\psi)(\varphi),\alpha\rangle=\langle \varphi, H_n(\psi)(\alpha)\rangle\ ,
$$
so the conclusion follows from Lemma~\ref{lemma:duality}.
\end{proof}

\begin{cor}\label{semipr}
Let $(C,d)$ be a normed chain complex and denote by $i\colon C\to \widehat{C}$ the inclusion of $C$
in its metric completion. Then the induced map in homology
$$
H_n(i)\colon H_n(C)\to H_n(\widehat{C})
$$
preserves the seminorm for every $n\in\mathbb{N}$.
\end{cor}
\begin{proof}
Since $i(C)$ is dense in $\widehat{C}$, the map $i$ induces the identity on $C'$, whence
on $H^\bullet_b(C')$. The conclusion follows from Corollary~\ref{semipr0}.
\end{proof}

Corollary~\ref{semipr} implies that, in order to compute seminorms,  one may usually reduce 
to the study of Banach chain complexes.

\section{Some results by Matsumoto and Morita}
In this section we describe some results taken from~\cite{Matsu-Mor}.

\begin{defn}\label{UBC}
Let $(C,d)$ be a normed chain complex. Then $(C,d)$ satisfies the $q$-uniform boundary condition
(or, for short, the $q$-UBC condition) if the following condition holds: there exists $K\geq 0$ 
such that, if $c$ is a boundary in $C_q$, then
there exists a chain $z\in C_{q+1}$ with the property that $d_{q+1}z=c$ and
$\|z\|_1\leq K\cdot \|c\|_1$. 
\end{defn}

\begin{thm}[\cite{Matsu-Mor}, Theorem 2.3]\label{MM1}
Let $(C,d)$ be a Banach chain complex with Banach dual cochain complex
$(C',\delta)$. Then
the following conditions are equivalent:
\begin{enumerate}
\item
$(C,d)$ satisfies $q$-UBC;
\item
the seminorm on $H_q(C)$ is a norm;
\item
the seminorm on $H^{q+1}_b(C')$ is a norm;
\item the Kronecker product induces an isomorphism 
$$
H^{q+1}_b(C')\ \cong (H_{q+1}(C))'\ .
$$
\end{enumerate}
\end{thm}
\begin{proof}
Let us denote by $Z_i$ and $B_i$ the spaces of cycles and boundaries of degree $i$ in $C_i$,
and by $Z^i_b$ and $B^i_b$ the spaces of cocycles and coboundaries of degree $i$ in $(C')^i$.
Being the kernels of bounded maps, the spaces $Z_i$ and $Z^i_b$ are always closed in $C_i$, $C_i'=(C')^i$, respectively.

(1) $\Leftrightarrow$ (2): We endow  the space $C_{q+1}/Z_{q+1}$ with the quotient seminorm. Being the kernel of a bounded
map, the space $Z_{q+1}$ is closed in $C_{q+1}$, so the seminorm is a norm, and the completeness
of $C_{q+1}$ implies the completeness of $C_{q+1}/Z_{q+1}$. 
Condition (1) is equivalent to the fact that the isomorphism $C_{q+1}/Z_{q+1}\cong B_q$ induced by $d_{q+1}$
is bi-Lipschitz. By the open mapping Theorem, this is in turn equivalent to the fact that $B_q$ is complete.
On the other hand, $B_q$ is complete if and only if it is closed in $Z_q$, i.e.~if and only if condition (2) holds.

(2) $\Leftrightarrow$ (3): Condition (2) is equivalent to the fact that the range of $d_{q+1}$ is closed, while (3) holds if and only if
the range of $\delta^q$ is closed. Now
the closed range Theorem~\cite[Theorem 4.14]{Rudin} implies that, if $f\colon V\to W$ is a bounded map between Banach spaces, then
the range of $f$ is closed if and only if the range of the dual map $f'\colon W'\to V'$ is closed, and this 
concludes the proof.

(2) $\Rightarrow$ (4): 
The surjectivity
of the map $H^{q+1}_b(C')\to (H_{q+1}(C))'$ is proved in Lemma~\ref{lemma:duality}. 
Let $f\in Z^{q+1}_b(C')$ be such that $[f]=0$ in $(H_{q+1}(C))'$. Then
$f|_{Z_{q+1}}=0$. 
If (2) holds, then $B_q$ is Banach, so by the open mapping
Theorem the differential $d_{q+1}$ induces a bi-Lipschitz isomorphism $C_{q+1}/Z_{q+1}\cong B_q$. 
Therefore, we have $f=g'\circ d_{q+1}$, where $g'\in (B_q)'$. By Hahn-Banach, $g'$ admits a continuous extension to a map
$g\in (C')^q$, so $f=\delta^q g$, and $[f]=0$ in $H^{q+1}_b(C')$.

The implication (4) $\Rightarrow$ (3) follows from the fact that any element of $H^{q+1}_b(C')$ with vanishing seminorm lies in the kernel of
the map $H^{q+1}_b(C')\to (H_{q+1}(C))'$.
\end{proof}

\begin{cor}[\cite{Matsu-Mor}, Corollary 2.4]\label{vanish0:cor}
Let $(C,d)$ be a Banach chain complex with Banach dual cochain complex
$(C',\delta)$, and let $q\in\mathbb{N}$. Then:
\begin{enumerate}
%\item If $H_1(C)=0$, then $H^1_b(C)=0$. 
\item If $H_{q}(C)=H_{q+1}(C)=0$, then $H_b^{q+1}(C')=0$.
\item If $H_b^{q}(C')=H_b^{q+1}(C')=0$, then $H_{q}(C)=0$.
\item $H_\bullet(C)=0$ if and only if $H^\bullet_b(C')=0$.
\end{enumerate}
\end{cor}
\begin{proof}
%(1): Observe that $H_0(C)=Z_0$ is Banach, so Theorem~\ref{MM1}
%implies that $H^1_b(C')=(H_1(C))'=0$.

 (1): By Theorem~\ref{MM1}, if $H_q(C)=0$ then $H_b^{q+1}(C')$ is isomorphic
to $(H_{q+1}(C))'=0$.

(2): By Theorem~\ref{MM1}, if $H^{q+1}_b(C')=0$, then $H_{q}(C)$ is Banach, so 
by Hahn-Banach $H_{q}(C)$
vanishes if and only if $(H_{q}(C))'=0$. But
$H^{q}_b(C')=0$ implies $(H_{q}(C))'=0$ by Lemma~\ref{lemma:duality}.

(3): By claims (1) and (2), we only have to show that $H_\bullet(C)=0$
implies $H^0_b(C')=0$. But $H_0(C)=0$ implies that $d_1\colon C_1\to C_0$
is surjective, so $\delta^0\colon (C')^0\to (C')^1$ is injective and
$H^0_b(C')=0$.
\end{proof}

As observed in~\cite{Matsu-Mor}, a direct application of Corollary~\ref{vanish0:cor} implies the vanishing of the $\ell^1$-homology
of a countable CW-complex with amenable fundamental group. In fact, a stronger result
holds: the $\ell^1$-homology of a countable CW-complex only depends on its fundamental
group. The first proof of this fact is due to Bouarich~\cite{Bua}. 
In Section~\ref{Loeh:sec} we will describe an approach to this result which
is due to L\"oh~\cite{Loeh}, and which is very close in spirit to Matsumoto and Morita's arguments.

We now show how Theorem~\ref{MM1} can be exploited to prove
that, for every group $\G$, the seminorm of $H^2_b(\G)$ is in fact a norm.
An alternative proof of this fact is given in~\cite{Ivanov2}.

\begin{cor}[\cite{Matsu-Mor,Ivanov2}]\label{Banach:h2}
 Let $\G$ be any group. Then $H^2_b(\G)$ is a Banach space, i.e.~the canonical seminorm
of $H^2_b(\G)$ is a norm. If $X$ is any countable CW-complex,
then $H^2_b(X)$ is a Banach space.
\end{cor}
\begin{proof}
By Theorem~\ref{gro-iva:thm}, the second statement is a consequence of the first one.
Following~\cite{Mitsu}, let us consider the map
$$
S\colon \G\to \cl_2(\G)\, ,\quad S(g)=\sum_{k=0}^\infty 2^{-k-1}(g^{2^k},g^{2^k})\ .
$$
We have $\|S(g)\|_1=1$ and $d_2S(g)=g$ for every $g\in \G$, so 
$S$ extends to a bounded map $S\colon \cl_1(\G)\to\cl_2(\G)$ such that
$d_2S={\rm Id}_{\cl_1(\G)}$. This shows that 
$\cl_\bullet (\G)$ satisfies $1$-UBC (and that $\hl_1(\G)=0$). Then the conclusion follows from
Theorem~\ref{MM1}.
\end{proof}

\section{Injectivity of the comparison map}
Let us now come back to the case when $(C,d)$ is a (possibly non-Banach)
normed chain complex, and let $(\widehat{C},\hat{d})$ be the completion
of $(C,d)$. For every $i\in\mathbb{N}$ we denote by $(C^i)^*=(C_i)^*$
the \emph{algebraic} dual of $C_i$, and we consider
the algebraic dual complex $(C^*,\delta)$ of $(C,d)$ (since this will not raise
any ambiguity, we denote by $\delta$ both the differential
of the algebraic dual complex and the differential of the normed dual complex).
We also denote by $H^\bullet(C^*)$ the cohomology of $(C^\ast,\delta)$.
In the case when $C=C_\bullet(X)$ or $C=C_\bullet(\G)$ we have $H^\bullet(C^*)=H^\bullet(X)$ or $H^\bullet(C^*)=H^\bullet(\G)$ respectively.

The inclusion of complexes $(C',\delta)\to (C^*,\delta)$ induces the
comparison map
$$
c\colon H_b^\bullet(C')\to H^\bullet(C^*)\ .
$$

\begin{thm}[\cite{Matsu-Mor}, Theorem 2.8]\label{MM2}
Let us keep notation from the preceding paragraph. Then the following conditions are equivalent:
\begin{enumerate}
 \item $(C,d)$ satisfies $q$-UBC.
\item $(\widehat{C},\hat{d})$ satisfies $q$-UBC and the space of $(q+1)$-cycles
of $C_{q+1}$ is dense in the space of $(q+1)$-cycles of $\widehat{C}_{q+1}$.
\item The comparison map $c_{q+1}\colon H^{q+1}_b(C')\to H^{q+1}(C^*)$ is injective.
\end{enumerate}
\end{thm}
\begin{proof}
Let us denote by $Z_i$ and $B_i$ (resp.~$\widehat{Z}_i$ and $\widehat{B}_i$)
the spaces of cycles and boundaries  in $C_i$ (resp.~in $\widehat{C}_i$),
and by $Z^i_b$ and $B^i_b$ the spaces of cocycles and coboundaries in $(C')^i$.

 (1) $\Rightarrow$  (2):
Suppose that $(C,d)$ satisfies $q$-UBC with respect to the constant $K\geq 0$, and
fix an element $b\in \widehat{B}_q$.
Using that $B_q$ is dense in $\widehat{B}_q$, it is easy to construct
a sequence $\{b_n\}_{n\in\mathbb{N}}\subseteq  B_q$ such that $\sum_{i=1}^\infty b_i=b$ and
$\sum_{i=1}^\infty \|b_i\|_1\leq 2\|b\|_1$. By (1), for every $i$ we may choose
$c_i\in C_{q+1}$ such that $dc_i=b_i$ and $\|c_i\|_1\leq K\|b_i\|_1$ for a uniform
$K\geq 0$. In particular, the sum $\sum_{i=1}^\infty c_i$ converges
to an element $c\in \widehat{C}_{q+1}$ such that $\hat{d}c=b$ and $\|c\|_1\leq 2K\|b\|_1$, so
$(\widehat{C},\hat{d})$ satisfies $q$-UBC. 

Let us now fix an element $z=\lim_{i\to \infty} c_i\in\widehat{Z}_{q+1}$, where
$c_i\in C_{q+1}$ for every $i$. Choose an element $c'_i\in C_{q+1}$
such that $dc'_i=-dc_i$ (so $c_i+c'_i\in Z_{q+1}$) and $\|c'_i\|_1\leq K\|d c_i\|$. 
Observe that
$$
\|c'_i\|_1\leq K\|dc_i\|_1=K\|\hat{d} (c_i-z)\|_1\leq (q+2)K\|c_i-z\|_1\ ,
$$
so $\lim_{i\to \infty} c'_i=0$. Therefore,
we have $z=\lim_{i\to \infty} (c_i+c'_i)$ and $Z_{q+1}$ is dense in $\widehat{Z}_{q+1}$.

(2) $\Rightarrow$ (1): Suppose that $(\widehat{C},\hat{d})$ satisfies $q$-UBC with
respect to the constant $K\geq 0$, and
fix an element $b\in {B}_q$. Then there exist
$c\in C_{q+1}$ such that $dc=b$ and  
$c'\in \widehat{C}_{q+1}$
such that $\hat{d}c'=b$ and $\|c'\|_1\leq K\|b\|_1$.
Since $Z_{q+1}$ is dense in $\widehat{Z}_{q+1}$ we may find
$z\in Z_{q+1}$ such that $\|c-c'-z\|_1\leq \|b\|_1$. Then we have
$d(c-z)=dc=b$ and 
$$
 \|c-z\|_1\leq \|c-c'-z\|_1+ \|c'\|_1\leq (K+1)\|b\|_1\ .
$$

(2) $\Rightarrow$ (3): Take $f\in Z^{q+1}_b$ such that $[f]=0$
in $H^{q+1}(C^*)$. By the Universal Coefficient Theorem, $f$
vanishes on $Z_{q+1}$. By density, $f$ vanishes on $\widehat{Z}_{q+1}$,
so it defines the null element of $(H_{q+1}(\widehat{C}))'$. But Theorem~\ref{MM1}
implies that, under our assumptions, the natural
map $H^{q+1}_b(C')\to (H_{q+1}(\widehat{C}))'$ is injective, so $[f]=0$ in $H^{q+1}_b(C')$.

(3) $\Rightarrow$ (2): By the Universal Coefficient Theorem, the kernel of the map
$H^{q+1}_b(C')\to (H_{q+1}(C))'$ is contained in the kernel of the comparison map
$H^{q+1}_b(C')\to H^{q+1}(C^*)$, so Theorem~\ref{MM1}
implies that $(\widehat{C},\hat{d})$ satisfies $q$-UBC.

Suppose now that $Z_{q+1}$ is not dense in $\widehat{Z}_{q+1}$. Then there exists
an element $f\in (C')^{q+1}$ which vanishes on $Z_{q+1}$ and is not null
on $\widehat{Z}_{q+1}$. The element $f$ vanishes on $B_{q+1}$, so it belongs to $Z^{q+1}_b$. Moreover,
since $f|_{Z_{q+1}}=0$ we have $[f]=0$ in $H^{q+1}(C^*)$. On the other hand,
since $f|_{\widehat{Z}_{q+1}}\neq 0$ we have $[f]\neq 0$ in $H_b^{q+1}(C')$.
\end{proof}

\begin{cor}
If $\G$ is a group and $q\geq 1$, then the comparison map
$H^q_b(\G)\to H^q (\G)$
 is injective if and only if $C_\bullet(\G)$ satisfies $(q-1)$-UBC.

If $X$ is a topological space and $q\geq 1$, then the comparison map
$H^q_b(X)\to H^q (X)$
 is injective if and only if $C_\bullet(X)$ satisfies $(q-1)$-UBC.
\end{cor}

\section{The translation principle}\label{Loeh:sec}
We are now ready to prove that the homology of a Banach chain complex is completely determined
by the bounded cohomology of its dual chain complex.

\begin{thm}[\cite{Loeh}]\label{Loeh}
Let $\alpha\colon C\to D$ be a chain map between Banach chain complexes, and let
$H_\bullet(\alpha)\colon H_\bullet(C)\to H_\bullet(D)$, 
$H^\bullet_b(\alpha)\colon H^\bullet_b(D')\to H^\bullet_b(C')$
be the induced maps in homology and in bounded cohomology. Then:
\begin{enumerate}
\item
The map
$H_n(\alpha)\colon H_n(C)\to H_n(D)$
 is an isomorphism for every $n\in\mathbb{N}$  if and only if $H^n_b(\alpha)\colon H^n_b(D')\to H^n_b(C')$ is
 an isomorphism for every $n\in\mathbb{N}$.
 \item
 If $H^n_b(\alpha)\colon H^n_b(D')\to H^n_b(C')$ is an isometric isomorphism for every $n\in\mathbb{N}$, then
 also $H_n(\alpha)\colon H_n(C)\to H_n(D)$ is.
 \end{enumerate}
\end{thm}
\begin{proof}
The proof is based on the fact that the question whether a given chain map induces an isomorphism
in (co)homology may be translated into a question about the vanishing of the corresponding mapping cone, that we
are now going to define.

For every $n\in\mathbb{N}$ we set 
$$
\cone(\alpha)_n=D_n+C_{n-1}
$$
and we define a boundary operator $\overline{d}_n\colon \cone(\alpha)_n\to \cone (\alpha)_{n-1}$ by setting 
$$
\overline{d}_n(v,w)=
(d^Dv+\alpha_{n-1}(w),-d^Cw)\ ,
%(-d^Cv, \alpha_{n-1}(v)+d^Dw)\ ,
$$
where $d^C,d^D$ are the differentials of $C,D$ respectively.
We also put on $\cone(\alpha)_n$ the norm obtained by summing the norms of its summands, thus endowing
$\cone(\alpha)_\bullet$ with the structure of a Banach complex.

Dually, we set
$$
\cone(\alpha)^n=(D')^{n}+(C')^{n-1}
$$
and  
$$
\overline{\delta}_n(\varphi,\psi)=
(-\delta_D\varphi,-\delta_C\psi-\alpha^n(\varphi))\ ,
%(-\delta_D\varphi, 
%(\alpha')^{n}(\varphi)+\delta_C\psi)\ ,
$$
where $\alpha^{n}\colon (D')^{n}\to (C')^{n}$ is the dual map of $\alpha_{n}\colon C_{n}\to D_{n}$,
and $\delta_C,\delta_D$ are the differentials of $C',D'$ respectively. Observe that $\cone(\alpha)^n$, when endowed with the norm given by the maximum of the norms
of its summands, is canonically
isomorphic to the topological dual of $\cone(\alpha)_n$ via the pairing
$$
(\varphi,\psi)(v,w)=\varphi(v)-\psi(w)\ ,
$$
and $(\cone(\alpha)^\bullet,\overline{\delta}^\bullet)$
coincides with the normed dual cochain complex of $(\cone(\alpha)_\bullet,\overline{d}_\bullet)$. As a consequence, by Corollary~\ref{vanish0:cor} we have
\begin{equation}\label{vanish:eq}
 H_\bullet(\cone(\alpha)_\bullet)=0\quad \Longleftrightarrow\quad
H^\bullet_b(\cone(\alpha)^\bullet))=0\ .
\end{equation}
Let us denote by $\Sigma C$ the \emph{suspension} of $C$, i.e.~the complex
obtained by setting $(\Sigma C)^n=C^{n-1}$, and
consider the short exact sequence of complexes
$$
0\tto{} D\tto {\iota} \cone(\alpha)_\bullet \tto{\eta} \Sigma C\tto{} {0}
$$
where $\iota(w)=(0,w)$ and $\eta(v,w)=v$. Of course we have
$H_n(\Sigma C)=H_{n-1} (C)$, so the corresponding long exact sequence
is given by
$$
\ldots \tto{} H_n(\cone(\alpha)_\bullet)\tto{}H_{n-1}(C)\tto{\partial}H_{n-1}(D)\tto{}H_{n-1}
(\cone(\alpha)_\bullet)\tto{}\ldots
$$
Moreover, an easy computation shows that the connecting homomorphism $\partial$ 
coincides with the map $H_{n-1}(\alpha)$. As a consequence, the
map $H_\bullet(\alpha)$ is an isomorphism in every degree if and only if
$H_\bullet(\cone(\alpha)_\bullet)=0$.

A similar argument shows that $H_b^\bullet(\alpha)$ is an isomorphism in every degree if and only if
$H_b^\bullet(\cone(\alpha)^\bullet)=0$, so claim (1) of the theorem follows from~\eqref{vanish:eq}. Point (2) is now an immediate consequence of Corollary~\ref{semipr0}.
\end{proof} 

The converse of claim (2) of the previous theorem does not hold in general (see~\cite[Remark 4.6]{Loeh} for a counterexample).

\begin{cor}
Let $X$ be a path connected countable CW-complex  with fundamental group $\G$.
Then 
$\hl_\bullet(X)$ is isometrically isomorphic to
$\hl_\bullet(\G)$. Therefore, the $\ell^1$-homology of $X$ only depends on its fundamental group.
\end{cor}
\begin{proof}
We retrace the proof of Lemma~\ref{normnon} to construct
a chain map of normed chain complexes $r_\bullet\colon \cl_\bullet(X)\to\cl_\bullet(\G)$.
Let $s\in S_n(X)$ be a singular simplex, 
let $\widetilde{X}$ be the universal covering of $X$, and let
$\widetilde{s}$ be a lift of $s$ in $S_n(\widetilde{X})$.
We choose
a set of representatives $R$ for the action of $\G$ on $\widetilde{X}$, and we set
$$
{r}_n (s)=(g_0^{-1}g_1,\ldots,g_{n-1}^{-1}g_n)\in \G^n\ , 
$$
where $g_i\in \G$ is such that the $i$-th vertex of $\widetilde{s}$
lies in $g_i(R)$. It is easily seen that $r_n(s)$ does not depend on
the chosen lift $\widetilde{s}$. Moreover, $r_n$ extends to a well-defined
chain map $r_\bullet \colon \cl_\bullet(X)\to\cl_\bullet(\G)$, whose dual map
$r^\bullet \colon C_b^\bullet(\G)\to C_b^\bullet (X)$ coincides with (the restriction
to $\G$-invariants) of the map described in Lemma~\ref{normnon}.
Theorem~\ref{gro-iva:thm} implies that $r^\bullet$ induces an isometric isomorphism on bounded
cohomology, so the conclusion follows from Theorem~\ref{Loeh}.
\end{proof}

\begin{cor}
Let $X,Y$ be countable CW-complexes and let $f\colon X\to Y$ be a continuous
map  inducing an epimorphism with amenable kernel on fundamental groups.
Then 
$$
\hl_n(f)\colon \hl_n(X)\to\hl_n(Y)
$$
is an isometric isomorphism for every $n\in\mathbb{N}$. 
\end{cor}
\begin{proof}
The conclusion follows from Corollary~\ref{mapping:cor} and Theorem~\ref{Loeh}.
\end{proof}

\section{Gromov equivalence theorem}
Let now $(X,Y)$ be a topological pair. The $\ell^1$-norm on $C_n(X)$ descends to a norm on the relative singular chain module
$C_n(X,Y)$, which induces in turn a seminorm $\|\cdot \|_1$ on $H_n(X,Y)$. Such a seminorm plays an important role in some contexts of our interest, e.g.~in the definition
of the simplicial volume of manifolds with boundary (see Chapter~\ref{simplicial:chapter}).

It makes sense to consider also other seminorms on $H_n(X,Y)$. Indeed, it can often be useful to perturb the usual $\ell^1$-seminorm by adding a term which 
takes into account the weight of the boundaries of \emph{relative} cycles. To this aim, 
Gromov introduced in~\cite{Gromov} the following one-parameter family of seminorms on $H_n(X,Y)$.

Let $\theta\in [0,\infty)$ and consider the norm $\|\cdot\|_1(\theta)$ on $C_n(X)$ 
defined by $\|c\|_1(\theta)=\|c\|_1+\theta\|\partial_n c\|_1$. 
Every such norm is equivalent to the usual norm $\|\cdot\|_1=\|\cdot\|_1(0)$ for 
every $\theta\in[0,\infty)$ and induces a quotient seminorm on relative homology, 
still denoted by $\|\cdot \|_1(\theta)$.
By passing to the limit, one can also define a seminorm $\|\cdot\|_1(\infty)$ that, 
however, may be non-equivalent to $\|\cdot\|_1$.
The following result is stated by Gromov in~\cite{Gromov}.

\begin{thm}[{Equivalence theorem, \cite[p.~57]{Gromov}}]\label{thm:equi}
Let $(X,Y)$ be a pair of countable CW-complexes, and let $n\geq 2$.
If the fundamental groups of all connected components of $Y$ are amenable,
then the seminorms $\|\cdot\|_1(\theta)$ on $H_n(X,Y)$
coincide for every $\theta\in[0,\infty]$.
\end{thm} 

In this section we describe the proof of Theorem~\ref{thm:equi} given in~\cite{BBFIPP}, which is based on the following strategy:
we first introduce a one-parameter family of seminorms in bounded cohomology (that were first defined by Park~\cite{Park}); building on the fact that
$H^n_b(X,Y)$ is isometrically isomorphic to $H^n_b(X)$ (see Theorem~\ref{reliso:thm}) we then show that all these seminorms coincide; finally, we conclude the proof via an argument
making use of duality.

So let us fix a pair $(X,Y)$ of countable CW-complexes, and assume from now on that the fundamental groups of all connected components of $Y$ are amenable.
Following~\cite{Park2}, we first exploit a cone construction to give a useful alternative description of the seminorms
$\|\cdot\|_1(\theta)$ on $H_n(X,Y)$ (see also \cite{Loeh}). 

Let us denote by
$i_n\colon C_n(Y)\to C_n(X)$ the map induced by the inclusion $i\colon Y\to X$.
The homology mapping cone complex of $(X,Y)$ 
is the complex
$\left(C_\bullet(Y\to X),\overline{d}_\bullet\right))$, where 
$$C_n(Y\to X)=C_n(X)\oplus C_{n-1}(Y)\, \quad
\overline{d}_n(u,v)=(\partial_n u +i_{n-1}(v),-\partial_{n-1}(v))
\, ,$$ 
where $\partial_\bullet$ denotes the usual differential both of $C_\bullet(X)$ and of $C_\bullet (Y)$.
The homology of the mapping cone $(C_\bullet(Y\to X), \overline{d}_\bullet)$ 
is  denoted by $H_\bullet(Y\rightarrow X)$. 
For every $n\in\mathbb{N}$, $\theta\in[0,\infty)$ one can endow $C_n(Y\to X)$ with the norm
$$\|(u,v)\|_1(\theta)=\|u\|_1+ \theta \|v\|_1\ , $$
which induces in turn a seminorm, still denoted by $\|\cdot \|_1 (\theta)$, 
on $\h_n(Y\rightarrow X)$ (in~\cite{Park2}, Park restricts her attention
only to the case $\theta\geq 1$).
The chain map
\begin{equation}\label{betaPark:eq}
\beta_n\colon C_n(Y\to X)\longrightarrow C_n(X,Y)\, ,\qquad
\beta_n (u,v)=[u]
\end{equation}
induces a map $H_n(\beta_n)$
in homology. 

\begin{lemma} \label{lem: theta iso} The map
$$H_n(\beta_n)\colon\left( \h_n(Y\rightarrow X),\|\cdot \|_1(\theta)\right)  \longrightarrow \left( \h_n(X,Y),\|\cdot \|_1(\theta)\right)$$
is an isometric isomorphism for every $\theta\in [0,+\infty)$.
\end{lemma}

\begin{proof} 
It is immediate to check that $H_n(\beta_n)$ admits
the inverse map 
$$
\h_n(X,Y)\to \h_n(Y\to X)\, ,\qquad 
[u]\mapsto [(u,-\partial_n u)]\, .
$$ 
%is well-defined and provides an inverse map to $\h(\beta_n)$.
Therefore, $H_n(\beta_n)$ is an isomorphism, and we are left to show
that it is norm-preserving.

Let us set
$$
\beta'_n\colon C_n(Y\to X)\to C_n (X)\, ,\qquad 
\beta'_n(u,v)=u\, .
$$
By construction, $\beta_n$ is the composition of $\beta_n'$ with the 
natural projection $C_n(X)\to C_n(X,Y)$.
Observe that an element $(u,v)\in C_n(Y\to X)$ is a cycle if and only if
$\partial_nu=-i_{n-1}(v)$. As a consequence,
although the map $\beta'_n$ is not norm non-increasing in general, 
it does preserve norms when restricted to the space of cycles $Z_n(Y\rightarrow X)$. 
Moreover, every chain in $C_n(X)$ representing a relative cycle is contained in
$\beta'_n (Z_n(Y\rightarrow X))$,
and this concludes the proof.
\end{proof}

For $\theta\in (0,\infty)$, the dual normed chain complex of 
$(C_n(Y\to X),\|\cdot \|_1(\theta))$
is
Park's mapping cone for relative bounded cohomology~\cite{Park}, that is the complex
$(\cb^n (Y\to X),\overline \delta^n)$, where
$$
\cb^n (Y\to X)=\cb^n(X)\oplus \cb^{n-1}(Y)\, ,\quad
\overline{\delta}^n (f,g)=(\delta^n f, -i^n(f)-\delta^{n-1}g)
%\left(\begin{smallmatrix}\delta^n &0\\-i^n&-\delta^{n-1}\end{smallmatrix}\right)\right)
$$
endowed with the norm 
$$
\|(f,g)\|_\infty(\theta)=\max\{\|f\|_\infty ,\theta^{-1}
\|g\|_\infty\}. 
$$
We endow the cohomology $H_b^\bullet(Y\rightarrow X)$
of the complex $(\cb^\bullet(Y\to X),\overline{\delta}^n)$ with the quotient seminorm, which will  still be denoted
by $\|\cdot\|_\infty(\theta)$.
The chain map
\bqn
\beta^n\colon \cb^n(X,Y) \longrightarrow \cb^n(Y\to X),\qquad
\beta^n (f)= (f,0)
\eqn
induces a map $H^n(\beta^n)\colon H_b^n(X,Y)\to H_b^n(Y\to X)$. Building on the fact that $H^n_b(X,Y)$ is isometrically isomorphic to $H^n_b(X)$
we can now prove the following:

\begin{prop}\label{fund:prop}
For every $n\geq 2$, $\theta\in (0,\infty)$,
the map 
$$
\h^n(\beta^n)\colon \left(H_b^n(X,Y),\|\cdot\|_\infty\right)\to \left(H_b^n(Y\to X), \|\cdot \|_\infty(\theta)\right)
$$
is an isometric isomorphism.
\end{prop}
\begin{proof} 
Let us first prove that $\h^n(\beta^n)$ is an isomorphism (here we do not use any hypothesis
on $Y$).
To this aim, it is enough to show that, if we denote by $Z^n_b(Y\to X)$ the space of bounded cocycles in $C^n_b(X\to Y)$, then the composition
\begin{equation}\label{comp:eq}
\xymatrix{
Z^n_b(X,Y)\ar[r]^{\beta^n} & Z^n_b(Y\to X) \ar[r]&  H_b^n(Y\to X)
}
\end{equation}
is surjective with kernel $\delta\cb^{n-1}(X,Y)$. For any $g\in \cb^\bullet(Y)$ 
we denote by $g'\in \cb^\bullet(X)$ the extension of $g$ 
that vanishes on simplices with image not contained in $Y$.
Let us take $(f,g)\in Z^n_b(Y\to X)$. From $\overline{\delta}^n(f,g)=0$
we deduce that $f+\delta g'\in Z^n_b(X,Y)$. Moreover, $(f+\delta g',0)-(f,g)=-\overline{\delta}^{n-1}(g',0)$, 
so the map~\eqref{comp:eq} above is surjective.
Finally, if $f\in Z_b^n(X,Y)$ and $\overline{\delta}^{n-1}(\alpha,\beta)=(f,0)$, 
then $\alpha+\delta\beta'$ belongs to $\cb^{n-1}(X,Y)$ and $\delta(\alpha+d\beta')=f$. 
This concludes the proof that $\h^n(\beta^n)$ is an isomorphism.

Let us now exploit the fact that the fundamental group of each component of $Y$ is amenable.
We
consider the
chain map $$\gamma^\bullet \colon \cb^\bullet (Y\to X)\to \cb^\bullet(X),\qquad(f,g)\mapsto f\ .$$
For every $n\in\mathbb{N}$
the composition $\gamma^n\circ \beta^n$ coincides with the inclusion
$j^n\colon \cb^n(X,Y)\to \cb^n(X)$.
By Theorem~\ref{reliso:thm}, for every $n\geq 2$ the map $\h(j^n)$ is an isometric isomorphism. 
Moreover, both $\h^n(\gamma^n)$ and $\h^n(\beta^n)$ are norm non-increasing, 
so we may conclude that the isomorphism $\h^n(\beta^n)$ is isometric for every $n\geq 2$.
\end{proof}

We have now collected all the results we need to prove Gromov equivalence theorem. 
Indeed, it follows from the very definitions that $(C^n_\bullet(X,Y),\|\cdot\|_\infty)$ is the topological dual complex
to $(C_\bullet(X,Y),\|\cdot\|_1)=(C_\bullet(X,Y),\|\cdot\|_1(0))$. Moreover, by construction
$(C^\bullet_b(Y\to X),\|\cdot\|_\infty(\theta))$ is the topological dual complex of
$(C_\bullet(Y\to X),\|\cdot\|_1(\theta))$, for every $\theta\in (0,+\infty)$. Therefore, Proposition~\ref{fund:prop} and Corollary~\ref{semipr} ensure that the map
$$
\h_n(\beta_n)\colon \left(\h_n(Y\to X),\|\cdot \|_1(\theta)\right)\to  \left(\h_n(X,Y),\|\cdot\|_1\right)
$$
is norm-preserving. 
On the other hand,
we know from Lemma~\ref{lem: theta iso} that 
the map
$$\h_n(\beta_n)\colon\left( \h_n(Y\rightarrow X),\|\cdot \|_1(\theta)\right)  \longrightarrow \left( \h_n(X,Y),\|\cdot \|_1(\theta)\right)$$
is an isometric isomorphism for every $\theta\in [0,+\infty)$.
Putting toghether these facts we deduce that the identity
between $\left(\h_n(X,Y),\|\cdot\|_1\right)$ and $\left(\h_n(X,Y),\|\cdot\|_1(\theta)\right)$ is an isometry for every $\theta>0$.
The conclusion follows from the fact that, by definition,
$\|\cdot \|_1(0)=\|\cdot \|_1$ and $\|\cdot\|_1(\infty)=\lim_{\theta\to \infty} \|\cdot\|_1(\theta)$.

\medskip

As noticed by Gromov, Theorem~\ref{thm:equi} admits the following equivalent formulation, 
which is inspired by Thurston~\cite[\S 6.5]{Thurston} and 
plays an important role in several results about the (relative) simplicial volumes of gluings and fillings (see Chapters~\ref{simplicial:chapter} and~\ref{additivity:chap}):

\begin{cor}\label{Thurston's version}
Let $(X,Y)$ be a pair of countable CW-complexes, and
suppose that the fundamental groups of all the components of $Y$ are amenable.
Let $\alpha\in \h_n(X,Y)$, $n\geq 2$. Then,
for every $\epsilon>0$, there exists an element $c\in C_n(X)$ with $\partial_n c\in C_{n-1}(Y)$ 
such that $[c]=\alpha\in \h_n(X,Y)$, 
$\|c\|_1<\|\alpha\|_1+\epsilon$ and $\|\partial_n c\|_1<\epsilon$.
\end{cor}
\begin{proof}
 Let $\theta=(\|\alpha\|_1+\epsilon)/\epsilon$. By Theorem~\ref{thm:equi} 
 we know that $\|\cdot\|_1(\theta)$ induces the norm $\|\cdot\|_1$ in homology, 
 so we can find a representative $c \in  C_n(X)$ of $\alpha$ 
 with $\|c\|_1(\theta)=\|c\|_1+\theta\|\partial_n c\|_1\leq \|\alpha\|_1+\epsilon$.
 This implies that $\|c\|_1\leq\|\alpha\|_1+\epsilon$ and $\|\partial_n c\|_1\leq (\|\alpha\|_1+\epsilon)/\theta = \epsilon$.
\end{proof}

\section{Further readings}\label{further:duality}

\subsection*{Locally finite chains on non-compact topological spaces}
In the case when $X$ is a \emph{non-compact} topological space, it is often interesting to consider the complex of \emph{locally finite} singular chains
on $X$. For example, if $X$ is a connected non-compact oriented manifold of dimension $n$, then in degree $n$ the singular homology of $M$ vanishes, while
the locally finite singular homology of $M$ (say, with real coefficients) is isomorphic to $\R$ and generated by a preferred element, called the \emph{fundamental class of $M$}. The $\ell^1$-seminorm of this fundamental
class (which can now be infinite) is by definition the simplicial volume of $M$ (see the next chapter, where the case of compact manifolds is treated in detail).
A very natural question is whether duality may be exploited to establish a useful relationship between the $\ell^1$-seminorm in this slightly different context and (a suitable version of) the $\ell^\infty$-seminorm in
bounded cohomology. We refer the reader to~\cite{Loeh, Lothesis} for a detailed discussion of this topic.

\subsection*{Gromov equivalence theorem}
The first proof of Gromov equivalence theorem appeared in~\cite{Gromov}, and was based on the theory of multicomplexes. 
In the particular case when $X$ is the (natural compactification of) a complete finite volume
hyperbolic manifold without boundary and $Y=\partial X$, an explicit construction of the representatives $c_i$ approaching the infimum of
both $\|c_i\|_1$ and $\|\partial c_i\|_1$ within their homology classes is given in~\cite{FM}.

\subsection*{Bounded-cohomological dimension of discrete groups}
As a recent application of duality, we would finally like to mention L\"oh's construction of groups with infinite bounded-cohomological dimension, 
and of (new examples of) groups with bounded-cohomological dimension equal to 0~\cite{Loeh:dim}

\chapter{Simplicial volume}\label{simplicial:chapter}
The simplicial volume
is an invariant of manifolds introduced by Gromov in his seminal paper~\cite{Gromov}. If $M$ is a closed (i.e. connected, compact and without boundary) oriented manifold, then $H_n(M,\matZ)$ is infinite cyclic generated
by the \emph{fundamental class} $[M]_\matZ$ of $M$. If $[M]=[M]_\R\in H_n(M,\R)$ denotes
the image of $[M]_\matZ$ via the change of coefficients map
$H_n(M,\matZ)\to H_n(M,\R)$, then the simplicial volume of $M$ is defined as
$$
\| M\| = \| [M]\|_1\ ,
$$
where $\|\cdot \|_1$ denotes the $\ell^1$-norm on singular homology introduced in Chapter~\ref{duality:chap}. The simplicial volume of $M$ does not depend on the choice
of the orientation, so it is defined for every closed orientable manifold. Moreover, if $M$
is non-orientable and $\widetilde{M}$ is the orientable double covering of $M$,
then the simplicial volume of $M$ is defined by $\|M\|=\|\widetilde{M}\|/2$. Finally, if $M$ is compact and possibly disconnected, then the simplicial
volume of $M$ is just the sum of the simplicial volumes of its components.

\section{The case with non-empty boundary}
Let $M$ be a manifold with boundary.
We identify $C_\bullet(\partial M,\R)$ with its image in $C_\bullet(M,\R)$ via
the morphism induced by the inclusion, and observe that
$C_\bullet(\partial M,\R)$ is closed in $C_\bullet(M,\R)$ with respect to the
$\ell^1$-norm. Therefore, the quotient seminorm $\|\cdot \|_1$ on
$C_n(M,\partial M,\R)=C_n(M,\R)/C_n(\partial M,\R)$ is a norm,
which endows $C_n(M,\partial M,\R)$ with the structure of a normed chain complex.

If $M$ is connected and oriented, then $H_n(M,\partial M,\matZ)$ is infinite cyclic generated by the preferred element $[M,\partial M]_\matZ$.
If $[M,\partial M]=[M,\partial M]_\R$ is the image of $[M,\partial M]_\matZ$ via the change of coefficients morphism, then the simplicial volume of $M$ is defined by
$$
\| M,\partial M\|=\| [M,\partial M]\|_1\ .
$$
Just as in the closed case we may extend the definition of $\|M,\partial M\|$
to the case when $M$ is not orientable and/or disconnected.

The simplicial volume may be defined also for non-compact manifolds. However, we restrict
here to the compact case. Unless otherwise stated, henceforth every manifold will be assumed to be compact and orientable.
Moreover, since we will be mainly dealing with (co)homology with real coefficients, unless otherwise stated henceforth we simply
denote by $H_i(M,\bb M)$, $H^i(M,\bb M)$  and $H^i_b(M,\bb M)$ the modules
$H_i(M,\bb M,\R)$, $H^i(M,\bb M,\R)$ and $H^i_b(M,\bb M,\R)$. We adopt the corresponding notation for group (co)homology.

\bigskip

Even if it depends only on the homotopy type of a manifold, the simplicial volume is deeply related
to the geometric structures that a manifold can carry. For example, closed manifolds which support 
negatively curved Riemannian metrics have non-vanishing simplicial volume, while the simplicial 
volume of closed manifolds with non-negative Ricci tensor is null (see~\cite{Gromov}).
In particular, flat or spherical closed manifolds have vanishing simplicial volume, 
while closed hyperbolic manifolds have positive simplicial volume (see Section~\ref{inoue}).

Several vanishing and non-vanishing results
for the simplicial volume
are available by now,
but the exact value of non-vanishing simplicial volumes is known only in  very few cases. 
If $M$ is (the natural compactification of) a complete finite-volume hyperbolic $n$-manifold
without boundary, then 
a celebrated result by Gromov and 
Thurston implies that the simplicial volume of $M$ is equal to the Riemannian volume of $M$
divided by the volume $v_n$ of the regular ideal geodesic $n$-simplex in hyperbolic space. 
The only other exact computation of non-vanishing simplicial volume of closed manifolds is for the product of two closed hyperbolic surfaces
or more generally manifolds locally isometric to the product of two hyperbolic planes \cite{Bucher3}. The first exact computations of $\|M,\bb M\|$
for classes of $3$-manifolds for which $\| \bb M\|>0$ are given in~\cite{BFP}. Incredibly enough, in dimension bigger than two no exact value is known for the simplicial
volume of any compact hyperbolic manifold with geodesic boundary (see~\cite{FP,BFP2} for some results on the subject).
Building on these examples, more values for the simplicial volume can be obtained by taking connected sums or amalgamated sums over submanifolds with amenable fundamental group
(or, more in general, by performing surgeries, see e.g.~\cite{Sambusan}).

In this monograph we will restrict ourselves to those results about the simplicial 
volume that more heavily depend on the dual theory of bounded cohomology. In doing so,
we will compute the simplicial volume of closed hyperbolic manifolds (Theorem~\ref{hyp:thm}),
prove Gromov proportionality principle for closed Riemannian manifolds (Theorem~\ref{prop:thm}),
and prove Gromov additivity theorem for manifolds obtained by gluings along
boundary components with amenable fundamental group (Theorem~\ref{simpl:thm}).

\section{Elementary properties of the simplicial volume}
Let $M,N$ be $n$-manifolds and let $f\colon (M,\bb M)\to (N,\bb N)$ be a continuous map of pairs of degree $d\neq 0$.
The induced map on singular chains sends every single simplex to a single simplex,
so it induces a norm non-increasing map
$$
H_\bullet(f)\colon H_\bullet(M,\bb M)\to H_\bullet(N,\bb N)\ .
$$
As a consequence we have
\begin{align}\label{degree:eq}
\begin{split}
\|N,\bb N\|& =\|[N,\bb N]\|_1=\frac{\|H_n(f)([M,\bb M])\|_1}{|d|}\\ & \leq
\frac{\|[M,\bb M]\|_1}{|d|}=\frac{\|M,\bb M\|}{|d|}\ .
\end{split}
\end{align}

This shows that, if $\|N,\partial N\|>0$, then the set of possible degrees
of maps from $M$ to $N$ is bounded. If $M=N$, then 
iterating a map of degree $d$, $d\notin \{-1,0,1\}$,  we obtain maps of arbitrarily large
degree, so we get the following:

\begin{prop}\label{degreed:prop}
If the manifold $M$ admits a self-map of degree different from $-1,0,1$, then
$\|M,\partial M\|=0$. In particular, if $n\geq 1$ then 
$\|S^n\|=\|(S^1)^n\|=0$, and if $n\geq 2$ then $\|D^n,\bb D^n\|=0$.
\end{prop}

The following result shows that the simplicial volume is multiplicative with respect
to finite coverings:

\begin{prop}\label{finitecover:prop}
 Let $M\tto{} N$ be a covering of degree $d$ between manifolds (possibly with boundary).
Then
$$
\|M,\bb M\|=d\cdot \|N,\bb N\|\ .
$$
\end{prop}
\begin{proof}
 If $\sum_{i\in I} a_i s_i\in C_n(N,\bb N,\R)$ is a fundamental cycle for $N$ and $\widetilde{s}_i^j$, $j=1,\ldots,d$ are the lifts of $s_i$ to $M$ (these lifts exist since $s_i$ is defined on a simply connected space), then $\sum_{i\in I} \sum_{j=1}^d a_i \widetilde{s}_i^j$ is
a fundamental cycle for $M$, so taking the infimum over the representatives
of $[N,\bb N]$ we get
$$
\| M,\bb M\| \leq d\cdot \| N,\bb N\|\ .
$$ 
Since a degree-$d$ covering is a map of degree $d$, the conclusion follows from~\eqref{degree:eq}.
\end{proof}

\section{The simplicial volume of Riemannian manifolds}
A Riemannian covering between Riemannian manifolds is a locally isometric
topological covering. 
Recall that two Riemannian manifolds $M_1$, $M_2$ are \emph{commensurable} if there exists a Riemannian
manifold which is the total space of a finite Riemannian covering of $M_i$ for $i=1,2$.
Since the Riemannian volume is multiplicative with respect to coverings, Proposition~\ref{finitecover:prop} implies that
$$
\frac{\|M_1,\partial M_1\|}{\vol(M_1)}=\frac{\|M_2,\partial M_2\|}{\vol(M_2)}
$$
for every pair $M_1,M_2$ of commensurable Riemannian manifolds. 
Gromov's proportionality principle extends this property to pairs of manifolds which share the same Riemannian universal covering:

\begin{thm}[proportionality principle~\cite{Gromov}]\label{prop:thm}
 Let $M$ be a Riemannian manifold. Then the ratio
$$
\frac{\|M,\partial M\|}{\vol(M)}
$$
only depends on the isometry type of the Riemannian universal covering of $M$.
\end{thm}

In the case of hyperbolic manifolds, Gromov and Thurston computed the exact value
of the proportionality constant appearing in Theorem~\ref{prop:thm}:

\begin{thm}[\cite{Thurston, Gromov}]\label{hyp:thm}
 Let $M$ be a closed hyperbolic $n$-manifold. Then
$$
\|M\|=\frac{\vol(M)}{v_n}\ ,
$$
where $v_n$ is the maximal value of the volumes of geodesic simplices in the hyperbolic
space $\matH^n$.
\end{thm}

As mentioned above, Theorem~\ref{hyp:thm} also holds in the case when
$M$ is the natural compactification of a complete finite-volume hyperbolic $n$-manifold~\cite{Gromov, Thurston,stefano,FP,FM,BB}
(such a compactification is homeomorphic to a compact topological manifold with boundary, so it also has a well-defined simplicial volume).
In particular, we have the following:

\begin{cor}\label{punctured:surf:simpl}
 Let $\Sigma_{g,n}$ denote the closed orientable surface of genus $g$ with $n$ disks removed (so $\Sigma_{g,n}$ is closed
 if $n=0$, and it is compact with boundary if $n\neq 0$). Then
 $$
 \|\Sigma_{g,n},\partial\Sigma_{g,n}\|=\max \{0,-2\chi(\Sigma_{g,n})\}\ .
 $$
\end{cor}
\begin{proof}
 First observe that $\chi(\Sigma_{g,n})\geq 0$ exactly in the following cases:
 if $g=0$, $n=0,1,2$, so that $\Sigma_{g,n}$ is either a sphere, or a disk, or an annulus, and $\|\Sigma_{g,n}\|=0$
 (for example because each of these spaces admits a self-map of degree bigger than one);
 if $g=1$, $n=0$, so that $\Sigma_{g,n}$ is the torus and again $\|\Sigma_{g,n}\|=0$. 
 
 We can therefore suppose that $\chi(\Sigma_{g,n})<0$. In this case, it is well-known that (the interior of) $\Sigma_{g,n}$ supports
 a complete hyperbolic metric with finite area, so we may apply Theorem~\ref{hyp:thm} to deduce that 
 $$
  \|\Sigma_{g,n},\partial\Sigma_{g,n}\|= \frac{{\rm Area}(\Sigma_{g,n})}{v_2}=\frac{-2\pi\chi(\Sigma_{g,n})}{\pi}=-2\chi(\Sigma_{g,n})\ ,
 $$
where we used Gauss-Bonnet Theorem in order to describe the area of $\Sigma_{g,n}$ in terms of its Euler characteristic,
together with the well-known fact that the maximal area of hyperbolic triangles is equal to $\pi$.
 \end{proof}

We will provide proofs of Theorem~\ref{prop:thm} (under some additional hypothesis) and of Theorem~\ref{hyp:thm} in Chapter~\ref{prop:simpl:chap}.
Moreover, in Section~\ref{surface:simplicial:sec} we provide a computation of the simplicial volume of surfaces that does not make use of the full power of Theorem~\ref{hyp:thm}.

\section{Simplicial volume of gluings}\label{constructions}
Let us now describe the behaviour of simplicial volume with respect to some
standard operations, like taking connected sums or products, or performing surgery.
The simplicial volume is additive with respect to connected sums~\cite{Gromov}.
In fact, this is a consequence of a more general result concerning
the simplicial volume of manifolds obtained by gluing manifolds along boundary components
with amenable fundamental group.

Let $M_1,\ldots,M_k$ be oriented $n$-manifolds, and let us fix a pairing $(S_1^+,S_1^-)$, $\ldots$, $(S_h^+,S_h^-)$ 
of some boundary components of $\sqcup_{j=1}^k M_j$, in such a way that every boundary component of
$\sqcup_{j=1}^k M_j$ appears at most once among the $S_i^\pm$.
For every $i=1,\ldots,h$, let also $f_i\colon S_i^+\to S_i^-$ be a fixed orientation-reversing homeomorphism (since every $M_j$ is oriented, 
every $S_i^\pm$ inherits a well-defined orientation).
We denote by $M$ the oriented manifold obtained by gluing $M_1,\ldots,M_k$ along $f_1,\ldots,f_h$, and we suppose that $M$
is connected. 

For every $i=1,\ldots,h$ we denote by $j^{\pm}(i)$ the index such that $S_i^\pm \subseteq M_{j^{\pm}(i)}$,
and by $K_i^\pm$ the kernel of the map $\pi_1(S_i^\pm)\to \pi_1(M_{j^{\pm}(i)})$ induced by the inclusion.
 We say that the gluings $f_1,\ldots,f_h$ are
 \emph{compatible} if the equality
$$
(f_i)_\ast \left(K_i^+\right)=K_i^-
$$
holds 
for every $i=1,\ldots,h$. Then we have the following:

\begin{thm}[Gromov additivity theorem\cite{Gromov, BBFIPP, Kuessner}]\label{simpl:thm}
Let $M_1,\ldots,M_k$ be $n$-dimensional manifolds, $n\geq 2$, 
suppose that the fundamental group of every boundary component of every $M_j$ is amenable, and let $M$
be obtained by gluing $M_1,\ldots,M_k$ along (some of) their boundary components.
Then
$$
\| M,\bb M\|\leq \|M_1,\bb M_1\|+\ldots+\|M_k,\bb M_k\|\ .
$$
In addition, if the gluings defining $M$ are compatible, then
$$
\| M,\bb M\|= \|M_1,\bb M_1\|+\ldots+\|M_k,\bb M_k\|\ .
$$
\end{thm}

Theorem~\ref{simpl:thm} will be proved in Chapter~\ref{additivity:chap}.

Of course, the gluings defining $M$ are automatically compatible if each $S_i^\pm$ is $\pi_1$-injective in $M_{j^{\pm}(i)}$
(in fact, this 
is the case in the most relevant applications of Theorem~\ref{simpl:thm}). Even in this special case, no inequality between $\|M,\bb M\|$ and $\sum_{j=1}^k \|M_j,\bb M_j\|$ holds in general
if we drop the requirement that the fundamental group of every  $S_i^\pm$ be amenable
(see Remark~\ref{counter}). On the other hand, even if the fundamental group of every 
$S_i^\pm$ is amenable, 
the equality in Theorem~\ref{simpl:thm} does not hold in general for \emph{non}-compatible
gluings (see again Remark~\ref{counter}).

Let us mention two important corollaries of Theorem~\ref{simpl:thm}.

\begin{cor}[Additivity for connected sums]\label{connectedsum}
Let $M_1,M_2$ be closed $n$-dimensional manifolds, $n\geq 3$.
Then
$$
\| M_1\# M_2\|=\|M_1\|+\|M_2\|\ ,
$$
where $M_1\# M_2$ denotes the connected sum of $M_1$ and $M_2$.
\end{cor}
\begin{proof}
Let $M_i'$ be
$M_i$ with one open ball removed. Since the $n$-dimensional disk has vanishing simplicial volume and $S^{n-1}$ is simply connected, Theorem~\ref{simpl:thm}
implies that $\|M_i',\bb M_i'\|=\|M_i\|$. Moreover, Theorem~\ref{simpl:thm} also implies
that
$\| M_1\# M_2\|=\|M'_1,\partial M'_1\|+\|M'_2,\partial M'_2\|$, and this concludes the proof.
\end{proof}

%\medskip

Let now $M$ be a closed $3$-dimensional manifold.
The prime decomposition Theorem, the JSJ decomposition Theorem and Perelman's proof 
of Thurston's geometrization conjecture imply
that $M$ can be canonically 
cut along spheres and $\pi_1$-injective tori into
the union of hyperbolic and Seifert fibered pieces. Since the simplicial volume of Seifert fibered spaces vanishes,
Theorem~\ref{simpl:thm} and Corollary~\ref{connectedsum}
imply the following:

\begin{cor}[\cite{Gromov}]\label{Soma:cor}
Let  $M$ be a $3$-dimensional manifold, and suppose that either $M$ is closed, 
or it is bounded by $\pi_1$-injective tori. Then the simplicial volume of $M$ is equal to the sum 
of the simplicial volumes of its hyperbolic pieces.
\end{cor}

A proof of Corollary~\ref{Soma:cor} 
(based on results from~\cite{Thurston})
may be found also in~\cite{Soma}.

\begin{rem}\label{counter}
The following examples show that 
the hypotheses of Theorem~\ref{simpl:thm} should not be too far from being the weakest
possible.

Let $M$ be a hyperbolic $3$-manifold with connected geodesic boundary. It is well-known
that $\partial M$ is 
$\pi_1$-injective in $M$. We fix a pseudo-Anosov homeomorphism $f\colon \partial M\to\partial M$, and for every $m\in \mathbb{N}$
we denote by $D_m M$ the twisted double obtained
by gluing two copies of $M$ along the homeomorphism $f^m\colon \partial M\to \partial M$
(so $D_0 M$ is the usual double of $M$).
It is shown in~\cite{Jungreis} that 
$$\| D_0 M\|<2\cdot \|M,\partial M \|\ .$$ 
On the other hand, 
by~\cite{Soma2} we have  $\lim_{m\to \infty} {\rm Vol}\, D_{m} M=\infty$,
so $\lim_{m\to \infty} \| D_{m} M \|=\infty$, and the inequality $$\|D_{m} M\|>2\cdot \|M,\partial M\|$$
holds for infinitely many $m\in\mathbb{N}$.
This shows that, 
even in the case when each $S_i^\pm$ is $\pi_1$-injective in $M_{j^{\pm}(i)}$, 
no inequality between $\|M,\bb M\|$ and $\sum_{j=1}^k \|M_j,\bb M_j\|$ holds
in general
if one drops the requirement that 
the fundamental group of every $S_i^\pm$ be amenable.

On the other hand, if $M_1$ is (the natural compactification of) the once-punctured torus and $M_2$ is the $2$-dimensional disk,
then the manifold $M$ obtained by gluing $M_1$ with $M_2$ along $\bb M_1\cong \bb M_2\cong S^1$ is a torus, 
so
$$
\|M\|=0<2+0=\| M_1,\bb M_1\|+ \|M_2,\bb M_2\|\ .
$$
This shows that, even 
in the case when the fundamental group of every $S_i^\pm$ is amenable, 
the equality $\|M,\bb M\|=\sum_{j=1}^k \|M_j,\bb M_j\|$
does not hold in general if one drops the requirement that the gluings
be compatible.
\end{rem}

\section{Simplicial volume and duality}\label{dual:simpl:sec}
Let $M$ be an $n$-manifold with (possibly empty) boundary
(recall that every manifold is assumed to be orientable and compact).
Then, the topological dual of the relative chain module
$C_i(M,\bb M)$ is the relative cochain module $C^\bullet_b(M,\bb M)$, whose cohomology gives
the relative bounded cohomology module $H^\bullet_b(M,\bb M)$
(see Section~\ref{relbounded:sec}).

As usual, if $\varphi\in H^i(M,\bb M)$ is a singular (unbounded) coclass, then we denote by
$\|\varphi\|_\infty\in [0,+\infty]$ the infimum of the norms of the
representatives of $\varphi$ in $Z^n(M,\bb M)$, i.e.~we set 
$$
\| \varphi\|_\infty =\inf \{\|\varphi_b\|_\infty\, |\, \varphi_b\in H^i_b(M,\bb M)\, ,\
c^i(\varphi_b)=\varphi\}\ ,
$$
where $\inf\emptyset =+\infty$ and $c^i\colon H^i_b(M,\bb M)\to
H^i(M,\bb M)$ is the comparison map induced by the inclusion of bounded relative
cochains into relative cochains.

We denote by $[M,\bb M]^*\in H^n(M,\bb M)$ the \emph{fundamental coclass}
of $M$, i.e.~the element $[M,\bb M]^*\in H^n(M,\bb M)$ such that
$$
\langle [M,\bb M]^*,[M,\bb M]\rangle =1 \ ,
$$ 
where $\langle\cdot , \cdot \rangle$ denotes the Kronecker product
(in Chapter~\ref{duality:chap} we defined the Kronecker product only in the context
of bounded cohomology, but of course the same definition makes sense also for
usual singular cohomology). Lemma~\ref{lemma:duality} implies the following:

\begin{prop}\label{prop:duality}
 We have 
$$
\|M,\bb M\|=\max \{ \langle\varphi_b,[M,\bb M]\rangle\, |\,  \varphi_b\in H_b^n(M,\bb M),\,  \|\varphi_b\|_\infty \leq 1\}\ ,
$$
so
$$
\|M,\bb M\|=\left\{\begin{array}{ll}
0 & \textrm{if}\ \|[M,\bb M]^*\|_\infty =+\infty\\
\|[M,\bb M]^*\|^{-1} & \textrm{otherwise.}
\end{array}\right.
$$
\end{prop}

\begin{cor}\label{criterion}
 Let $M$ be a closed orientable $n$-manifold, and let $c^n\colon H^n_b(M,\R)\to H^n(M,\R)\cong \R$ be the comparison map. Then $\|M\|=0$ if and only if
 $c^n=0$, and $\|M\|>0$ if and only if $c^n$ is surjective.
\end{cor}

Corollary~\ref{criterion} implies the following non-trivial vanishing result:

\begin{cor}\label{amvan}
 Let $M$ be a closed $n$-manifold with amenable fundamental group. Then
$$
\| M\|=0\ .
$$
In particular, any closed simply connected manifold has vanishing simplicial volume.
\end{cor}

\section{The simplicial volume of products}\label{prod:sec}
It is readily seen that the product $\Delta^n\times  \Delta^m$
of two standard simplices of dimension $n$ and $m$ can be triangulated by
$\binom{n+m}{n}$ simplices of dimension $n+m$.
Using this fact it is easy to prove that, if $M$, $N$ are closed manifolds of dimension
$m,n$ respectively, then
\begin{equation}\label{binom:eq}
\| M\times N\|\leq 
\binom{n+m}{m}
\|M\| \cdot \| N\|\ .
\end{equation}

In fact, an easy application of duality implies the following stronger result:

\begin{prop}[\cite{Gromov}]\label{product:prop}
 Let $M$, $N$ be closed manifolds of dimension
$m,n$ respectively. Then
$$
\| M\|\cdot \| N\|\leq \| M\times N\|\leq 
\binom{n+m}{m}
\|M\| \cdot \| N\|\ .
$$
\end{prop}
\begin{proof}
Thanks to~\eqref{binom:eq}, we are left to prove the 
inequality
$\| M\|\cdot \| N\|\leq \| M\times N\|$.
Let $[M]^*\in H^m(M)$, $[N]^*\in H^n(N)$ be the fundamental coclasses
of $M, N$ respectively. Then it is readily seen that
$$
[M]^*\, \cup\, [N]^* = [M\times N]^*\in H^n(M\times N)\, ,
$$
$$
\|[M\times N]^*\|_\infty=
\|[M]^*\, \cup\, [N]^* \|_\infty\leq \|[M]^*\|_\infty\cdot \| [N]^*\|_\infty\ ,
$$
where $\cup$ denotes the cup product.
So the inequality 
$$
\|M\times N\|\geq \|M\|\cdot \|N\|
$$
follows from Proposition~\ref{prop:duality}, and we are done.
\end{proof}

\section{Fiber bundles with amenable fibers}
A natural question is whether the estimates on the simplicial volume of a product in terms of the simplicial volumes of the factors carry over
to the context of fiber bundles. We refer the reader to Section~\ref{further:simplicial:readings} for a brief discussion of this topic. Here we just prove the following:

\begin{thm}\label{amenable:covers:simpl}
Let $p\colon E\to M$ be a locally trivial fiber bundle with fiber $F$, where $E,M,F$ are compact connected orientable manifolds. Assume that $\dim F\geq 1$, and that the image
$i_*(\pi_1(F))<\pi_1(E)$ of the fundamental group of $F$ via the map induced by the inclusion $F\hookrightarrow E$ is amenable. Then
$$
\| E\|=0\ .
$$
\end{thm}
\begin{proof}
Let $n=\dim E>\dim M$.
 Since a locally trivial fiber bundle is a fibration, we have an exact sequence
 $$
 \xymatrix{
 \pi_1(F) \ar[r]^{i_*} & \pi_1(E) \ar[r]^{p_*} & \pi_1(M) \ar[r] & \pi_0(F)\ .
 }
 $$
 Since $F$ is connected $\pi_0(F)$ is trivial, so our assumptions imply that $p_*\colon \pi_1(E)\to \pi_1(M)$ is an epimorphism
 with an amenable kernel. Then by Gromov Mapping Theorem (see Corollary~\ref{mapping:cor}), the map $H^n_b(p)\colon H^n_b(M,\R)\to H^n_b(E,\R)$
 is an isometric isomorphism. Let us now consider the commutative diagram
 $$
 \xymatrix{
 H^n_b(M,\R) \ar[d]^{c^n_M}\ar[r]^{H^n_b(p)} & H^n_b(E,\R) \ar[d]^{c_E^n}\\ 
 H^n(M,\R) \ar[r]^{H^n(p)} & H^n(E,\R)\ ,
 }
 $$
 where $c_M^\bullet$ (resp.~$c_M^\bullet$) denotes the comparison map. Since $H^n_b(p)$ is an isomorphism, we have
 $c_E^n=   H^n(p)\circ c_M^n\circ  H^n_b(p)^{-1}$. But $n>\dim M$, so $c_M^n=0$, hence $c_E^n=0$. By Corollary~\ref{criterion}, this implies $\| E\|=0$.
 \end{proof}

\section{Further readings}\label{further:simplicial:readings}

\subsection*{Simplicial volume and geometric invariants}
As already observed by Gromov in his seminal paper~\cite{Gromov}, the simplicial volume of manifolds is strictly related to many invariants of geometric nature. For example
if one defines the minimal volume ${\rm MinVol}(M)$ of a closed manifold $M$ as the infimum of the volumes of the Riemannian metrics supported by $M$ with sectional curvature bounded between $-1$ and $1$,
then the inequality $$\frac{\|M\|}{ (n-1)^n n!}\leq  {\rm MinVol}(M)$$ holds. In particular, non-vanishing of the simplicial volume implies non-vanishing of the minimal volume. Analogous estimates
hold when the minimal volume is replaced by the minimal entropy, an invariant which, roughly speaking, measures the rate of growth of balls in the universal covering. 
We refer the reader to~\cite{Gromov} for an extensive treatment of these topics.

\subsection*{Variations of the simplicial volume}
The notion of simplicial volume admits variations which turned out to be interesting for many applications. 
For example, the \emph{integral} simplicial volume $\|M\|_\mathbb{Z}$ of a manifold $M$ is the infimum (in fact, the minimum) of the $\ell^1$-seminorm of integral fundamental cycles of $M$.  
The integral simplicial volume is only submultiplicative with respect to finite coverings, but one can promote it to a multiplicative invariant by setting
\[ \stisv M := \inf \Bigl\{ \frac1d \cdot \isv {\overline M} 
                    \Bigm| \text{$d \in \matN$ and $\overline M \rightarrow M$ 
                           is a $d$-sheeted covering}
                    \Bigr\} 
   %\in \R_{\geq 0}
   .
\]
It is readily seen that
$$
\|M\|\leq \stisv M \leq \isv{M}\ ,
$$
but integral simplicial volume is often strictly bigger than the simplicial volume (e.g.,  
$\|M\|_\mathbb{Z}>0$ for every closed manifold $M$, while $\|M\|=0$ for many manifolds). 
It is natural to try to understand when stable integral simplicial volume is equal to the simplicial volume. 
Since integral cycles may be interpreted as combinatorial objects, 
when this is the case one can usually approach the study of the simplicial volume in a quite concrete way. 
For example,
the following purely topological problem on simplicial volume was
formulated by Gromov~\cite[p.~232]{gromovasymptotic}\cite[3.1.~(e) on
  p.~769]{gromov-cycles}:

\begin{quest}\label{Gromov-quest}
	Let $M$ be an
	oriented closed connected aspherical manifold. Does $\sv M = 0$ 
	imply $\chi (M) = 0$? 
\end{quest}

 A possible strategy to answer Question~\ref{Gromov-quest} in the affirmative is to replace
simplicial volume by stable integral 
simplicial volume: in fact, an easy application of Poincar\'e duality shows that 
$$
\stisv{M}=0\quad \Longrightarrow |\chi(M)|=0\ ,
$$
and this reduces Gromov's question to the question whether the vanishing of the simplicial volume implies the vanishing of the stable integral
simplicial volume (say, for aspherical manifolds with residually finite fundamental group).

More in general, one could ask which manifolds satisfy the equality $\|M\|=\stisv{M}$. This is the case for oriented closed surfaces of genus $g\geq 2$ (see Section~\ref{surface:simplicial:sec}) and for closed hyperbolic manifolds
of dimension $3$~\cite{FLPS}. On the contrary, $\stisv{M}$ is strictly bigger than $\|M\|$ for every closed hyperbolic $n$-manifold, $n\geq 4$~\cite{FFM}.

Yet another variation of the simplicial volume is the \emph{integral foliated simplicial volume} defined by Gromov in~\cite{Gromovbook}. 
In some sense, integral foliated simplicial volume interpolates the classical simplicial volume and the stable integral simplicial volume.
As suggested by Gromov himself~\cite[p. 305f]{Gromovbook}
and confirmed by Schmidt~\cite{Sthesis}, the vanishing of integral foliated simplicial volume implies the vanishing of $\ell^2$-Betti numbers, whence of the Euler characteristic.
Therefore, the integral foliated simplicial volume could probably be useful to approach Question~\ref{Gromov-quest}. For some recent results on this topic we refer the reader to~\cite{LP,FLPS}.

\subsection*{Simplicial volume of products and fiber bundles}
Let $M,F$ be closed orientable manifolds. We have seen above that the simplicial volume of $M\times F$ satisfies the bounds
$$
\| M\|\cdot \| F\|\leq \| M\times F\|\leq 
\binom{n+m}{m}
\|M\| \cdot \| F\|\ .
$$
A natural question is whether $\|M\times F\|$ should be equal to $c_{n,m} \|M\|\cdot \|F\|$ for a constant $c_{n,m}$ depending
only on $n=\dim M$, $m=\dim F$. The answer to this question is still unknown in general. 

Of course, products are just a special case of fiber bundles, so one may wonder whether the simplicial volume of the total space $E$ of
a fiber bundle with base $M$ and fiber $F$ could be estimated in terms of the simplicial volumes of $M$ and $F$. Since there exist closed hyperbolic $3$-manifolds
that fiber over the circle, no estimate of the form $\| E\|\leq \|M\| \cdot \| F\|$ can hold for any $k>0$ (while the simplicial volume of total spaces of fiber bundles with amenable fibers always vanishes, see Theorem~\ref{amenable:covers:simpl}). 
On the contrary, in the case when $F$ is a surface it was proved
in~\cite{HK} that the inequality $\| M\|\cdot \| F\|\leq \|E\|$ still holds. This estimate was then improved in~\cite{Michelle:fiber}, where it was shown that
$$
\frac{3}{2}\cdot \| M\|\cdot \| F\|\leq \|E\|
$$
for every fiber bundle $E$ with surface fiber $F$ and base $M$. This implies in particular that, if $S$ is a closed orientable surface, then
$$
\|M\times S\|\geq \frac{3}{2} \cdot \|M\|\cdot \|S\|\ .
$$
In fact, this inequality turns out to be an equality when $M$ is also a surface~\cite{Bucher3}. It it still an open question whether
the inequality
$$
\|E\|\geq \|M\|\cdot \|F\|
$$
holds for every fiber bundle $E$ with fiber $F$ and base $M$, without any restriction on the dimensions of $F$ and $M$.

\subsection*{Simplicial volume of non-compact manifolds}
If $M$ is a non-compact orientable $n$-manifold, then the $n$-th homology group of $M$ (with integral or real coefficients) vanishes. Therefore, in order to define the notion of simplicial volume, in this case
it is necessary to deal with \emph{locally finite} chains. With this choice it is still possible to define a fundamental class, and the simplicial volume is again the infimum of the $\ell^1$-norms 
of all the fundamental cycles of $M$. Of course, in this case such an infimum may be equal to $+\infty$: this is the case, for example, when $M=\R$, or when $M$ is the internal part of
a compact manifold with boundary $N$ such that $\|\partial N\|>0$. As mentioned in Section~\ref{further:duality}, suitable variations of the theory of bounded cohomology may be exploited
to establish interesting duality results, thus reducing the computation of the simplicial volume to a cohomological context~\cite{Loeh, Lothesis}. Nevertheless, 
bounded cohomology seems to be more powerful when dealing with compact manifolds, while
in the non-compact case the simplicial volume
seems to be more mysterious than in the compact case. 

For example, it is still not known whether the simplicial volume of the product of two punctured tori vanishes or not. Moreover, the proportionality principle fails in the non-compact case.
As already suggested by Gromov~\cite{Gromov}, probably more understandable objects may be obtained by defining suitable variations of the simplicial volume, such as the \emph{Lipschitz simplicial volume},
which was defined in~\cite{Gromov} and studied e.g.~in~\cite{Loh-Sauer, Loh-Sauer2, Franceschini1}. 

A slightly easier case arise when considering the internal part $M$ of a compact manifold $N$ with amenable boundary: in this case it is known that 
the simplicial volume of $N$, the Lipschitz simplicial volume of $M$ (with respect to any complete Riemannian metric), and the simplicial volume of $M$ all coincide 
(see e.g.~\cite{KimKue}). For other results dealing with non-amenable cusps we refer the reader to~\cite{KimKim,MichelleKim}.

\chapter{The proportionality principle}\label{prop:simpl:chap}
This chapter is devoted to the proofs of Theorems~\ref{prop:thm} and~\ref{hyp:thm}.
The proportionality principle for the simplicial volume of Riemannian manifolds
is due to Gromov~\cite{Gromov}. A detailed proof first appeared in
\cite{Loh}, where L\"oh exploited the approach via ``measure homology'' 
introduced by Thurston~\cite{Thurston}. Another proof is given 
in~\cite{Bucher}, where Bucher follows an approach which is based on the use of
bounded cohomology, and is closer in spirit to
the original argument by Gromov (however, it may be worth mentioning that, in~\cite{Loh}, 
the proof of the fact that measure homology is isometric to
the standard singular homology, which is the key step towards the proportionality 
principle, still relies on Gromov's isometric isomorphism between the bounded cohomology of a space and the one of its fundamental group, as well as on Monod's results about \emph{continuous} bounded cohomology of groups).
An independent proof of the proportionality principle, which does not make use of any sophisticated result from the theory of bounded cohomology, has recently been 
given by Franceschini~\cite{Franceschini1}.

Here we closely follow the account on Gromov's and Bucher's approach
to the proportionality principle described in~\cite{Frigerio}. However, in order to avoid some technical subtleties,
we restrict our attention to the study of non-positively curved compact Riemannian manifolds
(see Section~\ref{further:prop:sec} for a brief discussion of the general case).
Gromov's approach to the proportionality principle exploits
an averaging process which can be defined
only on sufficiently regular cochains. 
As a consequence, we will be lead to study the complex of cochains which are \emph{continuous}
with respect to the compact-open topology on the space of singular simplices.

\section{Continuous cohomology of topological spaces}
Let $M$ be an $n$-dimensional manifold.
For every $i\in\matN$, we endow the space of singular $i$-simplices ${S_i} (M)$ with the compact-open topology (since the standard simplex is compact,
this topology coincides with the topology of the uniform convergence with respect to any metric which induces the topology of $M$).
For later reference we point out the following elementary property of the compact-open topology (see e.g.~\cite[page 259]{Dug}):

\begin{lemma}\label{basi:top:lemma1}
Let $X,Y,Z$ be topological spaces, and $f\colon Y\to Z$, $g\colon X\to Y$ be continuous.
The maps $f_\ast \colon F(X,Y)\to F(X,Z)$, $g^\ast \colon F(Y,Z)\to F(X,Z)$ defined by $f_\ast (h)=f\circ h$,
$g^\ast (h)=h\circ g$ are continuous. 
%In particular, if $H\subseteq F(Y,Z)$ is compact, then
%$\{h\circ g\, |\, h\in H\}\subseteq F(X,Z)$ is compact.
%Moreover, if $Y$ is locally compact, then the map
%$\theta\colon F(X,Y)\times  F(Y,Z)\to F(X,Z)$ defined by $\theta (h,h')=h'\circ h$
%is continuous.
\end{lemma}

Throughout this chapter, all the (co)chain and (co)homology modules will be understood with real coefficients.
We say that a cochain $\varphi\in {C^i (M)}$ is 
\emph{continuous} if it restricts
to a continuous map on ${S_i(M)}$, and we denote by ${ \ccst (M)}$
the subcomplex of continuous cochains in ${ \cst} (M)$ (the fact that $\ccst(M)$ is indeed a subcomplex
of $\cst(M)$ is a consequence of Lemma~\ref{basi:top:lemma1}).
We also denote by ${ \cclimst} (M) = { \ccst (M)}\cap { \climst (M)}$
the complex of bounded continuous cochains.
The corresponding cohomology modules will be denoted by
${ \hcst (M)}$ and ${ \hclimst (M)}$.
The natural inclusions of cochains 
$$
i^\bullet\colon C_c^\bullet(M)\to C^\bullet (M),\qquad
i_b^\bullet\colon C_{b,c}^\bullet(M)\to C_b^\bullet (M)
$$
induce maps
$$
H^\bullet({ i^\bullet}) \colon  { \hcst (M)}\to { \hst (M)},
\qquad 
H^\bullet_b({ i_b^\bullet})  \colon  { \hclimst (M)}\to { \hlimst (M)}.
$$

Bott stated in~\cite{Bott} that, at least for ``reasonable spaces'', the map $H^n(i^n)$
is an isomorphism for every $n\in\mathbb{N}$. However, Mostow asserted in~\cite[Remark 2 at p.~27]{Mostow} that the natural proof of this fact
seems to raise some difficulties. 
More precisely, it is quite natural to ask whether continuous cohomology
satisfies Eilenberg-Steenrod axioms for cohomology.
First of all, continuous cohomology is functorial thanks to
Lemma~\ref{basi:top:lemma1}.
%Even if this is proved to be true in~\cite{russo}, the argument showing
%that continuous cohomology satisfies the axiom of excision 
%is quite subtle (see Remark~\ref{nuovorem} below). 
Moreover, it is not difficult to show that
continuous cohomology satisfies 
the so-called ``dimension axiom'' and ``homotopy axiom''.
However, if $Y$ is a subspace of $X$ it is in general not possible
to extend cochains in $\ccst (Y)$ to cochains in $\ccst (X)$, so that it is not clear if a natural long
exact sequence for pairs actually exists in the realm of continuous cohomology.
This difficulty can be overcome either by considering only pairs $(X,Y)$ where
$X$ is metrizable and
$Y$ is closed in $X$, or by exploiting a cone construction, 
as described in~\cite{russo}.
A still harder issue
arises about excision: even if the barycentric subdivision operator consists of
a finite sum (with signs) of continuous self-maps of $S_\bullet (X)$,
the number of times a simplex should be subdivided in order to become ``small''
 with respect to a given open cover 
depends in a decisive way on the simplex itself. These difficulties have been overcome independently by 
Mdzinarishvili and the author (see respectively~\cite{russo} and~\cite{Frigerio}). 
In fact, it turns out that the inclusion induces an isomorphism between continuous cohomology and singular cohomology
for every space 
having the homotopy type of a metrizable and locally contractible
topological space.

The question whether the map $H^\bullet_b(i^\bullet_b)$ is an isometric isomorphism is even more difficult. An affirmative answer is provided in~\cite{Frigerio}
in the case of spaces having the homotopy type of an aspherical CW-complex. 

As anticipated above, we will restrict our attention to the case when $M$ is a manifold supporting a non-positively curved Riemannian
structure. In this case  the existence of a \emph{straightening procedure} for simplices
makes things much easier, and allows us to prove that $H^\bullet(i^\bullet)$ and $H^\bullet_b(i^\bullet_b)$ are both isometric isomorphisms
via a rather elementary argument.

\section{Continuous cochains as relatively injective modules}
Until the end of the chapter we  denote by $p\colon\xtil\to M$ the universal covering of the closed
smooth manifold $M$, and we fix an identification
of $\G=\pi_1(M)$ with the group of the covering automorphisms of $p$.

Recall that the complex $C^\bullet(\xtil)$ (resp.~$C_b^\bullet(\xtil)$) is naturally endowed with the structure
of an $\R[\G]$-module (resp.~normed $\R[\G]$-module). For every $g\in\G$, $i\in\mathbb{N}$, the map $S_i(\xtil)\to S_i(\xtil)$
sending the singular simplex $s$ to $g\cdot s$ is continuous with respect to the compact-open topology
(see Lemma~\ref{basi:top:lemma1}). Therefore, the action of $\G$ on (bounded) singular cochains on $\xtil$
preserves continuous cochains, and $C_c^\bullet(\xtil)$ (resp.~$C_{b,c}^\bullet(\xtil)$) inherits the structure
of an $\R[\G]$-module (resp.~normed $\R[\G]$-module).
In this section we show that, for every $i\in\mathbb{N}$, the module $C_c^i(\xtil)$ (resp.~$C_{b,c}^i(\xtil)$)
is relatively injective according to Definition~\ref{inj2:def} (resp.~Definition~\ref{relativelyinjective}).
The same result was proved for ordinary (i.e.~possibly non-continuous) cochains in Chapter~\ref{res:chapter}.
The following lemma ensures that continuous cochains on $M$ canonically correspond to $\G$-invariant
continuous cochains on $\xtil$.

\begin{lemma}\label{sollevo:lemma}
The chain map
$p^\bullet\colon C^\bullet (M)\to C^\bullet (\xtil)$ restricts to the following isometric isomorphisms
of complexes:
$$
p^\bullet\colon C^\bullet (M)\to C^\bullet (\xtil)^\Gamma\, ,\qquad  
p^\bullet|_{\cclimst (M)}\colon \cclimst (X)\to \cclimst (\xtil)^\Gamma,
$$
which, therefore, induce isometric isomorphisms
$$
\hcst (M)%=H^\bullet (\ccst (X))
\cong H^\bullet (\ccst (\xtil)^\Gamma)\, ,\qquad
\hclimst (M)%=H^\bullet (\cclimst (X))
\cong H^\bullet (\cclimst (\xtil)^\Gamma).
$$
\end{lemma}
\begin{proof}
We have already used the obvious fact that $p^\bullet$ is an isometric embedding on the space of 
$\Gamma$--invariant cochains, thus the only non-trivial issue
is the fact that $p^\bullet (\varphi)$ is continuous if and only if $\varphi$
is continuous. 
However, it is not difficult to show that
the map $p_n\colon S_n(\xtil)\to S_n (M)$ induced by $p$ is a covering
(see~\cite[Lemma A.4]{Frigerio} for the details).
In particular, it is continuous, open and surjective, and this readily implies
that a map $\varphi\colon S_n (M)\to\R$ is continuous if and only
if $\varphi\circ p_n\colon S_n (\xtil)\to\R$ is continuous, whence the conclusion.
\end{proof}

It is a standard fact of algebraic topology that the action of $\Gamma$ on $\xtil$ is \emph{wandering},
\emph{i.e.}~every $x\in\xtil$ admits a neighbourhood $U_x$ such that $g(U_x)\cap U_x=\emptyset$
for every $g\in\Gamma\setminus \{1\}$.
In the following lemma we describe a particular
instance of \emph{generalized Bruhat function} (see~\cite[Lemma 4.5.4]{Monod} for a more general
result based on~\cite[Proposition 8 in VII $\S$2 N$^\circ$ 4]{bou}).

\begin{lemma}\label{tecnico}
There exists a continuous map 
$h_{\widetilde{M}}\colon \xtil\to [0,1]$ with the following properties:
\begin{enumerate}
\item
For every $x\in\xtil$ there exists a neighbourhood $W_x$ of $x$ in $\xtil$
such that the set
$
\left\{g\in\Gamma\, | \,  g(W_x) \cap
{\rm supp}\, h_{\widetilde{M}}
\neq\emptyset\right\}
$
is finite.
\item
For every $x\in\xtil$, we have 
$
\sum_{g\in \Gamma} h_{\widetilde{M}} (g\cdot x)=1
$
(note that the sum on the left-hand side is finite by (1)).
\end{enumerate}
\end{lemma}
\begin{proof}
Let us take a locally finite open 
cover $\{U_i\}_{i\in I}$ of $X$ such that
for every $i\in I$ 
there exists $V_i\subseteq \xtil$ with $p^{-1} (U_i)=\bigcup_{g\in \Gamma}
g(V_i)$ and $g(V_i)\cap g'(V_i)=\emptyset$ whenever $g\neq g'$.
Let $\{\varphi_i\}_{i\in I}$ be a partition of unity adapted to $\{U_i\}_{i\in I}$. It is easily seen that
the map $\psi_i\colon \widetilde{M}\to\R$ which concides with $\varphi_i\circ p$ on $V_i$
and is null elsewhere is continuous. We can now set $h_{\widetilde{M}}=\sum_{i\in I} \psi_i$.
Since $\{U_i\}_{i\in I}$ is locally finite, also $\{V_i\}_{i\in I}$,
whence $\{ {\rm supp}\, \psi_i\}_{i\in I}$, is locally finite, so $h_{\widetilde{M}}$
is indeed well defined and continuous. 

It is now easy to check that $h_{\widetilde{M}}$ indeed satisfies the required properties
(see~\cite[Lemma 5.1]{Frigerio} for the details).
\end{proof}

\begin{prop}\label{inj1:prop}
For every $n\geq 0$ the modules $\cc^n (\xtil)$ and $\cclim^n (\xtil)$ are relatively injective
(resp.~as an $\R[\G]$-module and as a normed $\R[\G]$-module).
\end{prop}
\begin{proof}
Let $\iota\colon A\to B$ be an injective map between $\R[\Gamma]$-modules,
with left inverse $\sigma\colon B\to A$, and suppose we are given a 
$\Gamma$-map $\alpha\colon A\to C_c^n (\xtil)$. 
We denote by $e_0,\ldots,e_n$ the vertices of the standard $n$--simplex, and define
$\beta\colon B\to C_c^n (\xtil)$ as follows: given $b\in B$, the cochain $\beta (b)$
is the unique linear extension of the map that on the singular
simplex $s$ takes the following value:
$$
\beta (b) (s) =\sum_{g\in \Gamma} h_{\widetilde{M}} \left(g^{-1} (s(e_0))\right)
\cdot \left(\alpha (g(\sigma (g^{-1}(b))))(s)\right),
$$
where $h_{\widetilde{M}}$ is the map provided by Lemma~\ref{tecnico}.
By Lemma~\ref{tecnico}--(1), the sum involved is in fact finite, so $\beta$ is well defined.
Moreover, for every $b\in B$, $g_0 \in\Gamma$ and $s\in S_n (\xtil)$ we have
$$
\begin{array}{lll}
\beta (g_0\cdot b)(s) & = & \sum_{g\in \Gamma} \hx \Big(g^{-1}\big(s(e_0)\big)\Big)\cdot \alpha \Big(g\big(\sigma(g^{-1}
g_0 (b))\big)\Big)(s)\\ &=&
\sum_{g\in \Gamma} \hx \Big(g^{-1}g_0 \cdot (g_0^{-1}\cdot s)(e_0)\Big)\cdot \alpha \Big(g_0 g_0^{-1}g \big(\sigma(g^{-1}
g_0 (b))\big)\Big)(s)\\ &=&
\sum_{k\in \Gamma} \hx \Big(k^{-1} (g_0^{-1}\cdot s)(e_0)\Big)\cdot \alpha \Big(g_0 k\big(\sigma(k^{-1}
(b))\big)\Big)(s)\\ &=&
\sum_{k\in \Gamma} \hx \Big(k^{-1} (g_0^{-1}\cdot s)(e_0)\big)\cdot \alpha \Big(k \big(\sigma(k^{-1}
(b))\big)\Big)(g_0^{-1}\cdot s)\\ &=&
\beta (b) (g_0^{-1}\cdot s)= \big(g_0 \cdot \beta (b)\big)(s),
\end{array} 
$$
so $\beta$ is a $\Gamma$--map. Finally, 
$$
\begin{array}{lll}
\beta(\iota(b))(s)&=& \sum_{g\in \Gamma} \hx \Big(g^{-1}\big(s(e_0)\big)\Big)\cdot 
\alpha \Big(g\big(\sigma(g^{-1}(\iota(b)))\big)\Big)(s)\\ &=& 
\sum_{g\in \Gamma} \hx \Big(g^{-1}\big(s(e_0)\big)\Big)\cdot 
\alpha \Big(g\big(\sigma(\iota(g^{-1}\cdot b))\big)\Big)(s)\\ &=&
\sum_{g\in \Gamma} \hx \Big(g^{-1}\big(s(e_0)\big)\Big)\cdot 
\alpha (b)(s)\\ &=&
\Bigg(\sum_{g\in \Gamma} \hx \Big(g^{-1}\big(s(e_0)\big)\Big)\Bigg)\cdot \alpha (b)(s)=\alpha(b)(s),
\end{array}
$$
so $\beta\circ\iota=\alpha$.
In order to conclude that $C^n_c (\xtil)$ is relatively injective we need to show that
$\beta(b)$ is indeed continuous.
However, if $s\in S_n (\xtil)$ is a singular $n$-simplex, then by Lemma~\ref{tecnico}--(1)
there exists a neighbourhood $U$ of $s$ in $S_n (\xtil)$ such that the set
$\{g\in\Gamma\, | \, h_{\widetilde{M}} (g^{-1} (s' (e_0)))\neq 0\ {\rm for\ some}\
s'\in U\}$ is finite. This readily implies that if $\alpha (A)\subseteq C_c^n (\xtil)$, 
then also $\beta(B)\subseteq C_c^n (\xtil)$. 

The same argument applies \emph{verbatim} 
if $C^n (\xtil)$ is replaced by $\clim^n (\xtil)$, and $A,B$ are normed modules: in fact, 
if $\alpha$ is bounded and $\|\sigma\|\leq 1$, then also $\beta$ is bounded, and
$\|\beta\|\leq \|\alpha\|$. Thus 
$\cclim^n (\xtil)$ is a relatively injective normed $\Gamma$--module.
\end{proof}
 
\section{Continuous cochains as strong resolutions of $\R$}\label{conv:sec} 
Recall from Section~\ref{revisited:sec} that, if $\widetilde{M}$ is contractible, then the complex $C^\bullet(\xtil)$, endowed with
the obvious augmentation, provides a strong resolution of the trivial $\G$-module $\R$.
Moreover, thanks to Ivanov's Theorem~\ref{ivanov:thm}, the augmented normed complex
$C_b^\bullet(\xtil)$ is a strong resolution of the normed $\G$-module $\R$ even without the assumption
that $\widetilde{M}$ is contractible. The situation is a bit different when dealing with continuous cochains.
Unfortunately, we are not able to prove that 
Ivanov's contracting homotopy (see Section~\ref{ivanov:sec}) preserves continuous cochains, so 
it is not clear whether 
$C_{b,c}^\bullet(\xtil)$ always provides a strong resolution of $\R$.

In order to prove that the augmented complexes
$C_{c}^\bullet(\xtil)$ and $C_{b,c}^\bullet(\xtil)$ are strong we need to make some further assumption on $M$.
Namely, asking that $\widetilde{M}$ is contractible (i.e.~that $M$ is a $K(\G,1)$) would be sufficient.
However, to our purposes it is sufficient to concentrate our attention on the easier case when
$M$ supports a non-positively curved Riemannian metric.
In that case, Cartan-Hadamard Theorem implies that, for every point $x\in\widetilde{M}$, the exponential
map at $x$ establishes a diffeomorphism between the tangent space at $x$ and the manifold $\widetilde{M}$.
As a consequence, every pair of points in $\widetilde{M}$ are joined by a \emph{unique} geodesic,
and (constant-speed parameterizations of) geodesics continuously depend on their endpoints:
we say that $\widetilde{M}$ is \emph{continuously uniquely geodesic}.

Therefore, until the end of the chapter, we  assume that the closed manifold $M$ is
endowed with a non-positively curved Riemannian metric.

For $i\in\matN$ we denote by $e_i\in \R^{\matN}$ the point
$(0,\ldots,1,\ldots)$,
where the unique non-zero coefficient is at the $i$-th entry (entries are indexed by $\matN$,
so $(1,\!0,\ldots)\!=e_0$). 
We denote by $\Delta^p$ the standard $p$-simplex, 
\emph{i.e.}~the convex hull of $e_0,\ldots,e_p$, and we observe
that with these notations we have $\Delta^p\subseteq \Delta^{p+1}$.

\begin{prop}\label{strong1:prop}
The complexes ${\ccst (\xtil)}$ and ${\cclimst (\xtil)}$ are strong resolutions of
$\R$ (resp.~as an unbounded $\Gamma$-module and as a normed $\R$-module). 
\end{prop}
\begin{proof}
Let us choose a basepoint $x_0\in\xtil$.
For $n\geq 0$, we define an operator $T_n\colon C_n (\xtil)\to C_{n+1}(\xtil)$
which sends any singular $n$-simplex $s$ to the $(n+1)$-simplex
obtained by coning $s$ over $x_0$. In order to properly define the needed coning procedure
we will exploit the fact that $\xtil$ is uniquely continuously geodesic.

For $x\in \xtil$ we denote by $\gamma_x\colon [0,1]\to \xtil$ the constant-speed
parameterization of the geodesic joining $x_0$ with $x$. 
Let $n\geq 0$. We denote by $Q_0$ the face of $\Delta^{n+1}$ opposite to $e_0$, and we consider the identification
$r\colon Q_0\to \Delta^n$ given by
$r(t_1 e_1+\ldots t_{n+1} e_{n+1})=t_1 e_0+\ldots t_{n+1} e_n$.
%Let now $s\colon \Delta_n\to\xtil$ be a singular simplex.
%For $n\geq 0$, 
We now define $T_n\colon C_n (\xtil)\to C_{n+1}(\xtil)$
as the unique linear map such that,
if $s\in S_n(\xtil)$, then the following holds:
if
$p=t e_0+(1-t)q \in\Delta^{n+1}$, where $q\in Q_0$, then $(T_n(s)) (p)=\gamma_{s(q)}(t)$.
In other words,
$T_n (s)$ is just the geodesic cone over $s$ with vertex $x_0$.
Using that $\xtil$ is uniquely continuously geodesic, one may easily check that
$T_n (s)$ is well defined and continuous. Moreover, the restriction of $T_n$ to
$S_n(\xtil)$ defines a map
$$
T_n\colon S_n(\xtil)\to S_{n+1}(\xtil)
$$
which is continuous with respect to the compact-open topology.
We finally define $T_{-1}\colon \R\to C_0 (\xtil)$
by $T_{-1} (t)=t x_0$.
It is readily seen that, if $d_\bullet$ is the usual (augmented)
differential on singular chains, then $d_{0} T_{-1}={\rm Id}_\R$, and
for every $n\geq 0$ we have
$T_{n-1} \circ d_n + d_{n+1} \circ T_n= {\rm Id}_{C_n (\xtil)}$.

For every $n\geq 0$, let now $k^n\colon C^n (\xtil)\to C^{n-1} (\xtil)$ be defined by $k^n (\varphi)(c)=
\varphi (T_{n-1} (c))$. Since $T_n\colon S_n(\xtil)\to S_{n+1}(\xtil)$ is continuous,
the map $k^n$ preserves continuous cochains, so $\{k^n\}_{n\in\matN}$ provides a contracting
homotopy for the complex 
${C_c^\bullet (\xtil)}$, which is therefore a strong resolution of $\R$ as an unbounded $\R[\G]$-module.

Finally, since $T_n$ sends every single simplex to a single simplex, if $\alpha\in \clim^n (\xtil)$ 
then  $\|k^n (\alpha)\|\leq \|\alpha\|$. Thus $k^\bullet$ 
restricts to a contracting homotopy for the complex of normed $\Gamma$--modules 
${\cclimst (\xtil)}$. Therefore, this complex gives a
strong resolution of $\R$ as a normed $\Gamma$-module.
\end{proof}

We have thus proved that, if $M$ supports a non-positively curved metric, then $C_c^\bullet(\xtil)$ and $C_{b,c}^\bullet(\xtil)$ provide
relatively injective strong resolutions of $\R$ (as an unbounded $\G$-module and as a normed $\G$-module, respectively). 
As proved in Lemma~\ref{sing:inj}, Proposition~\ref{sing:res}, Lemma~\ref{sing:bd:inj} 
and Lemma~\ref{cones},
the same is true for the complexes $C^\bullet(\xtil)$ and $C_{b}^\bullet(\xtil)$, respectively.
Since the inclusions $C_c^\bullet(\xtil)\to C^\bullet(\xtil)$, 
$C_{b,c}^\bullet(\xtil)\to C_b^\bullet(\xtil)$ are norm non-increasing chain maps
which extend the identity of $\R$, Theorems~\ref{ext:thm} and \ref{ext:bounded:thm}
and Lemma~\ref{sollevo:lemma} imply the following:

\begin{prop}\label{isononisom}
Let $M$ be a closed manifold supporting a non-positively curved metric. Then the maps
$$
H^\bullet(i^\bullet)\colon H_c^\bullet(M)\to H^\bullet(M)\, ,\qquad 
H^\bullet_b(i_b^\bullet)\colon H_{b,c}^\bullet(M)\to H_b^\bullet(M)
$$
are  norm non-increasing isomorphisms.
\end{prop}

In order to promote these isomorphisms to isometries it is sufficient to exhibit 
norm non-increasing chain $\G$-maps $\theta^\bullet\colon C^\bullet(\xtil)\to C_c^\bullet(\xtil)$, 
$\theta_b^\bullet\colon C_{b}^\bullet(\xtil)\to C_{b,c}^\bullet(\xtil)$ which extend the identity of $\R$.
To this aim we exploit a straightening procedure, which will prove useful several times in this book. 

\section{Straightening in non-positive curvature}
The \emph{straightening procedure} for simplices was introduced by Thurston in~\cite{Thurston}.
It was originally defined on hyperbolic manifolds, 
but it can be performed in the more general 
context of non-positively curved Riemannian manifolds. 

Let $k\in\matN$, and let $x_0,\ldots,x_k$ be points in $\widetilde{M}$. 
The \emph{straight} simplex $[x_0,\ldots, x_k]\in S_k (\widetilde{M})$
with vertices $x_0,\ldots,x_k$ is defined as follows:
if $k=0$, then $[x_0]$ is the $0$-simplex with image $x_0$; if straight simplices have
been defined for every $h\leq k$, then $[x_0,\ldots,x_{k+1}]\colon \Delta^{k+1}\to\widetilde{M}$
is determined by the following condition:
for every $z\in \Delta^k\subseteq \Delta^{k+1}$, the restriction
of $[x_0,\ldots,x_{k+1}]$ to the segment with endpoints
$z,e_{k+1}$ is the constant speed parameterization of
the unique geodesic joining $[x_0,\ldots,x_k] (z)$ to $x_{k+1}$.
The fact that $[x_0,\ldots,x_{k+1}]$ is well defined and continuous is an immediate
consequence of the fact that $\xtil$ is continuously uniquely geodesic.

\section{Continuous cohomology versus singular cohomology}
We are now ready to prove that (bounded) continuous cohomology is isometrically isomorphic
to (bounded) cohomology at least for non-positively curved manifolds:

\begin{prop}\label{continuous:isometric}
Let $M$ be a closed manifold supporting a non-positively curved metric. Then the maps
$$
H^\bullet(i^\bullet)\colon \cc^\bullet(M)\to C^\bullet(M)\, ,\qquad
H_b^\bullet(i_b^\bullet)\colon C_{b,c}^\bullet(M)\to C_b^\bullet(M)
$$
are isometric isomorphisms.
\end{prop}
%Suppose $X$ is paracompact, and 
%let ${F}^\bullet (\Gamma)$ (resp.~$F_b^\bullet (\Gamma)$) be the standard resolution of $\R$
%as an unbounded (resp.~Banach) $\Gamma$--module. 
%There exists a chain map $\beta^\bullet\colon {F}^\bullet (\Gamma) \to \ccst (\xtil)$ 
%which extends the identity of $\R$ and is such that $\beta^n$ is norm--decreasing for every $n\in\matN$.
%In particular, $\beta^\bullet$ restricts to a chain map $\beta_b^\bullet\colon {F}_b^\bullet (\Gamma) \to \cclimst (\xtil)$ 
%which extends the identity of $\R$ and is such that $\|\beta^n_b \|\leq 1$ for every $n\in\matN$.
%\end{prop}
\begin{proof}
As observed at the end of Section~\ref{conv:sec}, it is sufficient to exhibit 
norm non-increasing chain $\G$-maps $\theta^\bullet\colon C^\bullet(\xtil)\to C_c^\bullet(\xtil)$, 
$\theta_b^\bullet\colon C_{b}^\bullet(\xtil)\to C_{b,c}^\bullet(\xtil)$ which extend the identity of $\R$.
To this aim, we choose a basepoint $x_0\in\xtil$. If $\varphi\in C^n(\xtil)$ and
$s\in S_n(\xtil)$, then we set
$$
\theta^n(\varphi)(s)=
\sum_{(g_0,\ldots,g_n)\in\Gamma^{n+1}} \hx (g_0^{-1} \cdot {s}(e_0))\cdots
\hx (g_n^{-1} \cdot{s}(e_n))\cdot {\varphi} ([g_0x_0,\ldots,g_nx_0])\ .
$$
In other words, the value of $\theta^n(\varphi)$ is the weighted sum of the values
taken by $\varphi$ on straight simplices with vertices in the orbit of $x_0$, where
weights continuously depend on the position of the vertices of $s$. Using the properties
of $\hx$ described in Lemma~\ref{tecnico} it is easy to check that $\theta^\bullet$ is a well-defined
chain $\G$-map which extends the identity of $\R$. Moreover, $\theta^\bullet$ is obviously norm non-increasing in every degree, so it restricts
a norm non increasing chain $\G$-map $\theta_b^\bullet$, and this concludes the proof.
\end{proof}

\section{The transfer map}
Let us now denote
by ${G}$ the group of orientation--preserving isometries
of $\xtil$. It is well-known
that $G$ admits a Lie group structure inducing the compact--open topology. 
Moreover, there exists on $G$ a left-invariant regular Borel measure $\mu_G$, which 
is called \emph{Haar measure} of $G$ and  
is unique up to scalars.
Since $G$ contains a cocompact subgroup, 
its Haar measure is in fact
also right-invariant~\cite[Lemma 2.32]{Sauer}.
Since $\G$ is discrete in $G$ and 
$M\cong \xtil /\G$ is compact, there exists a Borel subset $F\subseteq G$
with the following properties: the family $\{\gamma\cdot F\}_{\gamma\in \G}$ provides a locally finite partition of $G$
(in particular, $F$ contains exactly one representative for each left coset of $\G$ in $G$), and $F$ is relatively compact in $G$. 
%We will call
%such an $F$ a \emph{fundamental region} for $\G$ in $G$.
From now on, we normalize the Haar measure
$\mu_G$ in such a way that $\mu_G (F)=1$.

%In order to avoid too heavy notations, if $H$ is a subgroup of $\G$
%we set
%${ H_blimst} (\xtil)^H=H^\bullet ({ C_blimst} (\xtil)^H)$, 
For $H=\G,G$, we will consider the homology
$H^\bullet ({ \ccst} (\xtil)^H)$ of the complex given by
$H$-invariant cochains in ${ \ccst} (\xtil)$.
We also endow $H^\bullet ({ \ccst} (\xtil)^H)$
%${ H_blimst} (\xtil)^H$, 
%${ \hcst}(\xtil)^H$ 
with the seminorm induced by 
%${ C_blimst} (\xtil)^H$, 
${ \ccst} (\xtil)^H$ (this seminorm is not finite in general).
Recall that by Lemma~\ref{sollevo:lemma} we 
have an isometric isomorphism 
$H^\bullet ({ \ccst} (\xtil)^\G)
\cong { \hcst} (M)$. 
%${ H_blimst}(\xtil)^\Gamma\cong { H_blimst} (X)$.
The chain inclusion ${ \ccst} (\xtil)^G\hookrightarrow 
{ \ccst} (\xtil)^\G$ induces a norm non-increasing map
$$
{\rm res}^\bullet \colon  
H^\bullet ({ \ccst} (\xtil)^G)
\longrightarrow 
H^\bullet ({ \ccst} (\xtil)^\G)
\cong { \hcst} (M).%\\
$$

Following~\cite{Bucher}, we will now construct a norm non-increasing left inverse 
of $\res^\bullet$. 
Take $\varphi\in { C_c^i} (\xtil)$ and $s\in { C_i} (\xtil)$, 
and consider the function $f_\varphi^s \colon G\to \R$
defined by $f_\varphi^s (g)=\varphi (g\cdot s)$. 
By Lemma~\ref{basi:top:lemma1}
$f_\varphi^s$ is
continuous, whence bounded on the relatively compact subset $F\subseteq G$.
% Moreover, 
%$F$ is relatively compact in $\G$ and $\varphi$
%is locally bounded, so the restriction of $f_\varphi^s$ to $F$ is bounded. 
Therefore, a well-defined
cochain $\tr^i (\varphi)\in { C^i} (\xtil)$ exists such that for every $s\in { S_i} (\xtil)$ we have
$$
\tr^i (\varphi) (s)=\int_F f_\varphi^s (g)\, d\mu_G (g) = \int_F \varphi (g\cdot s)\, d\mu_G (g) .
$$

\begin{prop}\label{funziona:prop}
The cochain $\tr^i (\varphi)$ is continuous. Moreover,
if $\varphi$ is $\G$--invariant, then $\tr^i (\varphi)$ is $G$--invariant,
while if $\varphi$ is $G$--invariant, then $\tr^i (\varphi)=\varphi$.
\end{prop}
\noindent {\sc Proof:}
%Let $H\colon F\times { S_i} (\xtil)\to \R$ be the map defined by 
%$H(g,s)=f_\varphi^s (g)=\varphi (g\cdot s)$. By Lemma~\ref{contact:lemma}, such a map is Borelian.
%By a standard result in measure theory (see \emph{e.g.}~the proof 
%of Fubini's Theorem in~\cite{bill}), this implies that the map 
%$$
%{ S_i (\xtil)}\to \R,\qquad
%s\mapsto \int_F H(g,s)\, d\mu_\G (g)
%$$ 
%is Borelian. But such a map is equal to the restriction
%of $\tr'_i (\varphi)$ to  ${ S_i} (\xtil)$. This shows that 
%$\tr'_i (\varphi)$ is Borelian.
%In order to show that $\tr'_i (\varphi)$ is locally bounded, 
For $s,s'\in S_i(\xtil)$ we set $d(s,s')=\sup_{x\in\Delta^i} d_{\xtil} (s(q),s'(q))$. Since
$\Delta^i$ is compact, $d_S$ is a well-defined distance on $S_i(\xtil)$, and it induces
the compact-open topology on $S_i(\xtil)$.

Let now $s_0\in { S_i} (\xtil)$ and $\vare >0$ be fixed. By Lemma~\ref{basi:top:lemma1}, the set
$\overline{F}\cdot s_0\subseteq { S_i} (\xtil)$ is compact. 
Since $\varphi$ is continuous, this 
easily implies that there exists $\eta>0$ such that $|\varphi (s_1)-\varphi (s_2)| \leq \vare$
for every $s_1\in \overline{F}\cdot s_0$, $s_2\in B_{d_S} (s_1,\eta)$, where 
$B_{d_S}(s_1,\eta)$ is the open ball
of radius $\eta$ centered at $s_1$. Take now $s\in B_{d_S} (s_0,\eta)$.
Since $G$ acts isometrically on ${ S_i (X)}$, for every $g\in F$ we have $d_S (g\cdot s_0, g\cdot s)=
d_S (s_0,s)<\eta$, so $|\varphi (g\cdot s)-\varphi (g\cdot s_0)| \leq \vare$.
Together with the fact that $\mu_G (F)=1$, this readily implies 
$$
%\begin{array}{lllll}
|\tr^i (\varphi) (s) -\tr^i (\varphi) (s_0)|  
%\left| \int_F \varphi(g\cdot s)\, d\mu_ (g) -
 %\int_F \varphi(g\cdot s_0)\, d\mu_\G (g)\right|   
 \leq 
 \int_F |\varphi (g\cdot s)-\varphi (g\cdot s_0) |\, d\mu_G (g)  \leq \vare .
 %\end{array}
 $$
We have thus proved that $\tr^i (\varphi)$ is continuous.

%bounded, this implies that an open neighbourhood $U$ of $\overline{F}\cdot s_0$ in ${ S_i} (\xtil)$
%and a constant $M>0$ exist such that $|\varphi (s)|\leq M$ for every $s\in U$.  
%Let now $\varepsilon=\inf \{d_S (s,s')\, |\, s\in \overline{F}\cdot s_0,\,
%s'\in U^c\}$. Since $\overline{F}\cdot s_0$ is compact and $U$ is open, we have $\varepsilon>0$.
%Moreover, the distance $d_S$ is $\G$--invariant, so if $B_{d_S}(s_0,\varepsilon/2)$ is the open ball
%of radius $\varepsilon/2$ centered at $s_0$, we have $g(B_{d_S}(s_0,\varepsilon/2))\subseteq U$
%for every $g\in F$. This implies that $|\varphi (g\cdot s)|\leq M$ for every $g\in F$, $s\in 
%B_{d_S}(s_0,\varepsilon/2)$, so $\tr'_i (\varphi)|_{B_{d_S}(s_0,\varepsilon/2)}$ is bounded,
%and $\tr'_i (\varphi)$ is locally bounded.

Now, if $\varphi$ is $G$--invariant then by the very definition we have $\tr^i (\varphi)=\varphi$, so in order to conclude it is sufficient to show that
$\tr^i (\varphi)$ is $G$--invariant if $\varphi$ is $\G$--invariant.
So, let us fix $g_0\in G$.
% and let $F\subseteq \G$ be a Borel set of representatives for the 
%action of $\G$ on $\G$ by left translations. Since $\\Gamma$ is a cocompact lattice
%of $\G$, we may assume that the family $\{\gamma\cdot F\}_{\gamma\in\\Gamma}$ is locally finite in $\G$, and the closure $\overline{F}$ of $F$ in $\G$ is compact.
Since the family $\{\gamma\cdot F\}_{\gamma\in \G}$ is locally finite in $G$
and the right multiplication by $g_0$ induces
a homeomorphism of $G$, also 
the family $\{\gamma\cdot F\cdot g_0\}_{\gamma\in \G}$ is locally finite in $G$. 
Together with the compactness of
$\overline{F}$, this readily implies that a finite number of distinct
elements $\gamma_1,\ldots,\gamma_r\in \G$ exist such that $F$ admits the finite partition
$$
F=\bigsqcup_{i=1}^r F_i ,
$$
where $F_i=F\cap \bigl(\gamma_i^{-1}\cdot F\cdot g_0\bigr)$ for every $i=1,\ldots,r$.
As a consequence we also have
$$
F\cdot g_0=\bigsqcup_{i=1}^r \gamma_i\cdot F_i .
$$
Let us now fix an element $\varphi\in { C_c^q} (\widetilde{M})^\G$ and
a simplex $s\in { S_q}(\widetilde{M})$.
Using the $\G$--invariance of $\varphi$ and the (left) $G$--invariance of
$d\mu_G$, for every $i=1,\ldots, r$ we obtain 
\begin{equation}\label{inveq}
\int_{F_i} \varphi (g\cdot s)\, d\mu_G (g) = \int_{F_i} \varphi (\gamma_i g\cdot s)\, 
d\mu_G (g) =\int_{\gamma_i F_i} \varphi (g\cdot s)\, d\mu_G (g) . 
\end{equation}
Therefore we get
%$$
\begin{align*}%{lll}
%\begin{array}{lllll}
 \tr^q (\varphi) (g_0\cdot s) &= 
 \int_F \varphi (gg_0\cdot s)\, d\mu_G (g) 
&\!\!\!\!\!\!\!\! &=  \int_{F\cdot g_0} \varphi (g\cdot s)\, d\mu_G (g) \\
 &= \sum_{i=1}^r \int_{\gamma_i\cdot F_i} \varphi (g\cdot s)\, d\mu_G(g)
&\!\!\!\!\!\!\!\! &=  \sum_{i=1}^r \int_{F_i} \varphi (g\cdot s)\, d\mu_G(g) \\
 &= \int_F \varphi (g\cdot s)\, d\mu_G (g) 
&\!\!\!\!\!\!\!\! &= \tr^q (\varphi) (s) , 
\end{align*}
%\end{array}
%$$
where the second equality is due to the (right) $G$--invariance of $d\mu_G$, and the fourth equality
is due to equation~\eqref{inveq}. We have thus shown that, if $\varphi$ is $\G$--invariant,
then $\tr^q (\varphi)$ is $G$--invariant,
whence the conclusion.
\qed\smallskip 

Proposition~\ref{funziona:prop} provides a well-defined map 
$\tr^\bullet\colon  { \ccst} (\xtil)^\G\to { \ccst} (\xtil)^G$. It is readily seen
that $\tr^\bullet$ is a chain map. With an abuse, we still denote by
$\tr^\bullet$ the map $H^\bullet(\tr^\bullet)\colon H^\bullet ({ \ccst} (\xtil)^\G)\to H^\bullet ({ \ccst} (\xtil)^G)$ induced in cohomology.
%be the map induced by $\tr'_\bullet$ in cohomology. 
Since $\tr^\bullet$ restricts to the 
identity on $G$--invariant cochains, we have the following commutative diagram:

$$
\xymatrix{
H^\bullet ({ \ccst} (\xtil)^G) \ar[r]_{\res^\bullet} \ar@/^1.5pc/[rr]^{\rm Id} &  
H^\bullet ({ \ccst} (\xtil)^\G)
\ar[r]_{\tr^\bullet} \ar@{<->}[d] &  H^\bullet ({ \ccst} (\xtil)^G)\\
& { \hcst} (M) &
}
$$
where the vertical arrow describes the isomorphism provided by Lemma~\ref{sollevo:lemma}.
Since $\tr^\bullet$ is obviously norm non-increasing, we get the following:

\begin{prop}\label{last:prop}
The map $\res^\bullet \colon 
H^\bullet ({ \ccst} (\xtil)^G)\to H^\bullet ({ \ccst} (\xtil)^\G)$ is an isometric embedding.
\qed
\end{prop}

\section{Straightening and the volume form}
We are now going to describe an explicit representative of the volume coclass on $M$. 
Since differential forms may be integrated only on \emph{smooth} simplices,
we first need to replace generic singular simplices with smooth ones. In general, this
can be done by means of well-known smoothing operators (see e.g.~\cite[Theorem 16.6]{Lee}).
However, under the assumption that $M$ is non-positively curved we may as well exploit the straightening
procedure described above.

Recall that a singular $k$-simplex $s$ with values in a smooth manifold is smooth
if, for every $q\in \Delta^k$, the map $s$ may be smoothly extended
to an open neighbourhood
of $q$ in the $k$-dimensional affine subspace of $\mathbb{R}^\mathbb{N}$ containing $\Delta^k$.
The space ${_s S_k(\widetilde{M})}$ (resp.~${_s S_k({M})}$) of smooth simplices
with values in $\widetilde M$ (resp.~in $M$) may be naturally endowed with the $C^1$-topology.

\begin{lemma}\label{C1}
For every $(k+1)$-tuple $(x_0,\ldots,x_k)$, the singular simplex
$[x_0,\ldots,x_k]$ is smooth. Moreover, if we endow
$\widetilde{M}^{k+1}$ with the product topology and ${_s S_k(\widetilde{M})}$
with the $C^1$-topology, then the map
$$
\widetilde{M}^{k+1}\to {_s S_k(\widetilde{M})}\, ,\qquad 
(x_0,\ldots,x_k)\mapsto [x_0,\ldots,x_k]
$$
is continuous.
\end{lemma}
\begin{proof}
We have already mentioned the fact that, 
since $\xtil$ is non-positively curved, for every $x\in \xtil$ the exponential
map $\exp_x \colon T_x\xtil\to \xtil$ is a diffeomorphism.
Moreover, if $z\in\Delta^{k-1}$ and $q=tz+(1-t)e_k\in\Delta^k$, then by definition
we have
$$
[x_0,\ldots,x_k](q)=\exp_{x_k}\left(t\exp^{-1}_{x_k}([x_0,\ldots,x_{k-1}](z))\right)\ .
$$
Now the conclusion follows by an easy inductive argument.
\end{proof}

Let us now concentrate on some important homological properties of the straightening.
We define  
$\strtil_k \colon C_k (\widetilde{M})\to
C_k (\widetilde{M})$ as the unique linear map such that 
for every $s\in S_k (\widetilde{M})$
$$
\strtil_k (s)=[s(e_0),\ldots,s(e_k)]\in S_k(\widetilde{M})\ .
$$

\begin{prop}\label{straight:prop}
The map $\strtil_\bullet\colon C_\bullet (\widetilde{M})\to C_\bullet (\widetilde{M})$ satisfies the following properties:
\begin{enumerate}
\item
$d_{k+1}\circ \strtil_{k+1} = \strtil_{k}\circ d_{k+1}$ for every $n\in\matN$;
\item
$\strtil_k (\gamma\circ s)=\gamma\circ \strtil_k(s)$
for every $k\in\matN$, $\gamma\in\Gamma$, $s\in S_k (\widetilde{M})$;
\item
the chain map $\str_\bullet\colon C_\bullet (\widetilde{M})\to  
C_\bullet (\widetilde{M})$
is $\Gamma$-equivariantly homotopic to the identity, via a homotopy
which takes
any smooth simplex into the sum of a finite
number of smooth simplices.
\end{enumerate}
\end{prop}
\begin{proof}
If $x_0,\ldots,x_k\in\widetilde{M}$, then 
it is easily seen that for every $i\leq k$ the $i$-th face of
$[x_0,\ldots,x_k]$ is given by
$[x_0,\ldots,\widehat{x}_i,\ldots,x_k]$; moreover since isometries preserve geodesics we
have $\gamma\circ [x_0,\ldots,x_k]=[\gamma (x_0),\ldots, \gamma (x_k)]$
for every $\gamma\in {\rm Isom} (\widetilde{M})$.
These facts readily imply points~(1) and (2) of the proposition.

Finally, for ${s}\in S_k (\widetilde{M})$, 
let $F_{s}\colon \Delta^k\times [0,1]\to \widetilde{M}$
be defined by $F_{s}(x,t)=\beta_x (t)$, 
where $\beta_x\colon [0,1]\to \widetilde{M}$ is the constant-speed parameterization of
the geodesic segment joining ${s} (x)$ with 
$\strtil({s}) (x)$. We set $T_k({s})=(F_{s})_\bullet (c)$,
where $c$ is the standard chain triangulating the prism $\Delta^k\times [0,1]$ by $(k+1)$-simplices.
The fact that $d_{k+1}T_k + T_{k-1}d_k={\rm Id}-\strtil_k$ is now easily checked,
while the $\Gamma$-equivariance of $T_\bullet$ is a consequence of the fact that geodesics are preserved by isometries. 
The final statement of the proposition easily follows from the definition of $T_\bullet$.
\end{proof}

As a consequence of the previous proposition, the chain map $\strtil_\bullet$ induces a chain map
$$
\str_\bullet\colon C_\bullet(M)\to C_\bullet (M)
$$
which is homotopic to the identity. We say that a simplex $s\in S_k(M)$ is \emph{straight} if it is equal
to $\str_k(s')$ for some $s'\in S_k(M)$, i.e.~if it may be obtained by composing a straight
simplex in $\xtil$ with the covering projection $p$.
By Lemma~\ref{C1}, any straight simplex in $M$ is smooth, and 
the map $\str_k\colon S_k(M)\to {_s S_k(M)}$ is continuous,
if we endow $S_k(M)$ (resp.~${_s S_k(M)}$) with the compact-open topology
(resp.~with the $C^1$-topology).

We are now ready to define the volume cocycle. Let $n$ be the dimension of $M$, and suppose that $M$ is oriented.
We define a 
map ${\rm Vol}_M\colon { S_n} (M)\to \R$ by setting
$$
{\rm Vol}_M (s)=\int_{\str_n(s)} \omega_M,
$$
where $\omega_M\in\Omega^n (M)$ is the 
volume form of $M$. Since straight simplices are smooth, this map is well defined. Moreover, 
since integration is continuous with respect to the $C^1$-topology,
the linear 
extension of ${\rm Vol}_M$,  which will still be denoted by ${\rm Vol}_M$,
defines an element in ${ C^n_c} (M)\subseteq C^n(M)$.
Using Stokes' Theorem and the fact that $\str_\bullet$ is a chain map, one may easily
check that the cochain ${\rm Vol}_M$ is a cocycle, and defines therefore elements
$[{\rm Vol}_M]\in { H^n} (M)$,
$[{\rm Vol}_M]_c\in { H^n_c} (M)$. Recall that $[M]^*\in H^n(M)$ denotes
the fundamental coclass of $M$.

\begin{lemma}\label{contvol:lemma}
We have $[{\rm Vol}_M]={\rm Vol} (M)\cdot [M]^*$.
\end{lemma}
\begin{proof}
Since $H^n (M)\cong \R$, we have 
$[{\rm Vol}_M]=\langle [{\rm Vol}_M],
[M]_\R\rangle \cdot [M]^\R$. Moreover, it is well-known that 
the fundamental class of $M$ can be represented by the sum 
$c=\sum s_i$ 
of the simplices in a
positively oriented smooth
triangulation of $M$. By Proposition~\ref{straight:prop}, the difference 
$c-\str_n(c)$ is the boundary of a smooth $(n+1)$-chain, so Stokes' Theorem implies
that
$$
\vol (M)=\int_c \omega_M=\int_{\str_n(c)} \omega_M = \vol_M(c)=\langle [{\rm Vol}_M],
[M]_\R\rangle\ .
$$
\end{proof}

Observe now that, under the identification $C^n_c(M)\cong C^n_c(\xtil)^\G$, the
volume cocycle $\vol_M$ corresponds to the cocycle
$$
\vol_{\xtil}\colon C_n(\widetilde{M})\to \R\, \qquad \vol_{\xtil} (s)=\int_{\strtil_n(s)} \omega_{\xtil}\ ,
$$
where $\omega_{\xtil}$ is the volume form of the universal covering $\xtil$. Of course, $\vol_{\xtil}$
is $G$-invariant, so it defines an element $[\vol_{\xtil}]_c^G$ in $H^n(C^\bullet_c(\xtil)^G)$
such that 
$\res^n([\vol_{\xtil}]^G)=[\vol_M]_c\in H^n_c(M)$ (recall that we have an isometric identification
$H^n_c(M)\cong H^n(C^\bullet_c(\xtil)^\G)$). 

\section{Proof of the proportionality principle}
We are now ready to prove the main result of this chapter:

\begin{thm}[Gromov Proportionality Theorem]\label{prop:thm:proof}
Let $M$ be a closed non-positively curved Riemannian manifold. Then 
$$
\|M\|=\frac{\vol (M)}{\| [\vol_{\xtil}]_c^G\|_\infty}
$$
(where we understand that $k/\infty=0$ for every real number $k$).
In particular, the proportionality constant between the simplicial volume and the Riemannian volume
of $M$ only depends on the isometry type of the universal covering of $M$.
\end{thm}
\begin{proof}
 The maps $\res^n\colon H^n(C^\bullet_c(\xtil)^G)\to H^n_c(M)$ and
 $H^n(i^n)\colon H^n_c(M)\to H^n(M)$ are isometric embeddings, so putting together the
 duality principle (see Proposition~\ref{prop:duality}) and Lemma~\ref{contvol:lemma} we get
 $$
 \|M\|=\frac{1}{\|[M]^*\|_\infty}=\frac{\vol(M)}{\|[\vol_M]\|_\infty}=\frac{\vol(M)}{\|H^n(i^n)(\res^n([\vol_{\xtil}]_c^G))\|_\infty}=
 \frac{\vol(M)}{\|[\vol_{\xtil}]_c^G\|_\infty}\ .
$$
 \end{proof}

\section{The simplicial volume of hyperbolic manifolds}\label{hyp:simpl:sec}
We now concentrate our attention on the simplicial volume of closed hyperbolic manifolds. Thurston's original
proof of Theorem~\ref{hyp:thm} is based on the fact that singular homology is isometrically isomorphic
to \emph{measure homology} (as already mentioned, the first detailed proof of this fact appeared in~\cite{Loh}). 
Another proof of Gromov-Thurston's Theorem, which
closely follows Thurston's approach but avoids the use of measure homology, may be found in~\cite{BePe}
(see also~\cite{Rat}). Here we present (a small variation of) a proof due to Bucher~\cite{Bucher}, which exploits
the description of the proportionality constant between the simplicial volume and the Riemannian volume described
in Theorem~\ref{prop:thm:proof}.

Let $M$ be a closed oriented hyperbolic $n$-manifold. The universal covering of $M$ is isometric
to the $n$-dimensional hyperbolic space $\matH^n$. As usual, we identify the fundamental group
of $M$ with the group $\G$ of the automorphisms of the covering $p\colon \matH^n\to M$, and we denote
by $G$ the group of orientation-preserving isometries of $\matH^n$. By Theorem~\ref{prop:thm:proof}, we are
left to compute the $\ell^\infty$-seminorm of the element
$[\vol_{\matH^n}]_c^G$ of
$H^n(C^\bullet_c(\xtil)^G)$. We denote by $\vol_{\matH^n}$ the representative of $[\vol_{\matH^n}]_c^G$ described
in the previous section, i.e.~the cocycle such that
$$
\vol_{\matH^n}(s)=\int_{\str_n(s)} \omega_{\matH^n}\ ,
$$
where $\omega_{\matH^n}$ is the hyperbolic volume form. Of course, in order to estimate the seminorm of the volume coclass
we first need to understand the geometry of hyperbolic straight simplices.

\section{Hyperbolic straight simplices}
We denote by $\mathno=\matH^n\cup \partial \matH^n$ the natural compactification of hyperbolic space, obtained by adding
to $\matH^n$ one point for each class of asymptotic geodesic rays in $\matH^n$.
We recall that every pair of points $\mathno$ is connected by a unique geodesic segment (which 
has infinite length if any of its endpoints lies in $\partial\matH^n$). 
A subset in $\mathno$ is \emph{convex} if whenever it contains a pair of points it also contains the geodesic segment connecting them. 
The \emph{convex hull} of a set $A$ is defined as usual as the intersection of all convex sets containing $A$.

A \emph{(geodesic) $k$-simplex} $\Delta$ in $\mathno$ is the convex hull of $k+1$ points in $\mathno$, called \emph{vertices}. A 
geodesic $k$-simplex is
\emph{finite} if all its vertices lie in $\matH^n$,
\emph{ideal} if all its vertices lie in $\partial\matH^n$, and
\emph{regular} if every permutation of its vertices is induced by an isometry of $\matH^n$. Since $\matH^n$ has constant
curvature, finite geodesic simplices are exactly the images of straight simplices. Moreover,
for every $\ell>0$ there exists, up to isometry, exactly one finite regular geodesic $n$-simplex of edgelength
$\ell$, which will be denoted by $\tau_\ell$. 
In the same way, there exists, up to isometry, exactly one regular ideal geodesic $n$-simplex, which will be denoted
by $\tau_\infty$.

It is well-known that the supremum $v_n$ of the volumes of geodesic $n$-simplices is finite for every $n\geq 2$. However,
computing $v_n$ and describing those simplices whose volume is exactly equal to $v_n$ is definitely non-trivial.
It is well-known that the area of every finite triangle is strictly smaller than $\pi$, while every ideal triangle
is regular and has area equal to $\pi$, so $v_2=\pi$. In dimension 3, Milnor proved that $v_3$ is equal to the volume
of the regular ideal simplex, and that any simplex having volume equal to $v_3$ is regular and ideal~\cite[Chapter 7]{Thurston}. 
This result was then extended to any dimension by Haagerup and Munkholm~\cite{HM} (see also~\cite{Pe} for a beautiful alternative
proof based on Steiner symmetrization):

\begin{teo}[\cite{HM, Pe}] \label{maximal:teo}
Let $\Delta$ be a geodesic $n$-simplex in $\mathno$. Then $\vol(\Delta)\leqslant v_n$, 
and $\vol(\Delta) = v_n$ if and only if $\Delta$ is ideal and regular.
\end{teo}

Moreover, the volume of geodesic simplices is reasonably continuous with respect
to the position of the vertices (see e.g.~\cite[Proposition 4.1]{Luo} or~\cite[Theorem 11.4.2]{Rat}), so
we have
\begin{equation}\label{taul}
\lim_{\ell \to \infty} \vol(\tau_\ell)=\vol(\tau_\infty)=v_n
\end{equation}
(see e.g.~\cite[Lemma 3.6]{FP} for full details).

\section{The seminorm of the volume form}\label{comput:hyp:sec}
Let us now come back to the study of the seminorm of
$[\vol_{\matH^n}]_c^G\in 
H^n(C^\bullet_c(\xtil)^G)$. 
As a corollary of Theorem~\ref{maximal:teo}, we immediately obtain the inequality
\begin{equation}\label{easyup}
\|[\vol_{\matH^n}]_c^G\|_\infty \leq \|\vol_{\matH^n}\|_\infty=v_n\ .
\end{equation}
We will show that this inequality is in fact an equality.
To do so, we need to compute
$$
\inf \left\{ \|\vol_{\matH^n}+\delta\varphi\|_\infty\, ,\ {\varphi\in C^{n-1}_c(\matH^n)^G}\right\}\ .
$$
Let us fix $\varphi\in C^{n-1}_c(\matH^n)^G$. Since an affine automorphism of the standard $n$-simplex
preserves the orientation if and only if the induced permutation of the vertices is even, 
the cochain $\vol_{\matH^n}$ is alternating. 
We thus have
$$
  \|\vol_{\matH^n}+\delta\alt^{n-1}(\varphi)\|_\infty =\|\alt^n(\vol_{\matH^n}+\delta\varphi)\|_\infty\leq
   \|\vol_{\matH^n}+\delta\varphi\|_\infty
$$
(recall from Section~\ref{altern:sec} that alternation provides a norm non-increasing chain map), so
we may safely assume that $\varphi$ is alternating. 

Let us denote by $s_\ell$ an orientation-preserving barycentric parameterization of
the geodesic simplex $\tau_\ell$, $\ell>0$ (see e.g.~\cite{FFM} for the definition of barycentric
parameterization of a geodesic simplex). Also let $\partial_i s_\ell$ be the $i$-th face of $s_\ell$, 
let $H$ be the hyperplane of $\matH^n$ containing $\partial_i s_\ell$,
and let
$v,w\in\matH^n$ be two vertices of $\partial_i s_\ell$. Let $h$ be the isometry of $H$
given by the reflection with respect to the $(n-2)$-plane containing all the vertices
of $\partial_i s_\ell$ which are distinct from $v,w$, and 
the midpoint of the edge joining $v$ to $w$. Then $h$ may be extended in a unique way to an orientation-preserving
isometry $g$ of $\matH^n$ (such an isometry interchanges the half-spaces of $\matH^n$ bounded by $H$).
%We have thus realized a transposition of the vertices of $\partial_i s_\ell$ via an orientation-
Since $\partial_i s_\ell$ is the barycentric
parameterization of a regular $(n-1)$-simplex, this readily implies that
$$
g\circ \partial_i s_\ell= \partial_i s_\ell\circ \sigma\ ,
$$
where $\sigma$ is an affine automorphism of the standard $(n-1)$-simplex which induces
an odd permutation of the vertices. But $\varphi$ is both alternating and $G$-invariant, so we have
$\varphi(\partial_i s_\ell)=0$, whence
$$
\delta\varphi (s_\ell)=0\ .
$$
As a consequence, for every $\ell>0$ we have
$$
\|\vol_{\matH^n}+\delta\varphi\|_\infty \geq
|(\vol_{\matH^n}+\delta\varphi)(s_\ell)|=|\vol_{\matH^n} (s_\ell)|=\vol (\tau_\ell)\ .
$$
Taking the limit of the right-hand side as $\ell\to\infty$ and using~\eqref{taul} we get
$$
\|\vol_{\matH^n}+\delta\varphi\|_\infty\geq v_n\ ,
$$
whence 
$$
\|[\vol_{\matH^n}]_c^G\|_\infty\geq v_n
$$
by the arbitrariety of $\varphi$. 
Putting this estimate together with inequality~\eqref{easyup} we finally obtain the following:

\begin{thm}[Gromov and Thurston]
 We have $\|[\vol_{\matH^n}]_c^G\|_\infty= v_n$, so
 $$
 \|M\|=\frac{\vol (M)}{v_n}
 $$
 for every closed hyperbolic manifold $M$.
\end{thm}

\section{The case of surfaces}\label{surface:simplicial:sec}
Let us denote by $\Sigma_g$ the closed oriented surface of genus $g$.
Recall that, if $g\geq 2$, then 
$\Sigma_g$ supports a hyperbolic metric, hence
by Gauss-Bonnet Theorem we have
$$
\|\Sigma_g\|=\frac{{\rm Area}(\Sigma_g)}{v_2}=\frac{2\pi |\chi(\Sigma_g)|}{\pi}=2|\chi(\sigma_g)|=4g-4\ .
$$
According to the
proof of this equality given in the previous section, we may observe
that the lower bound $\|\Sigma_g\|\geq 4g-4$ follows from 
the easy upper bound $\|\vol_{\matH^2}\|\leq v_2=\pi$, while the upper bound
$\|\Sigma_g\|\leq 4g-4$ follows from 
the much less immediate lower bound $\|[\vol_{\matH^2}]^G_c\|\geq v_2=\pi$. However, in the case of surfaces the inequality
$\|\Sigma_g\|\leq 4g-4$ may be deduced via a much more direct argument involving triangulations.

In fact, being obtained by gluing the sides of a $4g$-agon, the surface $\Sigma_g$ admits
a triangulation by $4g-2$ triangles
(we employ here the word ``triangulation'' in a loose sense, as is customary in geometric topology: 
a triangulation of a manifold $M$ is the realization of $M$ as the gluing of finitely many simplices via some simplicial pairing of their facets).
Such a triangulation defines an (integral) fundamental cycle whose $\ell^1$-norm is bounded above by $4g-2$, so
$\|\Sigma_g\|\leq 2|\chi(\Sigma_g)|+2$. Indeed, if we denote by $\|\cdot \|_\mathbb{Z}$ the \emph{integral} simplicial volume (see Section~\ref{further:simplicial:readings}),
then 
$$
\|\Sigma_g\|_\mathbb{Z}\leq 2|\chi(\Sigma_g)|+2\ .
$$
In order to get the desired estimate, we now need to slightly improve this inequality. To do so we exploit
a standard trick, which is based on the fact that the simplicial volume is multiplicative with respect to finite coverings. 

Since $\pi_1(\Sigma_g)$ surjects onto $H_1(\Sigma_g,\mathbb{Z})\cong \matZ^{2g}$,
for every $d\in\mathbb{N}$ we may construct a subgroup $\pi_1(\Sigma_g)$ of index $d$ (just take the kernel
of an epimorphism of $\pi_1(\Sigma_g)$ onto $\matZ_d$). Therefore, $\Sigma_g$ admits a $d$-sheeted covering,
whose total space is a surface $\Sigma_{g'}$ such that $|\chi(\Sigma_{g'})|=d|\chi(\Sigma_g)|$ (recall that
the Euler characteristic is multiplicative with respect to finite coverings). 
By the multiplicativity of the simplicial volume we now get 
$$
\|\Sigma_g\|=\frac{\|\Sigma_{g'}\|}{d}\leq \frac{\|\Sigma_{g'}\|_\mathbb{Z}}{d}\leq \frac{ 2|\chi(\Sigma_{g'})|+2}{d}=2|\chi(\Sigma_g)|+\frac{2}{d}\ .
$$
Since $d$ is arbitrary, we may conclude that $\|\Sigma_g\|\leq 2|\chi(\Sigma_g)|$, which gives the desired upper bound. 

In fact, this argument shows that the \emph{stable integral simplicial volume} $\stisv{\Sigma_g}$ of $\Sigma_g$ is equal to the classical simplicial volume $\|\Sigma_g\|$
(see Section~\ref{further:simplicial:readings} for the definition of $\stisv{\cdot}$). 
One may wonder whether this equality holds more in general (e.g. for manifolds admitting many finite coverings, such as aspherical manifolds
with residually finite fundamental groups). However, stable integral simplicial volume is strictly bigger than classical simplicial volume already
for closed hyperbolic manifolds of dimension greater than 3. 
We refer the reader to Section~\ref{further:simplicial:readings} for more details about this topic.

\section{The simplicial volume of negatively curved manifolds}\label{inoue}
Let us now  discuss the more general case when $M$ is a closed $n$-manifold endowed
with a Riemannian metric with sectional curvature $\leq -\varepsilon$, where $\varepsilon$
is a positive constant.
Then it is not difficult to show that there exists a finite constant $\kappa(n,\varepsilon)$
(which depends only on the dimension and on the upper bound on the sectional curvature)
which bounds from above
the volume of any straight simplex in $\widetilde{M}$ (see e.g.~\cite{Inoue}). If $G$ is the group of the orientation-preserving
isometries of $\xtil$, then
$$\|[\vol_{\xtil}]_c^G\|_\infty\leq \| \vol_{\xtil}\|_\infty\leq \kappa (n,\varepsilon)\ ,$$ 
so
$$
\|M\|= \frac {\vol(M)}{\|[\vol_{\xtil}]_c^G\|_\infty}=\frac{\vol (M)}{\kappa(n,\varepsilon)}>0\ .
$$
In particular, any closed manifold supporting a negatively curved Riemannian metric
has non-vanishing simplicial volume.

\section{The simplicial volume of flat manifolds}\label{flat:simplicial}
We have already mentioned the fact that the simplicial volume of
closed manifolds with non-negative Ricci tensor is null~\cite{Gromov}. 
In particular, if $M$ is a closed flat (i.e.~locally isometric to $\mathbb{R}^n$) manifold, then $\|M\|=0$.
Let us give a proof of this fact which is based on the techniques developed above.

Since $M$ is flat, the universal covering of $M$ is isometric to the Euclidean space $\R^n$.
If we denote by $G$ the group of the orientation-preserving isometries of $\R^n$, then
we are left to show that
\begin{equation}\label{infinity}
\|[\vol_{\R^n}]_c^G\|_\infty=\infty\ .
\end{equation}
Let $\tau_\ell$ be the Euclidean geodesic simplex of edgelength $\ell$, and observe
that $\lim_{\ell\to\infty} \vol(\tau_\ell)=\infty$.
The computations carried out in Section~\ref{comput:hyp:sec} in the hyperbolic case
apply \emph{verbatim} to the Euclidean case, thus showing that
$\|[\vol_{\R^n}]_c^G\|_\infty\geq \vol(\tau_\ell)$ for every $\ell>0$.
By taking the limit as $\ell$ tends to infinity we obtain~\eqref{infinity},
which in turn implies $\|M\|=0$.

More directly, after fixing a flat structure on the $n$-dimensional torus $(S^1)^n$, by Gromov proportionality principle 
we have $$\frac{\|M\|}{\vol(M)}=\frac{\|(S^1)^n\|}{\vol((S^1)^n)}=0\ ,$$ whence $\|M\|=0$.

Another proof of the vanishing of the simplicial volume of flat manifolds follows from Bieberbach Theorem:
if $M$ is closed and flat, then it is finitely covered by an $n$-torus, so $\|M\|=0$
by the multiplicativity of the simplicial volume with respect to finite coverings.

Alternatively, one could also observe that Bieberbach Theorem implies that 
$\pi_1(M)$ is amenable, so again $\|M\|=0$ by Corollary~\ref{amvan}.

\section{Further readings}\label{further:prop:sec}

\subsection*{Measure homology and other approaches to the proportionality principle}
The first complete proof of the proportionality principle is due to Clara L\"oh~\cite{Loh}. Her proof is based on the use of \emph{measure homology},
a homology theory introduced by Thurston in~\cite{Thurston} whose flexibility may be exploited to compare fundamental cycles on manifolds that only share the (metric) universal covering. 
A key step in the proof (and the main result of~\cite{Loh}) is the fact that, for reasonable spaces, measure homology is isometrically isomorphic to
singular homology (the isomorphism was previously shown to hold by Zastrow~\cite{Zastrow} and Hansen~\cite{Hansen}). The proof of this fact heavily relies on bounded cohomology, making use of results
by Monod on resolutions by \emph{continuous} cochains~\cite{Monod}. 
The purely cohomological proof described in~\cite{Frigerio}, being based on the fact that continuous cohomology is isometrically isomorphic to singular cohomology,
may be interpreted as the dual argument to L\"oh's. It is maybe worth mentioning that, as a byproduct of his results on the Lispchitz simplicial volume, Franceschini
produced in~\cite{Franceschini1} a new proof of Gromov's proportionality principle for compact Riemannian manifolds; as far as the author knows, Franceschini's argument provides the only proof which
avoids using the deep fact that the singular bounded cohomology of a space is isometrically isomorphic to the bounded cohomology of its fundamental group, even for non-aspherical spaces
(see Theorem~\ref{gro-iva:thm}).

\subsection*{The proportionality principle for symmetric spaces}
Of course, Gromov's proportionality principle is informative only for Riemannian manifolds with a non-discrete isometry group: indeed, if the isometry group of a closed manifold $M$
is discrete, then any closed manifold sharing the same universal covering with $M$ is in fact commensurable to $M$. Therefore, in this context the most interesting class of manifolds is probably given
by locally symmetric spaces. In fact, the study of the volume form on the universal covering has proved to be a very powerful tool to show non-vanishing of the simplicial volume
for many locally symmetric spaces. By exploiting the barycentric construction introduced by Besson, Courtois and Gallot in~\cite{BCG}, Lafont and Schmidt proved in~\cite{Lafont-Schmidt}
that the simplicial volume of a closed, locally symmetric space of non-compact type is positive, thus answering a question by Gromov~\cite{Gromov} (see also~\cite{michelleprimo} for a case not covered by
Lafont and Schmidt's argument).

Let us stress again that the proportionality constant between the Riemannian and the simplicial volume (in the case when the latter does not vanish) is known only
for manifolds covered by the hyperbolic space, or by the product of two hyperbolic planes. Of course, it would be of great interest to compute such constant in other
locally symmetric cases, e.g.~for complex hyperbolic manifolds.

\subsection*{Dupont's conjecture}
We have seen in this chapter that continuous cohomology can play an important role when we one considers non-discrete groups  acting on geometric objects (e.g.~isometry groups of symmetric spaces).
This provides a noticeable instance of the fact that  \emph{continuous} (bounded) cohomology of a Lie group $G$ is very useful when investigating the (bounded) cohomology
of discrete subgroups of $G$ (see~\cite{Monod, BM1,BM2} for several results in this spirit). 

For example, in order to prove the positivity of the simplicial volume of a closed $n$-dimensional locally symmetric space $M$ of non-compact type, 
Lafont and Schmidt 
proved that the comparison map $c^n\colon H^n_{b,c}(G)\to H^n_c(G)$ between the continuous bounded cohomology of $G$ and the continuous cohomology of $G$ is surjective,
where $G$ is the  Lie group associated to $\widetilde{M}$. 
More in general, it is a question of Dupont~\cite{Dupont} whether the comparison map $c^n\colon H^n_{b,c}(G)\to H^n_c(G)$ 
just mentioned is surjective,
whenever $G$ is 
a semisimple Lie group without compact factors, and $n\geq 2$ (see also~\cite[Problem A']{Monod:inv} and~\cite[Conjecture 18.1]{BIMW}). 
Dupont's conjecture has been attacked by a variety of methods, which may enlighten how bounded cohomology is very geometric in nature. 
In degree 2, Dupont's conjecture is known to hold thanks to the work by Domic and Toledo~\cite{DT2} and Clerc and {\O}rsted~\cite{CO}.
The above mentioned result by Lafont and Schmidt implies
surjectivity in top degree. More recently, Hartnick and Ott~\cite{OH1} settled Dupont's conjecture  for several specific classes of Lie groups, including Lie groups of Hermitian type, while Lafont and 
Wang~\cite{LafWang} proved that the comparison map $c^k\colon H^k_{b,c}(G)\to H^k_c(G)$ is surjective provided that $G\neq SL(3,\R),SL(4,\R)$, and that $k\geq n-r+2$. where $n,r$ are the dimension and the rank of the symmetric space
associated to $G$. 

One may even investigate in which cases the comparison map $c^n\colon H^n_{b,c}(G)\to H^n_c(G)$ is an isomorphism. For a recent result related to this sort of questions we refer the reader
to~\cite{OH2}, where it is shown that $H^4_{c,b}(G,\R)=0$ for every Lie group $G$ locally isomorphic to $SL(2,\R)$.

\chapter{Additivity of the simplicial volume}\label{additivity:chap}
This chapter is devoted to the proof of Theorem~\ref{simpl:thm}. Our approach makes an essential use of the duality
between singular homology and bounded cohomology. In fact, the additivity of the simplicial volume for gluings along boundary components with 
amenable fundamental group will be deduced from a suitable ``dual'' statement about bounded cohomology (see Theorem~\ref{bounded:graph:thm}). It is maybe
worth mentioning that, as far as the author knows, there exists no proof of Gromov additivity theorem which avoids the use of bounded cohomology:
namely, no procedure is known which allows to split a fundamental cycle for a manifold into (relative) efficient fundamental cycles for the pieces
obtained by cutting along hypersurfaces with amenable fundamental groups.

Let us fix some notation.
Let $M_1,\ldots,M_k$ be oriented $n$-manifolds, $n\geq 2$, such that the fundamental group of every component
of $\partial M_j$ is amenable.
We fix a pairing $(S_1^+,S_1^-),\ldots, (S_h^+,S_h^-)$ 
of some boundary components of $\sqcup_{j=1}^k M_j$, and 
for every $i=1,\ldots,h$ we fix an orientation-reversing homeomorphism $f_i\colon S_i^+\to S_i^-$.
We denote by $M$ the oriented manifold obtained by gluing $M_1,\ldots,M_k$ along $f_1,\ldots,f_h$, and we suppose that $M$
is connected. We also denote by $i_j\colon M_j\to M$ the obvious quotient map (so $i_j$ is an embedding provided that no boundary component of $M_j$ is
paired with another boundary component of $M_j$). 

For every $i=1,\ldots,h$ we denote by $j^{\pm}(i)$ the index such that $S_i^\pm \subseteq M_{j^{\pm}(i)}$,
and by $K_i^\pm$ the kernel of the map $\pi_1(S_i^\pm)\to \pi_1(M_{j^{\pm}(i)})$ induced by the inclusion.
 We recall that the gluings $f_1,\ldots,f_h$ are
 \emph{compatible} if the equality
$$
(f_i)_\ast \left(K_i^+\right)=K_i^-
$$
holds 
for every $i=1,\ldots,h$. Gromov additivity theorem states that
\begin{equation}\label{subadditivity:eq}
\| M,\bb M\|\leq \|M_1,\bb M_1\|+\ldots+\|M_k,\bb M_k\|\ ,
\end{equation}
and that, if the gluings defining $M$ are compatible, then
\begin{equation}\label{additivity:eq}
\| M,\bb M\|= \|M_1,\bb M_1\|+\ldots+\|M_k,\bb M_k\|\ .
\end{equation}

\section{A cohomological proof of subadditivity}
%For every $j$ we consider the restriction ${i}_j\colon M_j\to M$ of the quotient map defining $M$
%(so $i_j$ is an embedding provided that no pair of boundary components of $M_j$ ar identified in $M$).
If $S_i$ is a component of $\partial M_j$, we denote by $\overline{S}_i\subseteq M$ the image of
$S_i$ via $i_j$, and we set
%and with a slight
%abuse we identify $M_j$ with $i_j(M_j)\subseteq M$.
$\mathcal{S}=\bigcup_{i=1}^h \overline{S}_i\subseteq M$. Then the map $i_j$
is also a map of pairs ${i}_j\colon (M_j,\bb M_j)\to (M,\mathcal{S}\cup \bb M)$, and we denote by
$$H_b^n(i_j)\colon H_b^n(M,\mathcal{S}\cup \bb M)\to H_b^n(M_j,\bb M_j)$$ the induced map in bounded cohomology.
For every compact manifold $N$, the pair $(N, \bb N)$ has the homotopy type of a finite CW-complex~\cite{RS}, so
the inclusions of relative cochains into absolute cochains
induce isometric isomorphisms $H_b^n (M,\mathcal{S}\cup \bb M)\cong H_b^n(M)$, $H_b^n (M,\bb M)\cong H_b^n(M)$
(see Theorem~\ref{reliso:thm}).
As a consequence, also the inclusion $C_b^n (M,\mathcal{S}\cup \bb M)\to C_b^n(M,\bb M)$
induces an isometric isomorphism  
$$\zeta^n\colon H_b^n (M,\mathcal{S}\cup \bb M)\to H_b^n(M,\bb M)\ .$$ 
%With an abuse, we will denote by the symbol  and the map
%$H_b^n(M)\to H_b^n(M_j)$ induced by $i_j$.
For every $j=1,\ldots,k$, we finally define the map $$\zeta^n_j\ =\ H_b^n(i_j)\circ (\zeta^n)^{-1}\ \colon\ H_b^n(M,\bb M)\to H_b^n(M_j,\bb M_j)\ .$$ 
%by setting $\zeta^n_j=H_b^n(i_j)\circ (\zeta^n)^{-1}$. 

\begin{lemma}\label{somma}
For every $\varphi\in H_b^n(M,\bb M)$ we have
$$
\langle \varphi,[M,\bb M]\rangle=\sum_{j=1}^k \langle \zeta^n_j(\varphi), [M_j,\bb M_j]\rangle\ .
$$
\end{lemma}
\begin{proof}
Let $c_j\in C_n(M_j)$ be a real chain representing the fundamental class of $M_j$.
With an abuse, we identify any chain in $M_j$ (resp.~in $S_i^\pm$) with the corresponding chain in
$M$ (resp.~in $\overline{S}_i$), and we 
set
$c=\sum_{j=1}^k c_j \in C_n(M)$.
We now suitably modify $c$ in order to obtain a relative fundamental cycle for
$M$. It is readily seen that $\bb c_j$ is the sum of real fundamental cycles of the boundary components
of $M_j$. Therefore, since the gluing maps defining $M$ are
orientation-reversing, 
we may choose a chain $c'\in \oplus_{i=1}^N C_{n}(\overline{S}_i)$ such that $\bb c-\bb c'\in C_{n-1}(\bb M)$.
We  set $c''=c-c'$. By construction $c''$ is a relative cycle
in $C_n (M,\bb M)$, and it is immediate to check 
that it is in fact a relative fundamental cycle for $M$.
% Moreover, if $x$ is any point in $M_j\setminus \bb M_j\subseteq M$, then
%$c_j$ defines the real fundamental class of $\h_n(M_i,M_i\setminus \{x\})$, so
%$c''$ defines the real fundamental class of $\h_n(M,M\setminus \{x\})$, whence
%of $\h_n (M,\bb M)$.  
Let now $\psi\in \cb^n(M,\mathcal{S}\cup \bb M)$ be a representative of $(\zeta^n)^{-1}(\varphi)$.
By definition we have
$$
\psi(c)=\sum_{j=1}^k \psi(c_j)=\sum_{j=1}^k \langle \zeta^n_j(\varphi), [M_j,\bb M_j]\rangle \ .
$$
On the other hand, since $\psi$ vanishes on chains supported on $\mathcal{S}$, we also have
$$
\psi(c)=\psi(c''+c')=\psi(c'')=\langle \varphi,[M,\bb M]\rangle\ ,
$$
and this concludes the proof.
\end{proof}

We are now ready to exploit duality to prove subadditivity of the simplicial volume with
respect to gluings along boundary components having an amenable fundamental group.
By Proposition~\ref{prop:duality}
we may choose an element $\varphi\in\ H_b^n(M,\bb M)$
such that 
$$
\|M,\partial M\|=\langle  \varphi,[M,\bb M]\rangle \, ,\qquad 
\|\varphi\|_\infty\leq 1\ .
$$
Observe that $\|\zeta^n_j(\varphi)\|_\infty\leq \|\varphi\|_\infty\leq 1$, so by Lemma~\ref{somma}
$$
\|M,\partial M\| 
=
\langle \varphi,[M,\bb M]\rangle  =\sum_{j=1}^k \langle \zeta^n_j(\varphi),[M_j,\bb M_j]\rangle  \leq 
{\sum_{j=1}^k \|M_j,\bb M_j\|}\ .
$$
This concludes the proof of inequality~\eqref{subadditivity:eq}.

\begin{rem}\label{final:rem}
The inequality
$$
\| M,\bb M\|\leq \| M_1,\bb M_1\|+\ldots+\|M_k,\bb M_k\|
$$
may also be proved by showing that ``small'' fundamental cycles for $M_1,\ldots,M_k$ may be glued together
to construct a ``small'' fundamental cycle for $M$.

In fact, let $\vare>0$ be given.
By Thurston's version of Gromov equivalence theorem (see Corollary~\ref{Thurston's version}),
for every $j=1,\ldots,k$ we may choose a real relative fundamental cycle $c_j\in C_n(M_j,\bb M_j)$
such that 
$$\|c_j\|_1\leq \|M_j,\bb M_j\|+\vare\ ,\qquad \|\bb c_j\|_1\leq \vare\ .$$ 
Let us  set $c=c_1+\ldots+c_k\in C_n(M)$ (as above, we identify any chain in $M_j$  with
its image in $M$). As observed in Lemma~\ref{somma},
the chain $\bb c_j$ is the sum of real fundamental cycles of the boundary components
of $M_j$. Since the gluing maps defining $M$ are
orientation-reversing, this implies that $\partial c=\partial b+ z$ for some $b\in C_{n}(\mathcal{S})$, $z\in C_{n-1}(\partial M)$.

We now use again that the fundamental group of each $\overline{S}_i$ is amenable in order to bound the $\ell^1$-norm of $b$. In fact,
the vanishing of $H^n_b(\mathcal{S})$ implies that the singular chain complex of $\mathcal{S}$ satisfies the $(n-1)$-uniform boundary condition
(see Definition~\ref{UBC} and Theorem~\ref{MM2}). Therefore, 
we may choose the chain $b$ in such a way that the inequalities
$$
\|b\|_1\leq \alpha \|\partial c\|_1\leq \alpha k \varepsilon
$$
hold for some universal constant $\alpha$. The chain $c'=c-b$ is now a relative fundamental cycle for $(M,\partial M)$, and we have
$$
\| M,\bb M\|  \leq \|c'\|_1\leq \|c\|_1+\|b\|_1\leq \sum_{j=1}^k \|M_j,\bb M_j\| + k\vare+\alpha k\vare\ .
$$
Since $\varepsilon$ was arbitrary, this concludes the proof.
\end{rem}

\section{A cohomological proof of Gromov additivity theorem}
The proof that equality~\eqref{additivity:eq}
is based on the following extension property for bounded coclasses:
%We refer the reader to Section~\ref{special:section} for the definition of the module $C^n_{bs}(X,Y)$ of special
%cochains, whenever every component of $Y$ has an amenable fundamental group.

\begin{thm}\label{bounded:graph:thm}
Suppose that the gluings defining $M$ are compatible, let $\varepsilon>0$ and take an element
$\varphi_j\in H^n_b(M_j,\bb M_j)$ for every $j=1,\ldots,k$. Then there exists a coclass
$\varphi\in H^n_b(M,\bb M)$ such that $\zeta^n_j(\varphi)=\varphi_j$ for every $j=1,\ldots,k$, and
%For every $j=1,\ldots,k$, let $\varphi_j\in Z^n_{bs}(M_j,\bb M_j)$ be a special cocycle.
%Then there exists a bounded cocycle $\varphi\in C^n_b(M,\calS\cup\bb M)$ such that 
%$C^n_b(i_j)(\varphi)=\varphi_j$ for every $j=1,\ldots, k$, and
$$
\|\varphi\|_\infty \leq \max \{\|\varphi_j\|,\, j=1,\ldots,k\}\ +\ \varepsilon\ .
$$
\end{thm}

Henceforth we assume that the gluings defining $M$ are compatible.
We first show how Theorem~\ref{bounded:graph:thm} may be used to conclude the proof of Gromov additivity theorem.
By Proposition~\ref{prop:duality},
for every $j=1,\ldots,k$, 
we may choose an element $\varphi_j\in H_b^n(M_j,\bb M_j)$
such that 
$$
\|M_j,\partial M_j\|= \langle \varphi_j,[M_j,\bb M_j]\rangle \, ,\qquad 
\|\varphi_j\|_\infty\leq 1\ .
$$
By Theorem~\ref{bounded:graph:thm}, for every $\varepsilon>0$ 
there exists $\varphi\in H_b^n(M,\bb M)$ such that
$$
\|\varphi\|_\infty\leq 1+\varepsilon\, ,\quad \zeta^n_j(\varphi)=\varphi_j\, ,\ j=1,\ldots,k\ .
$$
Using Lemma~\ref{somma} we get
%\begin{align*}
$$
 \sum_{j=1}^k \|M_j,\bb M_j\| =
 \sum_{j=1}^k \langle \varphi_j,[M_j,\bb M_j]\rangle=\langle \varphi, [M,\bb M]\rangle
\leq (1+\varepsilon)\cdot \|M,\bb M\|\ .
$$
Since $\varepsilon$ is arbitrary, this implies that $\|M,\bb M\|$ cannot be strictly smaller
than the sum of the $\|M_j,\bb M_j\|$. Together with inequality~\eqref{subadditivity:eq}, this implies
that
$$
\| M,\bb M\|= \|M_1,\bb M_1\|+\ldots+\|M_k,\bb M_k\|\ .
$$
Therefore, in order to conclude the proof of Gromov additivity theorem we are left to prove
Theorem~\ref{bounded:graph:thm}.

%\section{Proof of Theorem~\ref{bounded:graph:thm}}

%\section{The universal covering of $M$ as a tree of spaces}
We denote by $p\colon \widetilde{M}\to M$ the universal covering of $M$. We
first describe the structure of $\widetilde{M}$ as a \emph{tree of spaces}. 
We construct a tree $T$ as follows. 
Let us set $N_j=M_j\setminus \bb M_j$ for every $j=1,\ldots,k$.
We pick a vertex for every connected component
of $p^{-1}(N_j)$, $j=1,\ldots,k$, and we 
join two vertices if the closures of the corresponding
components intersect (along a common boundary component). Therefore,
edges of $T$ bijectively correspond to the  connected components
of $p^{-1}(\calS)$. 
It is easy to realize $T$ as a retract of $\widetilde{M}$, so $T$ is simply connected, i.e.~it is a tree.
We denote by $V(T)$ the set of vertices of $T$, and
for every $v\in V(T)$ we denote by $\widetilde{M}_v$ the closure of the component of $\widetilde{M}$
corresponding to $v$.

The following lemma exploits the assumption that the gluings defining $M$ are compatible.

\begin{lemma}
For  every $v\in V(T)$ there exist $j(v)\in\{1,\ldots,k\}$ and a universal covering map
$p_v\colon \widetilde{M}_v\to M_{j(v)}$ such that the following diagram commutes:
$$
\xymatrix{\widetilde{M}_v \ar[r]^{p_v}\ar[dr]_{p|_{\widetilde{M}_v}} & M_{j(v)}\ar[d]^{i_j}\\
& i_j(M_j)
}
$$
\end{lemma}
\begin{proof}
An application of Seifert-Van Kampen Theorem shows that, under the assumption that
the gluings defining $M$ are compatible, for every $j$ the map
$\pi_1(M_j)\to \pi_1(M)$ induced by $i_j$ is injective. This implies in turn that
$p$ restricts to a universal covering $\widetilde{M}_v\setminus \bb \widetilde{M}_v\to N_{j(v)}$.
It is now easy to check that this restriction extends to the required map
$p_v\colon\widetilde{M}_v\to M_{v(j)}$.
\end{proof}

Let us now proceed with the proof of Theorem~\ref{bounded:graph:thm}.
For every $j=1,\ldots,k$, we are given a coclass $\varphi_j\in H^n_b(M_j,\bb M_j)$.
Recall that, since the fundamental group of every component of $\bb M_j$ is amenable, the inclusion of relative cochains into absolute ones
induces an isometric isomorphism in bounded cohomology. In particular, 
Corollary~\ref{special:cor} implies that $\varphi_j$ may be represented by a special
cocycle $f_j\in Z^n_{bs}(M_j,\bb M_j)$ such that $\|f_j\|_\infty<\|\varphi\|_\infty+\varepsilon$
(we refer the reader to Definition~\ref{special:section} for the definition of the module $C^n_{bs}(X,Y)$ of special
cochains).

We identify the group of the covering automorphisms of $p\colon\widetilde{M}\to M$ with the fundamental
group $\Gamma=\pi_1(M)$ of $M$. Moreover,
for every vertex $v\in V(T)$ we denote by $\Gamma_v$ the stabilizer of $\widetilde{M}_v$ in $\Gamma$.
Observe that $\Gamma_v$ is canonically isomorphic (up to conjugacy) to $\pi_1(M_{j(v)})$.
For every $v\in V(T)$, we denote by $f_v\in Z^n_{bs}(\widetilde{M}_v,\bb \widetilde{M}_v)^{\Gamma_v}$ the 
pull-back of $f_j$ via the covering $\widetilde{M}_v\to M_{j(v)}$. 
In order to prove Theorem~\ref{bounded:graph:thm} it is sufficient to show that
there exists a bounded cocycle $f\in Z^n_b(\widetilde{M},\bb \widetilde{M})^\Gamma$ which restricts to
$f_v$ on each $C_n(\widetilde{M}_v)$, $v\in V(T)$, and is such that $\|f\|_\infty\leq \max \{\|f_j\|_\infty,\, j=1,\ldots,k\}$.

Let $s\colon \Delta^n\to \widetilde{M}$ be a singular simplex, $n\geq 2$, and 
let $q_0,\ldots,q_n$ be the vertices of $s$ (i.e.~the images via $s$ of the vertices of the standard simplex).
We say that the vertex $v\in V(T)$ is a \emph{barycenter} of $s$ the following condition holds:
\begin{itemize}
\item
if $i,j\in\{0,\ldots,n\}$, $i\neq j$, then every path
in $\widetilde{M}$ joining $q_i$ and  $q_j$ intersects $\widetilde{M}_v\setminus \bb\widetilde{M}_v$.
%In this case, we say that ${v}$ is a \emph{barycenter} of $s$.
\end{itemize}

\begin{lemma}\label{onebary}
Let $s\colon \Delta^n\to \widetilde{M}$ be a singular simplex, $n\geq 2$.
Then $s$ has at most one barycenter.
 If $s$ is supported in $\widetilde{M}_v$ for some $v\in V(T)$, then
 $v$ is the barycenter of $s$ if and only if every component of $\bb\widetilde{M}_v$ contains at most one vertex of $s$.
If this is not the case, then
 $s$ does not have any barycenter.
\end{lemma}
\begin{proof}
Suppose by contradiction that $v_1$ and $v_2$ 
are distinct barycenters of $s$.
Let $A$ be the connected component of $\widetilde{M}\setminus \widetilde{M}_{v_1}$ that contains
$\widetilde{M}_{v_2}\setminus \bb\widetilde{M}_{v_2}$. Since $v_1$ is a barycenter of $s$,
at most one vertex of $s$ can be contained in $A$.
Moreover, the set $\widetilde{M}\setminus A$
is path-connected and disjoint from $M_{{v}_2}\setminus \bb \widetilde{M}_{v_2}$. Since $v_2$ is a barycenter of $s$,
this implies that at most one vertex of $s$ can belong to $\widetilde{M}\setminus A$. As a consequence,
$s$ cannot have more than two vertices, and this contradicts the assumption $n\geq 2$.

The second and the third statements of the lemma are obvious.
\end{proof}

Let now $v\in V(T)$ be fixed. 
We associate (quite arbitrarily) to $s$ a 
singular simplex $s_v$
with vertices $q_0',\ldots,q_n'$ in such a way that the following conditions hold:
\begin{itemize}
\item
$s_v$ is supported in $\widetilde{M}_v$;
\item
if $q_i\in \widetilde{M}_v$, then 
$q_i'=q_i$;
\item
if $q_i\notin\widetilde{M}_{v}$, then there exists a unique 
component $B$ of $\bb \widetilde{M}_{v}$ such that
every path joining $q_i$ with $\widetilde{M}_{v}$ intersects
$B$; 
in this case, we choose $q'_i$ to be any point of $B$.
\end{itemize}
The simplex $s_v$ may be thought as a ``projection''
of $s$ onto $\widetilde{M}_v$. We may now set, for every $v\in V(T)$, 
$$
\hat{f}_v(s)=f_v(s_v)\ .
$$
Observe that, even if $s_v$ may be chosen somewhat arbitrarily, the vertices of $s_v$ only depend on $s$, up
to identifying points which lie on the same connected component of $\bb\widetilde{M}_v$.
As a consequence, the fact that $f_v$ is special implies that $\hat{f}_v(s)$ is indeed well defined
(i.e.~it does not depend on the choice of $s_v$). 
%Moreover, if $s$ is already supported in $\widetilde{M}_v$,
%then $\hat{f}_v(s)=f_v(s)$: in fact, if $s$ is degenerate, then two of its vertices

\begin{lemma}\label{nobary}
 Suppose that $v$ is not a barycenter of $s$. Then $\hat{f}_v(s)=0$.
\end{lemma}
\begin{proof}
 If $v$ is not a barycenter of $s$
then there exists a component of $\bb\widetilde{M}_{v}$
containing at least two vertices of $s_v$, so $\hat{f}_v(s)=f_v(s_v)=0$ since $f_v$ is special
(see Remark~\ref{special:rem}).
\end{proof}

We are now ready to define the cochain $f$ as follows: for every singular $n$-simplex $s$ with values in $\widetilde{M}$
we set 
$$
f(s)=\sum_{v\in V(T)} \hat{f}_v(s)\ .
$$
By Lemmas~\ref{onebary} and~\ref{nobary}, the sum on the right-hand side either is empty or consists
of a single term. This already implies that 
$$
\|f\|_\infty\leq \max_{v\in V(T)} \|f_v\|_\infty =\max \{\|f_j\|_\infty,\ j=1,\ldots,k\}\ .
$$
Moreover, it is clear from the construction that $f$ is $\Gamma$-invariant, so
$f\in C^n_b(\widetilde{M},\bb\widetilde{M})^\Gamma$. Let us now suppose that
the singular $n$-simplex $s$ is supported in $\widetilde{M}_v$ for some $v\in V(T)$. 
Then we can set $s_v=s$, and from Lemmas~\ref{onebary} and~\ref{nobary} we deduce that
$f(s)=f_v(s)$. We have thus shown that $f$ indeed coincides
with $f_v$ on simplices supported in $\widetilde{M}_v$, so we are now left to prove that
$f$ is a cocycle.
So, let $s'$ be a singular $(n+1)$-simplex with values in $\widetilde{M}$, and denote by $s'_v$ a projection
of $s'$ on $\widetilde{M}_v$, according to the procedure described above in the case of $n$-simplices.
it readily follows from the definitions that $\partial_i (s'_v)$ is a projection of
$\partial_i s'$ on $\widetilde{M}_v$, so 
$$
\hat{f}_v(\partial s')=\sum_{i=0}^{n+1} (-1)^i f_v((\partial_i s')_v)=
\sum_{i=0}^{n+1} (-1)^i f_v(\partial_i s'_v))=f_v(\partial s'_v)=0\ ,
$$
where the last equality is due to the fact that $f_v$ is a cocycle.
As a consequence we get
$$
f(\bb s')=\sum_{v\in V(T)} \hat{f}_v(\bb s')=0\ ,
$$
so $f$ is also a cocycle. This concludes the proof of  Gromov additivity theorem.

\section{Further readings}

\subsection*{Isometric embeddings in bounded cohomology}
The proof of Gromov additivity theorem presented in this chapter consists of a purely topological description
of (a special case of) the arguments developed in~\cite{BBFIPP}. Indeed, the extension result for bounded cochains stated in Theorem~\ref{bounded:graph:thm}
easily descends from the following more general result:

\begin{thm}[\cite{BBFIPP}]\label{BBFIPP:thm}
 Let $\Gamma$ be the fundamental group of a graph of groups $\mathcal{G}$ based on the 
finite graph $G$.
Suppose that every vertex group of $\mathcal{G}$ is countable, and that every edge group
of $\mathcal{G}$ is amenable.
Then for every $n\in\mathbb{N}\setminus\{0\}$ there exists an isometric embedding 
$$
\Theta\colon
\bigoplus_{v\in V(G)} H_b^n (\Gamma_v) \longrightarrow H_b^n(\Gamma)
$$
which provides a right inverse to the map
$$
 \bigoplus_{v\in V(G)} \h(i_v) \ \colon \ H_b^n(\Gamma)\longrightarrow  \bigoplus_{v\in V(G)} H_b^n (\Gamma_v) \ .
$$
\end{thm}
The isometric embedding $\Theta$ is in general far from being an isomorphism: for example,
we know that the real vector space $H_b^2(\mathbb{Z}*\mathbb{Z})$ is infinite-dimensional (see Corollary~\ref{infinitedim:cor}), 
while $H_b^n(\mathbb{Z})\oplus H_b^n(\mathbb{Z})=0$ for every $n\geq 1$, since $\mathbb{Z}$ is amenable. 
Surprisingly enough, the proof of Theorem~\ref{BBFIPP:thm} runs into additional difficulties in the case of degree 2. 
In that case, even to define the map $\Theta$, 
it is necessary to use the fact that bounded cohomology can be computed via the complex of pluriharmonic functions~\cite{BM1}, 
and that such a realization has no coboundaries in degree 2 
due to the double ergodicity of the action of a group on an appropriate Poisson boundary \cite{Kaim, BM2}.

\subsection*{(Generalized) Dehn fillings}
Let $M$ be a compact orientable $3$-manifold whose boundary consists of tori. One can perform a \emph{Dehn filling} on $M$ by gluing a solid torus
to each of the tori in $\partial M$, thus getting a closed orientable $3$-manifold $N$ (whose homeomorphism type depends on the isotopy classes of the gluing maps).
Gromov's (sub)additivity theorem implies that $\|N\|\leq \|M,\partial M\|$. Indeed, it is a classical result by Thurston that, if $M$ is (the natural compactification of)
a complete finite-volume hyperbolic $3$-manifold, the strict inequality $\|N\|<\|M,\partial M\|$ holds. It is worth noting that Thurston's proof of this strict inequality
is based on the proportionality principle (see Theorems~\ref{prop:thm}, \ref{hyp:thm}), together with an estimate on \emph{Riemannian} volumes. 
A nice and explicit approach to the weaker inequality $\|N\|\leq \|M,\partial M\|$ is described in~\cite{FM}, where the fact that (a suitable generalization of) Dehn filling does not increase the simplicial volume
is extended also to higher dimensions. It would be very interesting to understand whether the strict inequality in the $3$-dimensional case could be understood by means of bounded cohomology.
If this were the case, then one could look for an extension of Thurston's strict inequality also to higher dimensions.

\subsection*{Alternative approaches to Gromov additivity theorem}
Gromov's original argument for the proof of the additivity theorem made an extensive use of the theory of multicomplexes~\cite{Gromov}. 
We refer the reader to~\cite{Kuessner} for a treatment of the matter which closely follows Gromov's original one.

\chapter{Group actions on the circle}\label{actions:chapter}
Bounded cohomology has been successfully exploited in the study of the dynamics of homemorphisms
of the circle. In this chapter we review some fundamental results mainly due to Ghys~\cite{Ghys0,Ghys1,Ghys2}, who proved that semi-conjugacy classes
of representations into the group of homeomorphisms of the circle are completely classified by their \emph{bounded Euler class}.
We also relate the bounded Euler class of a cyclic subgroup of homeomorphisms of the circle to the classical \emph{rotation number} of the generator of the subgroup,
and prove some properties of the rotation number that will be used in the next chapters. Finally, following Matsumoto~\cite{Matsu:numerical} we describe the canonical representative of the \emph{real} bounded Euler class,
also proving a characterization of semi-conjugacy in terms of the real bounded Euler class.

\section{Homeomorphisms of the circle and the Euler class}\label{omeo:sec}
Henceforth we identify $S^1$ with the quotient $\R/\matZ$ of the real line by the subgroup of the integers, and we
fix the corresponding quotient (universal covering) map $p\colon \R\to S^1$.
We denote by $\omeo$ the group of the orientation-preserving homeomorphisms of $S^1$. 
We also denote by $\omeot$ the group of the homeomorphisms of the real line obtained by lifting elements
of $\omeo$. In other words, $\omeot$ is the set of increasing homeomorphisms of $\R$ that commute with every integral translation.
Since $p$ is a universal covering, every $f\in\omeo$ lifts to a map $\widetilde{f}\in\omeot$. On the other hand,
every map $\widetilde{f}\in\omeot$ descends to a map $p_*(\widetilde{f})\in\omeo$ such that $\widetilde{p_*(\widetilde{f})}=\widetilde{f}$, so we have a well-defined
surjective homomorphism $p_*\colon \omeot\to\omeo$, whose kernel consists of the group of integral translations. Therefore,
we get an extension
\begin{equation}\label{Euler:extension}
1\tto{} \matZ\tto{\iota} \omeot\tto{p_*}\omeo\tto{} 1\ ,
\end{equation}
where $\iota(n)=\tau_n$ is the translation $x\mapsto x+n$
(henceforth we will often identify $\matZ$ with its image in $\omeot$ via $\iota$).
By construction, the above extension is central. Recall from Section~\ref{extension:sec} that to every central extension  of a group $\G$ by $\mathbb{Z}$ there is associated a class in
$H^2(\G,\mathbb{Z})$.

\begin{defn}
The coclass
$$
e\in H^2(\omeot,\matZ)
$$ 
associated to the central extension~\eqref{Euler:extension}
is the \emph{Euler class} of $\omeo$.
\end{defn}

The Euler class
 plays an important role in the study of the group of the homeomorphisms of the circle.
Before going on, we point out the following easy result, which will be often exploited 
later.

\begin{lemma}\label{conto}
Let $\widetilde{\psi}_1,\widetilde{\psi}_2\colon \R\to \R$ be (not necessarily continuous) 
maps such that $\widetilde{\psi}_i(x+1)=\widetilde{\psi}_i(x)+1$ for every $x\in\R$, $i=1,2$. 
Also suppose that $\widetilde{\psi}_i$ is strictly increasing on the set
$\{y_0\}\cup (x_0+\matZ)$.
Then 
$$
|\widetilde{\psi}_1(y_0)-\widetilde{\psi}_1(x_0)-(\widetilde{\psi}_2(y_0)-\widetilde{\psi}_2(x_0))|=
|\widetilde{\psi}_1(y_0)-\widetilde{\psi}_2(y_0)-(\widetilde{\psi}_1(x_0)-\widetilde{\psi}_2(x_0))|<1\ .
$$
In particular, if $\widetilde{f}\in\omeot$, then for every $x,y\in\R$ we have
$$
|\widetilde{f}(y)-y-(\widetilde{f}(x)-x)|<1\ .
$$
\end{lemma}
\begin{proof}
Since $\widetilde{\psi}_i$ commutes with integral translations, we may assume that $x_0\leq y_0<x_0+1$. 
Our hypothesis  implies that
$\widetilde{\psi}_i(x_0)\leq \widetilde{\psi}_i(y_0)< \widetilde{\psi}_i(x_0)+1$, 
so $0\leq  \widetilde{\psi}_i(y_0)- \widetilde{\psi}_i(x_0)<1$. This concludes the proof.
\end{proof}

\section{The bounded Euler class}\label{bounded:euler:section}
Let us fix a point $x_0\in \R$. Then we may define a set-theoretic section $s_{x_0}\colon\omeo\to\omeot$ such that $p_*\circ s_{x_0}={\rm Id}_{\omeo}$
by setting $s_{x_0}(f)=\widetilde{f}_{x_0}$, where $\widetilde{f}_{x_0}$ is the unique lift of $f$
such that $\widetilde{f}_{x_0}(x_0)-x_0\in [0,1)$. As described in Section~\ref{extension:sec},
the Euler class $e\in H^2(\omeo,\matZ)$ is represented by the 
cocycle
$$
c_{x_0}\colon \omeo\times\omeo \to \matZ\, ,\qquad (f,g)\mapsto \widetilde{(f\circ g)}_{x_0}^{-1}\widetilde{f}_{x_0}
\widetilde{g}_{x_0}\ \in\ \matZ
$$
(where as mentioned above we identify $\matZ$ with $\iota(\matZ)\in\omeot$). Equivalently, $c_{x_0}$ is defined by the condition
$$
 \widetilde{(f\circ g)}_{x_0}\circ \tau_{c_{x_0}(f,g)}=\widetilde{f}_{x_0}
\widetilde{g}_{x_0}\ .
$$
%By evaluating this equality, for every $x\in\R$ we obtain
%$$
%\widetilde{(f\circ g)}_{x_0} (x)+ c_{x_0}(f,g)=
%\widetilde{f}_{x_0}(\widetilde{g}_{x_0}(x))\ .
%$$
By Lemma~\ref{conto}, for every $f\in\omeo$, $x\in \R$ we have
$$
|\widetilde{f}_{x_0}(x)-x|<|\widetilde{f}_{x_0}(x)-x-(\widetilde{f}_{x_0}(x_0)-x_0)|+|\widetilde{f}_{x_0}(x_0)-x_0|<2\ ,
$$
whence also
\begin{equation}\label{stimaccia}
|\widetilde{f}_{x_0}^{-1}(x)-x|< 2\ .
\end{equation}

\begin{lemma}\label{only01}
The cocycle $c_{x_0}$ takes values into the set $\{0,1\}$ (in particular, it is bounded). Moreover, 
if $x_0,x_1\in\R$, then $c_{x_0}-c_{x_1}$ is the coboundary of a bounded integral $1$-cochain.
\end{lemma}
\begin{proof}
By definition, we have 
$$
c_{x_0}(f,g)=\widetilde{f}_{x_0}(\widetilde{g}_{x_0}(x_0)) - \widetilde{(f\circ g)}_{x_0} (x_0)\ .
$$
But $x_0\leq \widetilde{g}_{x_0}(x_0)<x_0+1$ implies 
$$x_0\leq \widetilde{f}_{x_0}(x_0)\leq \widetilde{f}_{x_0}(\widetilde{g}_{x_0}(x_0))<\widetilde{f}_{x_0}(x_0+1)=\widetilde{f}_{x_0}(x_0)+1 < x_0+2\ .$$
Since $x_0\leq \widetilde{(f\circ g)}_{x_0} (x_0)<x_0+1$, this implies that
$c_{x_0}(f,g)\in \{0,1\}$.

Let now $x_0,x_1\in\R$ be fixed and define $q\colon \omeo\to \matZ$ by setting
$q(f)=\widetilde{f}_{x_0}^{-1}\widetilde{f}_{x_1}$ (we identify again $\matZ$ with $\iota(\matZ)\in \omeot$).
In order to conclude, it is sufficient to show that $q$ is bounded and that $c_{x_1}-c_{x_0}=\delta q$.
Using~\eqref{stimaccia} we get
$$
|q(f)|=|\widetilde{f}_{x_0}^{-1}(\widetilde{f}_{x_1}(x_1))-x_1|\leq 
|\widetilde{f}_{x_0}^{-1}(\widetilde{f}_{x_1}(x_1))-\widetilde{f}_{x_1}(x_1)|+
|\widetilde{f}_{x_1}(x_1)-x_1|<3\ ,
$$
so $q$ is bounded. Moreover,
since $\tau_{c_{x_i}}(f,g)$ and $q(f)$, $q(g)$ commute with every element of $\omeot$, we get 
\begin{align*}
q(fg)^{-1}& =
\left(\widetilde{(f\circ g)}_{x_0}^{-1}\widetilde{(f\circ g)}_{x_1}\right)^{-1} \\ &=
\widetilde{(f\circ g)}_{x_1}^{-1}\widetilde{(f\circ g)}_{x_0}\\ &=
\left(\widetilde{g}_{x_1}^{-1}\widetilde{f}_{x_1}^{-1}\tau_{c_{x_1}(f,g)}\right)
\left(\widetilde{f}_{x_0}\widetilde{g}_{x_0}\tau_{-c_{x_0}(f,g)}\right)\\ &=
\tau_{c_{x_1}(f,g)-c_{x_0}(f,g)} \widetilde{g}_{x_1}^{-1}\widetilde{f}_{x_1}^{-1}
\widetilde{f}_{x_0}\widetilde{g}_{x_0}\\
& = \tau_{c_{x_1}(f,g)-c_{x_0}(f,g)}\widetilde{g}_{x_1}^{-1} q(f)^{-1}\widetilde{g}_{x_0}\\
& = \tau_{c_{x_1}(f,g)-c_{x_0}(f,g)} q(g)^{-1} q(f)^{-1}\ ,
\end{align*}
so
$$
\delta q(f,g) = q(fg)^{-1}q(f)q(g) = \tau_{c_{x_1}(f,g)-c_{x_0}(f,g)} \ .
$$
We have thus shown that $c_{x_1}-c_{x_0}=\delta q$, and this concludes the proof.
%Since $q$ is bounded, this concludes the proof.
\end{proof}

\begin{defn}
Let $x_0\in \R$, and let $c_{x_0}$ be the bounded cocycle introduced above. Then, 
the \emph{bounded Euler class} $e_b\in H^2_b(\omeo,\matZ)$ is defined by setting 
$e_b=[c_{x_0}]$ for some $x_0\in\R$. The previous lemma shows that $e_b$ is indeed well defined, and by construction
the comparison map $H^2_b(\omeo,\matZ)\to H^2(\omeo,\matZ)$ sends $e_b$ to $e$.
\end{defn}

Henceforth, for every $f\in \omeo$ we simply denote by $\widetilde{f}$ the lift $\widetilde{f}_0$ of $f$ such that
$\widetilde{f}(0)\in [0,1)$, and 
by $c$ the cocycle $c_0$.  

\section{The (bounded) Euler class of a representation}
Let $\G$ be a group, and consider a representation $\rho\colon \G\to \omeo$. Then the Euler class $e(\rho)$ and the bounded Euler
class of $e_b(\rho)$ of $\rho$ are the elements of $H^2(\G,\matZ)$ and $H^2_b(\G,\matZ)$ defined by
$$
e(\rho)=\rho^*(e)\, ,\qquad 
e_b(\rho)=\rho^*(e_b)\ ,
$$
where, with a slight abuse, we denote by $\rho^*$ both the morphisms $H^2(\rho)$, $H^2_b(\rho)$ induced by $\rho$ in cohomology and in bounded cohomology.

The (bounded) Euler class of a representation captures several interesting features of the dynamics of the corresponding action. It is worth mentioning that the bounded Euler class often carries much more information than the usual Euler class. For example, 
if $\G=\matZ$, then $H^2(\G,\matZ)=0$, so $e(\rho)=0$ for every $\rho\colon \G\to\omeo$. On the other hand, we have seen in 
Proposition~\ref{h2zz} that $H^2_b(\matZ,\matZ)$ is canonically isomorphic to $\R/\matZ$, so one may hope that $e_b(\rho)$ does not vanish at least for some
representation $\rho\colon \matZ\to\omeo$. And this is indeed the case, as we will show in the following section.

\section{The rotation number of a homeomorphism}\label{rotation:sec}
Let $f$ be a fixed element of $\omeo$, and let us consider the representation
$$
\rho_f\colon \matZ\to \omeo\, ,\qquad \rho_f(n)=f^n\ .
$$
Then we define the \emph{rotation number} of $f$ by setting
\begin{equation}\label{rot:def}
\rot (f)=e_b(\rho_f)\in \R/\matZ\ \cong \ H^2_b(\matZ,\matZ)\ ,
\end{equation}
where the isomorphism $\R/\matZ\ \cong \ H^2_b(\matZ,\matZ)$ is described in Proposition~\ref{h2zz}.
The definition of rotation number for a homeomorphism of the circle dates back to
Poincar\'e~\cite{Poincare}. Of course, the traditional definition of $\rot$ did not involve bounded cohomology, so
equation~\eqref{rot:def} is usually introduced as a theorem establishing a relationship between bounded cohomology and 
more traditional dynamical invariants. However, in this monograph we go the other way round, i.e.~
we recover 
the classical definition of the rotation number from the approach via bounded coohmology.

Let $p\colon \omeot\to\omeo$ be the usual covering projection, and denote by $c\colon \overline{C}^2(\omeo,\Z)$
the standard representative of the Euler cocycle, so that, 
if we denote by $\widetilde{f}=\widetilde{f}_0$  the preferred lift of any $f\in\omeo$, 
then
\begin{equation*}
c(f,g)=\widetilde{f}(\widetilde{g}(0))-\widetilde{f\circ g}(0)\ .
\end{equation*}
We first define a cochain $\psi\colon \overline{C}^1(\omeot,\Z)$ by setting  
\begin{equation*}%\label{ueq}
\psi(h)=h^{-1}\circ \widetilde{p(h)}\in \ker p=\Z\ .
\end{equation*}
for every $h\in\omeot$.
Then for every $h_1,h_2\in\omeot$ we have
\begin{align*}
p^*(c) (h_1,h_2) & =\widetilde{p(h_1)}(\widetilde{p(h_2)}(0))-\widetilde{p(h_1h_2)}(0)\\ &=
h_1(\tau_{\psi(h_1)}(h_2(\tau_{\psi(h_2)}(0))))-h_1(h_2(\tau_{\psi(h_1h_2)}(0))\\
& = \psi(h_1)+\psi(h_2)- \psi(h_1h_2) =
\overline{\delta}\psi (h_1,h_2)\ ,
\end{align*}
i.e.
\begin{equation}\label{psieq}
 p^*(c)=\overline{\delta}\psi\ .
\end{equation}
As a consequence, 
the coboundary of $\psi$ is bounded, i.e.~$\psi$ is a quasimorphism. 
Therefore, by Proposition~\ref{homogeneous:prop}, there exists 
a unique homogeneous quasimorphism $\psi_h\in Q^h(\omeot,\R)$  
at finite distance from $\psi$. Moreover, for every $h\in\omeot$ we have
$$
\psi_h(h)=\lim_{n\to +\infty} \frac{\psi(h^n)}{n}=\lim_{n\to -\infty} \frac{\psi(h^n)}{n}\ \in \R\ .
$$
Observe now that  by construction $\psi(h^n)=h^{-n}(\widetilde{p(h^n)}(0))$. Since
$\widetilde{p(h^n)}(0)\in [0,1)$, this implies that $\psi(h^n)\in h^{-n}([0,1))$, so that
$$
|\psi(h^n)- h^{-n}(0)|\leq 1
$$
and
$$
\psi_h(h)=\lim_{n\to -\infty} \frac{\psi(h^n)}{n}=\lim_{n\to -\infty} \frac{h^{-n}(0)}{n}=-\lim_{n\to +\infty} \frac{h^{n}(0)}{n}\ .
$$
Therefore, if we set 
$$
\rott\colon \omeot\to\R\, ,\qquad \rott(h)=\lim_{n\to +\infty} \frac{h^{n}(0)}{n}\ ,
$$
then we have proved the following:

\begin{prop}\label{rott:prop}
 The map $\rott$ is a well-defined homogeneous quasimorphism. Moreover, $\rott$ stays at bounded distance from
 the quasimorphism $-\psi$, where $\psi=p^*(c)$ is the pull-back of the canonical Euler cocycle.
 \end{prop}

Let now $f$ be a fixed element of $\omeo$, and consider the representation
$$
\rho_f\colon \matZ\to \omeo\, ,\qquad \rho_f(n)=f^n\ .
$$
Let us also fix the standard lift $\widetilde{f}\in\omeo$ of $f$, so that $\widetilde{f}(0)\in [0,1)$.
We define the integral chain $u\in\overline{C}^1(\Z,\Z)$ by setting
$$
u(n)=\psi(\widetilde{f}^n)\ .
$$
We have
$$
\rho^*(c)(n,m)=c(f^n,f^m)=p^*(c)(\widetilde{f}^n,\widetilde{f}^m)=\overline{\delta}\psi (\widetilde{f}^n,\widetilde{f}^m)=
\overline{\delta}u (n,m)\ .
$$
i.e.~$\overline{\delta} u=\rho^*(c)$. Moreover, 
$$
|u(n)+n\rott(\widetilde{f})|=|\psi(\widetilde{f}^n)-\psi_h (\widetilde{f}^n)|
$$
is uniformly bounded, 
so the map $b\colon \matZ\to\matZ$ 
$$
b(n)=u(n)+\lfloor \\rott(\widetilde{f}) n \rfloor
$$
is bounded (henceforth for every $x\in \R$ we denote by $\lfloor x\rfloor$ the largest integer which does not exceed
$x$). As a consequence the bounded Euler class $e(\rho_f)=[\rho_f^*(c)]$
may be represented by the cocycle
$\overline{\delta}^1 (u-b)$, i.e.~by the map
$$
(n,m)\ \mapsto \ \lfloor \rott(\widetilde{f}) (n+m) \rfloor - \lfloor \rott(\widetilde{f}) n \rfloor - \lfloor \rott(\widetilde{f}) m \rfloor\ .
$$
By Proposition~\ref{h2zz}, this implies that the element of $\R/\matZ$ corresponding to  the bounded Euler class of $\rho_f$ is given by $[\rott(\widetilde{f})]$, 
so our definition of rotation number gives $\rot(f)=[\rott(\widetilde{f})]$. 
We have thus recovered  the classical definition of rotation number:

\begin{prop}\label{rot:prop}
Let $f\in\omeo$, let $\widetilde{f}$ be \emph{any} lift of $f$ and $x_0\in \R$ be \emph{any} point. Then 
$$
\rot(f)=\left[\lim_{n\to\infty} \frac{\widetilde{f}^n(x_0)}{n} \right]\ \in\ \R/\matZ
$$
(in particular, the above limit exists).
\end{prop}
\begin{proof}
We have already proved the statement in the case when $\widetilde{f}=\widetilde{f}_0$ and $x_0=0$. However, an easy application of Lemma~\ref{conto}
implies that the limit $\lim_{n\to\infty} {\widetilde{f}^n(x_0)}/{n} $ does not depend on the choice of $x_0$. Moreover, if
$\widetilde{f}'=\tau_k\widetilde{f}$, then $\lim_{n\to\infty} {(\widetilde{f}')^n(x_0)}/{n} =k+\lim_{n\to\infty} {\widetilde{f}^n(x_0)}/{n} $,
and this concludes the proof.
\end{proof}

\begin{cor}\label{rot:hom}
For every $f\in\omeo$, $n\in\matZ$ we have
$$
\rot(f^n)=n\rot(f)\ .
$$
\end{cor}

Roughly speaking, Proposition~\ref{rot:prop} shows that $\rot(f)$ indeed measures the average rotation angle
of $f$.
For example, if $f$ is a genuine rotation of angle $0\leq \alpha<1$ (recall that $S^1=\R/\matZ$ has length one in our
setting), then $\widetilde{f}(x)=x+\alpha$,
so $\widetilde{f}^n(0)=n\alpha$ and $\rot(f)=[\alpha]$. On the other hand, if $f$ admits a fixed point, then 
we may choose a lift $\widetilde{f}$ of $f$ admitting a fixed point, so Proposition~\ref{rot:prop} implies
that $\rot(f)=0$. We will see in Corollary~\ref{rot:cor} that also the converse implication holds: if $\rot(f)=0$, then
$f$ fixes a point in $S^1$.

Recall that $\omeo$ is endowed with the compact-open topology. In the sequel it will be useful to know that the rotation number is continuous. To this aim, we first establish the following:

\begin{lemma}\label{esti:rot}
Let $\widetilde{f}\in\omeot$, and let $a\in\mathbb{Z}$.
 \begin{enumerate}
  \item If $\rott(\widetilde{f})>a$, then $\widetilde{f}(x)>x+a$ for every $x\in\R$.
  \item If there exists $x\in\R$ such that $\widetilde{f}(x)\geq x+a$, then $\rott(\widetilde{f})\geq a$.
  \item If $\rott(\widetilde{f})<a$, then $\widetilde{f}(x)<x+a$ for every $x\in\R$.
  \item If there exists $x\in\R$ such that $\widetilde{f}(x)\leq x+a$, then $\rott(\widetilde{f})\leq a$.
 \end{enumerate}
\end{lemma}
\begin{proof}
(1): Suppose by contradiction that $\widetilde{f}(x)\leq x+a$ for some $x\in\R$. Since $a$ is an integer and $\widetilde{f}$ commutes with integral translations, an obvious inductive argument shows that
$\widetilde{f}^n(x)\leq x+na$ for every $n\in\mathbb{N}$, and this implies in turn that $\rott(\widetilde{f})\leq a$.

(2): Since $a$ is an integer, an easy inductive argument shows that $\widetilde{f}^n(x)\geq x+na$ for every $n\in\mathbb{N}$, hence $\rott(\widetilde{f})\geq a$.

Claims (3) and (4) can be proved by analogous arguments.
\end{proof}

\begin{prop}\label{continuous:rot}
 The maps $\rott\colon \omeot\to \R$ and $\rot\colon \omeo\to\R/\Z$ are continuous.
\end{prop}
\begin{proof}
 Of course, it is sufficient to show that $\rott$ is continuous. So, suppose that $\widetilde{f}_i$, $i\in\mathbb{N}$ is a sequence of elements
 of $\omeot$ tending to $\widetilde{f}\in\omeot$. 
 Since $\R$ is locally compact, $\omeot$ is a topological group, i.e.~the composition is continuous with respect to the compact-open topology.
 In particular, 
$\widetilde{f}_i^n$ uniformly converges to $\widetilde{f}^n$ for every $n\in\mathbb{N}$.

Let now $\varepsilon>0$ be given. Then, there exist rational numbers $p/q,r/s$ such that
$$
\rott(\widetilde{f})-\varepsilon<p/q<\rott(\widetilde{f})<r/s<\rott(\widetilde{f})+\varepsilon\ .
$$
Since $\rott$ is a homogeneous quasimorphism, we have 
$\rott(\widetilde{f}^q)>p$, and Lemma~\ref{esti:rot} (1) implies that  $\widetilde{f}^q(x)>x+p$ for every $x\in\R$. Being periodic,
the function $x\mapsto \widetilde{f}^q(x)-x-p$ has thus a positive minimum. Since $\widetilde{f}^q_i$ converges to $\widetilde{f}^q$, this implies in turn
that there exist $i_0\in\mathbb{N}$, $x\in\R$ such that $\widetilde{f}_i^q(x)>x+p$ for every $i\geq i_0$. By Lemma~\ref{esti:rot} (2), we then have
$\rott(\widetilde{f}_i^q)\geq p$ for every $i\geq i_0$, hence $\rott(\widetilde{f}_i)\geq p/q$ for every $i\geq i_0$. This shows that
$$
\liminf_{i\to \infty} \rott(\widetilde{f}_i)\geq p/q\geq \rott(\widetilde{f})-\varepsilon\ .
$$

Using items (3) and (4) of Lemma~\ref{esti:rot}, one can analogously prove that
$$
\limsup_{i\to \infty} \rott(\widetilde{f}_i)\leq r/s\leq \rott(\widetilde{f})+\varepsilon\ .
$$

Due to the arbitrariness of $\varepsilon$, this implies that $\lim_{i\to \infty} \rott(\widetilde{f}_i)=\rott(\widetilde{f})$. 
\end{proof}

\section{Increasing degree one map of the circle}
Let $\G$ be a group. Two representations $\rho_1,\rho_2$ of $\G$ into
$\omeo$ are conjugate if there exists $\varphi\in \omeo$ such that
$\rho_1(g)=\varphi\rho_2(g)\varphi^{-1}$ for every $g\in \G$. It is easy to show that the bounded Euler
classes (whence the usual Euler classes) of conjugate representations coincide
(this fact also follows from Theorem~\ref{Euler:thm} below). Unfortunately
it is not true that representations sharing the same bounded Euler class are  conjugate.
However, it turns out that the bounded Euler class completely determines the \emph{semi-conjugacy} class
of a representation. 

The subject we are going to describe was first investigated by Ghys~\cite{Ghys0}, who 
first noticed the relationship between semi-conjugacy and the bounded Euler class (see also~\cite{Ghys1,Ghys2}). 
For a detailed account on the history of the notion of semi-conjugacy, together with a discussion
of the equivalence of several definitions (and of the non-equivalence of some variations) of semi-conjugacy, we refer the reader to~\cite{BFH}. 
%Here we adopt the definition of semi-conjugacy given 
%in~\cite{Bucherweb}. 

\begin{defn}
 Let us consider an ordered $k$-tuple $(x_1,\ldots,x_k)\in (S^1)^k$. We say that such a $k$-tuple is
 \begin{itemize}
  \item \emph{positively oriented} if there exists an orientation-preserving embedding $\gamma\colon [0,1]\to S^1$
such that $x_i=\gamma(t_i)$, where $t_i< t_{i+1}$ for very $i=1,\ldots,k-1$.
\item \emph{weakly positively oriented} if there exists an orientation-preserving immersion $\gamma\colon [0,1]\to S^1$
which restricts to an embedding on $[0,1)$, and is 
such that $x_i=\gamma(t_i)$, where $t_i\leq t_{i+1}$ for very $i=1,\ldots,k-1$.
 \end{itemize}
Observe that cyclic permutations leave 
the property of being (weakly) positively oriented invariant.
\end{defn}

\begin{defn}
 A (not necessarily continuous) map $\varphi\colon S^1\to S^1$ is increasing of degree one if the following condition holds:
 if $(x_1,\ldots,x_k)\in (S^1)^k$ is weakly positively oriented, then 
 $(\varphi(x_1),\dots,\varphi(x_k))$ is weakly positively oriented.
\end{defn}

For example, any constant map $\varphi\colon S^1\to S^1$ is increasing of degree one. The following lemma provides 
an alternative description of increasing maps of degree one. Observe that the composition of increasing maps of degree one is increasing of degree one.
We stress that, in our terminology, a (not necessarily continuous) map $\widetilde{\varphi}\colon \R\to\R$ will be said to be \emph{increasing}
if it is weakly increasing, i.e.~if $\widetilde{\varphi}(x)\leq \widetilde{\varphi}(y)$ whenever $x\leq y$.

\begin{lemma}\label{quadruple:lem}
Let $\varphi\colon S^1\to S^1$ be any map. Then the following conditions are equivalent:
\begin{enumerate}
 \item 
$\varphi$ is increasing of degree one;
\item if $(x_1,\ldots,x_4)\in (S^1)^4$ is positively oriented, then 
 $(\varphi(x_1),\dots,\varphi(x_4))$ is weakly positively oriented;
 \item 
there exists a (set-theoretical) increasing lift $\widetilde{\varphi}\colon \R\to\R$ of $\varphi$
such that $\widetilde{\varphi}(x+1)=\widetilde{\varphi}(x)+1$ for every $x\in\R$.
\end{enumerate}
\end{lemma}
\begin{proof}
(1) $\Rightarrow$ (2) is obvious, so we begin by showing that (2) $\Rightarrow$ (3).
It is immediate to check that $\varphi$ satisfies (2) (resp.~(3)) if and only if $r\circ \varphi$ does,
where $r\colon S^1\to S^1$ is any rotation. Therefore, we may assume that $\varphi([0])=[0]$. 
For $t\in [0,1)$, 
we define $\widetilde{\varphi}(t)\in [0,1]$
as follows. If $\varphi([t])=[0]$, then we distinguish two cases: if
$\varphi([s])=[0]$ for every $0\leq s\leq t$, then we set 
$\widetilde{\varphi}(t)=0$; otherwise we set $\widetilde{\varphi}(t)=1$.
If $\varphi([t])\neq [0]$, then $\widetilde{\varphi}(t)$ is the unique point in $(0,1)$
such that $[\widetilde{\varphi}(t)]=\varphi([t])$.

We have thus defined a map $\widetilde{\varphi}\colon [0,1)\to [0,1]$ such that $\widetilde{\varphi}(0)=0$. 
This map uniquely extends to a map
(still denoted by $\widetilde{\varphi}$) such that $\widetilde{\varphi}(x+1)=\widetilde{\varphi}(x)+1$.
In order to conclude we need to show that $\widetilde{\varphi}$ is increasing. By construction, this boils down
to showing that 
$\widetilde{\varphi}|_{[0,1)}$ is increasing, i.e.~that
$\widetilde{\varphi}(t_0)\leq \widetilde{\varphi}(t_1)$ whenever
$0\leq t_0<t_1<1$. Since $\widetilde{\varphi}(0)=0$ we may suppose $0<t_0$. 

Assume first that $\varphi([t_0])=[0]$. If $\varphi([s])=[0]$ for every $0\leq s\leq t_0$,
then $\widetilde{\varphi}(t_0)=0\leq \widetilde{\varphi}(t_1)$, and we are done. Otherwise 
$\widetilde{\varphi}(t_0)=1$, and there exists $0<s<t_0$ such that $\varphi([s])\neq [0]$. Observe that the  quadruple
$([0],[s],[t_0],[t_1])$ is positively oriented. As a consequence, the quadruple
$(\varphi([0]), \varphi([s]),\varphi([t_0]),\varphi([t_1]))=([0],\varphi([s]),[0],\varphi([t_1]))$
is weakly positively oriented. Since $\varphi([s])\neq [0]$, this implies in turn that
$\varphi([t_1])=0$, so by definition $\widetilde{\varphi}(t_1)=1$, and again
$\widetilde{\varphi}(t_0)\leq \widetilde{\varphi}(t_1)$.

Assume now that $\varphi([t_0])\neq [0]$. If $\varphi([t_1])=[0]$, then $\widetilde{\varphi}(t_1)=1$, and we are done.
Otherwise, observe that, since $\varphi$ takes positively oriented quadruples into weakly positively oriented
quadruples, $\varphi$ also takes positively oriented triples into weakly positively oriented triples.
Therefore, the triple $([0],\varphi([t_0]),\varphi([t_1]))$ is weakly positively oriented. 
Since $\varphi([t_i])\neq [0]$ for $i=0,1$, this is equivalent to the fact that 
$\widetilde{\varphi}(t_0)\leq \widetilde{\varphi}(t_1)$.

Let us now prove that (3) implies (1). Let $(x_1,\ldots,x_k)\in (S^1)^k$ be weakly positively oriented. We choose
$t_i\in \mathbb{R}$ such that $[t_i]=x_i$  and
$t_1\leq t_i<t_1+1$ for $i=1,\ldots,k$. Since  $(x_1,\ldots,x_k)\in (S^1)^k$ is positively oriented
we have $t_i\leq t_{i+1}$ for every $i=1,\ldots,k-1$, so
$$
\widetilde{\varphi}(t_1)\leq 
\widetilde{\varphi}(t_2)\leq\ldots\leq 
\widetilde{\varphi}(t_k)\leq
\widetilde{\varphi}(t_1)+1\ ,
$$
which readily implies that $(\varphi(x_1),\ldots,\varphi(x_k))$
is weakly positively oriented.
\end{proof}

If $\varphi$ is increasing of degree one, then a lift $\widetilde{\varphi}$ as in claim (3)
of the previous lemma is called a \emph{good lift} of $\varphi$.
For example, the constant map $\varphi\colon S^1\to S^1$ mapping every point to $[0]$ 
admits $\widetilde{\varphi}(x)=\lfloor x\rfloor$ as a good lift (in fact, for every $\alpha\in\R$ the maps
$x\mapsto \lfloor x+\alpha\rfloor$ 
and $x\mapsto \lceil x+\alpha\rceil$
are good lifts of $\varphi$).

\begin{rem}\label{burger:ext:rem}
Condition (3) of the previous lemma (i.e.~the existence of an increasing lift commuting with integral translations)
is usually described as the condition
defining the notion of increasing map of degree one. This condition is not equivalent to
the request that positively oriented triples are taken into weakly positively oriented triples.
For example, if $\varphi([t])=[0]$ for every $t\in\mathbb{Q}$ and
$\varphi([t])=[1/2]$ for every $t\in\mathbb{R}\setminus\mathbb{Q}$, then $\varphi$ takes any triple
into a weakly positively oriented one, but it does not admit \emph{any} increasing lift, so it is not increasing of degree one.
\end{rem}

%The following definition will prove useful later:

%\begin{defn}
%Let $\varphi\colon S^1\to S^1$ be an increasing map of degree one. Then $\varphi$ is \emph{upper semicontinuous} if it admits
%an upper semicontinuous good lift $\widetilde{\varphi}\colon\R\to\R$.
%\end{defn}

\section{Semi-conjugacy}

We are now ready to give the definition of semi-conjugacy for group actions on the circle. Here we adopt the definition given in~\cite{BFH}.

\begin{defn}\label{semi:def}
Let $\rho_i\colon \G\to\omeo$ be group representations, $i=1,2$. We say that $\rho_1$ is \emph{semi-conjugate}
to $\rho_2$ if the following conditions hold:
\begin{enumerate}
\item
There exists an increasing map $\varphi$ of degree one such that
$$\rho_1(g)\varphi=\varphi\rho_2(g)$$ for every $g\in \G$.
\item
There exists an increasing map $\varphi^*$ of degree one such that
$$\rho_2(g)\varphi^*=\varphi^*\rho_1(g)$$ for every $g\in \G$.
\end{enumerate}
\end{defn}

The main result of this chapter is the following theorem, which is due to Ghys~\cite{Ghys0}:

\begin{thm}\label{Euler:thm}
Let $\rho_1,\rho_2$ be representations of $\G$ into $\omeo$. Then
$\rho_1$ is semi-conjugate
to $\rho_2$ if and only if $e_b(\rho_1)=e_b(\rho_2)$.
\end{thm}

We begin by observing that
semi-conjugacy is an equivalence relation:

\begin{lemma}
Semi-conjugacy is an equivalence relation.
\end{lemma}
\begin{proof}
Reflexivity and symmetry are obvious, 
while transitivity readily follows from the fact that the composition
of increasing maps of degree one is an increasing map of degree one.
\end{proof}

It readily follows from the definitions that conjugate representations are semi-conjugate.
 However, semi-conjugacy is a much weaker condition than conjugacy:
for example, any representation admitting a fixed point is semi-conjugate to the trivial
representation (see Proposition~\ref{zeroprop}). On the other hand, semi-conjugacy implies conjugacy 
for the important class of \emph{minimal} representations, that we are now going to define.

\begin{defn}
 A representation $\rho\colon\G\to\omeo$ is \emph{minimal} if every $\rho(\G)$-orbit is dense in $S^1$.
\end{defn}

\begin{prop}\label{dense:orbits}
Let $\rho_1,\rho_2\colon \G\to \omeo$ be semi-conjugate representations, and suppose that every orbit
of $\rho_i(\G)$ is dense for $i=1,2$. Then $\rho_1$ is conjugate to $\rho_2$.
\end{prop}
\begin{proof}
Let $\varphi$ be an increasing map of degree one such that
$$
\rho_1(g)\varphi=\varphi\rho_2(g)\qquad {\rm for\ every}\ g\in\G\ ,
$$
and denote by $\widetilde{\varphi}$ a good lift of $\varphi$. The image of $\varphi$ is obviously $\rho_1(\G)$-invariant, so our assumptions
imply that ${\rm Im}\, \varphi$ is dense in $S^1$. This implies in turn that the image of $\widetilde{\varphi}$ is dense in $\R$. So
the map $\widetilde{\varphi}$, being increasing, is continuous and surjective.
Therefore, 
the same is true for $\varphi$, and we are left to show that $\varphi$ is also injective.

Suppose by contradiction that 
 there exist points $x_0,y_0\in S^1$ such that $\varphi(x_0)=\varphi(y_0)$, and choose lifts $\widetilde{x}_0,\widetilde{y}_0$ of $x_0,y_0$ in $\R$
such that $\widetilde{x}_0<\widetilde{y}_0<\widetilde{x}_0+1$. Since $\widetilde{\varphi}$ is increasing and commutes with integral translations,
we have either $\widetilde{\varphi}(\widetilde{y}_0)=\widetilde{\varphi}(\widetilde{x}_0)$ or 
$\widetilde{\varphi}(\widetilde{y}_0)=\widetilde{\varphi}(\widetilde{x}_0+1)$. In any case, $\widetilde{\varphi}$ is constant on a non-trivial interval,
so there exists an open subset $U\subseteq S^1$ such that $\varphi|_U$ is constant. Let now $x$ be any point of $S^1$. Our assumptions imply that
there exists $g\in\G$ such that $\rho_2(g)^{-1}(x)\in U$, so $x\in V=\rho_2(g)(U)$,  where $V\subseteq S^1$ is open. Observe
that 
$$
\varphi|_V=(\varphi\rho_2(g))|_U\circ \rho_2(g)^{-1}|_V=
(\rho_1(g)\varphi)|_U\circ \rho_2(g)^{-1}|_V
$$
is constant. We have thus proved that $\varphi$ is locally constant, whence constant, and this contradicts the fact that $\varphi$ is surjective.
\end{proof}
\smallskip

\section{Ghys' Theorem}
We start by showing that representations sharing the same bounded Euler class are semi-conjugate.

%As mentioned above, in the following paragraphs we will show that two representations $\rho_1$, $\rho_2$ are semi-conjugate if and only if
%$e_b(\rho_1)=e_b(\rho_2)$. In the following proposition we prove one implication:

\begin{prop}\label{Euler:char:prop}
Let $\rho_1,\rho_2$ be representations of $\G$ into $\omeo$ such that $e_b(\rho_1)=e_b(\rho_2)$. Then
$\rho_1$ is semi-conjugate
to $\rho_2$.
\end{prop}
\begin{proof}
By symmetry, it is sufficient  to exhibit an increasing map of degree one
$\varphi$ such that
$$
\rho_1(g)\varphi=\varphi\rho_2(g)$$
for every $g\in\G$.

 Since $e_b(\rho_1)=e_b(\rho_2)$ we have \emph{a fortiori} $e(\rho_1)=e(\rho_2)$. 
 By Proposition~\ref{central:extensions}, there exists a central extension
$$
1\tto{} \matZ\tto{\iota}\overline{\G}\tto{\pi} \G\tto{} 1
$$
 associated to $e(\rho_1)=e(\rho_2)$ (see Section~\ref{extension:sec}). 
Let $c_i=\rho_i^*(c)$, $i=1,2$, be the pull-back of the usual
representative of the Euler class of $\omeo$.
 If $s\colon \G\to\overline{\G}$ is any fixed section of $\pi$, then  the cocycle $c\colon \overline{C}^2(\G,\matZ)$ associated to $s$ defines the same cohomology class as $c_1$. Therefore,
there exists $v\colon \G\to\matZ$ such that 
$$
i(c_1(g_1,g_2))=s(g_1g_2)^{-1}s(g_1)s(g_2)i(v(g_1,g_2))^{-1}i(v(g_1))i(v(g_2))\ ,
$$
so by setting $s_1(g)=i(v(g))s(g)$ for every $g\in\G$ we obtain a section
 $s_1\colon \G \to\overline{\G}$  such that
$$
i(c_1(g_1,g_2))=s_1(g_1g_2)^{-1}s_1(g_1)s_1(g_2)\ .
$$
 
 Since $e_b(\rho_1)=e_b(\rho_2)$, there exists a bounded cochain $u\colon \G\to\matZ$ such that 
$\delta u=c_2-c_1$. 
If we set $s_2(g)=s_1(g)\iota(u(g))$, then we get
$$
i(c_2(g_1,g_2))=s_2(g_1g_2)^{-1}s_2(g_1)s_2(g_2)\ .
$$
Of course, once $i\in \{1,2\}$ is fixed, every element of $\overline{g}\in \overline{\G}$ admits a unique expression $\overline{g}=\iota(n)s_i(g)$.
For $i=1,2$ we define the map
$$
\overline{\rho}_i\colon \overline{\G}\to\omeot\, ,\qquad \overline{\rho}_i(\iota(n)s_i(g))=\tau_n \widetilde{{\rho}_i(g)}\ .
$$
%By the proof of Lemma~\ref{Lift}, 
Lemmas~\ref{euler-null} and~\ref{Lift} imply that
$\overline{\rho}_i$ is a homomorphism. 
For the sake of clarity, we give here a direct and short
proof of this fact.
Indeed, if we take elements $\overline{g}_j=i(n_j)s_i(g_j)\in\overline{\Gamma}$, $j=1,2$, then 
$$
\overline{g}_1\, \overline{g}_2=\iota(n_1)s_i(g_1)\iota(n_2)s_i(g_2)=\iota(n_1+n_2)\iota(c_i(g_1,g_2))s_i(g_1g_2)
$$
so
$$
\overline{\rho}_i(\overline{g}_1\, \overline{g}_2)=\tau_{n_1+n_2}\tau_{c_i(g_1,g_2)}\widetilde{{\rho}_i(g_1g_2)}
=\tau_{n_1} \widetilde{{\rho}_i(g_1)}\tau_{n_2} \widetilde{{\rho}_i(g_2)}=\overline{\rho}_i(\overline{g}_1)\overline{\rho}_i(\overline{g}_2)\ ,
$$
%$$
%\tau_{n_1} \widetilde{\overline{\rho}_i(g_1)}\tau_{n_2} \widetilde{\overline{\rho}_i(g_2)}=\tau_{n_1+n_2}\tau_{c_i(g_1,g_2)}\widetilde{\overline{\rho}_i(g_1g_2)}
%$$
and this proves that  $\overline{\rho}_i$ is a homomorphism.

Let now $\overline{g}=\iota(n_1)s_1(g)=\iota(n_1-u(g))s_2(g)$. Then
$$
\overline{\rho}_1(\overline{g})^{-1}\overline{\rho}_2(\overline{g})=\left(\tau_{n_1}\widetilde{{\rho}_1(g)}\right)^{-1}\tau_{n_1-u(g)}\widetilde{{\rho}_2(g)}=\tau_{-u(g)}
\left(\widetilde{{\rho}_1(g)}\right)^{-1}\widetilde{{\rho}_2(g)}\ .
$$
Recall now that there exists $M\geq 0$ such that $|u(g)|\leq M$ for every $g\in \G$. Moreover, using that
$\widetilde{{\rho}_i(g)}(0)\in [0,1)$ for every $g\in \G$, $i=1,2$, it is immediate to realize that
$x-2\leq \left(\widetilde{{\rho}_1(g)}\right)^{-1}(\widetilde{{\rho}_2(g)}(x))\leq x+2$ for every $x\in\R$. As a consequence, for every $x\in \R$ we have
$$
x-M-2\leq \overline{\rho}_1(\overline{g})^{-1}(\overline{\rho}_2(\overline{g})(x))\leq x+M+2\ ,
$$
so we may define the value
$$
\widetilde{\varphi}(x)=\sup_{\overline{g}\in \overline{\G}} \overline{\rho}_1(\overline{g})^{-1}(\overline{\rho}_2(\overline{g})(x))\ \in \ [x-M-2,x+M+2]\ .
$$
We will now check that $\widetilde{\varphi}$ realizes the desired semi-conjugacy between $\rho_1$ and $\rho_2$. 

Being the supremum of (strictly) increasing maps that commute with integral translations, the map $\widetilde{\varphi}\colon\R\to\R$
is (possibly non-strictly) increasing and commutes with integral translations, so it is a good lift of an increasing map of degree one $\varphi\colon S^1\to S^1$.

For every $\overline{g}_0\in\overline{\G}$, $x\in\R$, we have 
\begin{align*}
\widetilde{\varphi}(\overline{\rho}_2(\overline{g}_0)(x)) &=
\sup_{\overline{g}\in\overline{\G}} \overline{\rho}_1(\overline{g})^{-1}(\overline{\rho}_2(\overline{g})(\overline{\rho}_2(\overline{g}_0)(x)))\\
& = \sup_{\overline{g}\in\overline{\G}} \overline{\rho}_1(\overline{g})^{-1}(\overline{\rho}_2(\overline{g}\, \overline{g}_0)(x))\\
& = \sup_{\overline{g}\in\overline{\G}} \overline{\rho}_1(\overline{g}\, \overline{g}_0^{-1})^{-1}(\overline{\rho}_2(\overline{g})(x))\\
& = \sup_{\overline{g}\in\overline{\G}} \overline{\rho}_1(\overline{g_0})(\overline{\rho}_1(\overline{g})^{-1}(\overline{\rho}_2(\overline{g})(x)))\\
& = \overline{\rho}_1(\overline{g}_0)\left( \sup_{\overline{g}\in\overline{\G}} \overline{\rho}_1(\overline{g})^{-1}(\overline{\rho}_2(\overline{g})(x))\right)\\
& = \overline{\rho}_1(\overline{g}_0)(\widetilde{\varphi}(x))\ ,
\end{align*}
and this easily implies that
$$
\varphi \rho_2(g_0)=\rho_1(g_0)\varphi
$$
for every $g_0\in \G$.
\end{proof}

We are now ready to characterize the set of representations having vanishing bounded Euler class as the semi-conjugacy class
of the trivial representation:

\begin{prop}\label{zeroprop}
Let $\rho\colon \G\to\omeo$ be a representation. Then the following conditions are equivalent:
\begin{enumerate}
\item
$\rho$ fixes a point, i.e.~there exists a point $x_0\in S^1$ such that $\rho(g)(x_0)=x_0$ for every
$g\in\G$;
\item
$\rho$ is semi-conjugate to the trivial representation;
\item
$e_b(\rho)=0$.
\end{enumerate}
\end{prop}
\begin{proof}
(1) $\Longrightarrow$ (3): Suppose that $\rho$ fixes the point $x_0$, and consider the cocycle $c_{x_0}$ representing
the bounded Euler class $e_b\in H^2(\omeo,\matZ)$. By definition, the pull-back $\rho^*(c_{x_0})$ vanishes identically,
and this implies that $e_b(\rho)=\rho^*(e_b)=0$.

The implication (3) $\Longrightarrow$ (2) is an immediate consequence of Proposition~\ref{Euler:char:prop}.

(2) $\Longrightarrow$ (1): If $\rho$ is semi-conjugate to the trivial representation $\rho_0$, then there exists 
an increasing map  $\varphi\colon S^1\to S^1$  of degree one such that 
$\rho({g})\varphi=\varphi\rho_0({g})=\varphi$ for every $g\in\G$.
Let now $x_0=\varphi(y_0)$ be a point in the image of $\varphi$. By evaluating the above equality at $y_0$, we get
 $$\rho(g)(x_0)=\rho({g})\varphi(y_0)=\varphi(y_0)=x_0$$ for every $g\in\G$. Therefore,
 $\rho$ fixes $x_0$, and we are done.
\end{proof}

We go on in our analysis of semi-conjugacy with the following:

\begin{lemma}\label{costante:lem}
Let $\rho_1,\rho_2\colon\G\to\omeo$ be representations, and let $\varphi\colon S^1\to S^1$ be an increasing map
of degree one such that
$$
\rho_1(g)\varphi=\varphi\rho_2(g)\qquad {\rm
for\ every}\ g\in\G\ .
$$ 
Let also $\widetilde{\varphi}$ be a good lift of $\varphi$, and denote by $\widetilde{\rho}_i(g)$ the preferred
lift of $\rho_i(g)$, $g\in\G$, $i=1,2$ (recall that such a lift is determined by the request that
$\widetilde{\rho}_i(0)\in [0,1)$). 
If
$\rho_1(\G)$ does not fix a point, then there exists a bounded map $u\colon\G\to\matZ$ such that
$$
\widetilde{\rho_1(g)}\widetilde{\varphi}=\widetilde{\varphi}\widetilde{\rho_2(g)}\tau_{u(g)}
$$
for every $g\in\G$.
\end{lemma}
\begin{proof}
Let us fix $g\in\G$, and set $f_i=\rho_i(g)$ and
$\widetilde{f}_i=\widetilde{\rho_i(g)}$.
Since $\widetilde{f}_1\widetilde{\varphi}$ and $\widetilde{\varphi}\widetilde{f}_2$ are both lifts of the same
map $f_1\varphi=\varphi f_2$, for every $x\in\R$ the value
$\tau_g(x)=\widetilde{f}_1\widetilde{\varphi}(x)-\widetilde{\varphi}\widetilde{f}_2(x)$ is an integer. 
Moreover, of course we have $\tau_g(x+k)=\tau_g(x)$ for every $x\in \R$, $k\in\matZ$.
We need to show that $\tau_g\colon\R\to\matZ$ is constant, and bounded independently of $g$. 

Since
$\varphi$ cannot be constant (because otherwise $\rho_1(\G)$ would fix the point in the image of $\varphi$),
we may choose two points $x,y\in S^1$ such that $\varphi(x)\neq \varphi(y)$, and we set
$\widetilde{K}=p^{-1}(\{x,y\})$. By construction, the restriction of $\widetilde{\varphi}$ to $\widetilde{K}$ is injective.
Since $\widetilde{f}_1$ is injective, this implies at once that  the composition $\widetilde{f}_1\widetilde{\varphi}$
is also injective on $\widetilde{K}$. Moreover, since $\varphi$ is injective on $\{x,y\}$, we have
$f_1(\varphi(x))\neq f_1(\varphi(y))$, so $\varphi(f_2(x))\neq \varphi(f_2(y))$, and  $\widetilde{\varphi}\widetilde{f}_2|_{\widetilde{K}}$ is injective too.

We first show that $\tau_g|_{\widetilde{K}}$ is constant.
Let $\widetilde{k},\widetilde{k}'$ lie in $\widetilde{K}$. 
%Since $\tau$ is $\matZ$-periodic,
%we may assume that
%$\widetilde{k}< \widetilde{k}'<\widetilde{k}+1$.  
Since both $\widetilde{\psi}_1=\widetilde{f}_1\widetilde{\varphi}$ and $\widetilde{\psi}_2=\widetilde{\varphi}\widetilde{f}_2$
are strictly increasing on $\widetilde{K}$,
we may apply Lemma~\ref{conto}, thus getting 
$$
|\tau_g(\widetilde{k}')-\tau_g(\widetilde{k})|=|\widetilde{\psi}_2(\widetilde{k}')-\widetilde{\psi}_1(\widetilde{k})-
(\widetilde{\psi}_2(\widetilde{k}')-\widetilde{\psi}_1(\widetilde{k}))|<1\ .
$$
Since $\tau_g$ has integral values, this implies that $\tau_g(\widetilde{k}')=\tau_g(\widetilde{k})$.

Let now $\widetilde{z}\in\R\setminus \widetilde{K}$ be fixed. 
In order to show that $\tau_g$ is constant it is sufficient to show that there exists $\widetilde{k}\in\widetilde{K}$
such that $\tau_g(\widetilde{z})=\tau_g(\widetilde{k})$. Our assumptions imply that 
$\widetilde{K}\cap (\widetilde{z},\widetilde{z}+1)$ contains at least two elements. Since $\widetilde\varphi$ is strictly
monotone on $\widetilde{K}$, this implies that there exists $\widetilde{k}\in\widetilde{K}$ such that $\widetilde\varphi$ 
is strictly monotone on $\{\widetilde{z}\}\cup(\widetilde{k}+\matZ)$. Therefore, we may apply again 
Lemma~\ref{conto} and conclude that $|\tau_g(\widetilde{z})-\tau_g(\widetilde{k})|<1$, so $\tau_g(\widetilde{z})=\tau_g(\widetilde{k})$. We have thus shown that $\tau_g$ is constant.

We are left to show that the value taken by $\tau_g$ is bounded independently of $g$. 
Let $a\in\R$ be such that $\widetilde{\varphi}([0,1])\subseteq [a,a+1]$.
For every $g\in \G$, $i=1,2$ we have that $\widetilde{\rho_i(g)}(0)\in [0,1)$, so
$$
|\widetilde{\varphi}(\widetilde{\rho_2(g)}(0))|\leq |a|+1
$$
and
\begin{align*}
|\widetilde{\rho_1(g)}(\widetilde{\varphi}(0))|& \leq |\widetilde{\varphi}(0)|+ |\widetilde{\rho_1(g)}(\widetilde{\varphi}(0))-\widetilde{\varphi}(0)-(\widetilde{\rho_1(g)}(0)-0)|+ |\widetilde{\rho_1(g)}(0)-0|\\ & <|\widetilde{\varphi}(0)|+1+1\leq |a|+3
\end{align*}
(see Lemma~\ref{conto}).
Therefore, we may conclude that
$$
|u(g)|=|\widetilde{\rho_1(g)}(\widetilde{\varphi}(0))-\widetilde{\varphi}(\widetilde{\rho_2(g)}(0))|<|\widetilde{\rho_1(g)}(\widetilde{\varphi}(0))|+|\widetilde{\varphi}(\widetilde{\rho_2(g)}(0))|<2|a|+4\ ,
$$
and we are done.
\end{proof}

\begin{rem}\label{Elia:rem}
The following example, which was suggested to the author by Elia Fioravanti, shows that, even in the case when $\rho_1$ is semi-conjugate to $\rho_2$,
the previous lemma does not hold if drop the hypothesis that $\rho_1(\G)$ does not have any fixed point. 
With notation as in the proof of Lemma~\ref{costante:lem}, let $f_2\in\omeo$  admit the lift
$\widetilde{f}_2(x+n)=x^2+n$ for every $x\in [0,1)$, $n\in\mathbb{Z}$, 
and let $\widetilde{\varphi}(x)=\lfloor x+1/2\rfloor$ be a good lift of the increasing map
of degree one $\varphi$ which sends every point of $S^1$ to $[0]$. Then we have
$$
\varphi(f_2(x))=[0]=\varphi(x)
%f_1(\varphi(x))=f_1([0])=[0]=\varphi(x)
$$
for every $x\in S^1$. Therefore, if $\rho_1\colon\matZ\to \omeo$ is the trivial representation and
$\rho_2=\rho_{f_2}\colon\mathbb{Z}\to\omeo$ sends $1$ to $f_2$, then $\varphi$ satisfies the hypotheses of Lemma~\ref{costante:lem}
(and $\rho_1$ is semi-conjugate
to $\rho_2$ by Proposition~\ref{zeroprop}).
But if as a lift of $\rho_1(1)={\rm Id}_{S^1}$ we choose the identity $\widetilde{f}_1$ of $\R$, then 
$$
\widetilde{f}_1(\widetilde{\varphi}(0))-\widetilde{\varphi}(\widetilde{f}_2(0))=\widetilde{f}_1(0)-\widetilde{\varphi}(0)=0\ ,
$$
while
$$
\widetilde{f}_1(\widetilde{\varphi}(1/2))-\widetilde{\varphi}(\widetilde{f}_2(1/2))=\widetilde{f}_1(1)-\widetilde{\varphi}(1/4)=1\ .
$$
\end{rem}

\bigskip

We are now ready to prove the main result of this chapter, that we recall here for the convenience of the reader:

\begin{thm}
Let $\rho_1,\rho_2$ be representations of $\G$ into $\omeo$. Then
$\rho_1$ is semi-conjugate
to $\rho_2$ if and only if $e_b(\rho_1)=e_b(\rho_2)$.
\end{thm}
\begin{proof}
We already know from Proposition~\ref{Euler:char:prop} that representations sharing the same bounded Euler class are semi-conjugate.

Let us suppose that $\rho_1$ is semi-conjugate
to $\rho_2$. If $\rho_1(\G)$ fixes a point of $S^1$, then Proposition~\ref{zeroprop} implies that $e_b(\rho_1)=e_b(\rho_2)=0$, and we are done.
Therefore, we may assume that  $\rho_1(\G)$ does not fix a point in $S^1$.

We denote by $\widetilde{\varphi}$ a good lift of $\varphi$, and by $u\colon \G\to\matZ$ the bounded function provided by Lemma~\ref{costante:lem}. In order to conclude that $e_b(\rho_1)=e_b(\rho_2)$ it is sufficient to show that,
if $c$ is the canonical representative of the bounded Euler class of $\omeo$, then
\begin{equation}\label{boundary:eq}
\rho_1^*(c)-\rho_2^*(c)=\delta u\ .
\end{equation}
Let us fix $g_1,g_2\in \G$. By the definition of $u$ we have
\begin{align*}
\widetilde{\varphi}&=\widetilde{\rho_1(g_1)}\circ \widetilde{\varphi}\circ\left(\widetilde{\rho_2(g_1)}\right)^{-1}\circ \tau_{-u(g_1)}\ ,\\
\widetilde{\varphi}&=\widetilde{\rho_1(g_2)}\circ \widetilde{\varphi}\circ\left(\widetilde{\rho_2(g_2)}\right)^{-1}\circ \tau_{-u(g_2)}\ ,\\
\widetilde{\varphi}&=\left(\widetilde{\rho_1(g_1g_2)}\right)^{-1}\circ \widetilde{\varphi}\circ\widetilde{\rho_2(g_1g_2)}\circ \tau_{u(g_1g_2)}\ .
\end{align*}
Now we substitute the first equation in the right-hand side of the second, and the third equation in the right-hand side of the obtained
equation. Using that
$$
\tau_{c(\rho_i(g_1),\rho_i(g_2))}=\widetilde{\rho_i(g_1)}\widetilde{\rho_i(g_2)}\left(\widetilde{\rho_i(g_1g_2)}\right)^{-1}
$$
we finally get
$$
\widetilde{\varphi}=\widetilde{\varphi}\circ \tau_{\rho_1^*(c)(g_1,g_2)-\rho_2^*(c)(g_1,g_2)-u(g_1)-u(g_2)+u(g_1g_2)}\ ,
$$
which is equivalent to equation~\eqref{boundary:eq}. 
\end{proof}

\begin{comment}VERO
\begin{cor}\label{fixed:cor}
Let $\rho\colon \G\to \omeo$ be a representation. Then $\rho(\G)$ admits a fixed point if and only if
$e_b(\rho)=0$.
\end{cor}
\begin{proof}
Suppose first that $\rho(\G)$ fixes the point $x_0\in S^1$.
Then the cocycle $\rho^*(c_{x_0})$ identically vanishes, so $e_b(\rho)=0$. On the other hand,
if $e_b(\rho)=0$, then Theorem~\ref{Euler:thm} implies that $\rho$ is semi-conjugate to the trivial representation $\rho_0$.
As a consequence, there exists an increasing map of degree one $\varphi\colon S^1\to S^1$ such that
$$
\rho(g)(\varphi(x))=\varphi(\rho_0(g)(x))=\varphi(x)
$$
for every $g\in \G$, $x\in S^1$. Therefore, the image of $\varphi$ consists of points which are fixed by $\rho(\G)$.
\end{proof}
\end{comment}

In the case of a single homeomorphism we obtain the following:

\begin{cor}\label{rot:cor}
Let $f\in\omeo$. Then:
\begin{enumerate}
\item
$\rot(f)=0$ if and only if $f$ has a fixed point.
\item
$\rot(f)\in\mathbb{Q}/\matZ$ if and only if $f$ has a finite orbit.
\item
If $\rot(f)=[\alpha]\notin \mathbb{Q}/\matZ$ and every orbit of $f$ is dense, then $f$ is conjugate
to the rotation of angle $\alpha$ (the converse being obvious).
\end{enumerate}
\end{cor}
\begin{proof}
Claim (1) follows from Proposition~\ref{zeroprop}, and claim~(2) follows from claim (1) and the fact that
$\rot(f^n)=n\rot(f)$ for every $n\in\matZ$ (see Corollary~\ref{rot:hom}).

If $f$ is as in (3), then Theorem~\ref{Euler:thm} implies that $f$ is semi-conjugate to the rotation of angle $\rot(f)$, so the conclusion
follows from Proposition~\ref{dense:orbits}.
\end{proof}

\section{The canonical real bounded Euler cocycle}
As it should be clear from the previous chapters, bounded cohomology with real coefficients enjoys many more properties that bounded cohomology with integral coefficients. 
Therefore, it makes sense to study the image of the integral bounded Euler class via the change of coefficients map
$\Z\to\R$. It turns out that, when working with real cochains rather than with integral ones, a preferred representative of the bounded Euler class exists, which was defined by Matsumoto 
in~\cite{Matsu:numerical}. In fact, this representative coincides with the unique \emph{homogeneous} representative of the real bounded Euler class,
according to the definition of homogeneous cocyles given by Bouarich~\cite{Bua1} and described in Section~\ref{homo:coc}.

\begin{defn}\label{real:euler:defn}
The real Euler class $e^\R\in H^2(\omeo,\R)$ 
(resp.~bounded real Euler class $e_b^\R\in H_b^2(\omeo,\R)$) 
is the image of $e\in H^2(\omeo,\matZ)$ (resp.~of $e_b\in H_b^2(\omeo,\matZ)$) under the  change of coefficients
homomorphism. If $\rho\colon \G\to\omeo$
is any representation, then we also set $e^\R(\rho)=\rho^*(e^\R)$ and $e_b^\R(\rho)=\rho^*(e_b^\R)$.
\end{defn}

Let us recall from Section~\ref{rotation:sec} that, for every $f\in\omeo$, one has 
$$
\rot(f)=[\rott (\widetilde{f})]\ \in \R/\Z\ ,
$$
where $\widetilde{f}$ is any lift of $f$ to $\omeot$, and
$$
\rott(\widetilde{f})=\lim_{n\to+\infty} \frac{\widetilde{f}^n(0)}{n}\ \in \R\ . 
$$

Of course, we can interpret $\rott$ as an element of $\overline{C}^1(\omeot,\R)$, and we now set 
$$
\widetilde{\tau}=-\overline{\delta} \rott \ \in \overline{C}^2(\omeot,\R)\ .
$$
By construction, $\widetilde{\tau}$ is a cocycle, and 
$$
\widetilde{\tau}(\widetilde{f},\widetilde{g})=\rott(\widetilde{f}\widetilde{g})-\rott(\widetilde{f})-\rott(\widetilde{g})\ .
$$
It readily follows from the definition that, if $\tau_n$ is the translation $x\mapsto x+n$, then $\rott(\widetilde{f}\circ \tau_n)=\rott(\widetilde{f})+n$. 
Using this, it is immediate to check that $\widetilde{\tau}$ descends to a cochain $\tau\in \overline{C}^2(\omeo,\R)$ such that 
$$
\tau(f,g)=
\widetilde{\tau}(\widetilde{f},\widetilde{g})=\rott(\widetilde{f}\widetilde{g})-\rott(\widetilde{f})-\rott(\widetilde{g})
$$
for any pair of arbitrary lifts $\widetilde{f},\widetilde{g}$ of $f,g$, respectively. Since $\widetilde{\tau}$ is a cycle, also $\tau$ is so.

\begin{prop}
The cocycle $\tau$ is bounded. Moreover, 
$\tau$ is a representative of the real bounded Euler class $e_b^\R\in H^2_b(\omeo,\R)$.
\end{prop}
\begin{proof}
%The proof of Proposition~\ref{homo:eu} describes how to construct the homogeneous representative of any real bounded class of degree $2$. Let us 
%apply the argument developed there
%to our current case of interest. 
For every element $f\in\omeo$ we denote by $\widetilde{f}\in\omeot$ the unique lift of $f$ such that $\widetilde{f}(0)\in [0,1)$. 
Let $c\in \overline{Z}^2_b(\omeo,\Z)$ be the usual representative of $e_b$, so that
$$
c(f,g)=\widetilde{f}(\widetilde{g}(x_0))-\widetilde{fg}(x_0)
$$
for every $x_0\in\R$. 

Let $p\colon \omeot\to\omeo$ be the covering map. We have proved in Section~\ref{rotation:sec} that  $p^*(c)=\overline{\delta}\psi$, where $\psi\in \overline{C}^1(\omeot,\R)$ is defined by
$$
\psi(h)=\widetilde{h}^{-1}\circ \widetilde{p(h)}\in \ker p=\Z\ 
$$
for every $h\in\omeot$.
Moreover, the homogeneous quasimorphism $\psi_h$ associated to $\psi$ satisfies $\psi_h=-\rott$ (see again Section~\ref{rotation:sec}). 
Let now $\widetilde{b}\in \overline{C}^1_b(\omeot,\R)$ be the bounded cochain such that $\psi-\psi_h=\widetilde{b}$. 
We have already observed that,
if $\tau_n$ is the translation $x\mapsto x+n$, then 
$\psi_h(\widetilde{f}\circ \tau_n)=\psi_h(\widetilde{f})+n$. The same equality also holds when $\psi_h$ is replaced by $\psi$, so 
$\widetilde{b}(\widetilde{f}\circ\tau_n)=\widetilde{b}(\widetilde{f})$, and
$\widetilde{b}=p^*(b)$ for a bounded cochain $b\in\overline{C}^1_b(\omeo,\R)$. 
Recall now that $\overline{\delta}{\psi}$ is the pull-back via $p$ of the Euler cocycle $c$.
From $\widetilde{\tau}=\overline{\delta}\psi_h=\overline{\delta}(\psi-\widetilde{b})$ we thus deduce
$$
\tau=c-\overline{\delta}b\ .
$$
This shows at once that $\tau$ is bounded, and that $[\tau]=e_b^\R$ in $H^2_b(\omeo,\R)$.
\end{proof}

The previous results allow us to provide the following characterization of semi-conjugacy in terms of the \emph{real} bounded Euler class. We say that a subset $S$ of $\G$ generates
the first homology of $\G$ if the set $\{\pi(s),\, s\in S\}$ generates $\G/[\G,\G]$, where $\pi\colon \G\to \G/[\G,\G]$ is the natural projection.

\begin{thm}[\cite{Matsu:numerical}]\label{Matsu:char}
 Let $\rho_1,\rho_2\colon \G\to\omeo$ be representations, and suppose that $S\subseteq \G$ generates the first homology of $\G$. Then $\rho_1$ and $\rho_2$ are semi-conjugate if and only if the following conditions hold:
 \begin{enumerate}
  \item $\tau(\rho_1(g),\rho_1(g'))=\tau(\rho_2(g),\rho_2(g'))$ for every $g,g'\in \G$;
  \item $\rot(\rho_1(g))=\rot(\rho_2(g))$ for every $g\in S$, where $S$ generates the first homology of $\G$.
 \end{enumerate}
\end{thm}
\begin{proof}
 Let us first suppose that $\rho_1$ and $\rho_2$ are semi-conjugate. Then Theorem~\ref{Euler:thm} ensures that $e_b(\rho_1)=e_b(\rho_2)$, hence $e^\R_b(\rho_1)=e^\R_b(\rho_2)$.
 Now the cocycle $\tau$ is homogeneous according to the definition given in Section~\ref{homo:coc}. Therefore,
 by Proposition~\ref{homo:eu}, $\rho_i^*(\tau)$ is the unique homogeneous representative of $e_b^\R(\rho_i)$ in $H^2_b(\G,\R)$, $i=1,2$.
 Since $e_b^\R(\rho_i)=e_b^\R(\rho_2)$, we thus have $\rho_1^*(\tau)=\rho_2^*(\tau)$.
 Moreover, for every $g\in\G$ the representations $\alpha_1,\alpha_2\colon \Z\to\omeo$, $\alpha_i(n)=\rho_i(g^n)$ are semi-conjugate, and again by Theorem~\ref{Euler:thm} 
 (together with our definition of rotation number) this implies
 that $\rot(\rho_1(g))=\rot(\rho_2(g))$.
 
 Let us now assume that conditions (1) and (2) hold. The short exact sequence of complexes
 $$
 0\to C^\bullet_b(\G,\Z)\to C_b^\bullet(\G,\R)\to C^\bullet(\G,\R/\Z)\to 0
 $$
 gives rise to the following exact sequence:
 $$
 \xymatrix{
 H^1_b(\G,\R)=0 \ar[r] & H^1(\G,\R/\Z)\ar[r]^{\delta} & H^2_b(\G,\Z)\ar[r] & H^2_b(\G,\R)\ .
 }
 $$
Condition (1) implies that $e_b^\R(\rho_1)=e_b^\R(\rho_2)$, so there exists a unique $\psi\in H^1(\G,\R/\Z)=\Hom(\G,\R/\Z)$ such that
$e_b(\rho_1)-e_b(\rho_2)=\delta\psi$. Let us now fix $s\in S$, and consider the map $\alpha\colon \Z\to \G$ given by $n\mapsto s^n$. Let also
$\rho'_i\colon \Z\to \omeo$ be defined by $\rho_i'=\rho_i\circ\alpha$. Let us consider the commutative diagram:
$$
\xymatrix{
H^1(\G,\R/\Z)\ar[r]^{\delta} \ar[d]^{H^1(\alpha)} & H^2_b(\G,\Z)\ar[d]^{H^2_b(\alpha)}\\
H^1(\Z,\R/\Z)\ar[r]^{\delta} & H^2_b(\Z,\Z)\ .
 }
 $$
 Our definition of rotation number implies that
 $$
 \rot(\rho_i(s))=\delta^{-1}(e_b(\rho'_i))\ ,
 $$
 so condition (2) implies that 
 $$
 H^1(\alpha)(\psi)=\rot(\rho_1(s))-\rot(\rho_2(s))=0\ .
 $$
 Recall now that there is a canonical isomorphism $H^1(\G,\R/\Z)\cong \Hom (\G,\R/\Z)$, so we may consider $\psi$ as a homomorphism
 defined on $\G$, and the fact that $H^1(\alpha)(\psi)=0$ implies that $0=\psi(\alpha(1))=\psi(s)$. We have thus shown that
 $\psi$ vanishes on $S$. Since $\R/\Z$ is abelian and $S$ generates the first homology of $\G$, this suffices to conclude that $\psi=0$,
 hence $0=\delta\psi=e_b(\rho_1)-e_b(\rho_2)$, and $\rho_1$ is semi-conjugate to $\rho_2$ by Theorem~\ref{Euler:thm}.
 \end{proof}

 \begin{cor}
  Let $\rho\colon \G\to \omeo$ be a representation. Then the following conditions are equivalent:
  \begin{enumerate}
   \item $e_b^\R(\rho)=0$;
   \item $\rho$ is semi-conjugate to a representation
  into the subgroup $SO(2)<\omeo$ of rotations;
  \item $\tau(\rho(g),\rho(g'))=0$ for every $g,g'\in\G$.
   \end{enumerate}
   \end{cor}
\begin{proof}
(1) $\Longleftrightarrow$ (3): The cocycle $\tau$ is homogeneous according to the definition given in Section~\ref{homo:coc}. Therefore,
 by Proposition~\ref{homo:eu}, $\rho^*(\tau)$ is the unique homogeneous representative of $e_b^\R(\rho)$,
 and this concludes the proof.

 (3) $\Longrightarrow$ (2):
Condition (3) readily implies that the map $g\mapsto \rot(\rho(g))$ defines a homomorphism $\G\to \R/\Z$. Let now $\rho'\colon \G\to SO(2)$ be defined
by letting $\rho'(g)$ be equal to the rotation of angle $ \rot(\rho(g))$. Then $\rho'$ is indeed a representation, and $\rho,\rho'$ satisfy conditions (1), (2) of Theorem~\ref{Matsu:char}.
Therefore, $\rho$ and $\rho'$ are semi-conjugate.

(2) $\Longrightarrow$ (1):
 Suppose that $\rho$ is semi-conjugate to a representation $\rho'\colon \G\to SO(2)$. Since $SO(2)$ is abelian, we have $H^2_b(SO(2),\R)=0$. 
 Since the map $(\rho')^*\colon H^2_b(\omeo,\R)\to H^2_b(\G,\R)$ factors through $H^2_b(SO(2),\R)$, 
 this implies
 that $e_b^\R(\rho')=0$. But Theorem~\ref{Euler:thm} implies that $e_b(\rho)=e_b(\rho')$, hence $e_b^\R(\rho)=e_b^\R(\rho')=0$.
\end{proof}

\section{Further readings}
As we mentioned in the introduction of this book, bounded cohomology 
may be exploited to refine numerical or cohomological invariants first defined in the context of classical cohomology. 
Ghys' Theorem~\ref{Euler:thm} provides a noticeable example of this phenomenon.
Indeed, Ghys' theory settles other natural questions about the (bounded) Euler class of representations into $\omeo$, e.g.~it provides
a complete characterization of bounded classes that can be realized by circle actions:

\begin{thm}[{\cite[Theorem B]{Ghys0}}]
 Let $\G$ be a discrete countable group and take $\alpha\in H^2_b(\G,\mathbb{Z})$. There exists a representation $\rho\colon \Gamma\to\omeo$ such that $\alpha=e_b(\rho)$ if and only if $\alpha$ can be represented by 
a cocycle taking only the values 0 and 1.
\end{thm}

Ghys' theory of the bounded Euler class has found applications in many different directions. For example, it has recently been exploited 
in the bounded cohomology approach to higher Teichm\"uller theory~\cite{BIW1,BIW2,BSBH}.

The definition of semi-conjugacy given in this book first appeared in an old manuscript by Takamura~\cite{Taka}, which was never published. Its equivalence with other definitions occurring in the literature is
discussed in detail in~\cite{BFH}. 

 An interesting issue concerning the definition of semi-conjugacy concerns the regularity of the increasing degree one maps involved. 
 In fact, it can be shown that both the increasing degree one maps appearing in the definition of semi-conjugacy may be chosen to be upper semicontinuous (see e.g.~\cite{BFH}).
 On the contrary, there is no hope in general for semi-conjugate circle actions to be semi-conjugate via continuous maps (see e.g.~\cite[Example 6.1]{BFH}). Nonetheless,
  semi-conjugacy admits a characterization in terms of  continuous maps of Hopf-Brouwer degree one: namely, semi-conjugacy is the equivalence relation generated by continuous left-semi-conjugacies of Brouwer-Hopf degree one
  (the circle action $\rho_1$ is \emph{left}-semi-conjugate to the circle action $\rho_2$ if there exists an increasing map $\varphi$ of degree one such that
$\rho_1(g)\varphi=\varphi\rho_2(g)$ for every $g\in \G$). We refer the reader to \cite[Theorem 1.7]{Cal07}, \cite{BFH} for a proof of this fact;
note that what we call here a left-semi-conjugacy via a continuous map of Brouwer-Hopf degree one is simply called a semi-conjugacy in \cite{Cal07}, while 
the equivalence relation generated by continuous left-semi-conjugacies of Brouwer-Hopf degree one (which is eventually equivalent to semi-conjugacy) 
is called monotone equivalence in \cite{Cal07}.

\begin{comment}
 \begin{rem}\label{Michelle:equi:rem}
In~\cite{Bucherweb}, two representations $\rho_1,\rho_2\colon\G\to \omeo$ are defined to be semi-conjugate if and only if the following conditions hold:
\begin{enumerate}
 \item There exists an increasing map $\varphi\colon S^1\to S^1$ of degree one such that
$$
\rho_1(g)\varphi=\varphi\rho_2(g)\qquad {\rm
for\ every}\ g\in\G\ ,
$$
\item
there exist $\rho(\G_i)$-invariant subsets $K_i\subseteq S^1$ such that $\varphi$ restricts to a bijection between
$K_2$ and $K_1$.
\end{enumerate}
Our definition of semi-conjugacy is indeed equivalent to Bucher's one. In fact, if $\rho_1$ and $\rho_2$ are semi-conjugate (according to our definition),
then condition (1) above holds by definition, and condition (2) is implied by Lemma~\ref{uppersemi} and Proposition~\ref{Michelle:equal}.
On the other hand, if condition (1) holds, then $\rho_1$ and $\rho_2$ are semi-conjugate by the implication (2) $\Rightarrow$ (1) in 
Proposition~\ref{Michelle:equal}, which holds without any assumption on the upper semicontinuity of $\varphi$.
\end{rem}
\end{comment}

\chapter{The Euler class of sphere bundles}

This chapter is devoted to a brief digression from the theory of bounded cohomology. Namely, we review here some classical definitions and constructions
coming from the theory of fiber bundles, restricting our attention to the case when the fiber is a sphere. The main aim of the chapter is a detailed description
of the Euler class of a sphere bundle. The Euler class is usually defined in the context of cellular homology. Of course,
cellular and singular cohomology are canonically isomorphic for every CW-complex. However, we already stressed the fact that 
cellular cochains are not useful to compute the \emph{bounded} cohomology of a cellular complex: for example,
it is not easy to detect whether
the singular coclass corresponding to a cellular one admits a bounded representative or not. Therefore, differently from what is usually done, in this chapter we carry out the construction of the Euler cocycle 
directly in
the context of singular cohomology. 

 \section{Topological, smooth and linear sphere bundles}\label{fiber:bundle}
 If $X,Y$ are topological spaces, we denote by $\pi_X\colon X\times Y\to X$, $\pi_Y\colon X\times Y\to Y$ the natural projections.
 Let $M$ be a topological space, and fix $n\geq 1$. An \emph{oriented $n$-sphere bundle} over $M$ is a map
 $\pi\colon E\to M$ such that the following conditions hold:
 \begin{itemize}
  \item For every $x\in M$, the subspace $E_x=\pi^{-1}(x)\subseteq E$ is homeomorphic to the $n$-dimensional sphere $S^n$; 
  \item Each $E_x$ is endowed with an orientation;
  \item 
  For every $x\in M$ there exist a neighbourhood $U$ and a homeomorphism 
 $\psi\colon \pi^{-1}(U)\to U\times S^n$ such that the following diagram commutes
$$
\xymatrix{
\pi^{-1}(U) \ar[r]^\psi \ar[rd]_{p} & U\times S^n\ar[d]^{\pi_U} \\
& U
}
$$
and the homeomorphism $\pi_{S^n}\circ \psi|_{E_x}\colon E_x\to S^n$ is orientation-preserving for
every $x\in U$. Such a diagram provides a \emph{local trivialization} of $E$ (over $U$).
 \end{itemize}

Henceforth 
we refer to oriented $n$-sphere bundles
simply as to $n$-sphere bundles, or to sphere bundles when the dimension is understood. 
A $1$-sphere bundle will be called \emph{circle bundle}.
An isomorphism between two sphere bundles
$\pi\colon E\to M$, $\pi'\colon E'\to M$ is a homeomorphism $h\colon E\to E'$ such that
$\pi=\pi'\circ h$. 
 
The definition above fits into the more general context of (sphere) bundles with structure group. In fact,
when $G$ is a subgroup of $\omeon$,
we say that the map $\pi\colon E\to M$ is 
a \emph{sphere bundle with structure group $G$} (or, simply, a \emph{sphere $G$-bundle})
if $M$ admits an open cover $\mathcal{U}=\{U_i,\, i\in I\}$ such that the
following conditions hold: every $U_i\in\mathcal{U}$ is the base of a local trivialization
$$
\xymatrix{
\pi^{-1}(U_i) \ar[r]^{\psi_i} \ar[rd]_{p} & U_i\times S^n\ar[d]^{\pi_{U_i}} \\
& U_i
}
$$
and for every $i,j\in I$ there exists a map $\lambda_{ij}\colon U_i\cap U_j\to G$ such that
$$
\psi_i(\psi_j^{-1}(x,v))=(x,\lambda_{ij}(x)(v))
$$
for every $x\in U_i\cap U_j$ and $v\in S^n$. Such a collection of local trivializations is called a \emph{$G$-atlas} for $E$, and the 
$\lambda_{ij}$'s are called the \emph{transition functions} of the atlas. 
In the general theory of bundles with structure group $G$, usually $G$ is a \emph{topological} group, and transition maps are required to be continuous.
In this monograph, unless otherwise stated, we endow every subgroup $G<\omeon$ with the compact-open topology, i.e.~the topology 
of the uniform convergence with respect to any metric inducing the usual topology of $S^n$ (but see also Remark~\ref{flat:discrete}).
If this is the case, then the transition maps
are automatically continuous.
As usual, we understand that two atlases define the same structure of $G$-bundle if their union
is still a $G$-atlas. Equivalently, $G$-bundle structures on $M$ correspond to \emph{maximal} $G$-atlases for $M$. 

An isomorphism of sphere $G$-bundles $\pi\colon E\to M$, $\pi'\colon E'\to M$  is an isomorphism of sphere bundles
$h\colon E\to E'$
such that,  for every $x\in M$, there exist a neighbourhood $U$ of $x$ and trivializations $\psi\colon \pi^{-1}(U)\to U\times S^n$,
$\psi'\colon (\pi')^{-1}(U)\to U\times S^n$ in the maximal $G$-atlases of $E$ and $E'$ such that
$$
\psi'(h(\psi^{-1}(x,v)))=(x,f(x)(v))
$$
for some function $f\colon U\to G$.

In this monograph we will be interested just in the following cases:
\begin{enumerate}
\item $G=\omeon$:  in this case, standard  properties of the compact-open topology on 
$\omeon$ (in fact, on any space of homeomorphisms of a compact metric space) imply that
$E$ is just what we called a \emph{sphere bundle}. Moreover, two sphere bundles are isomorphic as $\omeon$-bundles if and only if they are
isomorphic according to the definition given at the beginning of the section. Henceforth, we will sometimes stress these facts by
saying that $E$ is a \emph{topological} sphere bundle.
\item $G={\rm Diffeo}^+ (S^n)$ (endowed with the $C^\infty$-topology): in this case we say that $E$ is a \emph{smooth sphere bundle}; we also say that two smooth sphere bundles
are smoothly isomorphic if they are isomorphic as sphere ${\rm Diffeo}^+ (S^n)$-bundles.
\item $G={\rm Gl}^+(n+1,\mathbb{R})$ (endowed with the compact-open topology, which coincides both with the $C^\infty$-topology and with the topology of $G$ as a Lie group), where the action of $G$ is defined via the identification of $S^n$ with the space of rays in $\mathbb{R}^{n+1}$:
in this case we say that $E$ is a \emph{linear sphere bundle}; we also say that two linear sphere bundles
are linearly isomorphic if they are isomorphic as sphere ${\rm Gl}^+(n+1,\mathbb{R})$-bundles.
\end{enumerate}

\begin{rem}
Since ${\rm Gl}^+(n+1,\mathbb{R})\subseteq {\rm Diffeo}_+(S^n)$, every linear sphere bundle admits a unique compatible structure of smooth
sphere bundle. On the other hand, there could exist sphere bundles which
are not isomorphic to smooth sphere bundles, and smooth sphere bundles that are not smoothly isomorphic (or even topologically isomorphic) to linear bundles.
Also observe that two smooth fiber bundles could be topologically isomorphic without being smoothly isomorphic, and two linear fiber bundles could be topologically isomorphic or 
even smoothly isomorphic without being linearly isomorphic.
\end{rem}

If $\pi\colon E\to M$ is a vector bundle, i.e.~a 
locally trivial bundle with fiber $\R^{n+1}$ and structure group
${\rm GL}^+(n+1,\R)$, then we may associate to $E$ the $n$-sphere bundle $S(E)$ in the following way.
Let us denote by $E_0\subseteq E$ the image of the zero section $\sigma_0\colon M\to E$.
As a topological space, $S(E)$ is the quotient of $E\setminus E_0$ by the equivalence relation given by
$v\equiv w$ if $v,w$ belong to the same fiber $E_x$ of $E$ and $v=\lambda\cdot w$ for some $\lambda\in\R^+$,
where the product is defined according to the natural structure of real vector space on $E_x$. It is immediate to check
that the bundle map $\pi\colon E\to M$ induces a continuous map $S(E)\to M$ which is an $n$-sphere bundle, called
the \emph{sphere bundle} associated to $E$. In fact, by construction $S(E)$ admits a ${\rm Gl}^+(n+1,\R)$-atlas,
so it admits a natural structure of \emph{linear} sphere bundle. Conversely, every linear sphere bundle $\pi\colon S\to M$ admits an atlas
with transition functions with values in ${\rm GL}^+(n+1,\R)$. Such an atlas also defines a vector bundle
$\pi\colon E\to M$ with fiber $\R^{n+1}$, and by construction we have $S(E)=S$. In other words, there is a natural correspondence
between linear $n$-sphere bundles and vector bundles of rank $n+1$ over the same base $M$.

\section{The Euler class of a sphere bundle}\label{Euler:bundle}
Let us briefly recall the definition of the Euler class $\eu(E)\in H^{n+1}(M,\matZ)$ associated
to an $n$-sphere bundle $\pi\colon E\to M$. We follow the approach via obstruction theory, i.e.~we describe $\eu(E)$
as the obstruction to the existence of a section of $\pi$
(and we refer the reader e.g.~to \cite{Steenrod} for more details).
%So, let us fix a cellular structure on $M$, and let us denote by $M^k$ the $k$-skeleton of $M$ for every $k\in \mathbb{N}$. 

Let $s\colon \Delta^k\to M$ be a singular $k$-simplex. A section $\sigma$ over $s$ is a continuous map  $\sigma\colon \Delta^k\to E$
such that $\sigma (x)\in E_{s(x)}$ for every $x\in\Delta^k$. Equivalently, if we denote by $E_s$ the sphere bundle over
$s$ given by the pull-back of $E$ via $s$, 
then
$\sigma_s$ may be seen as a genuine section of $E_s$. Observe that the sphere bundle $E_s$ is always trivial, 
since the standard simplex is contractible.

Let us now fix the dimension $n$ of the fibers of our sphere bundle $\pi\colon E\to M$.
We will inductively construct compatible sections over all the singular $k$-simplices with values in $M$,
for every $k\leq n$. More precisely, for every $k\leq n$, for every singular $k$-simplex
$s\colon \Delta^k\to M$, we will define a section $\sigma_s\colon \Delta^k\to E$ over $s$ in such a way
that the restriction of $\sigma_s$ to the $i$-th face of $\Delta^k$ coincides
with the section $\sigma_{\partial_i s}\colon \Delta^{k-1}\to E$ over the $i$-th face $\partial_i s\colon \Delta^{k-1}\to M$ of $s$.

We first choose, for every 0-simplex (i.e.~point) $x\in M$, a section $\sigma_x$ over $x$, i.e.~an element
$\sigma_x$ of the fiber $E_x$. Let $s\colon [0,1]\to M$ be a singular $1$-simplex such that
$\partial s=y-x$. Then, since $E_s$ is trivial and $S^n$ is path connected, it is possible to choose
a section $\sigma_s\colon [0,1]\to E$ over $s$ such that $\sigma_s(1)=\sigma_y$ and $\sigma_s(0)=\sigma_x$, and we fix such
a choice for every $1$-simplex $s$. If $n=1$ we are done. Otherwise, we assume that the desired sections have been constructed
for every singular $(k-1)$-simplex, and, for any given singular $k$-simplex $s\colon \Delta^k\to M$, we define $\sigma_s$ as follows.

%We now consider a $2$-simplex  $s\colon \Delta^2\to M$ be a singular $2$-simplex, and we would like to understand whether
%the sections just defined over $1$-simplices may be extended to $s$.
%To this aim, we fix a trivialization 
We fix a trivialization
$$
\xymatrix{
E_s \ar[r]^-{\psi_s} \ar[rd]_{\pi} & \Delta^k\times S^n\ar[d]^{\pi_{\Delta^k}} \\
& \Delta^k
}
$$
The inductive hypothesis ensures that the sections defined on the $(k-1)$-faces of $s$ coincide
on the $(k-2)$-faces of $s$, so they may be glued together into a section $\sigma_{\partial s}\colon \partial\Delta^k\to E_s$
of the restriction of $E_s$ to $\partial \Delta^k$. 
%In order to extend $\sigma_{\partial s}$ to a section over $\Delta^k$
%it is sufficient to extend the map 
Since $k\leq n$, we have $\pi_{k-1}(S^n)=0$, so
the composition
$$
\pi_{S^n}\circ \psi_s\circ \sigma_{\partial s}\colon \partial\Delta^k\to  S^n
$$
%to a section of the projection $\pi_{\Delta^k}\colon \Delta^k\times S^n\to \Delta^k$. However, the composition
%$\psi_s\circ \sigma_{\partial s}$ 
extends to a continuous $B\colon \Delta^k\to S^n$. Then, an extension of $\sigma_{\partial s}$ over $\Delta^k$
is provided
by the section
$$
\sigma_s\colon \Delta^k\to E_s\, ,\qquad \sigma_s(x)=\psi_s^{-1}(x,B(x))\ .
$$
%$\psi_s\circ \sigma_{\partial s}$ over $\Delta^k$ by setting

Let us now consider a singular $(n+1)$-simplex $s\colon \Delta^{n+1}\to M$. Just as we did above, we may fix a trivialization
$$
\xymatrix{
E_s \ar[r]^-{\psi_s} \ar[rd]_{\pi} & \Delta^{n+1}\times S^n\ar[d]^{\pi_{\Delta^{n+1}}} \\
& \Delta^{n+1}
}
$$
of $E_s$, and observe that the sections defined on the faces of $s$ glue into a section
$\sigma_{\partial s}\colon \partial \Delta^{n+1}\to E_s$. The obstruction to extend this section
over $\Delta^{n+1}$ is measured by the degree $z(s)$ of the map
%If $\partial s=\partial_0 s-\partial_1 s+\partial_2 s$, where $\partial_i s$ is the $i$-th face of $s$, then
%the sections $\sigma_{\partial_0 s}$, $\sigma_{\partial_1 s}$, $\sigma_{\partial_2 s}$ may be glued together
%to define a section $\sigma_{\partial s}\colon \partial\Delta^2\to E_s$. Let us now fix an orientation-preserving identification
%$h\colon S^1\to \partial\Delta^2$, and denote by $z(s)\in \mathbb{Z}$ the topological degree of the composition 
$$
\xymatrix{
%S^1\ar[r]^h & 
\partial\Delta^{n+1} \ar[r]^{\sigma_s} & E_s\ar[r]^-{\psi_s} &\Delta^{n+1}\times S^n\ar[r]^-{\pi_{S^n}} & S^n
}
$$
(this degree is well defined since $\partial\Delta^{n+1}$ inherits an orientation from the one of $\Delta^{n+1}$).
The following result shows that $z$ represents a well-defined cohomology class:

\begin{prop}\label{euler:welldef}
The singular
cochain $z\in C^{n+1} (M,\matZ)$ defined above satisfies the following properties:
\begin{enumerate}
 \item $z(s)$ does not depend on the choice of the (orientation-preserving) trivialization of $E_s$;
 \item $\delta z=0$, i.e.~$z$ is a cycle;
 \item different choices for the sections over $k$-simplices, $k\leq n$, lead to another cocycle $z'\in C^{n+1}(M,\matZ)$
 such that $[z]=[z']$ in $H^{n+1}(M,\matZ)$.
 \item if $z'\in C^{n+1}(M,\matZ)$ is a cycle such that $[z]=[z']$ in $H^{n+1}(M,\matZ)$, then $z'$ can be realized
 as the cocycle corresponding to suitable choices 
 for the sections over $k$-simplices, $k\leq n$.
\end{enumerate}
\end{prop}
\begin{proof}
 (1): It is readily seen that, if $\psi_s,\psi'_s$ are distinct trivializations of $E_s$ and  
 $\alpha,\alpha'\colon \partial \Delta^{n+1}\to S^n$ are given by
 $\alpha=\pi_{S^n}\circ \psi_s\circ \sigma_s$, $\alpha'=\pi_{S^n}\circ \psi'_s\circ \sigma_s$, then $\alpha'(x)=h(x)(\alpha(x))$
 for every $x\in\partial\Delta^{n+1}$, where $h$ is the restriction to $\partial\Delta^{n+1}$ of a continuous
 map $\Delta^{n+1}\to\omeon$. Being extendable to $\Delta^{n+1}$, the map $h$ is homotopic to a constant map. Therefore,
 $\alpha'$ is homotopic to the composition of $\alpha$ with a fixed element of $\omeon$, so $\deg\alpha=\deg\alpha'$,
 i.e.~$z(s)=z'(s)$,
 and this concludes the proof.
 
 (2): Let $s\in S_{n+2}(M)$, and fix a trivialization $\psi_s\colon E_s\to \Delta^{n+2}\times S^n$. 
 The fixed collection of compatible sections induces a section $\sigma'_s$ of $E_s$ over the $n$-skeleton $T$ of $\Delta^{n+2}$.
 For every $i=0,\ldots,n+2$, we denote by $F_i$ the geometric $i$-th face of $\Delta^{n+2}$, and we fix an integral fundamental cycle
 $c_i\in C_n(\partial F_i,\mathbb{Z})\subseteq C_n(T,\mathbb{Z})$ for $\partial F_i$ (here we endow $\partial F_i$ with the orientation inherited
 from $F_i$, which in turn is oriented as the boundary of $\Delta^{n+2}$). Then we have
 $\sum_{i=0}^{n+2} [c_i] =0$ in $H_n(T,\mathbb{Z})$. 
 Observe that
 $\psi_s$ restricts to a trivialization $\psi_i\colon E_{\partial_i s}\to \Delta^{n+1}\times S^n$ for any single face $\partial_i s$ of $s$.
 By (1), we may exploit these trivializations to compute $z(\partial_i s)$ for every $i$, thus getting
 $$z(\partial_i s)=H_n(\pi_{S^n}\circ\psi_s\circ \sigma'_{s})((-1)^i[c_i])\in \mathbb{Z} = H_n(S^n,\mathbb{Z})$$
 (the sign $(-1)^i$ is due to the fact that the orientations on $F_i$ induced respectively by the orientation of $\Delta^{n+1}$ and by the ordering of the vertices of $F_i$ agree if and only if $n$ is even). 
 Since the composition $\pi_{S^n}\circ\psi_s\circ \sigma'_{s}$ is defined over $T$, this implies that
 $$
 z(\partial s)=\sum_{i=0}^{n+2} (-1)^i z(\partial_i s)=
 H_n(\pi_{S^n}\circ\psi_s\circ \sigma'_{s})\left(\sum_{i=0}^{n+2} [c_i]\right)=0\ .
$$ 
 Since this holds for every $s\in S_{n+2}(M)$, we conclude that $z$ is a cycle.
 
 (3): Let us fix two  families of compatible sections $\{\sigma_s\, |\, s\in S_n(M)\}$, $\{\sigma'_s\, |\, s\in S_n(M)\}$, and observe that, thanks to compatibility, these families also
 define 
families of sections $\{\sigma_s\, |\, s\in S_j(M),\, j\leq n\}$, $\{\sigma'_s\, |\, s\in S_j(M),\, j\leq n\}$. 
Using that $\pi_k(S^n)=0$ for every $k<n$, it is not difficult to show that 
these families are \emph{homotopic} 
up to dimension $n-1$, i.e.~that 
there exists a collection of compatible homotopies $\{H_s\, |\, s\in S_j(M),\,  j\leq n-1\}$ between the $\sigma_s$'s and the $\sigma_s'$'s, where $s$ ranges
over all the singular simplices of dimension at most $n-1$. Here the compatibility
condition requires that the restriction of $H_s$ to the face $\partial_i s$ is equal to $H_{\partial_i s}$. This implies in turn that we may modify the  family of compatible sections
 $\{\sigma'_s\, |\, s\in S_n(M)\}$ without altering the induced Euler cocycle, in such a way that the restrictions of $\sigma_s'$ and of $\sigma_s$ to $\partial \Delta^{n}$
 coincide for every $s\in S_n(M)$. 

Let us now fix a singular simplex $s\in S_n(M)$, and a trivialization $\psi_s\colon E_s\to \Delta^{n}\times S^n$. 
Let $\Delta^n_1,\Delta^n_2$ be copies of the standard simplex $\Delta^n$,
and endow $\Delta^n_1$ (resp.~$\Delta^n_2$) with the same (resp.~opposite) orientation of $\Delta^n$. Also denote by $S$ the sphere obtained 
by gluing $\Delta^n_1$ with $\Delta^n_2$ along their boundaries via the identification $\partial\Delta^n_1\cong \partial \Delta^n\cong \partial \Delta^n_2$.
Observe that $S$ admits an orientation which extends the orientations of $\Delta^n_i$, $i=1,2$. 
As discussed above, we may assume that the restrictions of $\sigma_s'$ and of $\sigma_s$ to $\partial \Delta^{n}$
 coincide for every $s\in S_n(M)$, so we may define a continuous map $v\colon S\to E_s$ which coincides with $\sigma_s$ on $\Delta^n_1$
 and with $\sigma'_s$ on $\Delta^n_2$. After composing with $\psi_s$ and projecting to $S^n$, this map gives rise to a map
 $S\to S^n$. We denote its degree by $\varphi(s)\in\mathbb{Z}$. Now it is easy to show that, if $z$ and $z'$ are the Euler cocycles associated respectively
 with the $\sigma_s$'s and the $\sigma'_s$'s, then $z-z'=\delta\varphi$.

(4): Suppose that $z\in C^{n+1}(M,\matZ)$ is the representative of $\eu(E)$ corresponding to the  collection of compatible sections
$\{\sigma_s,\, s\in S_n(M)\}$. 
Our assumption implies that $z'=z-\delta\varphi$ for some $\varphi\in C^n(M,\matZ)$.
For every $s\in S_n(M)$, let us modify $\sigma_s$ as follows. Let $\Delta^n_1,\Delta^n_2$ and   $S$ be as in the previous point.
We consider a trivialization
$$
\xymatrix{
E_s \ar[r]^-{\psi_s} \ar[rd]_{\pi} & \Delta^n\times S^n\ar[d]^{\pi_{\Delta^n}} \\
& \Delta^n
}
$$
and denote by $g\colon \Delta^n_1\to S^n$ the map $g=\pi_{S^n}\circ\sigma_s$, where we are identifying $\Delta^1_n$ with $\Delta^n$.
We extend $g$ to a map $G\colon S\to S^n$ such that $\deg G=\varphi(s)$. If $\overline{g}\colon \Delta^n\cong \Delta^n_2\to S^n$ is the restriction
of $G$ to $\Delta^n_2$, then we set 
$$
\sigma_s'\colon \Delta^n\to E_s\, ,\qquad \sigma_s'(x)=\psi_s^{-1}(x,\overline{g}(x))\ .
$$
With this choice, $\sigma_s'$ and $\sigma_s$ coincide on $\partial \Delta^n$, so 
the sections of the collection  $\{\sigma'_s,\, s\in S_n(M)\}$ are still compatible.
 Moreover, it is easily checked that the Euler cocycle corresponding to the new collection of sections is  equal to $z'$.
\end{proof}

The previous proposition implies that the following definition is well posed:

\begin{defn}
 The  \emph{Euler class} $\eu(E)\in H^{n+1}(M,\matZ)$ of $\pi\colon E\to M$ is the coclass
represented by the cocycle $z$ described above. 
\end{defn}

\begin{rem}
In the case when $M$ is a CW-complex, the construction above may be simplified a bit. In fact, using  that
$\pi_k(S^n)=0$ for $k<n$,
one may choose a global section of $E$ over the $n$-skeleton $M^{(n)}$ of $M$. Then, working with 
$(n+1)$-cells rather than with singular $(n+1)$-simplices, one may associate an integer to every $(n+1)$-cell of $M$, thus getting
a cellular $(n+1)$-cochain, which is in fact a cocycle whose cohomology class does not depend on the choice of
the section on $M^{(n)}$. This approach is explained in detail e.g.~in \cite{Steenrod}. Our approach via singular (co)homology
has two advantages:
\begin{enumerate}
 \item
 It allows a neater discussion of the boundedness of the Euler class. 
 \item 
It shows directly that, even when $M$ admits the structure of a CW-complex, the Euler class $\eu(E)$
does not depend on the realization of $M$ as a CW-complex. 
\end{enumerate}
\end{rem}

\section{Classical properties of the Euler class}

The Euler class is natural, in the following sense:

\begin{lemma}\label{Euler:functorial}
 Let $\pi\colon E\to M$ be a sphere $n$-bundle, and let $f\colon N\to M$
 be any continuous map. If $f^*E$ denotes the pull-back of $E$ via $f$, then
 $$
 e(f^*E)=H^{n+1}(f)(e(E))\ .
 $$
\end{lemma}
\begin{proof}
 Via the map $f$, a collection of compatible sections of $E$ over the simplices in $M$ gives rise to
 a collection of compatible sections of $f^*E$ over the simplices in $N$. The conclusion follows at once.
\end{proof}

By construction, the Euler class provides an obstruction for $E$ to admit a section:

\begin{prop}\label{secteu}
If the sphere bundle $\pi\colon E\to M$ admits a section, then $\eu(E)=0$.
\end{prop}
\begin{proof}
Let  $\sigma\colon M\to E$ be a global section. If $s\colon \Delta^k\to M$ is any singular simplex, then
we may define a section $\sigma_{s}$ of $E_s$ simply as the pull-back of $\sigma$ via $s$. With this choice,
for every singular $n$-simplex $s$, the partial section $\sigma_{\partial s}$ of $E_s$ extends from
$\partial \Delta^n$ to $\Delta^n$. As a consequence, the Euler class admits an identically vanishing representative, so $\eu(E)=0$.
\end{proof}

It is well-known that the vanishing of the Euler class does not ensure the existence of a section in general (see e.g.~\cite[Example 23.16]{Bott-Tu}).
However, we have the following positive result in this direction:

\begin{prop}\label{eusect}
Let $\pi\colon E\to M$ be an $n$-sphere bundle over a simplicial complex.
Then $\eu(E)=0$ if and only if the restriction of $E$ to the $(n+1)$-skeleton of $M$ admits a section.
\end{prop}
\begin{proof}
By Proposition~\ref{secteu}, we need to construct a section over $M^{(n+1)}$ under the assumption that $\eu(E)=0$.
Let $S_{n}(M)$ be the set of singular $n$-simplices with values in $M$.
By Proposition~\ref{euler:welldef}--(4), there exists 
 a collection
$\{\sigma_s,\, s\in S_n(M)\}$ of compatible sections over the elements of $S_n(M)$ such that the corresponding Euler cocycle is identically equal to zero.
 These sections may be exploited to define a global
section $\sigma\colon M^{(n)}\to E$ over the $n$-skeleton of $M$.
Using that the corresponding Euler cocycle vanishes, it is easy to show that
this section extends over $M^{(n+1)}$, and we are done.
\end{proof}

Much more can be said in the context of \emph{circle} bundles: indeed, the Euler class completely classifies the isomorphism type of the bundle
(see Theorem~\ref{Euler:classifies} below). In particular, a circle bundle with vanishing Euler class is topologically trivial, so it admits a section. 
In arbitrary dimension, the existence of a section for an  $n$-sphere bundle is guaranteed when the base space is a closed oriented $(n+1)$-manifold:

\begin{prop}\label{euse}
Let $\pi\colon E\to M$ be an $n$-sphere bundle over a smooth closed oriented $(n+1)$-manifold.
Then $\eu(E)=0$ if and only if $E$ admits a section.
\end{prop}
\begin{proof}
Since $M$ can be smoothly triangulated, the conclusion follows from Proposition~\ref{eusect}.
\end{proof}

In the case described in the previous proposition, the Euler class may be completely described by an integer:
\begin{defn}
Let $\pi\colon E\to M$ be an $n$-sphere bundle, and suppose that $M$ is a closed oriented $(n+1)$-manifold.
Then $H^n(M,\matZ)$ is canonically isomorphic to $\matZ$, so $\eu(E)$ may be identified with an integer, still denoted
by $\eu(E)$. We call such an integer the \emph{Euler number} of the fiber bundle $E$.
\end{defn}

The problem of which integral cohomology classes may be realized as Euler classes of some sphere bundle is open in general:
for example, a classical result by Milnor, Atiyah and Hirzebruch implies that, if $m\neq 1,2,4$, then the Euler class
of any linear $n$-sphere bundle over $S^{2m}$ is of the form $2\cdot \alpha$ for some $\alpha\in H^{n+1}(S^{2m},\matZ)$
\cite{Mil,AH}, while, if $k$ is even, there is a number $N(k,n)$ such that for every $n$-dimensional CW-complex $X$ and every 
cohomology class $a\in H^k(X,\matZ)$, there exists a linear $(k-1)$-sphere bundle $E$ over $X$ such that $\eu(E)=2N(k,n)a$~\cite{GSW}.
Here we just prove the following:

\begin{prop}\label{euler:torsion}
Let $\pi\colon E\to M$ be a linear $n$-sphere bundle, where $n$ is even. Then $2\eu(E)=0\in H^{n+1}(M,\matZ)$.
\end{prop}
\begin{proof}
Since $E$ is linear, every element of the structure group of $E$ commutes with the antipodal map of $S^n$. As a consequence,
there exists a well-defined involutive bundle map $i\colon E\to E$ such that the restriction of $i$ to the fiber
$E_x$ has degree $(-1)^{n+1}=-1$ for every $x\in M$. 

Let now $S_n(M)$ be the set of singular $n$-simplices with values in $M$, take a collection of compatible sections
$\{\sigma_s,\, s\in S_n(M)\}$, and let $z$ be the corresponding Euler cocycle.
By composing each $\sigma_s$ with $i$, we obtain a new collection of compatible sections
$\{\sigma'_s,\, s\in S_n(M)\}$, which define the Euler cocycle $z'=-z$. Therefore, we have
$\eu(E)=-\eu(E)$, whence the conclusion.
\end{proof}

It is not known to the author whether Proposition~\ref{euler:torsion} holds for
any \emph{topological} sphere bundle.

\begin{cor}
Let $\pi\colon E\to M$ be a linear $n$-sphere bundle over a closed oriented $(n+1)$-manifold, where $n$ is even.
Then $\eu(E)=0$, so $E$ admits a section.
\end{cor}
\begin{proof}
Since $H^{n+1}(M,\matZ)\cong \matZ$ is torsion-free, the conclusion follows from Propositions~\ref{euler:torsion} and~\ref{euse}.
\end{proof}

\section{The Euler class of  oriented vector bundles}

The notion of Euler class may be extended also to the context of oriented vector bundles: namely,
if $\pi\colon E\to M$ is such a bundle, then we simply set
$$
\eu(E)=\eu(S(E))\ .
$$

If the rank of $E$ is equal to $2$, then we may associate to $E$ another (smooth) sphere bundle, the 
\emph{projective bundle} $\mathbb{P}(E)\to M$. Such a bundle is defined as follows: the fiber
of $\mathbb{P}(E)$ over $x\in M$ is just the projective space $\mathbb{P}(E_x)$, i.e.~the space of
$1$-dimensional subspaces of $E_x$. 
Using the fact that
the one-dimensional real projective space $\mathbb{P}^1(\R)$ is homeomorphic to the circle, it is not difficult to check
that $\mathbb{P}(E)\to M$ is a smooth circle bundle, according to our definitions (in fact, a ${\rm Gl}^+(2,\mathbb{R})$-atlas for
$E$ induces a $\mathbb{P}{\rm Gl}^+(2,\mathbb{R})$-atlas for
$\mathbb{P}(E)$).
Then, using that the natural projection $\R^2\supseteq S^1\to \mathbb{P}^1(\R)$ has degree 2, it is easy to show that
$$
\eu(\mathbb{P}(E))=2\eu(S(E))=2\eu(E)\ .
$$

%The following classical result compute 
The Euler number of the tangent bundle of a closed manifold is just equal to the Euler characteristic of $M$:
%, and show
%that the Euler number of an $n$-sphere bundle over an $(n+1)$-closed manifold vanishes if $n$ is even.
%The following well-known result computes the Euler number of the tangent bundle of a closed surface.

\begin{prop}\label{tangentchi}
Let $M$ be a closed oriented smooth manifold, and denote by $TM$ the tangent bundle of $M$. Then
%Let $g\geq 0$, endow $\Sigma_g$ with a smooth structure, and denote by 
%$\pi\colon T\Sigma_g\to \Sigma_g$ the associated tangent bundle. Then
$$
\eu(TM)=\chi(M)\ .
$$
 \end{prop}
 \begin{proof}
Let $\tau$ be a smooth triangulation of $M$. We first exploit $\tau$ to construct a continuous vector field $X\colon M\to TM$
having only isolated zeroes. Let $n=\dim M$. 
For every $k$-simplex $\Delta$ of $\tau$,  $k\leq n$, we choose a point $p_\Delta$ in the interior of $\Delta$ (where we understand that,
if $\dim \Delta=0$, i.e.~if $\Delta$ is a vertex of $\tau$, then $p_\Delta=\Delta$). Then we build our vector field
$X$ in such a way that the following conditions hold:
\begin{enumerate}
\item 
$X$ is null exactly at the $p_\Delta$'s;
\item 
for every $\Delta$ with $\dim\Delta=k$, there exist local coordinates $(x_1,\ldots,x_n)$ on a neighbourhood $U$ of $p_\Delta$ such that
 $U\cap{\Delta}=\{x_{k+1}=\ldots=x_n=0\}$ and 
 $$
 X|_U=-\left(x_1\frac{\partial}{\partial x_1}+\ldots+x_k\frac{\partial}{\partial x_k}\right) + x_{k+1}\frac{\partial}{\partial x_{k+1}}+\ldots x_n\frac{\partial}{\partial x_n}
 $$
(so $p_\Delta$ is a sink for the restriction of $X$ to $\Delta$ and a source for
the restriction of $X$ to a transversal section of $\Delta$). 
\end{enumerate}

The fact that such a vector field indeed exists is well-known; we refer the reader to Figure~\ref{field:fig} for a picture in the $2$-dimensional case.
%This may be achieved just by gluing suitable local vector fields that on each $n$-dimensional simplex of
%$\tau$ (see Figure~\ref{field:fig} for the case $n=2$). If $t,e,v$ are respectively the number of triangles, edges and vertices of $\tau$, then
%the vector field $X$ has exactly $v$ repulsive points, $e$ saddle points and $t$ attractive points. 

\begin{center}
 \begin{figure}
  \includegraphics{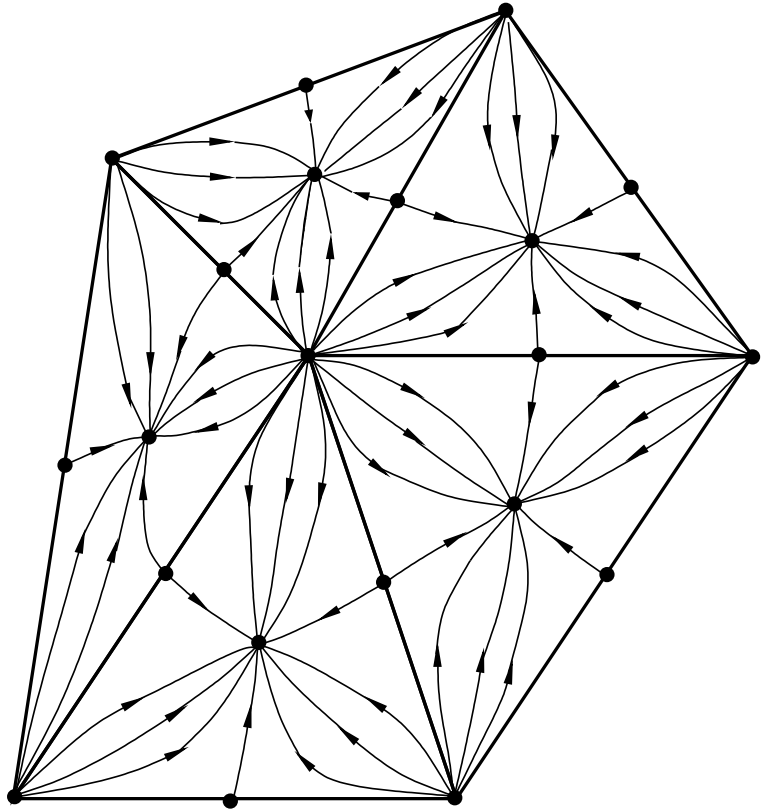}
 \caption{The vector field described in the proof of Proposition~\ref{tangentchi}. Here we describe a portion of a $2$-dimensional triangulation.}\label{field:fig}
 \end{figure}
\end{center}

Let us now choose another triangulation $\tau'=\{\Delta'_1,\ldots,\Delta'_m\}$ of $M$ such that no zero of $X$ lies on the $(n-1)$-skeleton of $\tau'$, 
every $\Delta'_i$ contains at most one zero of $X$ and, if $\Delta'_i$ contains a zero of $X$, then $\Delta'_i$ is so small that
the restriction of $X$ to $\Delta_i$ admits a local description as in item (2) above. We can choose  orientation-preserving 
parameterizations $s_i\colon \Delta^n\to M$ of the $\Delta'_i$
in such a way that the chain $c=s_1+\ldots+s_m$ represents a generator
of $H_n(M,\matZ)$. 

The vector field $X$ defines a nowhere vanishing section on the $(n-1)$-skeleton $(\tau')^{(n-1)}$ of $\tau'$, which induces in turn a section
from  $(\tau')^{(n-1)}$ to $S(TM)$. We may use this section to define $z(s_i)$ for every $i=1,\ldots, m$, where $z$ is the representative of the Euler class constructed above. If $\Delta'_i$ does not contain any singularity of $X$, then $X$ may be extended over $\Delta'_i$, so $z(s_i)=0$. On the other hand,
if $\Delta'_i$ contains the singularity $p_\Delta$, where $\dim\Delta=k$, then $z(s_i)$ is equal to the degree of the map
$$
S^{n-1}\to S^{n-1}\, ,\quad (x_1,\ldots,x_n)\mapsto (-x_1,\ldots,-x_k,x_{k+1},\ldots,x_n) ,
$$
so $z(s_i)=(-1)^k$. As a consequence we finally get 
\begin{align*}
\eu(TM)&=\sum_{i=1}^m z(s_i) =
\sum_{j=1}^n (-1)^k \cdot \#\{\Delta\in \tau\, |\, \dim\Delta=k\}=\chi(M)\ .
\end{align*}
 \end{proof}

\section{The euler class of circle bundles}
As mentioned above, many results about the Euler class of sphere bundles admit stronger versions in the $1$-dimensional case, i.e.~when the fiber is homeomorphic to $S^1$.
This is also due to some topological properties of $\omeo$ (like asphericity) that do not hold for $\omeon$, $n\geq 2$. 
Recall that, unless otherwise stated, we understand that $\omeo$ and $\omeot$  are endowed with the compact-open topology
(we refer the reader to Section~\ref{omeo:sec} for the definition of $\omeot$).
The following result provides an explicit description of the homotopy type of $\omeo$:

\begin{lemma}\label{topological:omeo}
\begin{enumerate}
\item
The topological group $\omeot$ is contractible. 
\item
The map $p_*\colon \omeot\to\omeo$ induced by the projection $p\colon\R\to S^1$ is a universal covering. 
\item
If $x$ is any point in $S^1$, the map 
$$
v_x\colon \omeo\to S^1\, ,\qquad v_x(f)=f(x)
$$
is a homotopy equivalence.
\end{enumerate}
\end{lemma}
\begin{proof}
(1): The map $\omeot\times [0,1]\to \omeot$ defined by
$(f,t)\mapsto tf+(1-t){\rm Id}_{\R}$ defines a contraction of $\omeot$, which, therefore, is contractible.

(2): The fact that $p_*$ is a covering is an easy exercise, so $p_*$ is a universal
covering by claim (1).

(3): Let $r_x\colon S^1\to\omeo$ be defined as follows: for every $y\in S^1$, $r_x(y)$ is the unique
rotation of $S^1$ taking $x$ into $y$. We obviously have $v_x\circ r_x={\rm Id}_{S^1}$, so we only have to check
that $r_x\circ v_x$ is homotopic to the identity of $\omeo$. Let us fix a point $\widetilde{x}\in p^{-1}(x)\subseteq \R$. For every $f\in\omeo$ we
fix an arbitrary lift $\widetilde{f}\in \omeo$ of $f$. Then, we define 
$$
\widetilde{H}\colon \omeo\times [0,1]\to\omeot\, ,\quad \widetilde{H}(f,t)=(1-t)\widetilde{f}+t\cdot \tau_{\widetilde{f}(\widetilde{x})-\widetilde{x}}\ ,
$$
and we set
$$
H\colon \omeo\times [0,1]\to \omeo\, ,\quad H=p_*\circ\widetilde{H}\ .
$$
It is immediate to check that the definition of $H(f,t)$ does not depend on the arbitrary choice of the lift $\widetilde{f}$. Using this, one easily checks
that $H$ is indeed continuous. The conclusion follows from the fact that $H(f,0)=f$ and $H(f,1)=r_x(v_x(f))$ for every $f\in\omeo$.
\end{proof}

\begin{cor}\label{quellocheserve}
Let $\sigma_1\colon S^1\to S^1$ be a continuous map,  let $b\colon S^1\to \omeo$ be a loop, and let $\sigma_2\colon S^1\to S^1$
be defined by $\sigma_2(x)=b(x)(\sigma_1(x))$. 
Then $b$ is null-homotopic if and only if $\deg \sigma_1=\deg\sigma_2$.
\end{cor}
\begin{proof}
%If $b$ is null homotopic, then there exists a homotopy $H\colon S^1\times [0,1]\to \omeo$ such that
%$H(x,0)=b(x)$ and $H(x,1)={\rm Id}_{S^1}$ for every $x\in S^1$. The map
%$H'(x,t)=H(x,t)(\sigma_1(x))$ provides a homotopy between $\sigma_2$ and $\sigma_1$, so
%$\deg\sigma_1=\deg\sigma_2$.

%Let us now suppose that $\deg\sigma_1=\deg\sigma_2$.

Let us consider the map
$$
F\colon S^1\times S^1\to S^1\, ,\qquad F(x,y)=b(x)(\sigma_1(y))\ .
$$
If we denote by $\gamma\colon S^1\to S^1\times S^1$ the loop $\gamma(x)=(x,x)$, then $\sigma_2=F\circ\gamma$. Observe now that, if 
$p_0\in S^1$ is any basepoint and
$\alpha\colon S^1\to S^1\times S^1$,
$\beta\colon S^1\to S^1\times S^1$ are the loops defined respectively by $\alpha(x)=(p_0,x)$, $\beta(x)=(x,p_0)$, then
$[\gamma]=[\alpha *\beta]$ in $H_1(S^1\times S^1,\matZ)$. Therefore, 
\begin{equation}\label{degrees:eq}
\deg \sigma_2\cdot [S^1]=(\sigma_2)_*([S^1])\cdot [S^1]=F_*(\gamma_*([S^1]))\cdot [S^1]=F_*((\alpha*\beta)_*([S^1]))\cdot [S^1]\ .
\end{equation}
But $F\circ (\alpha*\beta)=\gamma_1*\gamma_2$, where $\gamma_1(x)=b(p_0)(\sigma_1(x))$ and $\gamma_2(x)=b(x)(\sigma_1(p_0))$. 
Since $b(p_0)$ is an orientation-preserving homeomorphism we have $[\gamma_1]=[\sigma_1]=\deg\sigma_1\cdot [S^1]$ in $H_1(S^1,\matZ)$, so using~\eqref{degrees:eq}
we get
$$
\deg \sigma_2\cdot [S^1]=[\gamma_1]+[\gamma_2]=\deg\sigma_1\cdot [S^1] +[\gamma_2]\ ,
$$
which implies that $\deg \sigma_1=\deg\sigma_2$ if and only if $\gamma_2$ is null-homotopic. By Lemma~\ref{topological:omeo}, the latter condition holds if and only if $b$ is null-homotopic, whence the conclusion. 

\begin{comment}
Let $\widetilde{\sigma}_i\colon [0,1]\to \R$ be a lift of the composition $[0,1]\to S^1\tto{\sigma_i} S^1$. 
%of $\sigma_i$ with the obvious projection $[0,1]\to S^1$.
Let also $\widetilde{b}\colon [0,1]\to \omeot$ be a lift of the map $[0,1]\to S^1\tto{b}\omeo$ such that
$\widetilde{b}(0)(\widetilde{\sigma}_1(0))=\widetilde{\sigma}_2(0)$. Being lifts of the map
$[0,1]\to S^1\tto{\sigma_2} S^1$ which coincide in $0$, the paths 
$t\mapsto \widetilde{b}(t)(\widetilde{\sigma}_1(t))$ and $\widetilde{\sigma}_2$ coincide for every $t\in [0,1]$. In particular,
we have $\widetilde{b}(1)(\widetilde{\sigma}_1(1))=\widetilde{\sigma}_2(1)$, so
\begin{equation}\label{beq}
\begin{array}{cl}
\widetilde{b}(1)(\widetilde{\sigma}_1(0))+\deg \sigma_1 & =
\widetilde{b}(1)(\widetilde{\sigma}_1(0)+\deg\sigma_1)=
\widetilde{b}(1)(\widetilde{\sigma}_1(1)) =
\widetilde{\sigma}_2(1)\\ & =\widetilde{\sigma}_2(0)+\deg\sigma_2= \widetilde{b}(0)(\widetilde{\sigma}_1(0))+\deg\sigma_2\ .
\end{array}
\end{equation}
Therefore, $\deg\sigma_1=\deg\sigma_2$ if and only if $\widetilde{b}(1)(\widetilde{\sigma}_1(0))=
\widetilde{b}(0)(\widetilde{\sigma}_1(0))$. Since two lifts of $b$ coincide if and only if they coincide on a point, we conclude
that $\deg\sigma_1=\deg\sigma_2$ if and only if $\widetilde{b}$ is a loop. But $\widetilde{b}$ is a loop if and inly if
$b$ is homotopically trivial, and this concludes the proof.
\end{comment}
\end{proof}

We are now almost ready to prove that the Euler class completely classifies the isomorphism type of circle bundles. 
 
 \begin{lemma}\label{tecnico1}
 Let $\sigma,\sigma'\colon [0,1]\to S^1$ be continuous paths, and let $h_0,h_1\in\omeo$ be such that
 $h_i(\sigma(i))=\sigma'(i)$, $i=0,1$. Then there exists a path $\psi\colon [0,1]\to \omeo$ such that $\psi(i)=h_i$ for $i=0,1$,
 and $\psi(t)(\sigma(t))=\sigma'(t)$ for every $t\in [0,1]$.
 \end{lemma}
  \begin{proof}
  Let $\widetilde{\sigma},\widetilde{\sigma}'\colon [0,1]\to\R$ be continuous lifts of $\sigma,\sigma'$, and choose
  $\widetilde{h}_i\in\omeot$ which lifts $h_i$ and is such that
  $\widetilde{h}_i(\widetilde{\sigma}(i))=\widetilde{\sigma}'(i)$, $i=0,1$.
For $t\in [0,1]$ we now define 
\begin{align*}
\widetilde{v}_t\in\omeot\, ,\qquad  & \widetilde{v}_t(x)=(1-t)\widetilde{h}_0(x)+t\widetilde{h}_1(x)\ ,\\
\widetilde{\psi}_t\in\omeot\, , \qquad & \widetilde{\psi}_t(x)=\widetilde{v}_t(x)+
\widetilde{\sigma}'(t)-\widetilde{v}_t(\widetilde{\sigma}(t))\ .
\end{align*}  
Then the required path is obtained by projecting the path $t\mapsto\widetilde{\psi}_t$ from $\omeot$ to $\omeo$.
  \end{proof}
  
 \begin{thm}\label{Euler:classifies}
 Let $M$ be a simplicial complex, and let $\pi\colon E\to M$, $\pi'\colon E'\to M$ be circle bundles
 such that $\eu(E)=\eu(E')$ in $H^2(M,\mathbb{Z})$. Then $E$ is (topologically) isomorphic to $E'$.
 \end{thm}
\begin{proof}
Since $\eu(E)=\eu(E')$, we may choose collections of compatible sections $\{\sigma_s\}_{s\in S_1(M)}$ and
$\{\sigma'_s\}_{s\in S_1(M)}$ respectively for $E$ and $E'$ in such a way that the corresponding Euler cocycles
$z,z'\in C^2(M,\matZ)$ coincide (see Proposition~\ref{euler:welldef}--(4)).
Such sections glue up into partial sections $\sigma\colon M^{(1)}\to E_1$, $\sigma'\colon M^{(1)}\to E'_1$, where
$M^{(1)}$ is the $1$-skeleton of $M$, and $E_1$ (resp.~$E'_1$) is the restriction of $E$ (resp.~$E'$) to
 $M^{(1)}$. 
 
 For every vertex $v$ of $M$ we choose an orientation-preserving homeomorphism $h_v\colon E_v\to E'_v$ such that $h_v(\sigma(v))=\sigma'(v)$.
 The collections of these maps is just a bundle map between the restrictions of $E$ and of $E'$ to the $0$-skeleton of $M$.
 Using Lemma~\ref{tecnico1}, we may extend this map to
 a bundle map $h\colon E_1\to E'_1$ such that $h(\sigma(x))=\sigma'(x)$ for every $x\in M$.
 
 We will show that $h$ can be extended to a bundle map on the whole $2$-skeleton of $M$.
 Let $\Delta$ be a 2-dimensional simplex of $M$, and choose trivializations $\psi_\Delta\colon E_\Delta\to \Delta\times S^1$,
 $\psi'_\Delta\colon E'_\Delta\to \Delta\times S^1$ of the restrictions of $E$ and $E'$ over $\Delta$. 
 For every $x\in\partial\Delta$, the restriction of the composition $\psi'_\Delta\circ h\circ \psi_\Delta^{-1}$ to $\{x\}\times S^1$
 defines an orientation-preserving homeomorphism of $\{x\}\times S^1$, which gives in turn an element $b(x)\in\omeo$. Moreover, the map
 $b\colon \partial \Delta^2\to \omeo$ is continuous. Recall now that we have sections $\sigma|_{\partial\Delta}$ of $E$ and
 $\sigma'|_{\partial\Delta}$ of $E'$, which may be interpreted in the given trivializations as maps
 $\hat{\sigma},\hat{\sigma}'\colon \partial \Delta\to S^1$. Our choices imply that 
 $\deg\hat{\sigma}=\deg\hat{\sigma}'$, and that $b(x)(\hat{\sigma}(x))=\hat{\sigma}'(x)$ for every $x\in \partial \Delta$. By Corollary~\ref{quellocheserve},
 this implies that $b$ 
 may be extended to a continuous map $B\colon \Delta^2\to\omeo$. We
 define $k\colon \Delta^2\times S^1\to\Delta^2\times S^1$
 by setting $k(x,v)=(x,B(x)(v))$. Then, the composition
 $$(\psi')^{-1}_\Delta \circ k\circ \psi_\Delta\colon E_\Delta\to E'_\Delta$$ 
is a bundle map extending $h$ over $\Delta$. We have thus shown that the bundle map $h$ extends to a bundle map
$H$ between the restrictions of $E$ and $E'$ to the $2$-skeleton of $M$.

Suppose now that $H$ has been extended to the $k$-skeleton of $M$, where $k\geq 2$. 
If $\Delta$ is any $(k+1)$-simplex of $M$, 
the obstruction to extending
$H$ to $\Delta$ lies in the homotopy group $\pi_k(\omeo)$, which vanishes by Lemma~\ref{topological:omeo}.
Therefore, we can extend $H$ to a global bundle map between $E$ and $E'$, and this concludes the proof.
\end{proof}

 \section{Circle bundles over surfaces}
 We have seen above that recognizing which integral cohomology classes may be realized as Euler classes of some sphere bundle is highly non-trivial in general.
Nevertheless, in the case of closed surfaces it is easily seen that every integer is the Euler number of some
circle bundle:

\begin{prop}\label{esistono}
For every $n\in\matZ$ there exists a circle bundle $E_n$ over $\Sigma_g$ such that 
 $\eu(E_n)=n$.
 \end{prop} 
 \begin{proof}
 Let $\overline{D}\subseteq \Sigma_g$ be an embedded closed disk, let $D$ be the internal part of $\overline D$,
  and consider the surface with boundary $\Sigma_g'=\Sigma_g\setminus D$. Take the products 
  $E_1=\Sigma_g'\times S^1$ and $E_2=\overline{D}\times S^1$, and 
  fix an identification between $\partial \Sigma_g'=\partial\overline{D}$ and $S^1$, 
  which is orientation-preserving on $\partial \Sigma_g'$ and orientation-reversing on $\partial \overline{D}$.
Let now $n\in\matZ$ be fixed, and 
  glue $\partial E_1=(\partial \Sigma_g')\times S^1\cong S^1\times S^1$
  to $\partial E_2=(\partial \overline{D})\times S^1\cong S^1\times S^1$ via the attaching map $$ \partial E_1\to \partial E_2\, ,\qquad
  (\theta,\varphi)\mapsto
  (\theta,\varphi-n\theta)\ .$$
    If we denote by $E$ the space obtained via this gluing,
  the projections  $\Sigma_g'\times S^1\to\Sigma_g'$, $\overline{D}\times S^1\to\overline{D}$ glue into a well-defined
  circle bundle  $\pi\colon E\to \Sigma_g$. 
  
  Let us now fix a triangulation $\mathcal{T}$ of $\Sigma_g$ such that $\overline{D}$ coincides with one triangle $T_0$
  of $\mathcal{T}$, and consider the fundamental cycle $c$ for $\Sigma_g$ given by the sum
  of (orientation-preserving parameterizations of) the triangles of $\mathcal{T}$. Our construction provides preferred trivializations of $E_1$ and $E_2$.
  If $\sigma$ is the ``constant'' section of $E_1$ over $\Sigma_g'$, then
  $\sigma$ induces sections over all the $1$-simplices of $\mathcal{T}$. We may exploit such sections to compute
  $z(T)$ for every triangle $T$ of $\mathcal{T}$. By construction, we get $z(T_0)=n$ and $z(T)=0$ for every $T\neq T_0$.
  This implies that
  $\eu(E)=n$. 
 \end{proof}

 Recall now that Theorem~\ref{Euler:classifies} ensures that the Euler number completely classifies the isomorphism type  of topological circle bundles.
Therefore, putting together Proposition~\ref{esistono} and Theorem~\ref{Euler:classifies}
we get the following:

\begin{cor}
Let $\Sigma_g$ be a closed oriented surface. Then,
the Euler number establishes a bijection between the set of isomorphism classes of
topological circle bundles on $\Sigma_g$ and the set of integers.
\end{cor}

\section{Further readings}
There are many books where the interested reader may find more information on fiber bundles and on obstruction theory for fiber bundles. Among them,
without the aim of being in any way exhausting, here we just cite the books of Steenrod~\cite{Steenrod}, Husemoller~\cite{Husemoller}, and Milnor and Stasheff~\cite{MST}.

\chapter{Milnor-Wood inequalities and maximal representations}\label{Milnor-Wood:chapter}
This chapter is devoted to the study of a special class of fiber bundles, called \emph{flat bundles}. As we will see, (isomorphism classes of) flat bundles with fiber $F$ over a base space $M$ are
in bijection with conjugacy classes of representations of  $\pi_1(M)$ into the group of self-homeomorphisms of $F$. Therefore, the theory of flat bundles builds a bridge between the theory of 
fiber bundles
and the theory of representations. Here we will mainly concentrate our attention on circle bundles: in particular, we will show that the Euler class of a flat circle bundle corresponds (under the canonical
morphism between the cohomology of a space and the cohomology of its fundamental group) to the Euler class of the associated representation. This allows us to define the \emph{bounded} Euler class
of a flat circle bundle, which provides a more refined invariant than the classical Euler class. We will describe a phenomenon already anticipated in the introduction of this book: 
using the fact that we have an explicit bound on the norm of the Euler class, we will be able to bound from above the Euler number of flat circle bundles over surfaces, thus obtaining an explicit obstruction
for a circle bundle to admit a flat structure. This result was first proved (with different tools) by Milnor~\cite{Milnor} and Wood~\cite{Wood}, and is usually known under the name of \emph{Milnor-Wood} inequalities.

We will then study \emph{maximal} representations, i.e.~representations of surface groups
which attain the extremal values allowed by Milnor-Wood inequalities. A celebrated result by Goldman~\cite{Goldth} states that maximal representations
of surface groups into the group of orientation-preserving isometries of the hyperbolic plane are exactly the holonomies of hyperbolic structures. 
Following~\cite{BIW1}, we will give a proof of Goldman's theorem based on the use of bounded cohomology. In doing this, we will describe the Euler number of flat bundles over surfaces with boundary,
as defined in~\cite{BIW1}.

\section{Flat sphere bundles}
A \emph{flat} sphere bundle with structure group $G$ (or, simply, a \emph{flat} $G$-bundle) is 
a sphere bundle $\pi\colon E\to M$ endowed with an open cover $\mathcal{U}=\{U_i,\, i\in I\}$
of the base $M$ and a collection of local trivializations
$$
\xymatrix{
\pi^{-1}(U_i) \ar[r]^{\psi_i} \ar[rd]_{p} & U_i\times S^n\ar[d]^{\pi_{U_i}} \\
& U_i
}
$$
such that, for every $i,j\in I$, there exists a \emph{locally constant} map $\lambda_{ij}\colon U_i\cap U_j\to G$ such that
$$
\psi_i(\psi_j^{-1}(x,v))=(x,\lambda_{ij}(x)(v))
$$
for every $x\in U_i\cap U_j$, $v\in S^n$. Such a collection of local trivializations is  a \emph{flat $G$-atlas} for $E$, and the 
$\lambda_{ij}$'s are called the \emph{transition functions} of the atlas. Two flat 
$G$-atlases define the same structure of flat $G$-bundle if their union
is still a flat $G$-atlas.

An isomorphism of flat sphere $G$-bundles $\pi\colon E\to M$, $\pi'\colon E'\to M$ is an isomorphism of sphere bundles
$h\colon E\to E'$
such that,  for every $x\in M$, there exist trivializations $\psi_i\colon \pi^{-1}(U_i)\to U_i\times S^n$,
$\psi_{i'}\colon \pi^{-1}(U_{i'})\to U_{i'}\times S^n$ in the flat $G$-atlases of $E$ and $E'$ such that
$$
\psi_{i'}(h(\psi_i^{-1}(x,v)))=(x,f(x)(v))
$$
for some \emph{locally constant} function $f\colon U_i\cap U'_i\to G$.

\begin{rem}\label{flat:discrete}
If $G<\omeon$ is any group, we denote by $G^\delta$ the topological group
given by the abstract group $G$ endowed with the discrete topology. Then,
a flat $G$-structure on a sphere bundle is just a $G^\delta$-structure, and an isomorphism
between flat $G$-structures corresponds to an isomorphism between $G^\delta$-structures. 
Therefore, flat $G$-bundles may be inscribed in the general theory of $G$-bundles, provided
that the group $G$ be not necessarily endowed with the compact-open topology.
\end{rem}

We will call a flat sphere $G$-bundle \emph{flat topological sphere bundle} (resp.~\emph{flat smooth sphere bundle}, \emph{flat linear sphere bundle})
if $G=\omeon$ (resp.~$G={\rm Diffeo}_+ (S^n)$, $G={\rm Gl}^+(n+1,\mathbb{R})$). Of course, a structure of flat $G$-bundle
determines a unique structure of $G$-bundle. However, 
there exist sphere $G$-bundles which do not admit any flat $G$-structure (in fact, this chapter is mainly devoted to describe an important 
and effective obstruction on topological sphere bundles to be topologically flat). Moreover, the same $G$-structure may be induced by non-isomorphic
flat $G$-structures. Finally, whenever $G'\subseteq G$, a flat $G'$-structure obviously defines a unique
flat $G$-structure, but there could exist flat $G$-structures that cannot be induced by any $G'$-flat structure:
for example, in~Remark~\ref{smoothnolinear} we exhibit a flat circle bundle (i.e.~a flat $\omeo$-bundle) which does not admit any flat ${\rm Gl}^+(2,\R)$-structure.

\begin{rem}\label{flatvector}
The notion of flat bundle makes sense also in the case when the fiber is a vector space rather than a sphere:
namely, a flat vector bundle of rank $n+1$ is just a vector $G$-bundle, where $G={\rm Gl}^+(n+1,\mathbb{R})$ is endowed with the discrete topology, or, equivalently, it is a vector bundle equipped with an atlas with locally constant transition 
functions. Of course, for any base space $M$ there exists a one-to-one correspondence between (isomorphism classes of) flat vector bundles of rank $n+1$ over $M$ and (isomorphism classes of) flat linear sphere bundles
over $M$, given by associating to any flat vector bundle $E$ the corresponding sphere bundle $S(E)$. 
\end{rem}

In order to avoid pathologies, henceforth we assume that the base $M$ of every sphere bundle satisfies the following properties:
\begin{itemize}
 \item $M$ is path connected and locally path connected;
 \item $M$ is semilocally simply connected (i.e.~it admits a universal covering);
 \item $\pi_1(M)$ is countable.
\end{itemize}

Let $\pi\colon E\to M$ be a sphere bundle endowed with the flat $G$-atlas $\{(U_i,\psi_i)\}_{i\in I}$.  
A path $\overline{\alpha}\colon [0,1]\to E$ with values in $\pi^{-1}(U_i)$ for some $i\in I$ is \emph{elementary horizontal}
if the map $\pi_{S^n}\circ\psi_i\circ \overline{\alpha}\colon [0,1]\to S^n$ is constant (observe that this definition is independent of the choice of $i\in I$, 
i.e.~if $\overline{\alpha}$ is supported
in $\pi^{-1}(U_j)$ for some  $j\neq i$, then also the map $\pi_{S^n}\circ\psi_j\circ \overline{\alpha}\colon [0,1]\to S^n$ is constant).
A path $\overline{\alpha}\colon [0,1]\to E$ is \emph{horizontal} if it is the concatenation of elementary horizontal paths.
If $\alpha\colon [0,1]\to M$ is a path and $x$ is a point in the fiber of $\alpha(0)$, then there exists a unique horizontal path
$\overline{\alpha}\colon [0,1]\to E$ such that 
$\overline{\alpha}(0)=x$ and $\pi\circ\overline{\alpha}=\alpha$.

Let us now fix a basepoint $x_0\in M$, 
and denote by
$h\colon S^n\to E_{x_0}$ the homeomorphism induced by a local trivialization of $E$ over a neighbourhood of $x_0$
belonging to the atlas which defines the flat $G$-structure of $E$.
If
$\alpha\colon [0,1]\to M$ is a loop based at $x_0$, then for every $q\in E_{x_0}$ we consider the 
horizontal lift $\overline{\alpha}_q$ of $\alpha$ starting at $q$, and we set $t_\alpha(q)=\overline{\alpha}(1)\in E_{x_0}$.
Then $t_\alpha$ is a homeomorphism which only depends on the homotopy class of $\alpha$ in $\pi_1(M,x_0)=\G$.
Now, for every $g\in \G$ we take a representative $\alpha$ of $g^{-1}$, and we set $\rho(g)=h^{-1}\circ t_\alpha\circ h\colon S^n\to S^n$. 
It is not difficult to check that $\rho(g)$ belongs to $G$, and that
the map $\rho\colon \G\to G$ is a homomorphism (see~\cite[\S 13]{Steenrod} for the details). This homomorphism is usually called the \emph{holonomy} of the flat bundle
$E$. By  choosing a possibly different basepoint $x_0\in M$ and/or a possibly different identification of the corresponding fiber with $S^n$, one is lead to a possibly different
homomorphism, which, however, is conjugate to the original one by an element of $G$. Therefore, to any flat $G$-bundle on $M$ there is associated a holonomy representation
$\rho\colon \G\to G$, which is well defined up to conjugacy by elements of $G$. 

Let us now come back to our fixed flat $G$-bundle $E$. If we define the equivalence relation $\equiv$ on $E$
such that $x\equiv y$ if and only if there exists a horizontal path starting at $x$ and ending at $y$, then each equivalence class
of $\equiv$ is a leaf of a foliation $\mathcal{F}$ of $E$.  Observe that, if $F$ is a leaf of this foliation and $x$ is a point of $M$,
then the set $E_{x}\cap F$ may be identified with an orbit of a holonomy representation for $E$, so it is countable.
It readily follows that the foliation $\mathcal{F}$ is transverse to the fibers, according to the following:

\begin{defn}\label{foliation}
Let $\pi\colon E\to M$ be a (not necessarily flat) sphere $G$-bundle. A foliation
$\mathcal{F}$ of $E$ is \emph{transverse} (to the fibers) if the following condition holds:
there exists a $G$-atlas $\{(U_i,\psi_i)\}_{i\in I}$ for the $G$-structure of $E$ such that, for every 
leaf $F$ of $\mathcal{F}$ and every $i\in I$,
we have
$$\psi_i(F\cap \pi^{-1}(U_i))=U_i\times \Lambda$$ 
for some countable subset $\Lambda$ of $S^n$.
\end{defn}

We have shown that any flat $G$-structure on a sphere bundle  $E$ determines
a foliation transverse to the fibers. In fact, this foliation completely determines the flat structure on $E$:

\begin{thm}
 Let $\pi\colon E\to M$ be a sphere $G$-bundle. Then there exists a bijective correspondence between  flat $G$-structures on $E$
 compatible with the given $G$-structure on $E$, and  transverse foliations of $E$. Moreover, if $\pi\colon E\to M$, $\pi'\colon E'\to M$
 are flat $G$-bundles with corresponding foliations $\mathcal{F},\mathcal{F}'$, 
 then an isomorphism of $G$-bundles $h\colon E\to E'$ is an isomorphism of flat $G$-bundles if and only if it carries any leaf of $\mathcal{F}$ onto
 a leaf of $\mathcal{F}'$.
 \end{thm}
\begin{proof}
Suppose that
$E$ admits a transverse foliation $\mathcal{F}$, and let $\mathcal{U}$
be the maximal atlas with the properties described in Definition~\ref{foliation}. Since countable subsets of $S^n$ are totally disconnected and $M$ is locally path connected,
it is easily seen that the transition functions of $\mathcal{U}$ are locally constant, so $\mathcal{U}$ is a maximal flat $G$-atlas, and $E$
may be endowed with a flat $G$-structure canonically associated to $\mathcal{F}$. The 
theorem readily follows.
\end{proof}

 Flat $G$-bundles may be characterized also in terms of their holonomy. Set $\G=\pi_1(M)$ and let us first construct, 
 for any given representation $\rho\colon \G\to G$, a flat $G$-bundle $E_\rho$ with holonomy $\rho$.
We fix a point $\widetilde{x}_0\in\widetilde{M}$, we set $x_0=p(\widetilde{x}_0)\in M$, and we
 fix the corresponding identification of $\G$ with the group of the covering automorphisms of 
 $\widetilde{M}$, so that the automorphism $\psi$ gets identified 
with the element of $\pi_1(M,p(\widetilde{x}_0))$ represented by the projection of a path
joining $\widetilde{x}_0$ with $\psi(\widetilde{x}_0)$.
We
%We also fix a subgroup $G<\omeon$ and 
%a representation $\rho\colon \G\to G$.
%The following construction shows how to exploit $\rho$ to construct a flat sphere $G$-bundle over $M$.
set $\widetilde{E}=\widetilde{M}\times S^n$, and we define $E_\rho$ as the quotient of $\widetilde{E}$ by the diagonal action
of $\G$, where $\G$ is understood to act on $S^n$ via $\rho$. The composition of the projections $\widetilde{E}\to \widetilde{M}\to M$ induces a quotient map
$\pi\colon E_\rho \to M$. If the open subset $U\subseteq M$ is evenly covered and $\widetilde{U}\subseteq \widetilde{M}$ is an open
subset such that $p|_{\widetilde{U}}\colon \widetilde{U}\to U$ is a homeomorphism, then the  composition
$$
\xymatrix{
U\times S^n
\ar[rr]^-{(p^{-1},\mathrm{Id})}& &\widetilde{U}\times S^n\ar@{^{(}->}[r] & \widetilde{E}\ar[r] & E_\rho
}
$$ 
induces a homeomorphism $\psi\colon U\times S^n\to\pi^{-1}(U)$ which provides a local trivialization
of $E_\rho$ over $U$. It is easily checked that, when $U$ runs over the set
of evenly covered open subsets of $M$, the transition functions corresponding to these local trivializations
take values in $G$, and are locally constant. Indeed, 
the subsets of $\widetilde{E}$ of the form $\widetilde{M}\times \{p\}$ project onto the leaves of a transverse foliation
of $E_\rho$. As a consequence, $E_\rho$ is endowed with the structure of a flat $G$-bundle. Moreover, by construction
the action of $\G$ on $E_{x_0}$ via the holonomy representation is given exactly by $\rho$. 
The following proposition
implies that, up to isomorphism, every flat $G$-bundle may be obtained via the construction we have just described.

\begin{prop}\label{flat:rep}
The map
\begin{align*}
 \{\mathrm{Representations\ of}\ \G\ \mathrm{in}\ G\} &\to \{\mathrm{Flat}\ G-\mathrm{bundles}\}\\
 \rho &\mapsto E_\rho
\end{align*}
induces a bijective correspondence between conjugacy classes of representations of $\G$ into $G$
and isomorphism classes of flat $G$-bundles over $M$. In particular, every representation of $\G$ in $G$ arises as the holonomy
of a flat $G$-bundle over $M$, and flat $G$-bundles with conjugate holonomies are isomorphic as flat $G$-bundles.
\end{prop}
\begin{proof}
Let us first show that every flat sphere $G$-bundle 
is isomorphic to a bundle $E_\rho$ for some $\rho\colon \G\to G$.
To this aim, it is sufficient to prove that the holonomy completely determines the isomorphism type of the bundle.
So, let $E$ be a flat sphere $G$-bundle, fix points $x_0\in M$ and $\widetilde{x}_0\in p^{-1}(x_0)$ and
fix the identification between $\G=\pi_1(M,x_0)$ and the group of the covering automorphisms of
$\widetilde{M}$ such that the projection on $M$ of any path in $\widetilde{M}$ starting at $\widetilde{x}_0$
and ending at $g(\widetilde{x}_0)$ lies in the homotopy class corresponding to $g$. We also denote by
$h\colon S^n\to E_{x_0}$ the homeomorphism induced by a local trivialization of $E$ over a neighbourhood of $x_0$
belonging to the flat $G$-atlas of $E$, and we denote by $\rho\colon \G\to G$ the holonomy representation
associated to these data.

We will now show that $E_\rho$ and $E$ are isomorphic as flat $G$-bundles. To this aim we define a map
$\widetilde{\eta}\colon \widetilde{M}\times S^n\to E$ as follows:  for every $\widetilde{x}\in \widetilde{M}$ we fix a path $\widetilde{\beta}_{\widetilde{x}}$ joining
$\widetilde{x}_0$ to $\widetilde{x}$ in $\widetilde{M}$, we set $\beta_{\widetilde{x}}=p\circ\widetilde{\beta}_{\widetilde{x}}\colon [0,1]\to M$,
and we define $\widetilde{\eta}(\widetilde{x},q)$ as the endpoint of the horizontal lift of $\beta_{\widetilde{x}}$ starting at $q$. The map $\widetilde{\eta}$ is well defined 
(i.e.~it is independent of the choice of $\widetilde{\beta}_{\widetilde{x}}$), and induces the quotient map $\eta\colon E_\rho\to E$, which is easily seen to be an isomorphism
of $G$-bundles. Moreover, by construction $\eta$ carries the transverse foliation of $E_\rho$ into the transverse foliation of $E$, so
$\eta$ is indeed an isomorphism of flat $G$-bundles.

In order to conclude we are left to show that two representations $\rho,\rho'\colon \G\to G$ define isomorphic flat $G$-bundles if and only if they are conjugate. 
Suppose first that there exists $g_0\in G$ such that $\rho'(g)=g_0\rho(g)g_0^{-1}$ for every $g\in \G$.
Then the map $\widetilde{E}_\rho\to \widetilde{E}_{\rho'}$ given by $(x,v)\mapsto (g_0x,g_0v)$ induces an isomorphism of flat $G$-bundles
$E_\rho\to E_{\rho'}$. On the other hand, if $\rho,\rho'\colon \G\to G$ are representations and
$h\colon E_\rho\to E_{\rho'}$ is an isomorphism of flat $G$-bundles, then we may lift $h$ to a map
$\widetilde{h}\colon \widetilde{E}_\rho\to\widetilde{E}_{\rho'}$ such that $\widetilde{h}(\widetilde{x}_0,v)=(\widetilde{x}_0,g_0v)$
for some $g_0\in G$. Since $h$ preserves both the fibers and the transverse foliations, this easily implies that $\widetilde{h}(\widetilde{x},v)=(\widetilde{x},g_0v)$ for every 
$\widetilde{x}\in \widetilde{M}$, $v\in S^n$.
Moreover, for every $g\in \G$ there must exist $g'\in \G$ such that
$$
(g\widetilde{x}_0,g_0\rho(g)v)=\widetilde{h}(g\cdot (\widetilde{x}_0,v))=g'\cdot \widetilde{h}(\widetilde{x}_0,v)=(g'\widetilde{x}_0,\rho'(g')g_0v)\ .
$$
This equality implies $g'=g$, whence $g_0\rho(g)=\rho'(g)g_0$ for every $g\in\G$. We have thus shown that $\rho'$ is conjugate to $\rho$,
and this concludes the proof.
 \end{proof}

\section{The bounded Euler class of a flat circle bundle}
Let $\pi\colon E\to M$ be a flat circle bundle. 
We fix points $x_0\in M$ and $\widetilde{x}_0\in p^{-1}(x_0)$, and
we identify $\G=\pi_1(M,x_0)$ with the group of the covering automorphisms of
$\widetilde{M}$ so that the projection on $M$ of any path in $\widetilde{M}$ starting at $\widetilde{x}_0$
and ending at $g(\widetilde{x}_0)$ lies in the homotopy class corresponding to $g$.
We also choose a set of representatives 
$R$ for the action of $\G=\pi_1(M)$ on $\widetilde{M}$ containing $\widetilde{x}_0$.

By exploiting the flat structure on $E$, we are going to construct a bounded representative $z_b\in C^2_b(M,\matZ)$
for the Euler class $\eu(E)$. In fact, we will see that such a cocycle represents a well-defined bounded class in $H^2_b(M,\matZ)$.
By Proposition~\ref{flat:rep} we may suppose $E=E_\rho$ for some representation $\rho\colon\G\to\omeo$, and we denote
by $j\colon \widetilde{M}\times S^1\to E_\rho$ the quotient map with respect to the diagonal action of $\G$ on
$\widetilde{M}\times S^1$. Finally, 
we set $\theta_0=[0]\in S^1$.

For every $x\in M$, we denote by $\widetilde{x}$ the unique preimage of $x$ in $R\subseteq \widetilde{M}$, and
we choose the section $\sigma_x\in E_x$ defined by $\sigma_x=j(\widetilde{x},\theta_0)$. Let now $s\colon [0,1]\to M$
be a singular 1-simplex. We denote by $\widetilde{s}\colon [0,1]\to \widetilde{M}$ the unique lift of $s$ starting at a point
in $R$, and by $g_s$ the unique element of $\G$ such that $\widetilde{s}(1)\in g_s(R)$. We also choose a path
$\widetilde{h}_s(t)\colon [0,1]\to \omeot$ such that $\widetilde{h}_s(0)={\rm Id}$ and
$\widetilde{h}_s(1)=\widetilde{\rho(g_s)}$, where as usual $\widetilde{\rho(g_s)}$
denotes the lift of $\rho(g_s)$ taking $0\in\R$ into $[0,1)\subseteq \R$, and we denote by $h_s(t)$ the projection
of $\widetilde{h}_s(t)$ in $\omeo$. We finally define the section
$\sigma_s\colon [0,1]\to E$ by setting 
$$\sigma_s(t)=j(\widetilde{s}(t), h_s(t)(\theta_0))\ .$$
Our choices imply that $\sigma_s$ is indeed a section of $E$ over $s$ such that
$\sigma_s(0)=\sigma_{s(0)}$ and $\sigma_s(1)=\sigma_{s(1)}$.

Let us now evaluate the Euler cocycle $z_b$ corresponding to the chosen sections on a singular $2$-simplex
$s\colon \Delta^2\to M$. We denote by $e_0,e_1,e_2$ the vertices of $\Delta^2$, and for every $i$-dimensional simplex $v\colon \Delta^i\to M$ we denote by $\widetilde{v}\colon \Delta^i\to\widetilde{M}$ the unique
lift of $v$ whose first vertex lies in $R$. Let
$g_1,g_2$ be the elements of $\G$ such that
$\widetilde{s}(e_1)\in g_1(R)$, $\widetilde{s}(e_2)\in g_1g_2(R)$. 
Then we have 
$$
\partial_2 \widetilde{s}=\widetilde{\partial_2 s}\, , \quad
\partial_0 \widetilde{s}=g_1\circ \widetilde{\partial_0 s}\, , \quad
\partial_1 \widetilde{s} =\widetilde{\partial_1 s}\, .
$$
The pull-back $E_s$ of $E$ over $s$ admits the trivialization
$$
\psi_s\colon E_s\to \Delta^2\times S^1\, ,\quad \psi_s^{-1}(x,\theta)=j(\widetilde{s}(x),\theta)\ .
$$
Let $\sigma_{\partial s}\colon \partial \Delta^2\to E_s$ be the section obtained by concatenating $\sigma_{\partial_i s}$, $i=1,2,3$.
In order to compute $z_b(s)$ we need to write down an expression for the map $\pi_{S^1}\circ \psi_s\circ \sigma_{\partial s}\colon \partial \Delta^2\to S^1$.
%For convenience, we set $\varepsilon_i=(-1)^i$, and 

Recall that, to every $1$-simplex $\partial_i s$, there is associated a path $h_i=h_{\partial_i s}\colon [0,1]\to \omeo$ such that
$\sigma_{\partial_i s}(t)=j(\widetilde{\partial_i s}(t),h_i(t)(\theta_0))$. Moreover, we have $h_2(1)={\rho(g_1)}$, $h_0(1)={\rho(g_2)}$ and
$h_1(1)={\rho(g_1g_2)}$. 
Therefore,
if $\alpha_i\colon [0,1]\to S^1$ is defined by 
$$
\alpha_2(t)=h_{2}(t)(\theta_0)\, ,\quad
\alpha_0(t)=\rho(g_1)h_{0}(t)(\theta_0)\, ,\quad
\alpha_1(t)=h_{1}(t)(\theta_0)
$$
and $\beta_i(t)=(\partial_i\widetilde{s}(t),\alpha_i(t))\in \widetilde{M}\times S^1$, 
then
the desired section of $E_s$ over $\partial \Delta^2$ is obtained by projecting 
onto $E$ 
the concatenation $\beta_2*\beta_0*\beta_1^{-1}$
(as usual, for any path $\beta$ defined over $[0,1]$ we denote by $\beta^{-1}$ the inverse path, i.e.~the path such 
that $\beta^{-1}(t)=\alpha(1-t)$).
As a consequence, the value of $z_b(s)$ is equal to
the element of $\pi_1(S^1)\cong \matZ$ defined the map $\gamma=\alpha_2*\alpha_0*\alpha_1^{-1}\colon [0,3]\to S^1$.
In order to compute this element, we observe that the lift $\widetilde{\gamma}\colon [0,3]\to\R$
of $\gamma$ such that $\widetilde{\gamma}(0)=0$ is given by
$\widetilde{\gamma}=\widetilde{\alpha}_2*\widetilde{\alpha}_0*\widetilde{\alpha}_1^{-1}$, where 
$\widetilde{\alpha}_i$ is a suitable lift of $\alpha_i$ for $i=0,1,2$. Namely, since
$\widetilde{\gamma}(0)=0$, we have
$$
\widetilde{\alpha}_2(t)=\widetilde{h_{2}(t)}(0)\, ,
$$
and in order to make the endpoint of each lift coincide with the starting point of the following one we need to set 
$$
\widetilde{\alpha}_0(t)=\widetilde{\rho(g_1)}\widetilde{h_{0}(t)}(0)\, ,\quad
\widetilde{\alpha}_1(t)=\widetilde{\rho(g_1)}\widetilde{\rho(g_2)}\widetilde{\rho(g_1g_2)}^{-1}\widetilde{ h_{1}(t)}(0)\ .
$$
%(observe that our definitions imply that $\widetilde{h_{2}(1)}=\widetilde{\rho(g_1)}$,
%$\widetilde{h_{0}(1)}=\widetilde{\rho(g_2)}$,
%and $\widetilde{h_{1}(1)}=\widetilde{\rho(g_1g_2)^{-1}}$).
As a consequence, we have
$$
\widetilde{\gamma}(1)=\widetilde{\alpha}_1^{-1}(1)=\widetilde{\alpha}_1(0)=\widetilde{\rho(g_1)}\widetilde{\rho(g_2)}\widetilde{\rho(g_1g_2)}^{-1}(0)=\rho^*(c)(g_1,g_2)\ ,
$$
where $c\in C^2_b(\omeo,\matZ)$ is the canonical representative of the bounded Euler class
$e_b\in H^2_b(\omeo,\matZ)$.

We have thus shown that the representative $z_b\in C^2(M,\matZ)$ of the Euler class of $E$ is bounded. Moreover,
if $s\colon \Delta^2\to M$ is any singular simplex whose vertices lie respectively in
$R, g_1(R)$ and $g_1g_2(R)$, then $z_b(s)=\rho^*(c)(g_1,g_2)$. This fact may be restated by saying that 
$z_b=r^2(\rho^*(c))$, where 
$$
r^2\colon C^2_b(\G,\matZ)\to C^2_b(M,\matZ)
$$ is the map described in Lemma~\ref{normnon}. In other words, the natural map
$$
H^2_b(r^\bullet)\colon H^2_b(\G,\matZ)\to H^2_b(M,\matZ)
$$
described in Corollary~\ref{classifyingmap} takes the bounded Euler class $e_b(\rho)$ into the class
$[z_b]\in H^2_b(M,\R)$. Therefore, it makes sense to define the \emph{bounded Euler class} 
$\eu_b(E)$ 
of the \emph{flat}
circle bundle $E$ by setting 
$$
\eu_b(E)=[z_b]\in H^2_b(M,\R)\ .
$$
Recall that, if $E,E'$ are isomorphic flat circle bundles, then
$E=E_\rho$ and $E'=E_{\rho'}$ for conjugate representations $\rho,\rho'\colon \G\to\omeo$. Since 
conjugate representations share the same bounded Euler class, we have that $\eu_b(E)$ is a well-defined invariant
of the isomorphism type of $E$ as a flat circle bundle. We may summarize this discussion in the following:

\begin{thm}\label{eulerclassbounded}
 To every flat topological circle bundle $\pi\colon E\to M$ there is associated a bounded class
 $\eu_b(E)\in H^2_b(M,\matZ)$ such that:
 \begin{itemize}
 \item $\|\eu_b(E)\|_\infty\leq 1$;
  \item If $E=E_\rho$, then $\eu_b(E)=H^2_b(r^\bullet)(e_b(\rho))$;
  \item If $\pi'\colon E'\to M$ is isomorphic to $\pi\colon E\to M$ as a topological flat
  circle bundle, then $\eu_b(E')=\eu_b(E)$.
  \item The comparison map $H^2_b(M,\matZ)\to H^2(M,\matZ)$ takes $\eu_b(E)$ into $\eu(E)$.
 \end{itemize}
\end{thm}

Theorem~\ref{eulerclassbounded} already provides an obstruction for a circle bundle to admit
a flat structure:

\begin{cor}
 Let $E\to M$ be a circle bundle. If $E$ admits a flat structure, then $\eu(E)$ lies in the image of the unit
 ball via the comparison
 map $H^2_b(M,\matZ)\to H^2(M,\matZ)$.
\end{cor}

By exploiting real rather than integral coefficients, the above estimate on the seminorm of $\eu_b(E)$ may be improved. 
A celebrated result by Milnor and Wood implies that,
in the case of closed
surfaces, this strategy leads to sharp estimates on the Euler number of flat circle bundles.
In the following section we describe a proof of Milnor-Wood inequalities which makes use of the machinery introduced so far.

\section{Milnor-Wood inequalities}

Let us now concentrate our attention on circle bundles on surfaces. We recall that, for every $g\in\mathbb{N}$, the symbol $\Sigma_g$ 
denotes the closed
oriented surface of genus $g$. In  this section we describe classical results by Milnor~\cite{Milnor} and Wood~\cite{Wood}, which provide
sharp estimates on the Euler number of \emph{flat} circle bundles on $\Sigma_g$. To this aim, 
when considering both bundles and representations, it may be useful to consider
 \emph{real} (bounded) Euler classes
rather than integral ones. So, let us first define the real Euler class of a sphere bundle
(we refer the reader to Definition~\ref{real:euler:defn} for the definition of real (bounded) Euler class of a representation).

\begin{defn}
Let $\pi\colon E\to M$ be a topological $n$-sphere bundle. Then the \emph{real} Euler class $e^\R(E)\in H^{n+1}(M,\R)$ of $E$ is the image of
$e(E)\in H^{n+1}(M,\matZ)$ under the change of coefficients homomorphism. If $n=1$ and $E$ is topologically flat, then we denote by 
$e_b^\R(E)\in H_b^{2}(M,\R)$ the image of
$e_b(E)\in H_b^{2}(M,\matZ)$ under the change of coefficients homomorphism.
\end{defn}

Let us summarize what we can already deduce from the previous sections:

\begin{enumerate}
 \item If $g=0$, i.e.~$\Sigma_g=S^2$, then $\pi_1(\Sigma_g)$ is trivial, so for every $G<\omeo$ the space of representations
 of $\pi_1(\Sigma_g)$ into $G$ is trivial. As a consequence of Proposition~\ref{flat:rep}, the unique flat $G$-bundle over $S^2$
 is the trivial one. More precisely, the only circle bundle over $S^2$ supporting a flat $G$-structure is the topologically trivial one,
 and the unique flat $G$-structure supported by this bundle is the trivial one.
 \item If $g=1$, i.e.~$\Sigma_g=S^1\times S^1$, then $\pi_1(\Sigma_g)$ is amenable, so $H^2_b(\Sigma_g,\R)=0$. Therefore, if $\pi\colon E\to \Sigma_g$
 is a flat circle bundle, then
 $e_b^\R(E)=0$, whence $e^\R(E)=0$. But the change of coefficients map $H^2(\Sigma_g,\matZ)\to H^2(\Sigma_g,\R)$ is injective,
 so the Euler number of $E$ vanishes. By Theorem~\ref{Euler:classifies}, this implies that $E$ is topologically trivial. Therefore,
 the unique circle bundle over $S^1\times S^1$ supporting a flat structure is the trivial one. However, there are infinitely many non-conjugate
 representations of $\pi_1(S^1\times S^1)$ into $\omeo$. Therefore, Proposition~\ref{flat:rep} implies that the trivial circle bundle
 over $S^1\times S^1$ admits infinitely many pairwise non-isomorphic flat structures.
\end{enumerate}

The following proposition provides interesting examples of flat circle bundles over $\Sigma_g$, when $g\geq 2$.

\begin{prop}\label{2g-2:exists}
 For every $g\geq 2$, there exists a flat circle bundle 
 $\pi\colon E\to \Sigma_g$ such that 
 $$
 \eu(E)=2-2g\ .
 $$
\end{prop}
\begin{proof}
Let $\G_g$ be the fundamental group of $\Sigma_g$.
It is well-known that the surface $\Sigma_g$ supports a hyperbolic metric. As a consequence, we may identify
the Riemannian universal covering of $\Sigma_g$ with the hyperbolic plane $\matH^2$, and $\Sigma_g$ is realized
as the quotient of $\matH^2$ by the action of $\rho(\G_g)$, where $\rho\colon \G_g\to {\rm Isom}^+(\matH^2)$ is
a faithful representation with discrete image. The hyperbolic plane admits a natural compactification
$\overline{\matH}^2=\matH^2\cup\partial\matH^2$ which can be roughly described as follows (see e.g.~\cite{BePe} for the details):
points of $\partial\matH^2$ are equivalence classes of geodesic rays of $\matH^2$, where two such rays are equivalent if their images
lie at finite Hausdorff distance one from the other (recall that a geodesic ray in $\matH^2$ is just a unitary speed geodesic $\gamma\colon [0,+\infty)\to \matH^2$).
A metrizable topology on $\partial\matH^2$ may be defined by requiring that a sequence $(p_i)_{i\in \mathbb{N}}$ in $\partial\matH^2$
converges to $p\in \partial\matH^2$ if and only if there exists a choice of representatives $\gamma_i\in p_i$, $\gamma\in p$ such
that $\gamma_i\to\gamma$ with respect to the compact-open topology. Let us  
denote by $\pi\colon T_1\matH^2\to\matH^2$ the unit tangent bundle of $\matH^2$, and by $\gamma_v$, $v\in T_1\matH^2$, the geodesic ray in $\matH^2$ with initial speed $v$. 
Then the map 
$$
\Psi\colon T_1\matH^2\to \matH^2\times \partial\matH^2\, ,\qquad \Psi(v)=(\pi(v),[\gamma_v])
$$
is a homeomorphism. In particular, $\partial\matH^2$ is homeomorphic to $S^1$. Moreover, any isometry of $\matH^2$ extends to a homeomorphism of
$\partial\matH^2$, so $\rho(\G_g)$ acts on $\partial\matH^2$. Of course, $\rho(\G_g)$ also acts on $T_1\matH^2$, and it follows from the definitions
that the map $\Psi$ introduced above is $\rho(\G_g)$-equivariant with respect to these actions. Observe now that the quotient
of $T_1\matH^2$ by the action of $\rho(\G_g)$ may be canonically identified with the unit tangent bundle 
$T_1\Sigma_g$, so $\eu(T_1\Sigma_g)=2-2g$ by Proposition~\ref{tangentchi}.
On the other hand, since $\partial\matH^2$ is homeomorphic to $S^1$, the quotient of $\matH^2\times \partial\matH^2$
by the diagonal action of $\rho(\G_g)$ is a flat circle bundle on $\Sigma_g$. \end{proof}

\begin{rem}\label{smoothnolinear}
The previous result holds also if we require $E$ to be flat as a \emph{smooth} circle bundle.
In fact, $\partial\matH^2$ admits a natural smooth structure. With respect to this structure,
the map $\Psi$ introduced in the proof of Proposition~\ref{2g-2:exists} is a diffeomorphism, and
${\rm Isom}^+(\matH^2)$ acts on $\partial\matH^2$ via diffeomorphisms. As a consequence, $T_1\Sigma_g$
admits a structure of flat \emph{smooth} circle bundle. 

On the other hand, as a consequence of a classical result by Milnor (see Theorem~\ref{Milnor:thm} below), the bundle $T_1\Sigma_g$
does \emph{not} admit any flat \emph{linear} structure for $g\geq 2$. 
\end{rem}

Proposition~\ref{2g-2:exists}
shows that, when $g\geq 2$, in order to get bounds on the Euler number of flat bundles over $\Sigma_g$
a more refined analysis is needed. We first provide an estimate on the norm of $e_b^\R\in H^2_b(\omeo,\R)$.

\begin{lemma}\label{norm:euler}
 We have
 $$
 \|e^\R_b\|_\infty\leq \frac{1}{2}\ .
 $$
\end{lemma}
\begin{proof}
Recall from Lemma~\ref{only01} that $e_b$ admits a representative $c_{x_0}$
taking values in the set $\{0,1\}$. If $\varphi\colon C^1(\omeo,\R)^{\omeo}$ is
the constant cochain taking the values $-1/2$ on every pair $(g_0,g_1)\in \omeo^2$, then
the cocycle $c+\delta \varphi$ still represents $e_b$, and takes values
in $\{-1/2,1/2\}$. The conclusion follows.
\end{proof}

\begin{cor}\label{norm12}
 Let $\pi\colon E\to M$ be a flat circle bundle. Then $$\|\eu^\R_b(E)\|_\infty\leq \frac{1}{2}\ .$$
\end{cor}
\begin{proof}
 The conclusion follows from Lemma~\ref{norm:euler} and Theorem~\ref{eulerclassbounded}.
\end{proof}

We are now ready to prove Wood's estimate of the Euler number of flat circle bundles. For convenience, for
every $g$ we set $\chi_-(\Sigma_g)=\min \{\chi(\Sigma_g),0\}$, and we recall from Section~\ref{surface:simplicial:sec}
that the simplicial volume of $\Sigma_g$ is given by $\|\Sigma_g\|=2|\chi_-(\Sigma_g)|$.

\begin{thm}[\cite{Wood}]\label{Wood:thm}
 Let $\pi\colon E\to \Sigma_g$ be a flat circle bundle. Then
 $$
 |\eu(E)|\leq |\chi_-(\Sigma_g)|\ .
 $$
% (in particular, we get another proof of the fact that $\eu(E)=0$ for $g=0,1$). 
Moreover, this inequality is sharp, i.e.~for every
 $g\in\mathbb{N}$ there exists a flat circle bundle over $\Sigma_g$ with Euler number equal to $|\chi_-(\Sigma_g)|$. 
\end{thm}
\begin{proof}
 Let us denote by $\langle\cdot,\cdot\rangle\colon H^2_b(\Sigma_g,\R)\times H_2(\Sigma_g,\R)\to \R$
 the Kronecker product (see Chapter~\ref{duality:chap}). The Euler number of $E$ may be computed by evaluating $e(E)\in H^2(\Sigma_g,\matZ)$ on the integral fundamental
 class $[\Sigma_g]$ of $\Sigma_g$, or, equivalently, by evaluating the class $\eu^\R(E)\in H^2(\Sigma_g,\R)$ on the image $[\Sigma_g]_\R\in H_2(\Sigma_g,\R)$
 of the integral fundamental class of $\Sigma_g$
 under the change of coefficients homomorphism. As a consequence, we have
 $$
 |\eu(E)|=|\langle e^\R_b(E), [\Sigma_g]_\R\rangle|\leq \|e^\R_b(E)\|_\infty \|[\Sigma_g]\|\leq
\frac{1}{2} |2\chi_-(\Sigma_g)|\ .
$$
The fact that Wood's inequality is sharp is a direct consequence of Proposition~\ref{2g-2:exists}.
 \end{proof}

 As a consequence, we obtain the following corollary, which was originally proved by  Milnor:

\begin{thm}[\cite{Milnor}]\label{Milnor:thm}
 Let $\pi\colon E\to \Sigma_g$ be a flat linear circle bundle. Then
 $$
 |\eu(E)|\leq \frac{|\chi_-(\Sigma_g)|}{2}\ .
 $$
\end{thm}
\begin{proof}
If $\pi\colon E\to \Sigma_g$ is a flat linear circle bundle, then the associated projective bundle 
$\mathbb{P}(E)$ is a flat topological circle bundle, and
$2|\eu(E)|=|\eu(\mathbb{P}(E))|\leq |\chi_-(\Sigma_g)|$, whence the conclusion.
\end{proof}

In fact, the bounds in the last two theorems are sharp in the following strict sense: every Euler number
which is not forbidden is realized by some flat topological (resp.~linear) circle bundle. 

Theorem~\ref{Wood:thm} allows us to compute the exact norm of the real Euler class:

\begin{thm}\label{finale:mw}
 The norm of the Euler class $e^\R_b\in H_b^2(\omeo,\R)$ is given by
 $$
 \|e_b^\R\|_\infty =\frac{1}{2}\ .
 $$
\end{thm}
\begin{proof}
 The inequality $\leq$ was proved in Lemma~\ref{norm:euler}. Now let $g=2$, and consider the flat circle bundle on $\pi\colon E\to\Sigma_2$
 described in Proposition~\ref{2g-2:exists}. We have 
 $$
 2=2g-2=|\eu(E)|\leq \|e_b^\R\|_\infty \|\Sigma_2\|=4\|e_b^\R\|_\infty\ ,
 $$
 so $ \|e_b^\R\|_\infty\geq 1/2$, and we are done.
\end{proof}

\begin{rem}
 One could exploit the estimate described in the proof of Theorem~\ref{finale:mw} also to compute a (sharp) lower bound on the simplicial
 volume of $\Sigma_g$, $g\geq 2$. In fact, once one knows that $\|e_b^\R\|_\infty\leq 1/2$ and that there exists a flat circle bundle $E$ over $\Sigma_g$
 with Euler number $2g-2$, the estimate
 $$
 2g-2=|\eu(E)|\leq \|e_b^\R\|_\infty \|\Sigma_g\|\leq \|\Sigma_g\|/2
 $$ 
 implies that $\|\Sigma_g\|\geq 2|\chi(\Sigma_g)|$.
\end{rem}

\section{Flat circle bundles on surfaces with boundary}
We have seen in the previous chapter that to every (not necessarily flat) circle bundle on an oriented closed surface there is associated an Euler number. When $\Sigma$ is an oriented compact surface with non-empty boundary,
the cohomology group $H^2(\Sigma,\mathbb{Z})$ vanishes, so there is no hope to recover a non-trivial invariant from the classical Euler class. 
Indeed, by Theorem~\ref{Euler:classifies} the fact that $H^2(\Sigma,\mathbb{Z})=0$ implies that every circle bundle on any compact orientable surface with non-empty boundary is topologically trivial. 
One could work with the relative cohomology group
$H^2(\Sigma,\partial \Sigma,\mathbb{Z})\cong \mathbb{Z}$, but in order to make this strategy work it is necessary to fix trivializations (or, at least, sections) of the bundle over $\partial \Sigma$; the resulting invariant is thus
an invariant of the bundle relative to specific boundary sections. 

On the other hand, when working with \emph{flat} circle bundles, in the closed case we were able to define a \emph{bounded} Euler class. The point now is that, even when a surface $S$ has non-empty boundary, the bounded
cohomology of $\Sigma$ in degree $2$ may be highly non-trivial. In our case of interest, i.e.~when $\chi(\Sigma)<0$, the surface $\Sigma$ is aspherical with free fundamental group, so putting together 
Theorem~\ref{aspherical:thm} and Corollary~\ref{infinitedim:cor} we get that $H^2_b(\Sigma,\mathbb{R})\cong H^2_b(\pi_1(\Sigma),\mathbb{R})$ is infinite-dimensional (and, from this, it is not difficult to deduce that
$H^2_b(\Sigma,\mathbb{Z})$ is also infinitely generated as an abelian group). What is more, since every component of $\partial \Sigma$ has an amenable fundamental group, when working with real coefficients
the absolute bounded cohomology of $\Sigma$
is isometrically isomorphic to the bounded cohomology of $\Sigma$ relative to the boundary (see Theorem~\ref{reliso:thm}). 
Putting together these ingredients one can now define a meaningful Euler number for \emph{flat} circle bundles over
compact surfaces with boundary (or, equivalently, for circle actions  of fundamental groups of compact surfaces with boundary). The construction we are going to describe is taken from~\cite{BIW1,BIW2}.

Let us fix an oriented compact surface $\Sigma_{g,n}$ of genus $g$ with $n$ boundary components. Henceforth we assume
$\chi(\Sigma_{g,n})=2-2g-n<0$, and we denote by $\G_{g,n}$ the fundamental group of $\Sigma_{g,n}$. Let also $\rho\colon \G_{g,n}\to \omeo$ be a representation, and let
$e^\mathbb{R}_b(\rho)\in H^2_b(\G_{g,n},\mathbb{R})$ be the real bounded Euler class of $\rho$. 
Under the canonical isometric isomorphism $H^2_b(\Sigma_{g,n},\mathbb{R})\cong H^2_b(\G_{g,n},\mathbb{R})$, the class
$e^\mathbb{R}_b(\rho)$ corresponds to an element in $H^2_b(\Sigma_{g,n},\mathbb{R})$, which we still denote by $e_b^\R(\rho)$. Let us now recall from Theorem~\ref{reliso:thm}
that the inclusion $j^\bullet\colon C^n_b(\Sigma_{g,n},\partial \Sigma_{g,n},\mathbb{R})\to C^n_b(\Sigma_{g,n},\mathbb{R})$ of relative into absolute cochains induces
an isometric isomorphism 
$$
H^2_b(j^2)\colon  H^2_b(\Sigma_{g,n},\partial \Sigma_{g,n},\mathbb{R})\to H^2_b(\Sigma_{g,n},\mathbb{R})
$$	
on bounded cohomology, and let us set 
$$
e^{\R}_b(\rho, \Sigma_{g,n},\bb\sgn)=H^2_b(j^2)^{-1}(e^\mathbb{R}_b(\rho)) \ \in \ H^2_b(\Sigma_{g,n},\partial \Sigma_{g,n},\mathbb{R})\ .
$$
Finally, as usual we denote by $[\Sigma_{g,n},\partial \Sigma_{g,n}]\in H_2(\Sigma_{g,n},\partial \Sigma_{g,n},\R)$ the real fundamental class of the surface $\Sigma_{g,n}$.

\begin{defn}\label{releuler} 
 Let $\rho\colon \G_{g,n}\to \omeo$ be a representation. Then the Euler number $e(\rho, \Sigma_{g,n})$ of $\rho$ is defined by setting
 $$
 e(\rho, \Sigma_{g,n},\bb\sgn)=\langle e^{\R}_b(\rho, \Sigma_{g,n},\bb\sgn), [S_{g,n},\partial S_{g,n}]\rangle \ \in\ \mathbb{R}\ ,
 $$
 where $\langle \cdot,\cdot \rangle$ is the usual Kronecker product between $H^2_b(\Sigma_{g,n},\partial \Sigma_{g,n},\mathbb{R})$
 and $H_2(\Sigma_{g,n},\partial \Sigma_{g,n},\mathbb{R})$.
\end{defn}

%The Euler number of a representation is well defined, i.e.~it does not depend on the chosen identification of $\G_{g,n}$ with $\pi_1(\Sigma_{g,n})$. Indeed, if $f\colon (\Sigma_{g,n},\partial \Sigma_{g,n})\to 
%(\Sigma_{g,n},\partial \Sigma_{g,n})$ is any orientation-preserving homotopy equivalence of pairs, with induced map $f_*$ at the level of fundamental groups, then 
%\begin{align*}
% e(\rho\circ f_*, \Sigma_{g,n}) & =\langle e^{\R}_b(\rho\circ f_*, \Sigma_{g,n}), [S_{g,n},\partial S_{g,n}]\rangle\\ & = \langle e^{\R}_b(\rho,\sgn), H_2(f)([\Sigma_{g,n},\partial \Sigma_{g,n}])\rangle\\ &=
% \langle e^{\R}_b(\rho, \Sigma_{g,n}), [S_{g,n},\partial S_{g,n}]\rangle= e(\rho, \Sigma_{g,n})\ ,
% \end{align*}
% where we used that homotopy equivalences of pairs preserves the fundamental class of a manifold with boundary.
 
The Euler number of a representation $\rho\colon \G_{g,n}\to \omeo$
heavily depends on the identification of $\G_{g,n}$ with the fundamental group of $\sgn$.
%not only on the isomorphism class of $\G_{g,n}$, but also on the homeomorphism type of the surface $\sgn$.
Indeed, if $2-2g-n=2-2g'-n'$, where $n$ and $n'$ are positive, then
there exists a group isomorphism $\psi\colon \G_{g,n}\to \G_{g',n'}$, since  $\G_{g,n}$ and $\G_{g',n'}$ are both isomorphic to the free group $F_{2g+n-1}$ on $2g+n-1$ generators. However, 
if $\rho'\colon \G_{g',n'}\to\omeo$ is a representation, then it may happen that
$$
e(\rho',\Sigma_{g',n'},\bb\Sigma_{g',n'})\neq e(\rho'\circ\psi,\sgn,\bb\sgn)\ .
$$
 Indeed, even once the surface $\sgn$ is fixed, the Euler number of a representation $\rho\colon \G_{g,n}\to \omeo$ may depend on the choice of the identification
 $\G_{g,n}\cong \pi_1(\sgn)$. Indeed, when $n>1$ there exist group
 automorphisms $\psi\colon \G_{g,n}\to \G_{g,n}$ that are \emph{not} induced by homeomorphisms of $\sgn$ into itself. If $\psi$ is such an automorphism, then
 it may happen that $e(\rho,\sgn,\bb\sgn)\neq e(\rho\circ\psi,\sgn,\bb\sgn)$ (see Remark~\ref{differ:rem} for an explicit example).

 Let now $C_1,\ldots,C_n$ be the connected components of $\bb \Sigma_{g,n}$. After endowing each $C_i$ with the orientation induced by $\sgn$, we can consider an element 
 $g_i\in \G_{g,n}$ lying in the conjugacy class corresponding to the free homotopy class of $C_i$ in $\sgn$. Since the rotation number of an element of $\omeo$ only depends
 on its conjugacy class, the value
 $$
 \rot (\rho(C_i)) =\rot (\rho( g_i))
 $$
 is well defined (i.e., it does not depend on the choice of $g_i$). The following proposition implies in particular that the relative Euler number
$ e(\rho, \Sigma_{g,n},\bb\sgn)$ need not be an integer.

\begin{prop}\label{boundary:euler}
 Let $\rho\colon \G_{g,n}\to \omeo$ be a representation. Then
 $$
 [ e(\rho, \Sigma_{g,n},\bb\sgn)] = - \sum_{i=1}^n \rot(\rho(C_i))\ \in\ \R/\mathbb{Z}\ .
 $$
\end{prop}
\begin{proof}
 Let $c\in C^2_b(\G_{g,n},\mathbb{Z})$ be a representative of the bounded Euler class $e_b(\rho)\in H^2_b(\G_{g,n},\mathbb{Z})\cong H^2_b(\Sigma_{g,n},\mathbb{Z})$
 associated to the representation $\rho$. 
 We denote by $c_i\in C^2_b(C_i,\mathbb{Z})$ the restriction of $c$ to the boundary component $C_i$. 
 Since $\pi_1(C_i)\cong \Z$ is amenable, we have $H^2_b(C_i,\R)=0$, so 
 there exists a \emph{real} bounded cochain $b_i\in C^1_b(C_i,\mathbb{R})$ such that $\delta b_i=c_i$.

 Let now $b\in C^1_b(\sgn,\R)$ be a bounded cochain extending the $b_i$ (such a cochain may be defined, for example, by setting $b(s)=0$ for every singular $1$-simplex
 $s\colon \Delta^1\to\sgn$ whose image is not contained in any $C_i$, and $b(s)=b_i(s)$ if the image of $s$ is contained in $C_i$). Of course we have 
 $e_b^\R(\rho)=[c]=[c-\delta b]$ in $H^2_b(\sgn,\mathbb{R})$. Moreover, by construction the cocycle $c-\delta b$ vanishes on $\bb \sgn$. In other words, 
 $c-\delta b$ is a relative cocycle, and by construction its class in the relative bounded cohomology module $H^2_b(\sgn,\bb\sgn,\mathbb{R})$ coincides 
 with $e^{\R}_b(\rho, \Sigma_{g,n},\bb\sgn)$. Therefore, if $z$ is an integral fundamental cycle for $\sgn$, then we have
 $$
 e(\rho, \Sigma_{g,n},\bb\sgn)=\langle [c-\delta b],[\Sigma_{g,n},\bb\sgn]\rangle=\langle c,z\rangle - \langle \delta b,z\rangle\ .
 $$
 But $c$ is an integral cocycle, so $\langle c,z\rangle\in\mathbb{Z}$, and 
 $$
 [e(\rho, \Sigma_{g,n},\bb\sgn)]= - [\langle \delta b,z\rangle]\ \in \ \R/\Z\ .
 $$
 Since $z$ is a relative fundamental cycle for $\sgn$, we now have $\bb z=z_1+\ldots+z_n$, where $z_i$ is a fundamental cycle for $C_i$ for every $i=1,\ldots,n$. Therefore,
 we have 
 $$
 [e(\rho, \Sigma_{g,n},\bb\sgn)]= - [\langle \delta b,z\rangle]=-[\langle b,\bb z\rangle] =-\left[ \left\langle b,\sum_{i=1}^n z_i\right\rangle\right]=- \sum_{i=1}^n [\langle b_i,z_i\rangle]\ ,
 $$
 and in order to conclude we are left to show that 
 \begin{equation}\label{ee:eq}
 [\langle b_i,z_i\rangle]=\rot(\rho(C_i))\ {\rm in}\  \R/\Z 
 \end{equation}
 for every $i=1,\ldots,n$.

 Let $\overline{b}_i\in C^1(C_i,\R/\Z)$ be the cochain associated to $b_i$ via  the change of coefficients map. 
 Since $\delta b_i$ is an integral chain, $\overline{b}_i$ is a cocycle.
If we denote by $\delta'\colon H^1(C_i,\R/\Z)\to H^2_b(C_i,\Z)$ the connecting homomorphism
 associated to the short exact sequence $$0\to C^\bullet_b(C_i,\Z)\to C^\bullet_b(C_i,\R)\to C^\bullet (C_i,\R/\Z)\to 0\ ,$$ 
 since $b_i$ is bounded we have $\delta' [\overline{b}_i]=[\delta b_i]=e_b(\rho|_{H_i})$, 
 where $H_i\subseteq \G_{g,n}$ is the cyclic subgroup generated by a loop freely homotopic to $C_i$. Therefore,
 our definition of rotation number implies that
 $$
 \rot(\rho(C_i))=[\overline{b}_i]\ \in \ H^1(C_i,\R/\Z)\ \cong \R/\Z\ .
 $$
 Finally, since $[z_i]$ is the positive generator
 of $H_1(C_i,\Z)$, it is immediate to check that under the canonical identifications $$H^1(C_i,\R/\Z)={\rm Hom} (H_1(C_i,\mathbb{Z}),\R/\Z)=\R/\Z$$ the 
 class $[\overline{b}_i]$ corresponds to the element $[\langle b_i,z_i\rangle]\in\R/\Z$. This concludes the proof of equality~\eqref{ee:eq}, hence of the proposition.
 \end{proof}
 
 \begin{rem}\label{differ:rem}
  Let us consider the surface $\Sigma_{0,4}$, and let $$\G_{0,4}=\langle c_1,c_2,c_3,c_4| c_1c_2c_3c_4=1\rangle$$ be a presentation of $\pi_1(\Sigma_{0,4})$, where $c_i$ is represented by a loop
  winding around the $i$-th boundary component of $\Sigma_{0,4}$. Observe that $\G_{0,4}$ is freely generated by $c_1,c_2,c_3$.
  Let us fix a representation $\rho\colon \G_{0,4}\to \Isom^+ (\mathbb{H}^2)$ such that $\rho(c_2)=\rho(c_1)^{-1}$ and $\rot(\rho(c_3)\rho(c_2))\neq \rot (\rho(c_3))+\rot(\rho(c_2))$ (it is not difficult to show that
  such a representation indeed exists). 
  Let also $\psi\colon \G_{0,4}\to\G_{0,4}$ be the automorphism defined by $\psi(c_1)=c_1$, $\psi(c_2)=c_2c_1$, $\psi(c_3)=c_3c_2$, and set $\rho'=\rho\circ\psi$. 
  From $\rho(c_2)=\rho(c_1)^{-1}$ we deduce $\rot(\rho(c_1))+\rot(\rho(c_2))=0$ and $\rho(c_4)=\rho(c_3^{-1}c_2^{-1}c_1^{-1})=\rho(c_3)^{-1}$, hence
  $\rot(\rho(c_4))+\rot(\rho(c_3))=0$.
  Thus
  by Proposition~\ref{boundary:euler}
  we have that the class of $e(\rho,\Sigma_{0,4},\bb\Sigma_{0,4})$ in $\R/\Z$ is given by
  $$
  -\rot(\rho(c_1))-\rot(\rho(c_2))-\rot(\rho(c_3))-\rot(\rho(c_4))=0\ .
  $$
  On the other hand,
  \begin{align*}
  \rot(\rho(\psi(c_1)))&=\rot(\rho(c_1))\, ,\\
  \rot(\rho(\psi(c_2)))&=\rot(\rho(c_2)\rho(c_1))=0\, ,\\
  \rot(\rho(\psi(c_3)))&=\rot(\rho(c_3)\rho(c_2))\, ,\\
  \rot(\rho(\psi(c_4)))& =\rot(\rho(c_3^{-1})\rho(c_2^{-1})\rho(c_1^{-1}))=\rot(\rho(c_3^{-1}))=-\rot(\rho(c_3))\ ,
  \end{align*}
so the class of $e(\rho',\Sigma_{0,4},\bb\Sigma_{0,4})$ in $\R/\Z$ is given by
$$
-\rot(\rho(c_1)) -\rot(\rho(c_3)\rho(c_2)) +\rot(\rho(c_3))=\rot(\rho(c_2))-\rot(\rho(c_3)\rho(c_2)) +\rot(\rho(c_3))\ ,
$$
which is not null by our assumptions. This shows that
$$
e(\rho,\Sigma_{0,4},\bb\Sigma_{0,4})\neq e(\rho',\Sigma_{0,4},\bb\Sigma_{0,4})\ .
$$
  \end{rem}

The following proposition shows that the Euler number of a representation is additive with respect to gluings. We recall that, if $\Sigma$ is a compact surface and 
$C\subseteq \Sigma$ is a simple loop, then the (possibly disconnected) surface obtained by cutting $\Sigma$ open along $C$ is the surface $\Sigma_C$ obtained by compactifying
$\Sigma\setminus C$ by adding one copy of $C$ to each topological end of $\Sigma\setminus C$. There is a natural map $\Sigma_C\to \Sigma$, which
describes $\Sigma$ as a quotient of $\Sigma_C$. If $C$ is separating, this map
is an inclusion on each connected component
of $\Sigma_C$.

\begin{prop}\label{additive:euler}
 Let $\Sigma$ be a compact connected oriented surface with fundamental group $\G=\pi_1(\Sigma)$, and let $\rho\colon \G\to\omeo$ be a representation. 
 \begin{enumerate}
  \item Let $C$ be a separating simple  loop in $\Sigma$, denote by $\Sigma_1,\Sigma_2$ the surfaces obtained by
cutting $\Sigma$ open along $C$, and let $\rho_i=\rho\circ (\alpha_i)_*\colon \G_i\to \omeo$, where $\G_i=\pi_1(\Sigma_i)$ and $\alpha_i\colon \Sigma_i\to \Sigma$ is the natural inclusion. Then
$$
e(\rho,\Sigma,\bb\Sigma)=e(\rho_1,\Sigma_1,\bb \Sigma_1)+e(\rho_2,\Sigma_2,\bb\Sigma_2)\ .
$$
\item Let $C$ be a non-separating simple loop in $\Sigma$, denote by $\Sigma'$ the surface obtained by
cutting $\Sigma$ open along $C$, and let $\rho'=\rho\circ (\alpha)_*\colon \G'\to \omeo$, where $\G'=\pi_1(\Sigma')$ and $\alpha\colon \Sigma'\to \Sigma$ is the natural quotient map. Then
$$
e(\rho',\Sigma',\bb\Sigma')=e(\rho,\Sigma,\bb \Sigma)\ .
$$ 
\end{enumerate}
\end{prop}
\begin{proof}
 Let us prove (2), the proof of (1) being similar and described in~\cite[Proposition 3.2]{BIW1}. 
 We consider only cochains with real coefficients, and omit this choice from our notation.
 
  Let us consider the inclusions
  \begin{align*}
   \beta^\bullet& \colon C^\bullet_b(\Sigma,\bb\Sigma \cup C)\to C^\bullet_b(\Sigma,\bb\Sigma)\ ,\\ 
   j_1^\bullet &\colon C^\bullet_b(\Sigma,\bb\Sigma)\to C^\bullet_b(\Sigma)\ ,\\
   j_2^\bullet & \colon C^\bullet_b(\Sigma,\bb\Sigma \cup C)\to C^\bullet_b(\Sigma)\ ,
   \end{align*}
   so that $j_2^\bullet=j_1^\bullet\circ \beta^\bullet$. We know from Theorem~\ref{reliso:thm} that
 the maps induced by $j_1$ and $j_2$ in bounded cohomology are isometric isomorphisms (in degree bigger than 1), so the same is true for $H^\bullet_b(\beta^\bullet)$. 
 Now, if we define
 $\overline{\alpha}\colon (\Sigma'\,\bb\Sigma')\to (\Sigma,\bb\Sigma \cup C)$ to be the map of pairs induced by $\alpha$, and we denote by $j_3\colon C^\bullet_b(\Sigma',\bb\Sigma')\to C^\bullet_b(\Sigma')$
 the natural inclusion, then it is easy to check that the following diagram is commutative:
$$
 \xymatrix{
 H^2_b(\Sigma,\bb\Sigma)\ar[rd]^{H^2_b(j_1^2)} \ar[rr]^{H^2_b(\beta^2)^{-1}} & & H^2_b(\Sigma,\bb\Sigma\cup C) \ar[rr]^{H^2_b(\overline{\alpha})} & & H^2_b(\Sigma',\bb\Sigma')\ar[ld]^{H^2_b(j_3^2)}\\
 & H^2_b(\Sigma) \ar[rr]^{H^2_b(\alpha)} & & H^2_b(\Sigma')
 }
$$
% The thesis is equivalent to the equality
% $$
% \langle e_b^\R(\rho',\Sigma',\bb\Sigma'),[\Sigma',\bb\Sigma']\rangle=\langle e_b^\R(\Sigma,\bb\Sigma),[\Sigma,\bb\Sigma]\rangle\ .
% $$
Let $\varphi\in C^2_b(\Sigma,\bb\Sigma)$ be a representative of $e_b^\R(\rho,\Sigma,\bb\Sigma)\in H^2_b(\Sigma,\bb\Sigma)$.
 Since $H^2_b(\beta^2)$ is an isomorphism, we may safely suppose that $\varphi$ vanishes on simplices supported in $C$,
 so $\varphi$ is a representative also for $H^2_b(\beta^{2})^{-1}(e_b^\R(\rho,\Sigma,\bb\Sigma))$. Now, if
 $z'$ is a relative fundamental cycle of $\Sigma'$, then $z'$ is taken by $\alpha$ into the sum $C_2(\alpha)(z')$ of a relative fundamental
 cycle $z$ of $\Sigma$ and a chain $c\in C_2(C)$ supported in $C$. Therefore, we have
 \begin{align*}
 & \langle H^2_b(\overline{\alpha})(H^2_b(\beta^{2})^{-1}(e_b^\R(\rho,\Sigma,\bb\Sigma))),[\Sigma',\bb\Sigma'] \rangle=\langle H^2_b(\overline{\alpha})([\varphi]),[\Sigma',\bb\Sigma']\rangle\\ = & \varphi(C_2(\alpha)(z'))=
 \varphi(z+c)=\varphi(z)=\langle e_b^\R(\rho,\Sigma,\bb\Sigma),[\Sigma,\bb\Sigma]\rangle\ .
 \end{align*}
 Now the conclusion follows from the fact that 
 $$
 H^2_b(\overline{\alpha})(H^2_b(\beta^{2})^{-1}(e_b^\R(\rho,\Sigma,\bb\Sigma)))=e_r^\R(\rho',\Sigma',\bb\Sigma')
 $$
 thanks to the commutativity of the diagram above.
\end{proof}

It is useful to put a topology on the space of representations of a group $\G$ into $\omeo$. We have already endowed $\omeo$ with the compact-open topology.
Now, for any group $\G$ the space $\Hom(\G,\omeo)$ may be thought as a subset of the space $(\omeo)^\G$ of functions from $\G$ to $\omeo$.
We endow $(\omeo)^\G$ with the product topology, and $\Hom(\G,\omeo)$ with the induced topology. Then, it is readily seen that a sequence $\rho_n\colon \G\to \omeo$
converges to a representation $\rho\colon \G\to \omeo$ if and only if $\rho_n(g)$ converges to $\rho(g)$ (with respect to the compact-open topology) for every $g\in\G$.
We also recall that, both in $\omeo$ and in $\omeot$, convergence with respect to the compact-open topology coincides with uniform convergence.

We are now going to show that the Euler number defines a continuous map on the space of representations. We begin with the following:

\begin{lemma}\label{continuous:tau}
 The bounded cocycle
 $$
 \tau\colon \omeo\times\omeo \to \R
 $$
 is continuous.
\end{lemma}
\begin{proof}
 Take $f_0,g_0\in\omeo$. We can choose neighbourhoods $U$ of $f_0$ and $V$ of $g_0$ in $\omeo$ on which the covering projection
 $p\colon \omeot\to\omeo$ admits continuous sections $s_U,s_V$. Then, for every $(f,g)\in U\times V$ we have
 $$
 \tau(f,g)=\rott(s(f)s(g))-\rott(s(f))-\rott(s(g))\ ,
 $$
 and the conclusion follows from Proposition~\ref{continuous:rot}.
\end{proof}

It is well-known that the Euler number is a continuous function on the space of representations of \emph{closed} surface groups into $\omeo$. 
The fact that the relative Euler number defines a continuous function on the space of representations of punctured surface groups was proved in~\cite[Theorem 1]{BIW1} (for a suitable definition
of relative Euler number) for representations
with values in Hermitian groups. The following theorem extends this result to the case of representations into $\omeo$.

\begin{thm}\label{euler:bd:continuous}
 The map
\begin{align*}
 \Hom(\G_{g,n},\omeo) & \to \R\ , \\
 \rho & \mapsto e(\rho,\sgn,\bb \sgn)
 \end{align*}
is continuous.
 \end{thm}
\begin{proof}
 Let us recall from Lemma~\ref{normnon} the explicit construction of the isomorphism $H^2_b(\G_{g,n},\R)\to H^2_b(\sgn,\R)$. One first chooses a set
 of representatives $R$ for the action of $\G_{g,n}$ on $\widetilde{\sgn}$ via deck automorphisms of the covering $p\colon\widetilde{\sgn}\to\sgn$. We can choose such a set so that the following condition holds:
 for every connected component $C$ of $\bb\sgn$, the set $R\cap p^{-1}(\bb\sgn)$ is contained in a single connected component of $\bb\widetilde{\sgn}$. 
 Then, the isometric isomorphism $H^2_b(\G_{g,n},\R)\to H^2_b(\sgn,\R)$ is induced by the $\G_{g,n}$-chain map $r^\bullet\colon C^\bullet_b(\G_{g,n},\R)\to C^\bullet_b(\sgn,\R)$ obtained by dualizing the following chain map:
 for every $x\in \widetilde{\sgn}$, $r_0(x)$ is equal to 
 the unique
$g\in \G_{g,n}$ such that $x\in g(R)$. For $n\geq 1$, if $s$ is a singular $n$-simplex with values in $\widetilde{\sgn}$, then 
$r_n(s)=(r_0(s (e_0)),\ldots,r_n(s(e_n)))\in \G_{g,n}^{n+1}$, where
$s (e_i)$ is the $i$-th vertex of $s$.

Let now $\rho\colon \G_{g,n}\to\omeo$ be a representation. By definition, the cochain
$\psi=r^2(\rho^*(\tau))$ is a representative of $e_b^\R(\rho)\in H^2_b(\sgn,\R)$, where we are identifying  $H^2_b(\sgn,\R)$
with $H^2_b(\G_{g,n},\R)$ via the map $H^2_b(r^2)$ described above. 
We claim that $\psi$ vanishes on chains supported in $\bb{\sgn}$ (as usual, we are identifying $\G_{g,n}$-invariant cochains on $\widetilde{\sgn}$ with
cochains on $\sgn$). Indeed, if $s\colon \Delta^2\to {\sgn}$ is supported
$\bb{\sgn}$, then it is supported in one connected component $C$ of $\bb\sgn$, and
our choice for the set of representatives $R$ implies that $r^2(s)=(g_0g^{n_0},g_0g^{n_1},g_0g^{n_2})$ for some $g_0\in \G_{g,n}$, $n_i\in\mathbb{Z}$,
where $g$ is the generator of a (suitable conjugate of) $\pi_1(C)$ in $\G_{g,n}=\pi_1(\sgn)$. Therefore,
if we denote by $\widetilde{f}\in\omeo$ a fixed lift of $\rho(g)$, then 
\begin{align*}
\psi(s) & =\tau(\rho(g^{n_1-n_0}),\rho(g^{n_2-n_1})\\ &=\rott\left(\widetilde{f}^{n_1-n_0}\widetilde{f}^{n_2-n_1}\right)-
\rott\left(\widetilde{f}^{n_1-n_0}\right)-\rott\left(\widetilde{f}^{n_2-n_1}\right)=0\ ,
\end {align*}
because $\rott$ is a homogeneous quasimorphism.

We have thus proved  that $\psi$ is a relative cocycle representing the bounded Euler class of $\rho$,
so by definition
$[\psi]=e_b^\R(\rho,\sgn,\bb\sgn)$ in $H^2_b(\sgn,\bb\sgn,\R)$. Let now $c=\sum_{i=1}^k \alpha_i\sigma_i$ be a fundamental cycle for $(\sgn\bb\sgn)$,
and let $r^2(\sigma_i)=(a_i,b_i,c_i)\in \G_{g,n}^3$ for every $i=1,\ldots,k$. We have
$$
e(\rho,\sgn,\bb\sgn)=\sum_{i=1}^k \alpha_i\psi(\sigma_i)=\sum_{i=1}^k \alpha_i \tau(a_i,b_i,c_i)=\sum_{i=1}^k \alpha_i\tau (\rho(a_i^{-1}b_i),\rho(b_i^{-1}c_i))\ .
$$
By Lemma~\ref{continuous:tau}, this formula shows  that $e(\rho,\sgn,\bb\sgn)$ varies continuously with respect to $\rho$. 
\end{proof}

The previous results allow us to exhibit an explicit formula for the Euler number of representations of punctured surfaces.
Indeed, let us fix a presentation of $\G_{g,n}=\pi_1(\sgn)$ of the following form:
$$
\G_{g,n}=\langle a_1,b_1,\ldots,a_g,b_g,c_1,\ldots,c_n\, |\, [a_1,b_1]\cdot \ldots \cdot [a_g,b_g]\cdot c_1\cdot \ldots \cdot c_n = 1\rangle\ .
$$
Here, $a_i$ and $b_i$ correspond to simple non-separating loops that transversely (and positively) intersect at the basepoint of $\sgn$ for every $i=1,\ldots,g$,
while $c_j$ corresponds to a loop going around the $j$-th boundary component of $\sgn$ for every $j=1,\ldots,n$. 
Moreover, up to free homotopy, the orientation of $c_i$ corresponds to the orientation induced by $\sgn$ on its $j$-th boundary component (we refer the reader to Figure~\ref{generators:fig} for an explicit
description of these generators).

\begin{center}
 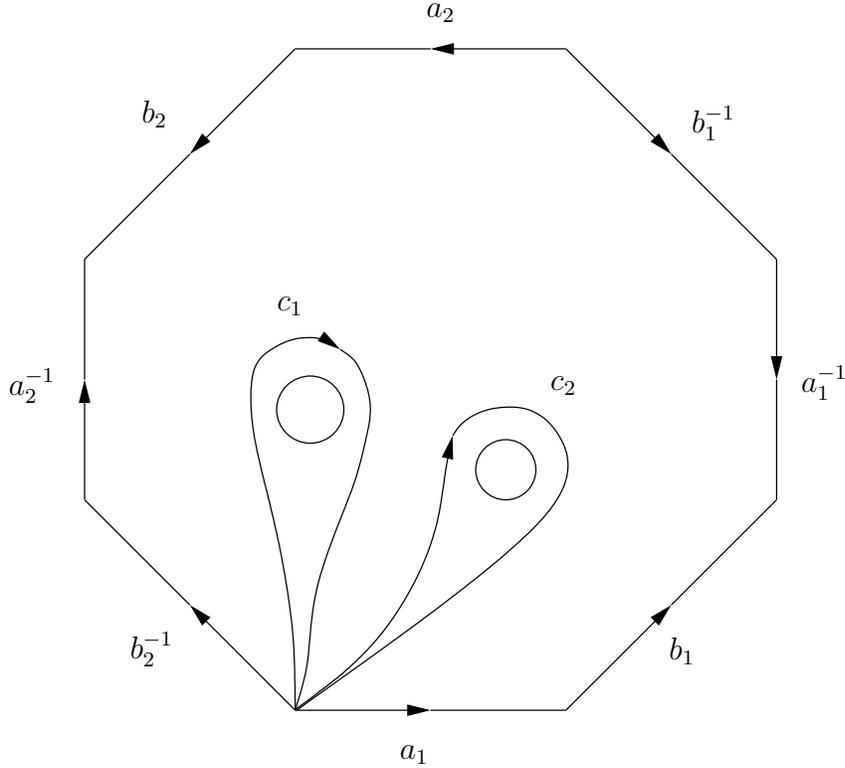
\begin{figure}
 \input{generators.pstex_t}
\caption{The standard generators of $\G_{2,2}$. Here the surface $\Sigma_{2,2}$ is described as a twice-punctured octagon with sides identified in pairs.}\label{generators:fig} 
 \end{figure}
\end{center}

Suppose $n\geq 1$ and
recall that, since $\G_{g,n}$ is free, every representation of $\G_{g,n}$ into $\omeo$ lifts to a representation into $\omeot$.

\begin{thm}\label{euler:bd:formula}
Let $n\geq 1$. Let $\rho\colon \G_{g,n}\to\omeo$ be a representation, and fix a lift $\widetilde{\rho}\colon \G_{g,n}\to \omeot$ of $\rho$. Then
 $$
e(\rho,\sgn,\bb\sgn)=\sum_{j=1}^n -\rott(\widetilde{\rho}(c_j))\ .
 $$
\end{thm}
\begin{proof}
Let us first prove that the right hand side of the equality in the statement does not depend on the choice of $\widetilde{\rho}$. Indeed,
if $\widetilde{\rho}'\colon \G_{g,n}\to\omeot$ also lifts $\rho$, then $\widetilde{\rho}'(a_i)=\widetilde{\rho}(a_i)\circ\tau_{n_i}$,
$\widetilde{\rho}'(b_i)=\widetilde{\rho}(b_i)\circ\tau_{m_i}$, $\widetilde{\rho}'(c_j)=\widetilde{\rho}(c_j)\circ\tau_{l_j}$ for some
$n_i,m_i,l_j\in\mathbb{Z}$. It is readily seen that the condition that both $\widetilde{\rho}$ and $\widetilde{\rho}'$ are representations
implies that $\sum_{j=1}^n l_j=0$. Therefore, 
$$
\sum_{j=1}^n -\rott(\widetilde{\rho}'(c_j))=\sum_{j=1}^n -\rott(\widetilde{\rho}(c_j)\circ \tau_{l_j})=
\sum_{j=1}^n -\rott(\widetilde{\rho}(c_j))-l_j=\sum_{j=1}^n -\rott(\widetilde{\rho}(c_j))\ .
$$

Now let us consider both $e(\cdot,\sgn,\bb\sgn)$ and $\sum_{j=1}^n -\rott(\widetilde{\cdot}(c_j))$ as real functions
on the space $\Hom(\G_{g,n},\omeo)$.
By Theorems~\ref{euler:bd:continuous} the map $e(\cdot,\sgn,\bb\sgn)$ is continuous. Moreover, using that $\omeot\to\omeo$ is a covering, it is  easily checked
that one can choose continuously the lifts to $\omeot$ of the representations lying in a small neighboourhood of a fixed $\rho\colon \G\to \omeo$. Together with Theorem~\ref{continuous:rot}, 
this shows that also the map $\sum_{j=1}^n -\rott(\widetilde{\cdot}(c_j))$ is continuous. Moreover,
when composed with the projection $\R\to\R/\Z$, the functions $e(\cdot,\sgn,\bb\sgn)$ and $\sum_{j=1}^n -\rott(\widetilde{\cdot}(c_j))$ coincide 
by Proposition~\ref{boundary:euler}. Finally, these functions both obviously vanish at the trivial representation,
so they coincide everywhere, since $\R\to \R/\Z$ is a covering and $\Hom(\G_{g,n},\omeo)$ is connected (because $\G_{g,n}$ is free). 
\end{proof}

It is interesting to observe that the formula for the Euler number obtained in Theorem~\ref{euler:bd:formula} may be of use also to compute the Euler number of representations
of closed surface groups:

\begin{thm}\label{closed:formula}
Let $\rho\colon \G_g\to \omeo$ be a representation of
the fundamental group of the closed surface $\Sigma_g$ of genus $g$, let $a_1,b_1,\ldots,a_g,b_g$ be the standard generators of $\G_g$, and denote by
$\widetilde{a}_i$ (resp.~$\widetilde{b}_i$) an arbitrary lift of $\rho(a_i)$ (resp.~$\rho(b_i)$) to $\omeot$. Then
$$
[\widetilde{a}_i,\widetilde{b}_i]\cdot\ldots\cdot [\widetilde{a}_g,\widetilde{b}_g]=\tau_e\ ,
$$
where $\tau_e$ is the translation by 
$$
e=e(\rho)\ .
$$
\end{thm}
\begin{proof}
 Let $D\subseteq \Sigma_g$ be a closed embedded disc, and let $C=\partial D$. Let also $\Sigma_{g,1}\subseteq \Sigma_g$ be the surface with boundary obtained by removing the internal part of $D$ from $\Sigma_g$.
 Let $\rho'\colon \pi_1(\Sigma_{g,1})\to\omeo$ be the representation obtained by composing $\rho$ with the map induced by the inclusion $\Sigma_{g,1}\to\Sigma_g$, and let
 $\widetilde{\rho'}\colon \pi_1(\Sigma_{g,1})\to \omeot$ be the lift of $\rho'$ defined by
 $\widetilde{\rho'}(a_i)=\widetilde{a}_i$ and $\widetilde{\rho'}(b_i)=\widetilde{b}_i$ (here, we are identifying the standard generators of $\pi_1(\Sigma_{g,1})$ with their images in $\pi_1(\Sigma_g)$). Since $D$ is simply connected,
 by Proposition~\ref{additive:euler} we have $$ e(\rho)=e(\rho',\Sigma_{g,1},\bb\Sigma_{g,1})\ .$$
 On the other hand, since $\rho([a_1,b_1]\cdot\ldots\cdot [a_g,b_g])=1$, we have 
 $[\widetilde{a}_1,\widetilde{b}_1]\cdot\ldots\cdot [\widetilde{a}_g,\widetilde{b}_g]=\tau_k$ for some $k\in\mathbb{Z}$. Now the conclusion follows from Theorem~\ref{euler:bd:formula}, together with the obvious
 fact that $\rott(\tau_k)=k$.
\end{proof}

Milnor-Wood inequalities also hold for the Euler number of representations of punctured surface groups:

\begin{thm}\label{MW:boundary}
Let $\rho\colon \G_{g,n}\to \omeo$ be a representation, and suppose that $\chi(\sgn)=2-2g-n<0$. Then
$$
|e(\rho,\sgn,\bb\sgn)|\leq {|\chi(\sgn)|}\ .
$$
\end{thm}
\begin{proof}
By Lemma~\ref{norm:euler} and Corollary~\ref{punctured:surf:simpl} we have
\begin{align*}
|e(\rho,\sgn,\bb\sgn)|&=\langle e_b^\R(\rho,\sgn,\bb\sgn),[\sgn,\bb\sgn]\rangle\\ &\leq
\|e_b^\R(\rho,\sgn,\bb\sgn)\|_\infty\cdot \|\sgn,\bb\sgn\|\\ & \leq {|\chi(\sgn)|}\ .
\end{align*}
\end{proof}

It is still true that the inequality appearing in the previous theorem is sharp.
Moreover, just as in the case of closed surfaces (which is treated in detail in the next section)
equality is attained if and only if  $\rho$ is geometric is a suitable sense. 
We refer the reader to Section~\ref{further:maximal} for further details about this issue.

\section{Maximal representations}\label{maximal:section}
Milnor-Wood inequalities provide a sharp bound on the possible Euler numbers of flat circle bundles over closed surfaces, or, equivalently, on the possible Euler numbers
of representations of surface groups into $\omeo$. The study of such space of representations is a very important and active research field, and is related to many areas of mathematics, from $1$-dimensional dynamics to
low-dimensional topology. A particularly interesting (and very well understood) class of representations is the class
of \emph{maximal} ones, according to the following definition (recall that $\G_g$ is the fundamental group of the closed oriented surface $\Sigma_g$ of genus $g$).

\begin{defn}
Let $g\geq 2$. A representation $\rho\colon \G_g\to \omeo$ is \emph{maximal} if $|e(\rho)|=2g-2$.
\end{defn}

Thanks to Milnor-Wood inequalities, maximal representations have the maximal (or minimal) Euler number among all possible group actions of $\G_g$. As anticipated above, the space of maximal representations
into $\omeo$ admits a complete and very neat characterization. Before describing some classical results in this direction we need to recall the definition of \emph{geometric} representation.
Let $\Sigma_g$ be endowed with a hyperbolic metric, and fix an orientation-preserving identification of the metric universal covering $\widetilde{\Sigma}_g$ of $\Sigma_g$ with the hyperbolic plane $\mathbb{H}^2$.
Via this identification, every covering automorphism of $\widetilde{\Sigma}_g$ corresponds to an orientation-preserving isometry of $\mathbb{H}^2$. Hence,
any identification of $\G_g$ with the group of automorphisms of the universal covering $\widetilde{\Sigma}_g$ induces an isomorphism between $\G_g$ and a subgroup of $\Isom^+ (\mathbb{H}^2)$ acting
freely and properly discontinuously on $\mathbb{H}^2$. We say that such an isomorphism $\rho\colon \G_g\to \Isom^+ (\mathbb{H}^2)$ is a \emph{holonomy representation} associated to the hyperbolic metric
on $\Sigma_g$ we started with. It is easy to check that two holonomy representations associated to the same hyperbolic metric may differ one from the other only by the precomposition
with an automorphism of $\G_g$ and by  conjugation by an element in $\Isom^+ (\mathbb{H}^2)$.

\begin{defn}
 Let $\rho\colon \G_g\to \Isom^+ (\mathbb{H}^2)$ be a representation. Then $\rho$ is \emph{geometric} if it is the holonomy of a hyperbolic structure on the closed oriented surface $\Sigma_g$ of genus $g$, or if it is the holonomy
 of a hyperbolic structure on the closed oriented surface $\overline{\Sigma}_g$ obtained by reversing the orientaton of $\Sigma_g$.
 %In other words, $\rho$ is geometric if the following holds: the surface $\Sigma_g$ can be endowed with a hyperbolic metric in such a way that $\rho$ 
 %or, equivalently,
 %if $\rho$ establishes an isomorphism between $\G_g=\pi_1(
 %if $\rho(\G_g)$ is a subgroup of $\Isom^+ (\mathbb{H}^2)$ acting freely and properly discontinuously on $\mathbb{H}^2$ in such a way that $\mathbb{H}^2/\G_g$ is homeomorphic to $\Sigma_g$. 
 \end{defn}

 We have already seen in the proof of Proposition~\ref{2g-2:exists} that $\Isom^+ (\mathbb{H}^2)$ naturally acts on $\partial \mathbb{H}^2\cong S^1$, so that
 $\Isom^+ (\mathbb{H}^2)$ may canonically identified with a subgroup of $\omeo$ (when realizing $\mathbb{H}^2$ as the half-plane $\mathcal{H}$ of complex numbers with positive imaginary part, isometries of the hyperbolic metric correspond
 to biholomorphisms of $\mathcal{H}$, which in turn may be represented by fractional linear transformations with real coefficients; in this way we get an identification between
 $\Isom^+(\mathbb{H}^2)$ and $PSL(2,\mathbb{R})$, and the action of $\Isom^+(\mathbb{H}^2)$ on $S^1=\partial \mathcal{H}=\mathbb{R}\cup\{\infty\}=\mathbb{P}^1(\mathbb{R})$ corresponds to the natural action
 of $PSL(2,\mathbb{R})$ on the real projective line). Therefore, to every representation $\rho\colon \G_g\to \Isom^+ (\mathbb{H}^2)$ we can associate its Euler number $e(\rho)$, and Proposition~\ref{2g-2:exists}
 may be restated as follows:
 
 \begin{prop}\label{Goldman:easy}
  Let $\rho\colon \G_g\to \Isom^+ (\mathbb{H}^2)$ be a geometric representation. Then $\rho$ is maximal. 
 \end{prop}
\begin{proof}
 Proposition~\ref{2g-2:exists} shows that, if $\rho$ is the holonomy of a hyperbolic structure on  $\Sigma_g$, then $e(\rho)=2-2g$. On the other hand,
 by reversing the orientation of the surface, the Euler number just changes its sign, and this concludes the proof.
\end{proof}

A celebrated result by Goldman implies that also the converse of Proposition~\ref{Goldman:easy} holds, i.e.~geometric representations in $PSL(2,\mathbb{R})$ are precisely the maximal ones:

\begin{thm}[\cite{Goldth}]\label{Goldman:thm}
 Let $\rho\colon \G_g\to \Isom^+ (\mathbb{H}^2)$ be a representation. Then $\rho$ is geometric if and only if it is maximal. 
\end{thm}

It is not difficult to show that a representation $\rho\colon \G_g\to \Isom^+ (\mathbb{H}^2)$ is geometric if and only if it is topologically conjugate to a geometric representation, and this holds in turn
if and only if it is semi-conjugate to a geometric representation.
Indeed, much more is true: even without restricting to actions via isometries of the hyperbolic plane, the whole space of maximal representations consists of a single
semi-conjugacy class (see e.g.~\cite{Matsu, iozzi, calegari:max}; we refer the reader to Section~\ref{further:maximal} for a brief discussion of related topics).

Our proof of Theorem~\ref{Goldman:thm} is based on a slight variation of the argument described in~\cite{BIW1}. In particular, at a critical passage we will make use of Euler numbers
of representations of \emph{punctured} surface groups. The following lemma will prove useful:

\begin{lemma}\label{punctured:torus}
 Let $\rho\colon \G_{1,1}\to \Isom^+ (\mathbb{H}^2)$ be a representation, and let $a,b$ be the standard generators
 of $\G_{1,1}=\pi_1(\Sigma_{1,1})$. If $\rho(a)$ is elliptic, then
 $$
|e(\rho,\Sigma_{1,1},\bb\Sigma_{1,1})|<1 \ .
 $$
\end{lemma}
\begin{proof}
 The group $\G_{1,1}$ has   the standard presentation $\langle a,b,c\, |\, [a,b]c=1\rangle$. Therefore, if $\widetilde{\rho}\colon \G_{1,1}\to \omeot$ is a lift of $\rho$, then by Theorem~\ref{euler:bd:formula} we have
 $$
e(\rho,\Sigma_{1,1},\bb\Sigma_{1,1})=-\rott(\widetilde{\rho}(c)) =\rott(\widetilde{\rho}([a,b]))\ ,
 $$
 where we used that $\rott(\widetilde{f}^{-1})=-\rott(\widetilde{f})$ for every $\widetilde{f}\in\omeot$. 
 For the sake of simplicity, let us set $\widetilde{a}=\widetilde{\rho}(a)$, $\widetilde{b}=\widetilde{\rho}(b)$.
 Up to conjugacy we may suppose that $\rho(a)$ is a rotation, so up to suitably choosing the lift $\widetilde{\rho}$ of $\rho$ we may assume that there exists $\theta\in [0,1)$ such that 
 $\widetilde{a}(x)=x+\theta$ for every $x\in\R$. We then have $\widetilde{a}^{-1}(x)=x-\theta$, so
 \begin{equation}\label{a1b2:eq}
 \widetilde{a}(\widetilde{b}(\widetilde{a}^{-1}(\widetilde{b}^{-1}(x))))=\widetilde{a}(\widetilde{b}(\widetilde{b}^{-1}(x)-\theta)=\widetilde{b}(\widetilde{b}^{-1}(x)-\theta)+\theta\ .
 \end{equation}
 Recall now that $\widetilde{b}$ is monotone and commutes with integral translations, so from $\widetilde{b}^{-1}(x)-1<\widetilde{b}^{-1}(x)-\theta\leq \widetilde{b}^{-1}(x)$ we obtain
 $$
 x-1=\widetilde{b}(\widetilde{b}^{-1}(x)-1)<\widetilde{b}(\widetilde{b}^{-1}(x)-\theta)\leq \widetilde{b}(\widetilde{b}^{-1}(x))=x .
 $$
 From equation~\eqref{a1b2:eq} we thus get
 $$
x-1\leq x-1+\theta< [\widetilde{a},\widetilde{b}](x)\leq x+\theta<x+1\ .
 $$
 Since the function $x\mapsto [\widetilde{a},\widetilde{b}](x)-x$ is periodic, this implies that there exists $\varepsilon>0$ such that
 $$
 x-1+\varepsilon\leq [\widetilde{a},\widetilde{b}](x)\leq x+1-\varepsilon
 $$
 for every $x\in\R$. An easy induction now implies that
 $$
 x-n+n\varepsilon\leq [\widetilde{a},\widetilde{b}]^n(x)\leq x+n-n\varepsilon\ ,
 $$
 for every $x\in\R$, $n\in\mathbb{N}$,
 so
 $$
 -1+\varepsilon = \lim_{n\to \infty} \frac{n+n\varepsilon}{n}\leq \lim_{n\to\infty} \frac{[\widetilde{a},\widetilde{b}]^n(0)}{n}\leq \lim_{n\to \infty} \frac{n-n\varepsilon}{n}=1-\varepsilon\ .
 $$
 By definition of $\rott$, this implies that 
 $$
 |e(\rho,\Sigma_{1,1},\bb\Sigma_{1,1})|=|\rott(\widetilde{\rho}([a,b]))|\leq 1-\varepsilon\ .
 $$
 \end{proof}

The previous lemma shows that elliptic elements prevent a representation of the punctured torus group from attaining the extremal Euler numbers. On the other hand, it
is a classical result that elliptic elements are in a sense the only obstruction for a subgroup of $\Isom^+ (\mathbb{H}^2)$ to be discrete. More precisely, we have
 the following classical characterization of geometric representations:

\begin{prop}\label{geometric:prop}
 Let $\rho\colon \G_g\to \Isom^+ (\mathbb{H}^2)$ be a representation. Then the following conditions are equivalent:
 \begin{enumerate}
  \item $\rho$ is geometric;
  \item $\rho$ is injective, and $\rho(\G_g)$ acts freely and properly discontinuously on $\mathbb{H}^2$;
  \item For every $g\in \G_g\setminus\{1\}$, the isometry $\rho(g)$ is not elliptic.
 \end{enumerate}
\end{prop}
\begin{proof}
 (1) $\Longrightarrow$ (3): 
 We know that, if $\rho$ is geometric, then $\rho(g)$ acts freely on $\mathbb{H}^2$ for every $g\in \G_g\setminus\{1\}$. But an element of $\Isom^+ (\mathbb{H}^2)$
 has fixed points if and only if it is elliptic, and this concludes the proof.

(3) $\Longrightarrow$ (2): 
Since the identity is an elliptic isometry, injectivity of $\rho$ is clear. Moreover, an isometry of the hyperbolic plane has fixed points if and only if it is elliptic,
so $\rho(\G_g)$ acts freely  on $\mathbb{H}^2$. It is readily seen that a subgroup of $\Isom^+ (\mathbb{H}^2)$ acts  properly discontinuously on $\mathbb{H}^2$ if and only if it is discrete,
so we are left to show that $\rho(\G_g)$ is indeed discrete. To this aim we exploit a very classical result on Fuchsian groups, which states that non-elementary subgroups
of $\Isom^+ (\mathbb{H}^2)$ that do not contain elliptic elements are discrete (see e.g.~\cite[Theorem 8.3.1]{Beardon} for a nice geometric proof of this fact, and for the definition
of elementary subgroup of $\Isom^+ (\mathbb{H}^2)$). Therefore, in order to conclude we just need to check that $\rho(\G_g)$ cannot be elementary.
It is known  that every elementary subgroup of $\Isom^+ (\mathbb{H}^2)$ is solvable. But $\rho(\G_g)\cong \G_g$ is not solvable
(for example, because it admits an epimorphism on the free group on two generators), and this concludes the proof.

(2) $\Longrightarrow$ (1): Our hypothesis implies that $\mathbb{H}^2/\rho(\G_g)$ is a hyperbolic surface $\Sigma$, and in order to conclude we just need to prove that
$\Sigma$ is homeomorphic to $\Sigma_g$. However, by construction $\pi_1(\Sigma)=\pi_1(\Sigma_	g)$, and it is well-known that this suffices to conclude
that $\Sigma$ and $\Sigma_g$ are homeomorphic.
\end{proof}

Henceforth and until the end of the chapter, every (co)homology module will be understood to be with \emph{real} coefficients.

\begin{lemma}\label{covering:euler}
 Let $p\colon \Sigma'\to\Sigma$ be a finite covering of degree $d\in\mathbb{N}$ between closed oriented surfaces, let $\rho\colon \G\to\omeo$ be a representation, and let 
 $\rho'=\rho\circ p_*\colon \G'\to\omeo$, where $\G=\pi_1(\Sigma)$, $\G'=\pi_1(\Sigma')$, and $p_*\colon \G'\to \G$ is the map induced by $p$. then
 $$
 e(\rho')=d\cdot e(\rho)\ .
 $$
\end{lemma}
\begin{proof}
 It readily follows from the definitions that $e_b^\R(\rho')=H^2_b(p_*)(e_b^\R(\rho))$, so
 (the identifications $H^2_b(\Sigma)=H^2_b(\G)$, $H^2_b(\Sigma')=H^2_b(\G')$ being understood):
\begin{align*}
 e(\rho')&=\langle e_b^\R(\rho'),[\Sigma']\rangle=\langle H^2_b(p_*)(e_b^\R(\rho)),[\Sigma']\rangle\\ &=
 \langle e_b^\R(\rho), H_2(p)([\Sigma'])\rangle=d\langle e_b^\R(\rho),[\Sigma]\rangle =d\cdot e(\rho)\ .
 \end{align*}
\end{proof}

We are now ready to prove Theorem~\ref{Goldman:thm}. So, let us suppose that $\rho\colon \G_{g}\to \Isom^+ (\mathbb{H}^2)$ is a maximal representation. Thanks to Proposition~\ref{geometric:prop},
it is sufficient to show that $\rho(g)$ is not elliptic for every $g\in\G_{g}\setminus\{1\}$. Assume by contradiction that there exists $\overline{g}\in \G_{g}\setminus\{1\}$ with $\rho(\overline{g})$ elliptic.
A result by Scott~\cite{Scott} ensures that we can find a finite covering $p\colon \Sigma'\to \Sigma_g$ and an element $g'\in\G'=\pi_1(\Sigma')$ such that $p_*(g')=\overline{g}$,
and $g'$ is represented by a simple loop $A$ in $\Sigma'$, where $p_*$ denotes the map induced by $p$ at the level of fundamental groups. Up to passing to a double covering of $\Sigma'$, we
can also assume that $A$ is separating.
Let $\rho'=\rho\circ p_*\colon \G'\to \Isom^+ (\mathbb{H}^2)$ be the representation induced by $\rho$.
If we denote by $d$ the degree of the covering $p$, then $\chi(\Sigma')=d\chi(\Sigma)$, while $e(\rho')=d\cdot e(\rho)$ by the previous lemma, so $\rho'$ is still maximal. 

Since $A$ is non-separating, there exists an embedded punctured torus $\Sigma_{1,1}\subseteq \Sigma'$ such that $A\subseteq \Sigma_{1,1}$ is the support of a standard generator of $\pi_1(\Sigma_{1,1})$. 
Let now $C\subseteq \Sigma'$ be the boundary of $\Sigma_{1,1}$, and denote by $\Sigma_1=\Sigma_{1,1}$ and $\Sigma_2$ the surfaces obtained by cutting $\Sigma'$ open along $C$.
Let also $\rho_i\colon \pi_1(\Sigma_i)\to \Isom^+ (\mathbb{H}^2)$ be the representations induced by $\rho'$, $i=1,2$, and observe that, if $a\in\pi_1(\Sigma_{1,1})$ is the standard generator corresponding
to the loop $A$, then $\rho_1(a)=\rho'(g')=\rho(\overline{g})$ is elliptic, so Lemma~\ref{punctured:torus} implies that 
$$
|e(\rho_1,\Sigma_1,\bb\Sigma_1)|<1=|\chi(\Sigma_1)|\ .
$$
On the other hand, Milnor-Wood inequalities for punctured surfaces (see Theorem~\ref{MW:boundary}) imply that 
$$
|e(\rho_2,\Sigma_2,\bb\Sigma_2)|\leq |\chi(\Sigma_2)|\ .
$$
By Proposition~\ref{additive:euler} we can conclude that
$$
|e(\rho',\Sigma',\bb\Sigma')|= |e(\rho_1,\Sigma_1,\bb\Sigma_1)| + |e(\rho_2,\Sigma_2,\bb\Sigma_2)|<|\chi(\Sigma_1)|+|\chi(\Sigma_2)|=|\chi(\Sigma')|\ ,
$$
i.e.~$\rho'$ is not maximal. This gives the desired contradiction, and concludes the proof of Theorem~\ref{Goldman:thm}.

\section{Further readings}\label{further:maximal}

\subsection*{Milnor-Wood inequalities}
Milnor-Wood inequalities provided the first effective tool to construct obstructions for a circle bundle to be flat, or 
for a rank-2 vector bundle to support a flat connection. The original proofs of Theorems~\ref{Milnor:thm} and~\ref{Wood:thm}, which are due respectively to Milnor and Wood, were based on 
the computation of Euler numbers described in Theorem~\ref{closed:formula}, together with an accurate analysis of the geometry of commutators in $\omeot$. Indeed, Milnor's and Wood's arguments contained \emph{in nuce}
the notion of quasimorphism, and exploited the fact that effective estimates may be provided for the evaluation of quasimorphisms on products of commutators. We refer the reader
to~\cite{Gold:survey} for a very nice survey about the ideas that originated from Milnor's and Wood's seminal work. 
 For an interpretation
of Milnor-Wood inequalities in terms of hyperbolic geometry see also~\cite{Mathews}.

\subsection*{Milnor-Wood inequalities and rotation numbers}
Theorems~\ref{euler:bd:formula} and~\ref{closed:formula} provide explicit formulas relating the Euler number of a representation to the rotation numbers
of (the images of) suitably chosen words in the generators of the fundamental group of a surface.
For example, as a consequence of Theorems~\ref{euler:bd:formula} and of Milnor-Wood inequalities for surfaces with boundary (see Theorem~\ref{MW:boundary}), we get that, if $F=\G_{g,1}$ is a free group
of rank $2g$ generated by $a_1,b_1,\ldots,a_g,b_g$ and $\rho\colon F\to\omeot$ is a representation,
then the maximal value that the quantity
$$
\rott(\rho([a_1,b_1]\cdot\ldots\cdot [a_g,b_g])
$$
can attain is precisely $|\chi(\Sigma_{g,1})|$. This result inscribes into a wider research area, which deals with the study of
which inequalities are satisfied by the rotation numbers of (specific) elements in $\rho(\G)$, where $\rho\colon \G\to \omeot$ is a fixed representation.
Several questions of this type are addressed in~\cite{calegari-walker}, where a particular attention is payed to the case when $\G$ is free.

\subsection*{Maximal representations}
The strategy adopted by Goldman in his proof of Theorem~\ref{Goldman:thm} is quite different from the one described here. Indeed, Goldman's approach to the study of the Euler number of representations
in $PSL(2,\mathbb{R})$ makes an intensive use of hyperbolic geometry, and is based on the study of sections of flat bundles associated to representations: in fact, it is not difficult to see
that a representation is the holonomy of a geometric structure if and only if the associated bundle admits a section which is transverse to the leaves of the canonical foliation. 

There exist by now several quite different
proofs of Theorem~\ref{Goldman:thm}. Two different  proofs of geometricity of maximal representations entirely based on techniques coming from bounded cohomology are described in~\cite{iozzi, BIW1}, while
a proof based on a complete different method is due to Calegari~\cite{calegari:max}.
In fact, the results proved in~\cite{iozzi, calegari:max} imply the following stronger result, which was first obtained by Matsumoto~\cite{Matsu}:

\begin{thm}\label{matsuiozzi}
Let $\rho\colon \G_g\to\omeo$ be a maximal representation. Then $\rho$ is semi-conjugate to a geometric representation into $\Isom^+ (\mathbb{H}^2)$.
\end{thm}

Indeed, it is shown in~\cite{calegari:max}  that 
 there is a \emph{unique} element $\varphi\in H^2_b(\G_g,\R)$ 
 such that $\|\varphi\|_\infty=1$ and $\langle \varphi, [\Sigma_g]\rangle=2\chi(\Sigma_g)$. Therefore,
 if $\rho\colon \G_g\to\omeo$ is maximal, then necessarily $e_b(\rho)=\varphi/2$. Thanks to Ghys' Theorem, this implies in turn 
 that maximal representations lie in a unique semi-conjugacy class.

\subsection*{The real bounded Euler class}
Let us recall that the Euler number of a representation depends only on its \emph{real} bounded Euler class. It is interesting to investigate how much information is lost in passing from integral to real coefficients.
We have already described in Theorem~\ref{Matsu:char} a characterization of semi-conjugacy in terms of the real Euler bounded class (and, even more interestingly, in terms of its canonical representative $\tau$).
When $\G$ is a lattice in a locally compact second countable group $G$, 
in~\cite{Burger:ext} Burger described a complete characterization of extendability of representations $\rho\colon \G\to\omeo$ to the ambient group $G$
in terms of the real bounded Euler class of $\rho$, thus  providing a unified treatment of rigidity results of Ghys, Witte-Zimmer, Navas and Bader-Furman-Shaker. Among the results proved in~\cite{Burger:ext} there is also a
complete characterization of representations $\rho\colon \G\to\omeo$ for which $\|e_b^\R(\rho)\|_\infty=\|e_b^\R\|_\infty=1/2$.

\subsection*{Components of spaces of representations}
The Euler number provides a continuous map from $\Hom(\G_g,\omeo)$ to $\Z$. As a consequence, representations with distinct Euler number
necessarily lie in distinct connected components of $\Hom(\G_g,\omeo)$. If one restricts to considering representations with values in
$\Isom^+ (\mathbb{H}^2)$, then a celebrated result by Goldman says that if two elements in  $\Hom(\G_g,\Isom^+ (\mathbb{H}^2))$
have the same Euler number, then they can be connected by a path in $\Hom(\G_g,\Isom^+ (\mathbb{H}^2))$: therefore, 
the connected components of $\Hom(\G_g,\Isom^+ (\mathbb{H}^2))$ are \emph{precisely}
the preimages under the Euler number of the integers in $[\chi(\Sigma_g),-\chi(\Sigma_g)]$ (see~\cite{Gold:top}). Things get much more complicated when studying the connected components
of $\Hom(\G_g,\omeo)$. Theorem~\ref{matsuiozzi} implies that maximal representations define a connected component of 
$\Hom(\G_g,\omeo)$. Nevertheless, it is possible to construct (necessarily non-maximal) elements of $\Hom(\G_g,\omeo)$ which share the same Euler number, and which lie in distinct
connected components of $\Hom(\G_g,\omeo)$. Indeed, it is still unknown whether the number of the connected components of
$\Hom(\G_g,\omeo)$ is finite or not. We refer the reader to~\cite{Mann, Matsu:new} for recent results about this topic.

\subsection*{Surfaces with boundary}
In developing his theory, Goldman
also defined an Euler number for representations of fundamental groups of surfaces with boundary, and our definition of Euler number for such representations (which is taken from ~\cite{BIW1}) coincides in fact with Goldman's one.
However, Goldman's definition only works for representations that send peripheral loops into parabolic or hyperbolic elements (coherently with this fact and with Theorem~\ref{MW:boundary}, Goldman's Euler number
is always an integer), so it seems less suited to the study of deformations of geometric structures with cone angles at punctures (such structures necessarily have elliptic peripheral holonomy).

We have defined geometric representations only for fundamental groups of closed surfaces. Indeed, one can easily extend this notion to cover also the case of compact surfaces with boundary: 
a representation $\rho\colon \G_{g,n}\to \Isom^+ (\mathbb{H}^2)$ is \emph{geometric} if it is the holonomy representation of a complete finite area
hyperbolic structure with geodesic boundary on $\sgn'$, where $\sgn'$ is obtained from $\sgn$ by removing some of its boundary components (so that one cusp is appearing at each topological end of $\sgn'$). 
It is still easy to see that geometric representations realize the maximal Euler number. Moreover, just as in the closed case, also the converse is true: maximal representations are geometric also for surfaces
with boundary:

\begin{thm}[\cite{BIW1, calegari:max}]
 Let $\rho\colon \G_{g,n}\to \Isom^+ (\mathbb{H}^2)$ be a representation. Then $|e(\rho,\sgn,\bb\sgn)|=|\chi(\sgn)|$ if and only if $\rho$ is geometric.
\end{thm}

In the case with boundary, Calegari's characterization of bounded classes of unitary norm taking 
maximal value on the (relative) fundamental class of the
surface may be restated as follows:
there is a unique
homogeneous
quasimorphism on $\G_{g,n}$ of defect 1 which takes the maximal value on the
boundary
conjugacy class (or classes if $n\geq 2$).
Indeed, the case of
closed surfaces easily follows from that of bounded surfaces, which is the main topic of study in~\cite{calegari:max}.

\subsection*{Milnor-Wood inequalities in higher dimensions}
Milnor-Wood inequalities have been generalized in several directions by now. 
Namely, Goldman's Theorem implies that, at least in the case of fundamental groups of closed surfaces, maximal representations in $\Isom^+ (\mathbb{H}^2)$ define a connected component of the whole space of representations, and 
their conjugacy classes are in natural bijection with Teichm\"uller space. When replacing $\Isom^+ (\mathbb{H}^2)$ with a semisimple Lie group $G$, one may wonder
whether some components (or specific subsets, when $\Sigma$ is non-compact) of $\Hom(\pi_1(\Sigma),G)$ can play the role of (a suitable analogue of) Teichm\"uller space. This circle of ideas has lead
to what is currently known as \emph{Higher Teichm\"uller theory}. It is not possible to list here even a small number of papers that could introduce the reader to the subject.
For an approach which is very closely related to the topics developed in this book, we just refer to the beautiful survey~\cite{BIW2}.

\chapter{The bounded Euler class in higher dimensions \\and the Chern conjecture}
In the previous chapter we have discussed some applications 
of the boundedness of the Euler class
to the study of flat topological circle bundles.
We now analyze the case
of flat $n$-sphere bundles for $n\geq 2$. We concentrate here on the case of \emph{linear} sphere bundles,
which is better understood.

The first description of a uniformly bounded representative for the (simplicial) Euler class of a flat linear sphere bundle
$E$ is due to Sullivan~\cite{Sullivan}. Since $E$ is linear, we have that $E=S(V)$ for some
flat rank-$(n+1)$ vector bundle (see Remark~\ref{flatvector} for the definition of flat vector bundle). Sullivan observed that,
in order to compute the Euler class of $E$, one may analyze just \emph{affine} sections of $V$ 
over singular simplices, where the word affine makes sense exactly because $V$ is a flat vector bundle.
Using this, he proved that the Euler class of a flat linear bundle over a simplicial complex is represented by a 
\emph{simplicial} cocycle of norm at most one. 
Of course, any simplicial cochain on a compact simplicial complex is bounded.
However, the fact that the norm of Sullivan's cocycle is bounded by $1$ already implies
that, if $M$ is a triangulated oriented manifold of dimension $(n+1)$, then the Euler number of $E$ is bounded by the number of top-dimensional
simplices in a triangulation of $M$. 
In order to get better estimates, one
 would like to promote the bounded simplicial Euler cocycle
to a bounded singular cocycle. This would also allow to replace the number of top-dimensional
simplices in a triangulation of $M$ with the simplicial volume of $M$ in the above upper bound for the Euler number of any flat linear sphere bundle on $M$.
This can be done essentially in two ways: one could either invoke a quite technical result by Gromov~\cite[\S 3.2]{Gromov}, which ensures that the bounded cohomology (with real coefficients)
of $M$ is isometrically isomorphic to the simplicial bounded cohomology of a suitable multicomplex $K(M)$, thus reducing computations in singular cohomology
to computations in simplicial cohomology, or explicitly describe a bounded singular cocycle representing the Euler class. 
Here we describe the approach to
the second strategy developed by Ivanov and Turaev in~\cite{IvTu} (we refer the reader also to \cite{BuMo} for stronger results in this direction).

Finally, we show how the study of the bounded Euler class of flat linear bundles may be exploited to get partial results towards the Chern conjecture, which predicts that the Euler characteristic
of a closed affine manifold should vanish.

\section{Ivanov-Turaev cocycle}
Let $\pi\colon E\to M$ be a flat linear $n$-sphere bundle, and recall that $E=S(V)$ for some
flat rank-$(n+1)$ vector bundle. We will prove that, just as in the case of flat topological circle bundles, the Euler class of $E$ admits a bounded representative, which is equal to the pull-back of a 
``universal'' bounded Euler class 
in the bounded cohomology of the structure group of the bundle. This section is devoted to the construction of such an element
$[\E]_b\in H^{n+1}_b(\GL^+(n+1,\R),\R)$, as described by Ivanov and Turaev in~\cite{IvTu}.

Recall that a subset $\Omega\subseteq \R^{n+1}$ is \emph{in general position} if, for every subset $\Omega_k$ of $\Omega$ with
exactly $k$ elements, $k\leq n+2$, the dimension of the smallest affine subspace
of $\R^{n+1}$ containing  $\Omega_k$ is equal to $k-1$.

\begin{defn}
Let $N$ be an integer (in fact, we will be interested only in the case $N\leq n+1$, and mainly in the case $N=n+1$).
 An $(N+1)$-tuple $(v_0,\ldots,v_{N})\in (\R^{n+1})^{N+1}$ is \emph{generic} if the following condition holds: for every
 $(\varepsilon_0,\ldots,\varepsilon_{N})\in \{\pm 1\}^{N+1}$,
the set $\{0,\varepsilon_0 v_0,\ldots,\varepsilon_{N}v_{N}\}$ 
is in general position
(in particular, the set $\{0,\varepsilon_0 v_0,\ldots,\varepsilon_{N}v_{N}\}$  consists of $N+2$ distinct points).
\end{defn}

Point (6) of the following lemma will prove useful to get representatives of the Euler class of small norm, as suggested
in~\cite{Smillie}.

\begin{lemma}\label{Smillie:lemma}
Let $(v_0,\ldots,v_{n+1})$ be a generic $(n+2)$-tuple in $\R^{n+1}$.
%distinct points in $\R^{n+1}\setminus \{0\}$.
%and suppose that, for any 
For every $I=(\varepsilon_0,\ldots,\varepsilon_{n+1})\in \{\pm 1\}^{n+2}$,
%the set $\{0,\varepsilon_0 v_0,\ldots,\varepsilon_{n+1} v_{n+1}\}$  is in general position.
%such that the following conditions hold: for any $I=(\varepsilon_0,\ldots,\varepsilon_{n+1})\in \{\pm 1\}^{n+2}$,
%\begin{itemize}
% \item[(a)] $0$ does no
%set of the form $\{\varepsilon_0 v_0,\ldots,\widehat{\varepsilon_i v_i},\ldots,\varepsilon_{n+1} v_{n+1}\}$,
%\item[(b)] the points $\varepsilon_0v_0,\ldots,\varepsilon_{n+1}v_{n+1}$ are affinely independent.
%\end{itemize}
let $\sigma_I\colon \Delta^{n+1}\to \R^{n+1}$
be the affine map defined by $$\sigma_I(t_0,\ldots,t_{n+1})=\sum_{i=0}^{n+1} t_i\varepsilon_iv_i\ ,
$$
and let $S\cong S^n$ be the sphere
of rays of $\R^{n+1}$. Then:
\begin{enumerate}
 \item The restriction of $\sigma_I$ to $\partial\Delta^{n+1}$ takes values in $\R^{n+1}\setminus \{0\}$, thus defining
 a map $\hat{\sigma}_I\colon \partial\Delta^{n+1}\to S$.
 \item The map $\sigma_I$ is a smooth embedding.
 \item If the image of $\sigma_I$ does not contain $0$, then $\deg \hat{\sigma}_I=0$.
 \item If the image of $\sigma_I$ contains $0$ and $\sigma_I$ is orientation-preserving, then
 $\deg \hat{\sigma}_I=1$.
 \item If the image of $\sigma_I$ contains $0$ and $\sigma_I$ is orientation-reversing, then
 $\deg \hat{\sigma}_I=-1$.
 \item %Either the image of $\sigma_I$ does not contain $0$ for every $I\in \{\pm 1\}^{n+2}$, or 
 There exist exactly two elements $I_1,I_2\in \{\pm 1\}^{n+2}$ such that the image of $\sigma_{I_j}$
contains $0$ for $j=1,2$. For such elements we have $\sigma_{I_1}(x)=-\sigma_{I_2}(x)$ for every $x\in\partial\Delta^{n+1}$.
 \end{enumerate}
\end{lemma}
\begin{proof}
 (1) and (2) follow from the fact that 
$\{0,\varepsilon_0 v_0,\ldots,\varepsilon_{n+1} v_{n+1}\}$ is in general position.
If the image of $\sigma_I$ does not contain
$0$, then $\hat{\sigma}_I$ continuously extends to a map from $\Delta^{n+1}$ to $S$, and this implies (3). Suppose now that the image of $\sigma_I$ contains $0$.
Since $\{0,\varepsilon_0 v_0,\ldots,\varepsilon_{n+1} v_{n+1}\}$ is in general position, $0$ cannot belong to the image of any face of $\Delta^{n+1}$, so the image
of $\sigma_I$ contains a neighbourhood of $0$ in $\R^n$, and $\hat{\sigma}_I\colon \partial \Delta^{n+1}\to S$
is surjective. Since $\sigma_I$ is an embedding and the image
of $\sigma_I$
is star-shaped with respect to $0$, the map $\hat{\sigma}_I\colon \partial \Delta^{n+1}\to S$
is also injective, so it is a homeomorphism. Moreover, $\hat{\sigma}_I$ is orientation-preserving if and only if $\sigma_I$ is, and this concludes the proof of (4) and (5). 

Let us prove (6). Since $v_0,\ldots,v_{n+1}$ are linearly dependent in $\R^{n+1}$, we have
$\sum_{i=0}^{n+1} \alpha_iv_i=0$ for some $(\alpha_0,\ldots,\alpha_{n+1})\in \R^{n+2}\setminus \{0\}$. 
Since the set $\{0,v_0,\ldots, v_{n+1}\}$ is in general position,
we have $\alpha_i\neq 0$ for every $i$, and, if $\sum_{i=0}^{n+1}\lambda_iv_i=0$, then
$(\lambda_0,\ldots,\lambda_{n+1})=\mu(\alpha_0,\ldots,\alpha_{n+1})$ for some $\mu\in \R$.
We set
$$
\overline{t}_i=\frac{|\alpha_i|}{\sum_{i=0}^{n+1} |\alpha_i|}\, \in\, [0,1]\, ,\qquad \varepsilon^1_i={\rm  sign}(\alpha_i)\ .
$$
Since $\sum_{i=0}^{n+1} \overline{t}_i \varepsilon^1_i v_i=0$ and $\sum_{i=0}^{n+1} \overline{t}_i=1$, 
if we set $I_1=(\varepsilon^1_0,\ldots,\varepsilon^1_{n+1})$, $I_2=(-\varepsilon^1_0,\ldots,-\varepsilon^1_{n+1})$, then
$0$ belongs to the image of $\sigma_{I_1}$ and $\sigma_{I_2}$, and $\sigma_{I_1}(x)=-\sigma_{I_2}(x)$ for every $x\in\partial\Delta^{n+1}$.

Suppose now that $\sum_{i=0}^{n+1} t_i\varepsilon_i v_i=0$, where $t_i\in [0,1]$, $\sum_{i=0}^{n+1} t_i=1$. Then there exists $\mu\in\R$
such that $t_i\varepsilon_i=\mu \overline{t}_i \varepsilon^1_i$ for every $i$. 
This readily implies that $\varepsilon_i=\varepsilon_i^1$ for every $i$ 
(if $\mu>0$), or  $\varepsilon_i=-\varepsilon_i^1$ for every $i$  (if $\mu<0$; observe that $\mu=0$ is not possible since $t_i\neq 0$ for some $i$).
In other words, $(\varepsilon_0,\ldots,\varepsilon_{n+1})$ is equal either to $I_1$ or to $I_2$, and this concludes the proof.
\end{proof}

%Let $\Lambda$ be a countable subgroup of ${\rm Gl}(n+1,\R)$. We say that an $(n+2)$-tuple 
%$(v_0,\ldots,v_{n+1})\in (\R^{n+1})^{n+2}$ is \emph{$\Lambda$-generic} if, for every
%$(\varepsilon_0,\ldots,\varepsilon_{n+1})\in \{\pm 1\}^{n+2}$ and every 
%$(g_0,\ldots,g_{n+1})\in \Lambda^{n+2}$, the set $\{\varepsilon_0g_0(v_0),\ldots,\varepsilon_{n+1}g_{n+1}(v_{n+1})\}$ consists of $n+2$ distinct points in general position.

Until the end of the section we simply denote by $G$ the group $\GL^+(n+1,\R)$, and
by $D$ the closed unit ball in $\R^{n+1}$. We will understand that $D$ is endowed with the standard
Lebesgue measure (and any product $D^k$ is endowed with the product of the Lebesgue measures of the factors).

\begin{defn}
Take $\overline{g}=(g_0,\ldots,g_{n+1})\in G^{n+2}$. We say that an $(n+2)$-tuple  $(v_0,\ldots,v_{n+1})\in D^{n+2}$
is $\overline{g}$-\emph{generic} if $(g_0v_0,\ldots,g_{n+1}v_{n+1})$ is generic.
\end{defn}

\begin{lemma}\label{full:measure:lemma}
For every $\overline{g}\in G^{n+2}$,
%Let $\Lambda$ be a countable subgroup of ${\rm Gl}(n+1,\R)$. 
the set of $\overline{g}$-generic $(n+2)$-tuples 
%$$
%\{(v_0,\ldots,v_{n+1})\in D^{n+2}\, |\, (g_0v_0,\ldots,g_{n+1}v_{n+1})\ \rm{is\ generic}\}
%$$
has full measure in $D^{n+2}$.
%$\Lambda$-generic $(n+2)$-tuples has full measure in $(\R^{n+1})^{n+2}$ (in particular, it is non-empty).
\end{lemma}
\begin{proof}
Let $\overline{g}=(g_0,\ldots,g_{n+1})$.
For every fixed $(\varepsilon_0,\ldots,\varepsilon_{n+1})\in \{\pm 1\}^{n+2}$, the subset of $(\R^{n+1})^{n+2}$ of elements $(v_0,\ldots,v_{n+1})$
such that $$\{0,\varepsilon_0g_0(v_0),\ldots,\varepsilon_{n+1}g_{n+1}(v_{n+1})\}$$ consists of $n+3$ distinct points in general position
is the non-empty complement of a real algebraic subvariety of
$(\R^{n+1})^{n+2}$, so it has full measure in $(\R^{n+1})^{n+2}$. Therefore, being the intersection of a finite number of full measure sets,  
the set of $\overline{g}$-generic $(n+2)$-tuples has itself full measure.
\end{proof}

For convenience, in this section we work mainly with cohomology with \emph{real} coefficients (but see Corollary~\ref{Zbounded}).

We are now ready to define the Euler cochain $\E\in C^{n+1}_b(G,\R)$. First of all, for every $(n+2)$-tuple
$\overline{v}=(v_0,\ldots,v_{n+1})\in (\R^{n+1})^{n+2}$ we define a value $t(\overline{v})\in \{-1,0,1\}$ as follows.
If the set $\{0,v_0,\ldots,v_{n+1}\}$ is not in general position, then $t(\overline{v})=0$.
Otherwise,
if $\sigma_{\overline{v}}\colon \Delta^{n+1}\to \R^{n+1}$ is the affine embedding with vertices $v_0,\ldots,v_{n+1}$, then:
$$
t(\overline{v})=\left\{\begin{array}{cl}
                       1 & {\rm if}\ 0\in {\rm Im}\, \sigma_{\overline{v}}\quad {\rm and}\  \sigma_{\overline{v}}\ \ {\rm is\ positively\ oriented}\\
                       -1 & {\rm if}\ 0\in {\rm Im}\, \sigma_{\overline{v}}\quad {\rm and}\  \sigma_{\overline{v}}\ \ {\rm is\ negatively\ oriented}\\
                       0 & {\rm otherwise}\ .
                      \end{array}\right.
                      $$
Then, for every $(g_0,\ldots,g_{n+1})\in G^{n+2}$ we set
$$
\E(g_0,\ldots,g_n)=\int_{D^{n+2}} t(g_0v_0,\ldots,g_{n+1}v_{n+1})\, dv_0\ldots dv_{n+1}\ .
$$
                      
\begin{lemma}\label{Ecycle}
 The element $\E\in C^{n+1}_b(G,\R)$ is a $G$-invariant alternating cocycle. Moreover:
 \begin{enumerate}
  \item $\|\E\|_\infty\leq 2^{-n-1}$, and $\E=0$ if $n$ is even.
  \item $\E(g_0,\ldots,g_{n+1})=0$ if 
  there exist $i\neq j$ such that 
  $g_i,g_j\in SO(n+1)$.
 \end{enumerate}
\end{lemma}
\begin{proof}
 The fact that $\E$ is $G$-invariant follows from the fact that, if $g\cdot\overline{v}$ is the $(n+2)$-tuple
 obtained by translating every component of $\overline{v}\in (\R^{n+1})^{n+2}$ by $g\in G$, then 
 $\sigma_{g\cdot\overline{v}}=g\circ \sigma_{\overline{v}}$, so $t(g\cdot\overline{v})=t(\overline{v})$. Moreover,
 $\E$ is alternating, since $t$ is. Let us prove that $\E$ is a cocycle. So, let $\overline{g}=(g_0,\ldots,g_{n+2})\in G^{n+3}$,
 and set $\partial_i \overline{g}=(g_0,\ldots,\widehat{g}_i,\ldots,g_{n+2})$. We need to show that
 $$
 \delta \E(\overline{g})=\sum_{i=0}^{n+2} (-1)^i \E(\partial_i \overline{g})=0\ .
 $$
 Let us denote by $\Omega\in D^{n+2}$ the set of $(n+2)$-tuples which are generic for every $\partial_i\overline{g}$.
Then, $\Omega$ has full measure in $D^{n+2}$, so 
$$
\E(\partial_i\overline{g})=\int_{\Omega} t(g_0v_0,\ldots,\widehat{g_iv_i},\ldots, g_{n+2}v_{n+2})\, dv_0\ldots\widehat{dv_i}\ldots dv_{n+2}\ .
 $$
 Therefore, in order to conclude it is sufficient to show that 
 $$
T(w_0,\ldots,w_{n+2})= \sum_{i=0}^{n+2} (-1)^i\, t(w_0,\ldots,\hat{w_i},\ldots,w_{n+2})=0
 $$
 for every $(n+3)$-tuple $(w_0,\ldots,w_{n+2})$ such that
 $(w_0,\ldots,\hat{w_i},\ldots,w_{n+2})$ is generic for every $i$. Let $\lambda\colon \partial\Delta^{n+2}\to \R^{n+1}$ be the 
 map which sends the $i$-th vertex of $\Delta^{n+2}$ to $w_i$, and is affine on each face of $\Delta^{n+2}$. Then, 
 by computing the degree of $\lambda$ as the sum of the local degrees at the preimages of $0$,
 one sees 
 that $T(w_0,\ldots,w_{n+2})$ is equal to the degree of $\lambda$, which is null since $\R^{n+1}$ is non-compact. This concludes the proof that 
 $\E$ is a $G$-invariant alternating cocycle.
 
 Let us now prove (1). We fix $\overline{g}=(g_0,\ldots,g_{n+1})\in G^{n+2}$, and denote by $\Omega\subset D^{n+2}$
 the set of $\overline{g}$-generic $(n+2)$-tuples, so that
 $$
\E(g_0,\ldots,g_{n+1})=\int_{\Omega} t(g_0v_0,\ldots,g_{n+1}v_{n+1})\, dv_0\ldots dv_{n+1}\ .
$$
For every $I=(\varepsilon_0,\ldots,\varepsilon_{n+1})\in \{\pm 1\}^{n+2}$
we set 
$$
t_I(w_0,\ldots,w_{n+1})=t(\varepsilon_0w_0,\ldots,\varepsilon_{n+1}w_{n+1})\ ,
$$
$$
\E_I(g_0,\ldots,g_{n+1})=\int_{\Omega} t_I(g_0v_0,\ldots,g_{n+1}v_{n+1})\, dv_0\ldots dv_{n+1}\ .
$$
Since the map
$$
(v_0,\ldots,v_{n+1})\mapsto (\varepsilon_0v_0,\ldots,\varepsilon_{n+1}v_{n+1})
$$
is a measure-preserving automorphism of $\Omega$ and each $g_i$ is linear, we have $\E=\E_I$ for every $I$, 
so
\begin{align*}
& \E(g_0,\ldots,g_{n+1})=2^{-n-2}\sum_I \E_I(g_0,\ldots,g_{n+1})=\\ & 2^{-n-2} \int_\Omega \left(
\sum_I t_I(g_0v_0,\ldots,g_{n+1}v_{n+1})\right) \, dv_0\ldots dv_{n+1}\ .
\end{align*}
But claim (6) of Lemma~\ref{Smillie:lemma} implies that, for every $(v_0,\ldots,v_{n+1})\in\Omega$, 
there exist exactly two multiindices $I_1,I_2$ such that $t_I(g_0v_0,\ldots,g_{n+1}v_{n+1})$ does not vanish.
Moreover, since $v\mapsto -v$ is orientation-preserving (resp.~reversing) if $n$ is odd (resp.~even),
we have 
\begin{align*}
t_{I_1}(g_0v_0,\ldots,g_{n+1}v_{n+1})=t_{I_2}(g_0v_0,\ldots,g_{n+1}v_{n+1})=\pm 1 \qquad &
\textrm{if}\ n\ \textrm{is\ odd}, \\
t_{I_1}(g_0v_0,\ldots,g_{n+1}v_{n+1})=-t_{I_2}(g_0v_0,\ldots,g_{n+1}v_{n+1})\qquad & \textrm{if}\ n\ \textrm{is\ even}\ .
\end{align*}
This concludes the proof of (1).

Suppose now that there exist $g_i,g_j$ such that $g_i,g_j\in SO(n+1)$, $i\neq j$. The map $\psi\colon D^{n+2}\to D^{n+2}$
which acts on $D^{n+2}$ as $g_i^{-1}$ (resp.~$g_j^{-1}$) on the $i$-th (resp.~$j$-th) factor of $D^{n+2}$, and as the identity
on the other factors, is a measure-preserving automorphism of $D^{n+2}$. Therefore, if $g'_i=g'_j=1$ and $g'_k=g_k$ for every $k\notin \{i,j\}$, then
$$
\E(g_0,\ldots,g_{n+1})=\E(g_0',\ldots,g_{n+1}')=0\ ,
$$
where the last equality is due to the fact that $\E$ is alternating, and $g'_i=g'_j$.
\end{proof}

\begin{comment}
Let $\Lambda$ be a countable subgroup of ${\rm Gl}(n+1,\R)$, and fix a $\Lambda$-generic
$(n+2)$-tuple $(v_0,\ldots,v_{n+1})$. We now define a cochain $\eul\in C^{n+1}_b(\Lambda,\R)$ as
follows. For every $\overline{g}=(g_0,\ldots,g_{n+1})\in \Lambda^{n+2}$, $I=(\varepsilon_0,\ldots,\varepsilon_{n+1})\in \{\pm 1\}^{n+2}$,
we denote by $\sigma(\overline{g},I)\colon \Delta^{n+1}\to \R^{n+1}$
the singular simplex given by
$$
\sigma(\overline{g},I)(t_0,\ldots,t_{n+1})=\sum_{i=0}^{n+1} t_i\varepsilon_i g_i(v_i)\ .
$$
By Lemma~\ref{Smillie:lemma}, the singular simplex $\sigma(\overline{g},I)$ induces a map
$\widehat{\sigma}(\overline{g},I)\colon\partial\Delta^{n+1}\to S$, and we set 
\begin{align*}
\eul_I(g_0,\ldots,g_{n+1})&=\deg \widehat{\sigma}(\overline{g},I)\, ,\\
\eul(g_0,\ldots,g_{n+1})&=2^{-n-2} \sum_I\eul_I(g_0,\ldots,g_{n+1})\ .
\end{align*}

It is easy to show that both $\eul_I$ and $\eul$ are cocycles, but this fact will not be used in this monograph. 
Moreover, it follows by the very definitions that both $\eul_I$ and $\eul$ are
$\Lambda$-invariant, so Lemma~\ref{Smillie:lemma} implies the following:

\begin{lemma}\label{smillie:norm}
 We have $\eul_I\in C^{n+1}_b(\Lambda,\matZ)^\Lambda$, $\eul\in C^{n+1}_b(\Lambda,\R)^\Lambda$, and
 $$
 \|\eul_I\|_\infty \leq 1\, ,\qquad
 \|\eul\|_\infty \leq 2^{-n-1}\ .
 $$
\end{lemma}
\qed

\end{comment}

\begin{defn}
 The $(n+1)$-dimensional \emph{bounded Euler class} is the element
 $$
 [\E]_b\in H^{n+1}_b(\GL(n+1,\R),\R)\ .
 $$
\end{defn}

By construction, thanks to Lemma~\ref{Ecycle} the $(n+1)$-dimensional bounded Euler class satisfies
$$
\|[\E]_b\|_\infty\leq 2^{-n-1}\ .
$$

\section{Representing cycles via simplicial cycles}
As mentioned above, in order to prove that the Euler class of a linear sphere bundle
is bounded, it is convenient to make simplicial chains come into play. To this aim we introduce
a machinery which is very well-suited to describe singular cycles as push-forwards of
simplicial cycles. We refer the reader to~\cite[Section 5.1]{Loeh:IMRN} for an alternative description of this construction.

Let $M$ be a topological space, and let $\overline{z}=\sum_{i=1}^k a_i \overline{\sigma}_i$ be an $n$-dimensional cycle in
$C_n(M,\R)$. 
%For technical reasons that will be clear later, we first take the double barycentric subdivision $\overline{z}=\sum_{i=1}^k a_i \sigma_i$ of $z$,
%where 
%$\sigma_i$ is a singular $n$-simplex on $M$, and $a_i\in\R$
%for every $i$. 
Also assume that $\overline{\sigma}_i\neq \overline{\sigma}_j$ for  $i\neq j$.
We now construct a $\Delta$-complex $\overline{P}$ associated to $\overline{z}$ ($\Delta$-complexes slightly generalize simplicial complexes; the interested reader can find the definition and the basic properties of $\Delta$-complexes
in~\cite[Chapter 2]{Hatcher}; however, no prerequisite on $\Delta$-complexes is needed in order to understand what follows). 

Let us consider $k$ distinct copies
$\Delta^n_1,\ldots,\Delta^n_k$ of the standard $n$-simplex $\Delta^n$.
For every $i$ we fix an identification between $\Delta^n_i$ and $\Delta^n$, so that we may consider
$\sigma_i$ as defined on $\Delta^n_i$.
For every $i=1,\ldots,k$, $j=0,\ldots, n$, we denote by $F^i_j$ the $j$-th face of $\Delta^n_i$,
and by $\partial^i_j\colon \Delta^{n-1}\to F^i_j\subseteq \Delta_i^n$ the usual face inclusion.
We say that the faces $F^i_j$ and $F^{i'}_{j'}$ are \emph{equivalent} if  
$\overline{\sigma}_i|_{F^i_j}=\overline{\sigma}_{i'}|_{F^{i'}_{j'}}$, or, more formally, if $\bb^i_j\circ \overline{\sigma}_i=\bb^{i'}_{j'}\circ\overline{\sigma}_{i'}$. 
We now define a $\Delta$-complex $P$ as follows.
The simplices of $\overline P$ are $\Delta^n_1,\ldots,\Delta^n_k$, and, 
if $F_i^j$, $F_{i'}^{j'}$ are equivalent, then we identify them via the affine diffeomorphism
$\partial_{i'}^{j'}\circ (\partial_i^j)^{-1}\colon F_i^j\to F_{i'}^{j'}$. 
The only phenomenon that could prevent $\overline P$ to be a simplicial complex is the fact that different simplices may share more than one face, and that distinct faces of the same simplex may be identified to each other.
In order to exploit the properties of genuine simplicial complexes, we turn $\overline P$ into a simplicial complex just by taking the second barycentric subdivision $P$ of $\overline{P}$.

By construction, the maps $\overline{\sigma}_1,\ldots,\overline{\sigma}_k$ glue up to a well-defined continuous map
$f\colon |\overline{P}|\to M$, where $|\overline{P}|$ is the topological realization of $\overline{P}$. Moreover, 
for every $i=1,\ldots, k$, let  $\hat{\sigma}_i\colon \Delta^n\to \overline{P}$ be the simplicial simplex obtained by composing
the identification $\Delta^n\cong \Delta^n_i$ with the quotient map with values in $|\overline{P}|$,
and let us consider the simplicial chain $z_{\overline{P}}=\sum_{i=1}^k a_i \hat{\sigma}_i$. By construction, $z_{\overline{P}}$ is a cycle, and
the push-forward of $z_{\overline{P}}$ via $f$ is equal to $\overline{z}$. If we denote by $z_P$ the second barycentric subdivision of $z_{\overline{P}}$, then $z_P$ is a real simplicial
cycle on the simplicial complex $P$. Moreover, $f$ can also be considered as a map from $|P|$ to $M$, and the push-forward of $z_P$ via $f$ coincides with the second barycentric subdivision
of our initial cycle $\overline{z}$, thus lying in the same homology class as $\overline{z}$.
We have thus proved the following:

\begin{lemma}\label{basic:lemma}
Take an element $\alpha\in H_n(M,\R)$. Then, there exist a finite simplicial complex $P$, a real simplicial $n$-cycle $z_P$ on $P$ and a continuous
map $f\colon |P|\to M$ such that $H_n(f)([z_P])=\alpha$ in $H_n(M,\R)$.
\end{lemma}

\section{The bounded Euler class of a flat linear sphere bundle}
Let now $\pi\colon E\to M$ be a flat linear $n$-sphere bundle, and set $\G=\pi_1(M)$.
We denote by 
$$
r^M_\bullet\colon C_\bullet(\widetilde{M},\R)\to C_\bullet(\G,\R)\ ,
$$
$$r_M^\bullet \colon C^\bullet_b(\G,\R)^\G\to C^\bullet_b(\widetilde{M},\R)^\G=C^\bullet_b(M,\R)$$
the classifying maps defined in Lemma~\ref{normnon}.
We know from Proposition~\ref{flat:rep} (and Remark~\ref{flatvector})
that $E$ is (linearly) isomorphic to $E_\rho$ for a representation $\rho\colon \pi_1(M)\to \GL^+(n+1,\R)$.
We define the real bounded Euler class $e_b^\R(E)\in H^{n+1}_b(M,\R)$ of $E$ as the pull-back of $[\E]_b$ via $\rho$
(and the classifying map), i.e.~we set
$$
e_b^\R(E)=H^{n+1}_b(r_M^\bullet)\circ H^{n+1}_b(\rho^\bullet)([\E]_b)\ \in \ H^{n+1}_b(M,\R)\ .
$$
We are now ready to show that the real bounded Euler class of $E$ is mapped by the comparison map
onto the Euler class of $E$. We first deal with the simplicial case, and then we reduce the general case to the simplicial one.

\begin{prop}\label{simplicialcase}
 Suppose that $M$ is a simplicial complex, and let $z$ be an $(n+1)$-dimensional simplicial cycle on $M$. 
Then
$$r_M^{n+1}(\rho^{n+1}(\eul))(z)=\langle e^\R(E),[z]\rangle\ .$$
 \end{prop}
\begin{proof}
Let $q_1,\ldots,q_N$ be the vertices of $M$, and for every $i=1,\ldots,N$
let us take an element $v_i\in S^n$. After choosing a suitable trivialization of $E$
over the $q_i$'s, we will exploit the $v_i$ to define a section over the $q_i$'s. Then, 
we will affinely extend such a section over the $n$-skeleton of $M$, and use the resulting
section to compute the Euler class of $E$. In order to do so, we need to be sure
that the resulting section does not vanish on the $n$-skeleton of $M$, and to this aim we have to carefully
choose the $v_i$'s we start with. 

We fix points $x_0\in M$ and $\widetilde{x}_0\in p^{-1}(x_0)\in\widetilde{M}$, and
we identify $\G=\pi_1(M,x_0)$ with the group of the covering automorphisms of
$\widetilde{M}$ so that the projection on $M$ of any path in $\widetilde{M}$ starting at $\widetilde{x}_0$
and ending at $g(\widetilde{x}_0)$ lies in the homotopy class corresponding to $g$.
We also choose a set of representatives 
$R$ for the action of $\G=\pi_1(M)$ on $\widetilde{M}$ containing $\widetilde{x}_0$, and we denote by
 $j\colon \widetilde{M}\times S^n\to E_\rho=E$ the quotient map with respect to the diagonal action of $\G$ on
$\widetilde{M}\times S^n$.

Let us now fix 
an $N$-tuple $(v_1,\ldots,v_N)\in D^N$. If $v_i\neq 0$ for every $i$, such an $N$-tuple gives rise to a section $s^{(0)}$ of $E$ over the $0$-skeleton
$M^{(0)}$ which is defined as follows: if $\widetilde{q}_i$ is the unique lift of $q_i$ in $R$,
then $s^{(0)}(q_i)=j(\widetilde{q}_i,v_i/\|v_i\|)$. We are now going to affinely extend such a section over the $n$-skeleton
of $M$. However, as mentioned above, this can be done only under some additional hypothesis on 
$(v_1,\ldots,v_N)$, which we are now going to describe.

For every $k$-simplex $\tau$ of $M$, $k\leq n+1$, we denote by $\widetilde{\tau}\subseteq \widetilde{M}$ the 
lift of $\tau$ having the first vertex in $R$ (recall that a total order on the set of vertices of $M$ has been fixed from the very beginning). Let now $q_{j_i}$ be the $i$-th vertex of $\tau$. Then there exists $g_i\in\G$
such that the $i$-th vertex $\widetilde{\tau}_i$ of $\widetilde{\tau}$ is equal to $g_i(\widetilde{q}_{j_i})$. 
Moreover, we may choose the classifying map $r_\bullet^M$ in such a way that
$$
r^{k}_M(\widetilde{\tau})=(g_0,\ldots,g_{k})
$$
(see Lemma~\ref{normnon}).
We now say that
$(v_1,\ldots,v_N)$ is \emph{$\tau$-generic} if the $(k+1)$-tuple 
$$
(\rho(g_0)(v_{j_0})=v_{j_0},\rho(g_1)(v_{j_1}),\ldots,\rho(g_n)(v_{j_k}))\in D^{k+1}
$$
is generic.
If this is the case and $k\leq n$, then the composition 
$$
\xymatrix{
\tau \ar[rr]^-{l} & & \widetilde{\tau} \ar[rr]^-{{\rm Id}\times a} & & \widetilde{M}\times S^n \ar[rr]^-j & & E\ ,
}
$$
where $l$ is just the lifting map, and
$$
a(t_0\widetilde{\tau}_0+\ldots+t_k\widetilde{\tau}_k)=
\frac{t_0 \rho(g_0)(v_{j_0})+\ldots
+t_k\rho(g_k)(v_{j_k})}{\|t_0 \rho(g_0)(v_{j_0})+\ldots
+t_k\rho(g_k)(v_{j_k})\|}
$$
is a well-defined section of $E$ over $\tau$, which extends $s^{(0)}$. 

We say that the $N$-tuple $(v_1,\ldots,v_N)$ is \emph{generic} if it is $\tau$-generic for
every $k$-simplex $\tau$ of $M$, $k\leq n+1$, and we denote by $\Omega\subseteq D^{N}$ the subset of generic
$N$-tuples. A similar argument to the proof of Lemma~\ref{full:measure:lemma}
implies that there exist Borel subsets $\Omega_i\subseteq D$, $i=1,\ldots,N$, such that each $\Omega_i$ has full measure
in $D$ and
$$
\Omega_1\times\ldots\times \Omega_N\subseteq \Omega
$$
(in particular, $\Omega$ has full measure in $D^N$). In fact, one can set $\Omega_1=D$, and define inductively
$\Omega_{i+1}$ by imposing that $v\in \Omega_{i+1}$ if and only if the following condition holds: for every 
$(v_1,\ldots,v_i)\in \Omega_1\times\ldots\times \Omega_i$, 
the $(i+1)$-tuple $(v_1,\ldots,v_i,v)$ is $\tau$-generic for every $k$-simplex $\tau$ of $M$, $k\leq n+1$, with vertices in $\{q_1,\ldots,q_{i+1}\}$ (observe that
it makes sense to require that an $(i+1)$-tuple is $\tau$-generic, 
provided that the vertices of $\tau$ are contained in $\{q_1,\ldots,q_{i+1}\}$). It is not difficult to check that indeed $\Omega_i$
has full measure in $D$ for every $i$.

Let us now pick an element $\overline{v}=(v_1,\ldots,v_N)\in \Omega$. Such an element defines a 
family of compatible sections over all the simplicial $k$-simplices of $M$, $k\leq n$,
%section $s_\tau\colon \tau\to E|_{\tau}$
%on every $n$-simplex $\tau$ of $M$. 
%Moreover, such sections 
%glue up into a globally well defined
%section $s^{(n)}\colon M^{(n)}\to E|_{M^{(n)}}$, which 
which can be extended in turn to a
family of compatible sections over all the singular $n$-simplices in $M$. 
If we denote by $\varphi_{\overline{v}}$ the  Euler cocycle associated to this  family of compatible sections, then 
Lemma~\ref{Smillie:lemma} implies that, for every $(n+1)$-dimensional simplex $\tau$ of $M$, we have
$$
\varphi_{\overline{v}}(\tau)=t(\rho(g_0)v_{j_0},\ldots,\rho(g_{n+1})v_{j_{n+1}})\ ,
$$
where  $q_{j_i}$ is the $i$-th vertex of $\tau$, and 
$$
(g_0,\ldots,g_{n+1})=r_{n+1}^M(\widetilde{\tau})\ .
$$
Therefore, if $z$ is a fixed real simplicial $(n+1)$-cycle, then
$$
\langle e^\R(E),[z]\rangle=\varphi_{\overline{v}}(z)\ ,
$$
so
\begin{equation}\label{z1eq}
\langle e^\R(E),[z]\rangle=\int_{\overline{v}\in \Omega_1\times\ldots\times \Omega_N} \varphi_{\overline{v}}(z)\ ,
\end{equation}
(where we used that the product of the $\Omega_i$'s has full, i.e.~unitary, measure in $D^N$).

On the other hand, for every $(n+1)$-simplex $\tau$ we have
\begin{align*}
\int_{\overline{v}\in \Omega_1\times\ldots\times \Omega_N}\varphi_{\overline{v}}(\tau)&=\int_{\overline{v}\in \Omega_1\times\ldots\times \Omega_N}
t(\rho(g_0)v_{j_0},\ldots,\rho(g_{n+1})v_{j_{n+1}})\, d\overline{v}\\ &=
\int_{v_{j_i}\in \Omega_{j_i}} t(\rho(g_0)v_{j_0},\ldots,\rho_(g_{n+1})v_{j_{n+1}})\, dv_{j_0}\ldots dv_{j_{n+1}}\\ &=
\int_{D^{n+2}} t(\rho(g_0)v_{0},\ldots,\rho(g_{n+1})v_{{n+1}})\, dv_0\ldots dv_{n+1}\\ &=r^{n+1}_M(\rho^{n+1}(\eul))(\tau)\ ,
\end{align*}
(where we used again that $\Omega_i$ has full (whence unitary) measure in $D$), so by linearity
\begin{equation}\label{z2eq}
\int_{\overline{v}\in \Omega_1\times\ldots\times \Omega_N}\varphi_{\overline{v}}(z)=
r^{n+1}_M(\rho^{n+1}(\eul))(z)\ .
\end{equation}
Putting together equations~\eqref{z1eq} and~\eqref{z2eq} we finally get that
$\langle e^\R(E),[z]\rangle=\varphi_{\overline{v}}(z)=r^{n+1}_M(\rho^{n+1}(\eul))(z)$, whence the conclusion.
\end{proof}

We are now ready to prove that, via the classifying map, the group cochain $\eul$ indeed provides a representative
of the Euler class of $E$, even in the case when $M$ is not assumed to be a simplicial complex:

\begin{prop}
  The cochain $r_M^{n+1}(\rho^{n+1}(\eul))\in C^{n+1}_b(M,\R)$ 
 is a representative of the real Euler class of $E$.
\end{prop}
\begin{proof}
 By the Universal Coefficient Theorem, it is sufficient to show that 
 \begin{align}\label{quellacheserve}
 r_M^{n+1}(\rho^{n+1}(\eul))(z)=\langle e^\R(E),[z]\rangle
 \end{align}
 for every singular cycle $z\in C_{n+1}(M,\R)$. So, let us fix such a cycle, and take a finite simplicial complex $P$, a real simplicial $(n+1)$-cycle $z_P$ on $P$ and a continuous
map $f\colon |P|\to M$ such that $H_{n+1}(f)([z_P])=[z]$ in $H_{n+1}(M,\R)$ (see Lemma~\ref{basic:lemma}). 
Since $z$ is homologous to $C_{n+1}(f)(z_P)$, we have 
\begin{equation}\label{zP}
r_M^{n+1}(\rho^{n+1}(\eul))(z)=C^{n+1}(f)(r_M^{n+1}(\rho^{n+1}(\eul))) (z_P)\ .
\end{equation}
Moreover, one may choose a classifying map
$$
r_P^\bullet\colon C^\bullet(\pi_1(|P|),\R)^{\pi_1(|P|)}\to C^\bullet(\widetilde{|P|},\R)^{\pi_1(|P|)}=C^\bullet(|P|,\R)
$$
in such a way that the diagram
$$
\xymatrix{
C^\bullet(\G,\R)^\G \ar[d]^{r^\bullet_{M}} \ar[rr]^-{f_*^{\bullet}} & &  C^\bullet(\pi_1(|P|),\R)^{\pi_1(|P|)}  \ar[d]^{r^\bullet_P}\\
C^\bullet(|M|,\R)  \ar[rr]^-{C^{\bullet}(f)} & & C^\bullet(|P|,\R)
}
$$
commutes, where we denote by $f_*$ the map induced by $f$ on fundamental groups.
%, and by $f_*^{n+1}$ the map
%induced by $f_*$ on group cochains. 
Therefore, if we denote by $\rho'\colon \pi_1(|P|)\to G$ the composition
$\rho'=\rho\circ f_*$, then 
$$
C^{n+1}(f)(r_M^{n+1}(\rho^{n+1}(\eul)))=r_P^{n+1}(f^{n+1}_*(\rho^{n+1}(\eul)))=
r_P^{n+1}((\rho')^{n+1}(\eul))\ .
$$
%\begin{align*}
%C^{n+1}(f)(r_M^{n+1}(\rho^{n+1}(\eul)))(z_P)&=r_P^{n+1}(f^{n+1}_*(\rho^{n+1}(\eul)))(z_P)
%\\ &=
%r_P^{n+1}((\rho')^{n+1}(\eul))(z_P)\ .
%\end{align*}
Putting this equality together with \eqref{zP} we obtain that
\begin{equation}\label{zP2}
 r_M^{n+1}(\rho^{n+1}(\eul))(z)=r_P^{n+1}((\rho')^{n+1}(\eul))(z_P)\ .
\end{equation}

On the other hand, the bundle $E$ pulls back to a flat linear $S^n$-bundle
$f^*E$ on $|P|$, and it readily follows from the definitions that 
$f^*E$ is isomorphic to the sphere bundle associated to the representation
$\rho'\colon \pi_1(|P|)\to G$ just introduced.
Therefore, Proposition~\ref{simplicialcase} (applied to the case $M=|P|$) implies that
\begin{equation}\label{quella3}
 r_P^{n+1}((\rho')^{n+1}(\eul))(z_P)=\langle e^\R(f^*E),[z_P]\rangle\ ,
\end{equation}
while 
Lemma~\ref{Euler:functorial} (the statement with integral coefficients implies the one with real coefficients)
gives
\begin{equation}\label{quella2}
\langle e^\R(f^*E),[z_P]\rangle=
\langle H^{n+1}(f)(e^\R(E)),[z_P]\rangle= 
\langle e^\R(E),[z]\rangle\ .
\end{equation}
Puttin together~\eqref{zP2}, \eqref{quella3} and \eqref{quella2} we finally obtain 
the desired equality~\eqref{quellacheserve}, whence the conclusion.
\end{proof}

Putting together the previous proposition and Lemma~\ref{Ecycle} 
we obtain the following result, which generalizes  Corollary~\ref{norm12} to higher dimensions:

\begin{thm}\label{eullin}
 Let $\pi\colon E\to M$ be a flat linear $n$-sphere bundle, and let $\rho\colon \G\to \GL^+(n+1,\R)$ be the 
 associated representation, where $\G=\pi_1(M)$. Then the real Euler class of $E$ is given by 
 $$
e_\R(E)=c(e^\R_b(E))\ ,
 $$
 where $c\colon H^{n+1}_b(\G,\R)\to H^{n+1}(\G,\R)$ is the comparison map. Therefore, 
 $$
 \|e^\R(E)\|_\infty \leq 2^{-n-1}\ .
 $$
\end{thm}

Theorem~\ref{eullin} has the following immediate corollary, which provides a higher dimensional analogue of Milnor-Wood inequalities:

\begin{thm}\label{mw:higher}
Let $E$ be a flat vector bundle of rank $n$ over 
a closed oriented $n$-manifold $M$. Then
$$
|e(E)|\leq \frac{\|M\|} {2^{n}} \ .
$$
\end{thm}
\begin{proof}
 By definition, we have
 $$
 |e(E)|=|\langle e^\R(E),[M]\rangle| \leq \|e^\R(E)\|_\infty \|[M]\|_1 \leq 2^{-n}\|M\|\ .
 $$
\end{proof}

We have shown that the Euler class of a flat linear sphere bundle is bounded. By Proposition~\ref{surjiffsurj},
this implies that such a class is bounded even as an \emph{integral} class:

\begin{cor}\label{Zbounded}
Let $\pi\colon E\to M$ be a flat linear $n$-sphere bundle. Then the integral Euler class
$e(E)$ admits a bounded representative.
\end{cor}

We conclude the section with the following:

\begin{conj}\label{flattop}
 Let $\pi\colon E\to M$ be a \emph{topologically} flat sphere bundle. Then $\eu(E)$ admits a bounded representative.
\end{conj}

\section{The Chern conjecture}
By Bieberbach Theorem, closed flat manifolds are finitely covered by tori, so they have vanishing Euler characteristic.
Also the simplicial volume of closed flat manifolds vanishes (see Section~\ref{flat:simplicial}), and the vanishing of the Euler characteristic may be  interpreted also as a consequence
of this fact.
Indeed, a flat $n$-manifold $M$ admits an atlas whose transition maps
are (restrictions of) Euclidean isometries. Such an atlas induces a flat structure on the tangent bundle $TM$ of $M$, so by Theorem~\ref{mw:higher} we have
$$
|\chi(M)|=|e(TM)|\leq \frac{\|M\|} {2^{n}}=0\ .
$$
It is very natural to look for an  extension of this result to  closed \emph{affine} manifolds, that are manifolds which admit an atlas whose transition maps are (restrictions of) affine isomorphisms.
In fact, the tangent bundle of affine manifolds is obviously flat, and it is reasonable to expect that affine manifolds could share a lot of properties with flat ones. Surprisingly enough,
the attempt to generalize the
above result to affine manifolds resulted in the formulation of a long-standing open conjecture:

\begin{conj}[Chern conjecture]\label{Chern:conj}
 Let $M$ be a closed affine manifold. Then $\chi(M)=0$.
\end{conj}

Several particular cases of the Chern conjecture have been settled by now (see Section~\ref{further:chern} for a brief discussion of this topic). For example,
thanks to Kostant and Sullivan~\cite{KS}, the Chern conjecture is known to hold for closed \emph{complete} affine manifolds (also known as \emph{affine space-forms}), i.e.~for 
compact quotients of the Euclidean space $\R^n$ with respect to the action of free and proper discontinuous groups of affine isomorphisms of $\R^n$. 

One may even wonder whether the hypothesis of being affine could be relaxed to the weaker condition of being  \emph{tangentially flat}, i.e.~of having a flat tangent bundle. 
We have  observed above that any affine manifold is tangentially flat, but the converse
is not quite true: indeed, affine manifolds can be characterized as those manifolds whose tangent bundle admits a linear flat and \emph{symmetric} connection.
The question whether the Euler characteristic of a closed tangentially flat manifold is necessarily zero was explicitly raised by Milnor~\cite{Milnor} and Kamber and Tondeur~\cite[p. 47]{KT}, and
according to Hirsch and Thurston~\cite{HT} the answer was commonly expected to be positive. Nevertheless, examples of tangentially flat manifolds with non-vanishing Euler characteristic
were constructed by Smillie in~\cite{Smillie:counter} in every even dimension bigger than $2$. None of Smillie's manifolds is aspherical, and indeed
another open conjecture is the following:

\begin{conj}\label{Chern2:conj}
The Euler characteristic of a closed aspherical tangentially flat manifold vanishes.
\end{conj}

Milnor-Wood's inequalities may be exploited to prove the Chern conjecture in dimension two:

\begin{prop}
 The only closed orientable surface admitting an affine structure is the torus. In particular, 
the Chern conjecture holds in dimension 2.
 \end{prop}
\begin{proof}
Let $S$ be a closed affine surface. Since the tangent bundle of an affine manifold is flat, putting together Proposition~\ref{tangentchi}
 and Theorem~\ref{Milnor:thm} we get
 $$
 |\chi(S)|=|e(TS)|\leq \frac{|\chi_-(S)|}{2}\ ,
 $$
 which readily implies $\chi(S)=0$.
\end{proof}

The boundedness of the Euler class in every dimension is the key ingredient of a simple proof of the fact
that the Chern conjecture holds for manifolds with an amenable fundamental group. The following straightforward argument first appeared in~\cite{BePe}
(see also~\cite{HT} for a different proof).

\begin{thm}
 Let $M$ be an affine manifold with an amenable fundamental group. Then $\chi(M)=0$.
\end{thm}
\begin{proof}
Recall that the simplicial volume of $M$ vanishes by Corollary~\ref{amvan}.
Therefore, if $n=\dim M$, then by Theorem~\ref{mw:higher} we have
$$
|e(E)|\leq \frac{\|M\|} {2^{n}}=0 
$$
for any flat vector bundle $E$ of rank $(n+1)$. In particular, we obtain
$$
\chi(M)=e(TM)=0\ ,
$$
where the first equality is due to
Proposition~\ref{tangentchi}, while the second one to the fact that the tangent bundle of any affine manifold is flat.
\end{proof}

In a recent preprint~\cite{BCL}, Bucher, Connell and Lafont formulated the following conjecture:

\begin{conj}\label{BDL:conj}
 If $M$ is a closed manifold supporting an affine structure, then $\|M\|=0$.
\end{conj}

Recall that the tangent bundle of any affine manifold admits a flat structure.
Therefore, if $n=\dim M$, then by Theorem~\ref{mw:higher} we have
$$
|\chi(M)|=|e(TM)|\leq \frac{\|M\|} {2^{n}}\ .
$$
As a consequence, Conjecture~\ref{BDL:conj} would imply the Chern conjecture.

\section{Further readings}\label{further:chern}

\subsection*{The norm of the Euler class}
We have seen that Ivanov-Turaev's $n$-dimensional Euler cocycle satisfies the inequality
$$
\|\E\| \leq 2^{-n}\ .
$$
Of course, it is of interest to compute the exact seminorm of the bounded cohomology class defined by $\eul$, or even of the induced ordinary cohomology class.
Both tasks were achieved by Bucher and Monod in~\cite{BuMo}: after constructing a new bounded representative of the Euler class having the same norm 
as Ivanov-Turaev's one, the authors showed
that, in even dimension $n$, both the bounded and the unbounded Euler class have norm equal to $2^{-n}$. In particular,
Bucher-Monod's and Ivanov-Turaev's cocycles both have the smallest possible norm.

\subsection*{A canonical bounded Euler class}
Another natural question is whether the bounded cohomology class defined by Ivanov-Turaev's cocycle (or by the cocycle constructed by Bucher and Monod) is in any sense canonical.
It is worth mentioning that the space $H^{n}_b(\GL^+(n,\R),\R)$ has not been computed yet: in particular, it is not known whether the classical Euler class is represented by a unique bounded cohomology
class (i.e., whether the comparison map $H^{n}_b(\GL^+(n,\R),\R)\to H^{n}(\GL^+(n,\R),\R)$ is injective or not). Nevertheless, Bucher and Monod proved in~\cite{BuMo}
that the subspace of \emph{antisymmetric} classes in $H^{n}_b(\GL^+(n,\R),\R)$ is one-dimensional, where a class is said to be antisymmetric if conjugacy by an orientation-reversing
automorphism of $\R^n$ changes its sign. 
Since both Bucher-Monod's and Ivanov-Turaev's cocycles are antysimmetric, this implies that
the classical Euler class admits a unique bounded antisymmetric representative, whose norm coincides with the norm of the Euler class. As usual,
the existence of such a \emph{canonical} bounded Euler class may lead to a refinement of the invariants determined by the classical Euler class.

\subsection*{The Chern conjecture}
As observed above, the boundedness of the Euler class of flat bundles seems to provide a promising ingredient for an approach to the Chern conjecture. 
Nonetheless, Conjectures~\ref{Chern:conj} and~\ref{Chern2:conj} seem still quite far from being settled.  Many
attemps have been made to generalize Milnor-Wood inequality to other dimensions, but very
little progress has been made until very recently. In~\cite{BuGe} Bucher and Gelander
proved that the Euler class of flat bundles over closed oriented manifolds whose universal cover
is isometric to the product of two hyperbolic planes  satisfies an inequality of Milnor-Wood type, thus confirming both
conjectures for all manifolds which are locally isometric to a product of surfaces of constant
curvature. We refer the reader to~\cite{BuGe2} for other results in this direction.

With somewhat different methods, Bucher, Connell and Lafont recently proved that the Chern conjecture holds for any aspherical affine manifold for which the holonomy 
is injective, and contains at least one non-trivial translation. Indeed, they proved that the simplicial volume of any such manifold vanishes, and this suffices to get that also the Euler characteristic does.

Finally, we would like to mention that, following a strategy which is closer to the original circle of ideas that lead to the formulation of the Chern conjecture, Klingler has recently  proved that
the Chern conjecture holds for every special affine manifold~\cite{Kli} (an affine manifold is special if its holonomy representation has values in the subgroup of affinities with linear part
of determinant $1$). 

\begin{comment}

While it has been proven that the boundedness of the Euler class of flat bundles generalizes
to all dimensions, Conjectures 1 and 2 remain elusive. Indeed one needs explicit inequalities in
order to determine whether the tangent bundle can be ruled out from the flat ones or not. Many
attemps have been made to generalize Milnor-Wood inequality to other dimensions, but very
little progress has been made until very recently. In 2011 Michelle Bucher and Tsachik Gelander
proved that the Euler class of flat bundles over closed oriented manifolds whose universal cover
is isometric to (H2 )n satisfies an inequality of Milnor-Wood type [8], thus confirming both
conjectures for all manifolds which are locally isometric to a product of surfaces of constant
curvature. Their work, which takes up the largest part of our exposition, uses the theory of
bounded cohomology developed by Mikha ̈ Gromov in 1982 [17] and some deep results about
ıl
the super-rigidity of lattices in semisimple Lie groups due to Gregori Margulis [23].
\end{comment}

\bibliographystyle{amsalpha}
\bibliography{biblionote}

\end{document}

%% file: generators.pstex_t
\begin{picture}(0,0)%
\includegraphics{generators.pstex}%
\end{picture}%
\setlength{\unitlength}{4144sp}%
\begingroup\makeatletter\ifx\SetFigFont\undefined%
\gdef\SetFigFont#1#2#3#4#5{%
  \reset@font\fontsize{#1}{#2pt}%
  \fontfamily{#3}\fontseries{#4}\fontshape{#5}%
  \selectfont}%
\fi\endgroup%
\begin{picture}(4770,4668)(1873,-6631)
\put(5840,-5969){\makebox(0,0)[lb]{\smash{{\SetFigFont{12}{14.4}{\rmdefault}{\mddefault}{\updefault}{$b_1$}%
}}}}
\put(5975,-2812){\makebox(0,0)[lb]{\smash{{\SetFigFont{12}{14.4}{\rmdefault}{\mddefault}{\updefault}{$b_1^{-1}$}%
}}}}
\put(6628,-4365){\makebox(0,0)[lb]{\smash{{\SetFigFont{12}{14.4}{\rmdefault}{\mddefault}{\updefault}{$a_1^{-1}$}%
}}}}
\put(2683,-2752){\makebox(0,0)[lb]{\smash{{\SetFigFont{12}{14.4}{\rmdefault}{\mddefault}{\updefault}{$b_2$}%
}}}}
\put(1888,-4386){\makebox(0,0)[lb]{\smash{{\SetFigFont{12}{14.4}{\rmdefault}{\mddefault}{\updefault}{$a_2^{-1}$}%
}}}}
\put(2615,-5961){\makebox(0,0)[lb]{\smash{{\SetFigFont{12}{14.4}{\rmdefault}{\mddefault}{\updefault}{$b_2^{-1}$}%
}}}}
\put(4386,-2122){\makebox(0,0)[lb]{\smash{{\SetFigFont{12}{14.4}{\rmdefault}{\mddefault}{\updefault}{$a_2$}%
}}}}
\put(4228,-6562){\makebox(0,0)[lb]{\smash{{\SetFigFont{12}{14.4}{\rmdefault}{\mddefault}{\updefault}{$a_1$}%
}}}}
\put(3493,-3877){\makebox(0,0)[lb]{\smash{{\SetFigFont{12}{14.4}{\rmdefault}{\mddefault}{\updefault}{$c_1$}%
}}}}
\put(5128,-4356){\makebox(0,0)[lb]{\smash{{\SetFigFont{12}{14.4}{\rmdefault}{\mddefault}{\updefault}{$c_2$}%
}}}}
\end{picture}%